\documentclass[reqno,10pt]{amsart}
\usepackage{amssymb, amsmath}
\allowdisplaybreaks[4]
\usepackage[margin=2cm]{geometry} 
\usepackage{soul}
\usepackage{hyperref} 
\usepackage{graphicx} 
\usepackage{amsmath, amssymb, amsfonts, bbm, amsthm}
\usepackage{amsthm, amsmath,amsfonts,amssymb,euscript,graphics,color,slashed}
\usepackage{ulem}
\usepackage{xcolor}

\usepackage{mathrsfs}
\usepackage{comment}
\usepackage{import}
\usepackage{latexsym}
\usepackage[makeroom]{cancel}

\let\al=\alpha

\let\e=\varepsilon

\let\r=\rho

\let\f=\frac

\let\p=\psi

\let\wt=\widetilde

\def\cE{{\mathcal E}}

\def\virgp{\raise 2pt\hbox{,}}
\def\cdotpv{\raise 2pt\hbox{;}}

\def\C{\mathop{\mathbb C\kern 0pt}\nolimits}
\def\DD{\mathop{\mathbb D\kern 0pt}\nolimits}
\def\EE{\mathop{{\mathbb E \kern 0pt}}\nolimits}
\def\K{\mathop{\mathbb K\kern 0pt}\nolimits}
\def\N{\mathop{\mathbb N\kern 0pt}\nolimits}
\def\Q{\mathop{\mathbb Q\kern 0pt}\nolimits}
\def\R{\mathop{\mathbb R\kern 0pt}\nolimits}
\def\SS{\mathop{\mathbb S\kern 0pt}\nolimits}
\def\ZZ{\mathop{\mathbb Z\kern 0pt}\nolimits}
\def\TT{\mathop{\mathbb T\kern 0pt}\nolimits}
\def\P{\mathop{\mathbb P\kern 0pt}\nolimits}

\newcommand{\vv}[1]{\boldsymbol{#1}}

\def\dive{\mathop{\rm div}\nolimits}
\def\curl{\mathop{\rm curl}\nolimits}

\def\na{\nabla}
\def\p{\partial}

\newcommand{\beq}{\begin{equation}}
\newcommand{\eeq}{\end{equation}}
\newcommand{\ben}{\begin{eqnarray}}
\newcommand{\een}{\end{eqnarray}}
\newcommand{\beno}{\begin{eqnarray*}}
\newcommand{\eeno}{\end{eqnarray*}}

\numberwithin{equation}{section}
\newtheorem{lemma}{Lemma}[section]
\newtheorem{proposition}{Proposition}[section]
\newtheorem{theorem}{Theorem}[section]
\newtheorem{remark}{Remark}[section]
\newtheorem{corollary}{Corollary}[section]

\DeclareUnicodeCharacter{FF1A}{:}

\begin{document}

\title[Large box limit]{Global existence and stability of viscous Alfv\'en waves in the large-box limit for MHD systems}

\author[Li XU]{Li Xu}
\address{School of Mathematical Sciences, Beihang University\\  100191 Beijing, China}
\email{xuliice@buaa.edu.cn}

\author[Jiahui ZHANG]{Jiahui Zhang}
\address{Department of Mathematical Sciences, Tsinghua University\\ 100084 Beijing, China}
\email{zhangjh20@mails.tsinghua.edu.cn}

\date{}

\begin{abstract}
This paper provides a rigorous analysis of  how the {\it large box limit} fundamentally alters the global existence theory and dynamical behavior of the incompressible magnetohydrodynamics (MHD) system with small viscosity/resistivity $(0<\mu\ll 1)$ on periodic domains $Q_L=[-L,L]^3$, in the presence of a strong background magnetic field. While the existence of global solutions (viscous Alfv\'en waves) on the whole space $\R^3$ was previously established in \cite{He-Xu-Yu}, such results cannot be expected for general finite periodic domains.  We demonstrate that global solutions do exist on  $Q_L=[-L,L]^3$ precisely when the domain size exceeds a critical threshold  $L_\mu\geq e^{\f1\mu}$, thereby providing the first quantitative characterization of the transition to infinite-domain-like behavior. 

Our analysis focuses on the nonlinear stability of viscous Alfv\'en waves in this large-box setting, with two principal results: 
\begin{enumerate}
\item {\bf Local existence:} For $L\geq e^{\f{1}{\mu}}$, we establish the  existence of local small solutions on $[0,\log L]\times Q_L$ with uniform energy estimates independent of $L$. This is achieved via a novel hyperbolic energy method employing carefully designed weighted norms and characteristic foliation; 
\item {\bf Global extension:}  We prove that there exists $L_\mu\geq e^{\f{1}{\mu}}$ such that for all $L\geq L_\mu$,  the local solution can be extended globally. A linear-driving decay mechanism first reduces the wave energy to $\e_\mu (\ll \mu)$ by time $t_{n_0}=\log L$. The system then enters a small-data parabolic regime that ensures global existence.
\end{enumerate}

The proof heavily relies on two key stabilizing mechanisms:  the seperation of bidirectional-propagating Alfv\'en waves within characteristic hypersurface regions entirely contained in $[0,\log L]\times Q_L$, and the long-time effect of small diffusion. The large box limit ($L\gg 1$) is also crucial, as it suppresses finite-size effects while preserving the intrinsic bidirectional propagation of nonlinear Alfv\'en waves along the magnetic field. 
\end{abstract}

\maketitle

\tableofcontents

\section{Introduction}

Magnetohydrodynamics (MHD) studies the motion of electrically conducting fluids (e.g. plasmas or liquid metals) in the presence of magnetic fields, combining principles from fluid dynamics and electromagnetism. A fundamental phenomenon in MHD is the Alfv\'en wave - a transverse wave where magnetic tension acts as the restoring force. Alfv\'en waves, first theoretically predicted by Hannes Alfv\'en in 1942(see \cite{Alfven}), are a fundamental mode of magnetohydrodynamic (MHD) waves in magnetized plasmas. These waves arise from the interplay between magnetic field line tension and plasma inertia, producing transverse oscillations that propagate bidirectionally along the magnetic field. Distinct from sound waves or electromagnetic waves, Alfv\'en waves are a feature of MHD systems and play a critical role in diverse fields, including space physics, astrophysics, geodynamics, and fusion research.

From a mathematical perspective, Alfv\'en waves appear as solutions to the incompressible ideal MHD equations under strong background magnetic fields. In the ideal case, these solutions exhibit neither dissipation nor dispersion, maintaining constant energy and waveform during propagation. However, physical plasmas necessarily incorporate finite dissipative effects ($0<\mu\ll1$), where viscosity and resistivity introduce characteristic energy dissipation and wave damping. While the three-dimensional whole-space case with small dissipation was rigorously analyzed in \cite{He-Xu-Yu}, establishing global existence, long-time dynamics, and stability results. The analogous theory for finite periodic domains $Q_L=[-L,L]^3$ remains incomplete and presents fundamentally different challenges.

Our research resolves this theoretical gap through systematic investigation of the large box limit $(L\gg 1)$, where periodic systems approximate infinite-domain behavior.  This fundamental approximation in MHD theory considers system scales that significantly exceed characteristic plasma scales, achieving two crucial objectives: 
\begin{itemize}
\item[(1)] Effective elimination of boundary artifacts while preserving intrinsic plasma dynamics;
\item[(2)] Enabling rigorous analysis of key phenomena including MHD turbulence and energy cascade mechanisms, and linear stability properties of Alfv\'en waves and tearing modes.
\end{itemize}

The large box approximation, while indispensable for fundamental MHD theory and widely employed in both astrophysical and fusion plasma modeling, introduces several nontrivial theoretical challenges that demand careful consideration. Principal among these is the reconciliation of infinite-domain predictions with physically realistic finite systems, particularly regarding: i) nonlocal magnetic interactions in bounded domains; ii) computational constraints in numerical implementations; iii) boundary effects on wave propagation characteristics.

This study establishes a rigorous mathematical framework demonstrating that periodic MHD systems with sufficiently large scales ($L\geq L_\mu$) can indeed exhibit global dynamics properties comparable to their infinite-domain counterparts. Our work provides precise criteria for when periodic boundary conditions can faithfully represent realistic MHD phenomena, with particular significance for fusion plasma confinement (e.g. tokamak optimization \cite{Hasegawa-Chen}), space plasma dynamics (e.g., solar wind turbulence modeling \cite{Verschweren}) and astrophysical systems (e.g., accretion disk stability, pulsar magnetospheres\cite{Goldreich-Sridhar}). For comprehensive treatments of the physical foundations, we refer to \cite{Biskamp,Davidson}.

\subsection{Problem formulation and literature review}

 In this paper, we study the incompressible MHD system with small viscosity and resistivity ($0<\mu\ll1$) on a 3D periodic domain  $Q_L=[-L,L]^3$ as follows:
 \begin{equation}\label{eq:1}
\left\{\begin{aligned}
&\partial_{t}\vv v+\vv v \cdot \nabla\vv v-\mu \Delta\vv v+\nabla p= (\nabla \times \vv b)\times\vv b,\quad t>0,\,x\in Q_L,\\
&\partial_{t}\vv b+\vv v \cdot \nabla \vv b-\mu \Delta \vv b=\vv b \cdot \nabla\vv v, \\
&\operatorname{div}\vv v=0,\quad \operatorname{div}\vv b=0,\\
\end{aligned}\right. 
\end{equation}
where $\vv b$ is the magnetic field, $\vv v$ and $p$ are the velocity field and scalar pressure of the fluid respectively, and the constant $\mu>0$ is the viscosity coefficient, which also serves as the magnetic diffusivity (or resistivity). The system is supplemented with periodic boundary conditions for all unknowns in \eqref{eq:1}.

The Lorentz force term $(\nabla \times\vv b) \times\vv b$ in the momentum equation can be rewritten as
\beno
 (\nabla \times\vv b) \times\vv b = -\nabla \left(\frac{1}{2} |\vv b|^2 \right) + \vv b \cdot \nabla\vv b.
\eeno
Here the first term $-\nabla \left(\frac{1}{2} |\vv b|^2 \right)$ is referred to as the magnetic pressure, as it takes the same gradient form as the fluid pressure. The second term $\vv b\cdot\nabla\vv b=\na\cdot(\vv b\otimes\vv b)$ represents the magnetic tension force, responsible for generating Alfv\'en waves. By reusing  $p$ to denote the total pressure $p + \frac{1}{2} |\vv b|^2$ (instead of the fluid pressure),  the momentum equation can be expressed in the following form:
\begin{equation*}
\partial_t\vv v +\vv v \cdot\nabla\vv v-\mu \Delta\vv v+\nabla p=\vv b\cdot\nabla\vv b.
\end{equation*}

By introducing the Els\"{a}sser variables:
\begin{equation*}
Z_+ =\vv v +\vv b, \quad Z_- =\vv v -\vv b,
\end{equation*}
the MHD system \eqref{eq:1} can be reformulated in the following form
\begin{equation}\label{eq:2}
\left\{\begin{aligned}
&\partial_t Z_+ + Z_- \cdot \nabla Z_+-\mu\Delta Z_+ = -\nabla p, \quad t>0,\,x\in Q_L,\\
&\partial_t Z_- + Z_+ \cdot \nabla Z_--\mu\Delta Z_- = -\nabla p, \\
&\operatorname{div} Z_+ = \operatorname{div} Z_- = 0.
\end{aligned}\right. 
\end{equation}

Let $B_0=(0,0,1)$ be a uniform background magnetic field and set
\beno
z_+ = Z_{+}-B_0, \quad
z_- = Z_{-}+B_0.
\eeno
The MHD system \eqref{eq:2} can then be rewritten as
\begin{equation}\label{eq:MHD}
\left\{\begin{aligned}
&\partial_t z_+ + Z_- \cdot \nabla z_+-\mu\Delta z_+ = -\nabla p, \quad t>0,\,x\in Q_L,\\
&\partial_t z_- + Z_+ \cdot \nabla z_- -\mu\Delta z_-= -\nabla p,\\
&\operatorname{div} z_+ = \operatorname{div} z_- = 0.
\end{aligned}\right. 
\end{equation}

Let $j_+=\curl z_+$ and $j_-=\curl z_-$. Taking the curl of \eqref{eq:MHD}, we obtain
\begin{equation}\label{eq:j}
\left\{\begin{aligned}
&\partial_t j_{+} + Z_- \cdot \nabla j_{+} - \mu \Delta j_{+} = -\nabla z_- \wedge \nabla z_+, \quad t>0,\,x\in Q_L,\\
&\partial_t j_{-} + Z_+ \cdot \nabla j_{-} - \mu \Delta j_{-} = -\nabla z_+ \wedge \nabla z_-.
\end{aligned}\right. 
\end{equation}
We remark that $\dive j_{+}=\dive j_{-}=0$. The explicit expressions of the nonlinear terms on the righthand side are
\beno
\nabla z_- \wedge \nabla z_+ =\varepsilon_{ijk} \partial_i z^l_- \partial_l z^j_+ \partial_k , \quad \nabla z_+ \wedge \nabla z_- =\varepsilon_{ijk} \partial_i z^l_+ \partial_l z^j_- \partial_k ,
\eeno
where $\varepsilon_{ijk}$ is the totally anti-symmetric symbol associated to the volume form of $\mathbb{R}^3$.

The linearized analysis of \eqref{eq:MHD} and \eqref{eq:j} on the whole space $\R^3$ leads to dispersion relation
\beq\label{dispersion}
w_\pm(\xi)=-i\mu|\xi|^2\pm \xi_3,\quad\forall\,\xi\in\R^3,
\eeq
where the corresponding plane wave solutions are identified as (viscous) Alfv\'en waves.
This relation reveals that these waves are $1+1$ dimension waves propagating bidirectionally along the background magnetic field with $v_A=|B_0|=1$, while damped by a weak dissipation (if $0<\mu\ll1$). For the ideal MHD case (i.e., $\mu=0$), this linearized analysis reproduces the original results obtained by H. Alfv\'en in his foundational work \cite{Alfven}.

This paper establishes a fundamental extension of the global existence theory for the MHD system \eqref{eq:MHD} from the whole space to periodic domains, with three principal contributions:
\begin{enumerate}
\item {\bf Threshold for global solutions:} We prove that global solutions exist on periodic domains $Q_L=[-L,L]^3$ precisely when $L$ exceeds a threshold $L_\mu\geq e^{1/\mu}$, marking the transition to infinite-domain-like dynamics.
\item {\bf Stabilizing mechanisms:} For $L \geq L_\mu$, viscous Alfv\'en waves stabilize due to the separation of bidirectional Alfv\'en waves and cumulative viscous dissipation over characteristic timescales.
\item 
{\bf Mathematical framework:} We develop novel energy methods and weighted norms to quantify the finite-to-infinite domain transition.
\end{enumerate}

\medskip

We conclude this subsection with a brief survey of the incompressible MHD systems involving Alfv\'en waves.
The theoretical foundation was established by H. Alfv\'en (\cite{Alfven}) in 1942,  who first predicted the existence of linear Alfv\'en waves in ideal MHD with strong background magnetic fields. For the ideal MHD system in $\R^2$ and $\R^3$, Bardos, Sulem and Sulem (\cite{Bardos}) first constructed the global solutions (ideal Alfv\'en waves) in weighted H\"{o}lder space $C^{1,\alpha}$ (not in the energy spaces), by reformulating the MHD system in terms of the Els\"{a}sser variables (see \eqref{eq:MHD} with $\mu=0$) and developing a wave-based approach inspired by Klainerman's work \cite{Klainerman 1} on nonlinear wave equations. For small diffusion case in $\R^3$, the authors in \cite{He-Xu-Yu} provided a rigorous mathematical proof for the existence, propagation and stability of (ideal/viscous) Alfv\'en wave in the nonlinear setting. The approach is inspired by the stability of Minkowski spacetime \cite{Ch-K}. The proof combined the hyperbolic energy method, which relied heavily on the weighted energy and the energy flux, with parabolic estimates, achieved through the design of the weights for the energy. An alternative demonstration of a similar global existence result was subsequently established in \cite{Wei-Zh 2} via the comparison principle and in \cite{Cai-Lei} using the so-called ghost weight technique. For a sufficiently small ratio of viscosity and resistivity, the authors in \cite{Wei-Zh 1} proved the global existence in the H\"older space $C^{1,\al}$ by unifying the method used in \cite{Bardos} and the design of the weights for the case in \eqref{eq:MHD}. For the ideal MHD on thin domains $\Omega_\delta=\R^2\times(-\delta,\delta)$,  Xu (\cite{Xu}) constructed global small solutions to the ideal incompressible MHD system in three-dimensional thin domains $\Omega_\delta = \mathbb{R}^2 \times (-\delta, \delta)$, and proved the convergence from the 3D Alfv\'{e}n waves in $\Omega_\delta$  to the plane Alfv\'{e}n waves in $\mathbb{R}^2$ as the slab thickness $\delta \to 0$.  In \cite{Li-Yu}, the authors established a rigidity theorem for the ideal Alfv\'{e}n waves at infinity, showing that if the scattering fields of Alfv\'{e}n waves vanish at infinity, then there are no Alfv\'{e}n wave at all emanating from the plasma. 

 \smallskip



Finally, we give some comments on the short review.
  All aforementioned results for MHD systems were established on infinite domains (either whole space or infinite strips), where the background magnetic field is oriented along the infinitely extended direction. This setup ensures that the two families of Alfv\'en waves (solutions of the MHD system) propagate indefinitely in opposite directions along the magnetic field. This fundamental geometric property, where the seperation of counter-propagating Alfv\'en waves serves as the stable mechanism, guarantees the global existence of solutions.

\smallskip

In contrast to previous studies, the present work differs fundamentally in the following aspects:
\begin{enumerate}
\item We consider the MHD system on finite periodic domains $Q_L=[-L,L]^3$, where due to the domain's compactness, the seperation of counter-propagating Alfv\'en wave families is disrupted near boundaries $x_3=\pm L$. This implies the breakdown of the stabilization mechanism inherent to infinite-domain systems.
\item 
Our analysis proves the existence of a scale $L_\mu (\geq e^{1/\mu})$ such that for the domain scale $L\geq L_\mu$, the MHD system admits global solutions provided the initial energy $\e_0$ is sufficiently small. We emphasize that this $\e_0$ is independent of both the domain scale $L$ and the viscosity coefficient $\mu$.
\end{enumerate}

\subsection{Difficulties, key observations and strategies}

The main difficulties in constructing global solutions to this problem arise from the boundedness of the domain $Q_L$. The specific challenges manifest in the following aspects:
\begin{enumerate}
\item[1).] The separation property of Alfv\'en waves is disrupted, leading to difficulties in nonlinear control. In previous works on infinite domains, the separation property of Alfv\'en waves serves as the intrinsic mechanism for MHD stability, enabling control of nonlinear terms with null structure (see \eqref{eq:MHD}, \eqref{eq:j}).
\item[2).] The pressure $p$ satisfies an elliptic equation 
\beq\label{eq:p}
-\Delta p=\sum_{j,k=1}^3\p_j\p_k(z_-^jz_+^k),\quad\forall\, x\in Q_L,
\eeq
and its specific expression becomes significantly more complex compared to that in $\R^3$, making the control of pressure terms considerably more difficult.
\item[3).] The treatment of viscous terms presents significant challenges, especially for the lowest-order energy estimates.  The approach in \cite{He-Xu-Yu}utilizes Hardy's inequality to reformulate these estimates in terms of the basic energy identity. However, in the finite domain $Q_L$, while Poincaré inequality can replace Hardy's inequality, its associated constant exhibits $L$-dependence - a critical limitation that introduces new complications in completing the energy estimates.
\item[4).] In \cite{He-Xu-Yu}, the global existence of solutions could be directly obtained through hyperbolic energy estimates combined with the basic energy identity. However, in our current work, the finite domain nature of $Q_L$ fundamentally alters the properties of MHD solutions (notably, the breakdown of Alfv\'en waves' separation). Consequently, the framework used for proving local existence cannot be directly extended to establish global existence.
\end{enumerate} 

\medskip

To overcome these difficulties, we will comprehensively utilize the separation property of Alfv\'en waves, the large-scale domain characteristics ($L\gg1$), and weak dissipation effects. Our analysis yields the following key observations:

{\it (1). Control within finite time interval $[0,\log L]$}:
When the initial perturbations are confined to $[-L,L]^2\times[-L/4,L/4]$, the characteristic propagation regions of the solutions (Alfv\'en waves) remains entirely within the spacetime domain $[0,\log L]\times Q_L$.  In this regime, Alfv\'en waves' propagation dynamics mirror the whole-space $\R^3$ scenario, allowing their separation property to effectively control the corresponding nonlinear terms. For initial perturbations with spatial decay, the exterior region (outside $[-L,L]^2\times[-L/4,L/4]$) contributes only $O(1/{L^\al})$ terms, ensuring that nonlinear interactions remain controllable when $L$ is sufficiently large.

{\it (2). Dissipation-dominated phase ($t\gg 1/\mu$)}:
Inspired by \cite{He-Xu-Yu}'s results for the whole space $\R^3$, we observe that beyond the hyperbolic characteristic time $1/\mu$, the system's weak dissipation mechanism becomes operative, inducing energy decay. After sufficiently long evolution,  the system energy reduces to viscosity-dominated levels, enabling its reformulation as a small-energy parabolic system. This permits application of classical theory to guarantee global solutions of MHD system with eventual dissipation of Alfv\'en waves.

\medskip

Based on these fundamental observations, we formulate our principal strategies as follows:

{\bf Step 1. Local existence on the time interval $[0,\log L]$.}  To establish local existence, we implement a hyperbolic energy method that fundamentally relies on weighted energies and energy fluxes through characteristic hypersurfaces. Our approach consists of four key components:

{\it (1). Novel foliations of  $[0,\log L]\times Q_L$.} 
The solutions $z_+$ and $z_-$ of the MHD system \eqref{eq:MHD} represent $1+1$ dimension (Alfv\'en) waves propagating bidirectionally along the background magnetic field, subject to weak diffusion.  The nonlinear terms in \eqref{eq:j}  of form $\nabla z_{+} \wedge \nabla z_{-}$ exhibit null structure similar to that in nonlinear wave equations. To control these terms, we foliate the spacetime slab $[0,\log L]\times Q_L$ by the families of characteristic hypersurfaces:
\beno
\Big(\bigcup\limits_{-\frac{L}{4}< u_{+}<\frac{L}{4}} C_{t,u_{+}}^{+}\Big) \bigcup W_{t,+}^{\leq -\frac{L}{4}} \bigcup W_{t,+}^{\geq \frac{L}{4}}\quad
\text{and}\quad
\Big(\bigcup\limits_{-\frac{L}{4}< u_{-}<\frac{L}{4}} C_{t,u_{-}}^{-}\Big) \bigcup W_{t,-}^{\leq -\frac{L}{4}} \bigcup W_{t,-}^{\geq \frac{L}{4}},
\eeno
where the characteristic regions $\bigcup\limits_{-\frac{L}{4}< u_{+}<\frac{L}{4}} C_{t,u_{+}}^{+}$ and $\bigcup\limits_{-\frac{L}{4}< u_{-}<\frac{L}{4}} C_{t,u_{-}}^{-}$ consist of the hypersurfaces  along which $z_-$ and $z_+$ propagate, respectively, and are entirely contained within $[0,\log L]\times Q_L$.  The sets $W_{t,+}^{\leq -\frac{L}{4}} \bigcup W_{t,+}^{\geq \frac{L}{4}}$ and $W_{t,-}^{\leq -\frac{L}{4}} \bigcup W_{t,-}^{\geq \frac{L}{4}}$ represent the remainder regions. 

{\it (2). Design of the weighted energies and energy fluxes.} Since the viscous terms are not compatible with the hyperbolic approach,  we shall unify the hyperbolic energy method with the parabolic estimates via designing the different weights for the lowest-order, intermediate-order and top order energies (incorporating $L$-dependence). Due to the finite size of the domain, the geometry of the Alfv\'en wave characteristics changes in regions away  from the interior of $[0,\log L] \times Q_L$, leading to a breakdown in the separation of the waves there. Consequently, we control the nonlinear terms using an energy flux defined entirely inside   $[0,\log L]\times Q_L$, along with the large-scale effect of $L$ in the remainder regions. Precisely, we define the energy flux of $z_\pm$ as follows:
\beno
F_{\pm}(t)= \sup_{|u_{\pm}|\leq \frac{L}{4}}\int_{\overline{C}_{t,u_{\pm}}^{\pm}} \left(\log \langle w_{\mp}\rangle\right)^{4}\left|z_{\pm}\right|^{2} d\sigma_{\pm}.
\eeno
The higher-order energy fluxes are defined similarly. We remark that the hypersurfaces for these fluxes are entirely inside $[0,\log L]\times Q_L$.  For more details, see Subsetion \ref{geometries subsection}.

{\it (3). Lowest-order energy estimates.} While higher-order energy estimates follow directly from the vorticity equations (cf.~\eqref{eq:j}), which avoid nonlocal pressure terms, the lowest-order estimates require careful treatment of the original MHD system \eqref{eq:MHD}. Since the pressure term $p$ is governed by an elliptic equation \eqref{eq:p} on the bounded domain $Q_L$,  the estimates for $p$ on $Q_L$ are significantly more delicate due to the $L$-dependence of Green's function. The estimate is reduced to control the viscous terms such as
\beq\label{eq:add 1}
\mu\int_{0}^{t} \int_{Q_L} \frac{\left|z_{+}\right|^{2}}{\langle w_{-}\rangle^{2}} dx d\tau.
\eeq
For the analogous problem on  $\R^3$ (see \cite{He-Xu-Yu}), Hardy's inequality bounds this term by $\mu\int_0^t\|\na z_+\|^2_{L^2(\R^3)}d\tau$, which is controlled via the basic energy identity. However, on $Q_L$, the natural substitute-Poinc\'are’s inequality-fails due to its $L$-dependence constant. Our resolution is  to extend the integral to $\R^3$ and exploit the largeness of $L$. Roughly speaking,
\beno
\mu\int_{0}^{t} \int_{Q_L} \frac{\left|z_{+}\right|^{2}}{\langle w_{-}\rangle^{2}} dx ds\lesssim\mu\int_{0}^{t} \|\na z_+\|_{L^2(Q_L)}^2ds+\f{\mu}{L^2}t\|z_+\|_{L^2(Q_L)}^2,
\eeno
both terms are bounded by $\|z_{+,0}\|_{L^2(Q_L)}^2$ under the scaling $L\gg \f{1}{\mu}$ and $t\lesssim L$.

{\it (4). Uniform energy estimates with respect to (w.r.t.) $L$ and $\mu$.} For $L\geq e^{1/\mu}$, if the initial weighted energy satisfies $\cE^w(0)\leq \e_0^2$ (with $\e_0$ independent of $L$ and $\mu$), the nonlinear terms are controlled by the total energy and flux rather than diffusion. Consequently, we obtain the uniform energy estimates with respect to $L$ and $\mu$.

\smallskip

{\bf Step 2. Global extension.} To extend the local solution globally beyond $t=\log L$, we exploit the cumulative effect of weak dissipation over long times to derive the decay of the system's energy. The proof includes two key strategies:

{\it (1). Linear-driven decay estimates.} Building upon the uniform energy estimates (independent of $L$ and $\mu$) for local solutions, we establish the decay energy estimates for the physical energy of the solutions (see \eqref{linear driven decay estimates}). In view of the decay estimates, we show that there exists a critical domain scale $L_\mu\geq e^{1/\mu}$ such that when $L\geq L_\mu$, the physical energy decays to $\e_\mu (\ll \mu)$ by time $t_{n_0}=\log L$ and the system consequently enters {\it a small-data parabolic regime}. 

{\it (2). Global existence for parabolic system.} By taking $t_{n_0}=\log L$ as the new initial time and applying classical parabolic theory to the MHD system \eqref{eq:MHD} with initial energy bounded above by $\e_\mu (\ll\mu)$, we establish both the global existence of solutions on $[\log L,+\infty)$ and their asymptotic convergence to equilibrium. Consequently, we successfully extend the local solutions on $[0,\log L]$  globally.






\subsection{The characteristic geometries}\label{geometries subsection}

Given $T^*= \log L$, we study the spacetime $[0,T^*]\times Q_L$ associated with either the MHD solution $(\vv v,\vv b)$, or its equivalent formulation $(z_+,z_-)$ from \eqref{eq:MHD}. More precisely,  assuming the existence of a smooth solution $(z_{+}, z_{-})$  on  $[0,T^*]\times Q_L$, we investigate the foliation of characteristic hypersurfaces generated by $(z_{+}, z_{-})$. The periodic nature of $Q_L$ requires careful treatment in constructing this  foliation. We first recall that $[0,T^*]\times Q_L$ admits a natural time foliation $[0,T^*]\times Q_L=\bigcup_{0\leq t\leq T^*}\Sigma_t$, where each time slice $\Sigma_t$ is  defined by
\begin{equation*}
\Sigma_{t} = \left\{(\tau,x) \in \mathbb{R} \times Q_L : \tau = t\right\}.
\end{equation*}
To construct the characteristic foliation associated with $(z_{+}, z_{-})$,  we extend $z_{\pm}$ to  $\mathbb{R}^3$ by imposing $2L$-perodicity  in all spatial directions ($x_1$, $x_2$, and $x_3$). This extension allows us to treat $z_\pm$ as $2L$-perodic vector fields on $[0,T^*]\times \mathbb{R}^3$. 

We introduce the characteristic space-time vector fields  $L_+$ and $L_-$ as follows:
\beno
 L_{+}=T+Z_{+}, \quad L_{-}=T+Z_{-},
\eeno
where the time vector field $T$ is the usual $\p_t$ defined in the Cartesian coordinates. We also use the same notations to denote the partial differential operators $L_\pm=\partial_{t}+ Z_\pm\cdot \nabla$.
Using these operators, we define two optical functions $x_{3}^{\pm} = x_{3}^{\pm}(t,x)$ on $[0,T^*]\times \mathbb{R}^3$ as the solutions to
\beq\label{def of x 3 pm}
L_{+} x_{3}^{+}=0, \quad x_{3}^{+}\big|_{t=0}=x_{3},\quad\text{and}\quad 
L_{-} x_{3}^{-}=0, \quad x_{3}^{-}\big|_{t=0}=x_{3}.
\eeq
Then we construct the functions $u_{\pm}= u_{\pm}(t,x)$ on $[0,T^*]\times Q_L$ as follows:
\begin{equation}
u_{+}(t,x)=
\begin{cases}
x_{3}^{+}(t,x), &\text{if }\,x_{3}^{+}(t,x)\in(-L,L]\\
x_{3}^{+}(t,x_1,x_2,x_3+2L), &\text{if }\,x_{3}^{+}(t,x)\in(-2L,-L],
\end{cases} \label{eq:g3}
\end{equation}
\begin{equation}
u_{-}(t,x)=
\begin{cases}
x_{3}^{-}(t,x), &\text{if }\,x_{3}^{-}(t,x)\in[-L,L),\\
x_{3}^{-}(t,x_1,x_2,x_3-2L), &\text{if }\,x_{3}^{-}(t,x)\in[L,2L).
\end{cases} \label{eq:g4}
\end{equation}
 We remark that for smooth periodic functions $Z_\pm$ on $[0,T^*]\times Q_L$, the functions $x_3^\pm$ maintain smoothness on $[0,T^*]\times Q_L$, whereas $u_\pm$ are piecewise smooth on the same domain and satisfy
\beq\label{prop for u pm}
u_\pm\in[-L,L],\quad\text{and}\quad u_\pm|_{x_3=L}=u_\pm|_{x_3=-L}
\eeq
as a consequence of \eqref{def of x 3 pm} and \eqref{eq:e9c}. One can check the details in Remark \ref{rmk for u pm}.

\smallskip 

Therefore, for any $a \in [-L,L]$, we define the characteristic hypersurfaces $C_{a}^\pm $ as the level sets:
\beno\begin{aligned}
&C_{a}^{+}=\left\{(t,x)\in [0,T^*]\times Q_L:u_{+}(t,x)=a\right\}=\bigcup_{0\leq t\leq T^*} S^+_{t,a},\\
&
C_{a}^{-}=\left\{(t,x)\in [0,T^*]\times Q_L:u_{-}(t,x)=a\right\}=\bigcup_{0\leq t\leq T^*} S^-_{t,a},
\end{aligned}\eeno
where the time slice $S^\pm_{t,a}=C_a^\pm\cap\Sigma_t$. 
Similarly, for any $a\in(-2L,2L)$, we define $\overline{C}_{a}^{\pm}$ as follows:
\beno\begin{aligned}
&
\overline{C}_{a}^{+}=\left\{(t,x)\in [0,T^*]\times Q_L:x_{3}^{+}(t,x)=a\right\}=\bigcup_{0\leq t\leq T^*}\overline{S}_{t,a}^+,\\
&
\overline{C}_{a}^{-}=\left\{(t,x)\in [0,T^*]\times Q_L:x_{3}^{-}(t,x)=a\right\}=\bigcup_{0\leq t\leq T^*}\overline{S}_{t,a}^-,
\end{aligned}\eeno
where the time slice $\overline{S}_{t,a}^{\pm}=\overline{C}_{a}^{\pm}\bigcap \Sigma_{t}$.
By construction,  $L_{+}$ is tangent to both $C_{a}^{+}$ and $\overline{C}_{a}^{+}$, while  $L_{-}$ is tangent to both $C_{a}^{-}$ and $\overline{C}_{a}^{-}$.

For any $0 \leq t_1 < t_2 \leq T^*$, we define the characteristic surface sections $C_{[t_{1},\,t_{2}],\,u_{\pm}}^{\pm}$ and $\overline{C}_{[t_{1},\,t_{2}],\,a}^{\pm}$ as follows:
\beno
C_{[t_{1},\,t_{2}],\,u_{\pm}}^{\pm}= C_{u_{\pm}}^{\pm}\bigcap \big([t_1,t_2]\times Q_L\big),\quad
\overline{C}_{[t_{1},\,t_{2}],\,a}^{\pm}= \overline{C}_{a}^{\pm}\bigcap \big([t_1,t_2]\times Q_L\big).
\eeno
For simplicity, we denote by
\beno
C_{t,u_{\pm}}^{\pm}= C_{[0,\,t],\,u_{\pm}}^{\pm}
\quad
\overline{C}_{t,a}^{\pm}= \overline{C}_{[0,\,t],\,a}^{\pm},\quad\forall\,t\in[0,T^*].
\eeno

In order to specify the region where the energy estimates are taken place, for any $0 \leq t \leq T^*$, $a \in(-2L,2L)$, $-\frac{L}{4} \leq u_{+}^{1} < u_{+}^{2} \leq \frac{L}{4}$ and $-\frac{L}{4} \leq u_{-}^{1} < u_{-}^{2} \leq \frac{L}{4}$, we define the following hypersurfaces/ regions:
\begin{itemize}
\item[(i)] Time-slice regions:
\beno\begin{aligned}
&\Sigma_{t,+}^{\left[u_{+}^{1}, u_{+}^{2}\right]}=\bigcup\limits_{u_{+}^{1}\leq u_{+}\leq u_{+}^{2}} S_{t, u_{+}}^{+}, \quad \Sigma_{t,-}^{\left[u_{-}^{1}, u_{-}^{2}\right]}=\bigcup\limits_{u_{-}^{1}\leq u_{-}\leq u_{-}^{2}} S_{t, u_{-}}^{-},\\
&\Sigma_{t,\pm}^{\leq a}=\bigcup\limits_{u_{\pm}\leq a} S_{t, u_{\pm}}^{\pm}, \quad \Sigma_{t,\pm}^{\geq a}=\bigcup\limits_{u_{\pm}\geq a} S_{t, u_{\pm}}^{\pm}, \\
&\overline{\Sigma}_{t,+}^{> -L}=\bigcup\limits_{a> -L} \overline{S}_{t,a}^{+}, \quad \overline{\Sigma}_{t,-}^{<L}=\bigcup\limits_{a< L} \overline{S}_{t,a}^{-};
\end{aligned}\eeno
\item[(ii)] Spacetime regions:
\beno
\begin{aligned}
&W_{t,+}^{\left[u_{+}^{1}, u_{+}^{2}\right]}=\bigcup\limits_{u_{+}^{1}\leq u_{+}\leq u_{+}^{2}} C_{t, u_{+}}^{+}, \quad W_{t,-}^{\left[u_{-}^{1}, u_{-}^{2}\right]}=\bigcup\limits_{u_{-}^{1}\leq u_{-}\leq u_{-}^{2}} C_{t, u_{-}}^{-},\\
&W_{t,\pm}^{\leq a}=\bigcup\limits_{u_{\pm}\leq a} C_{t, u_{\pm}}^{\pm}, \quad W_{t,\pm}^{\geq a}=\bigcup\limits_{u_{\pm}\geq a} C_{t, u_{\pm}}^{\pm}, \\
&\overline{W}_{t,+}^{>-L}=\bigcup\limits_{a> -L} \overline{C}_{t,a}^{+}, \quad \overline{W}_{t,-}^{< L}=\bigcup\limits_{a< L} \overline{C}_{t,a}^{-}.
\end{aligned} 
\eeno
\end{itemize}
Thus, in the notations established above, the spacetime domain $[0,t] \times Q_L$ admits the following two characteristic foliations:
\beno
\Big(\bigcup\limits_{-\frac{L}{4}< u_{+}<\frac{L}{4}} C_{t,u_{+}}^{+}\Big) \bigcup W_{t,+}^{\leq -\frac{L}{4}} \bigcup W_{t,+}^{\geq \frac{L}{4}}\quad
\text{and}\quad
\Big(\bigcup\limits_{-\frac{L}{4}< u_{-}<\frac{L}{4}} C_{t,u_{-}}^{-}\Big) \bigcup W_{t,-}^{\leq -\frac{L}{4}} \bigcup W_{t,-}^{\geq \frac{L}{4}}.
\eeno

With the standard Euclidean metric on $\mathbb{R}^4=\mathbb{R}_{t}\times \mathbb{R}_{x}^3$  and ignoring the $x_1$, $x_2$ directions, we schematically depict the unit outward normals to the boundaries of the spacetime regions $W_{t,+}^{\left[u_{+}^{1}, u_{+}^{2}\right]}$ and $W_{t,-}^{\left[u_{-}^{1}, u_{-}^{2}\right]}$ as follows:

  \includegraphics[width=0.7\textwidth]{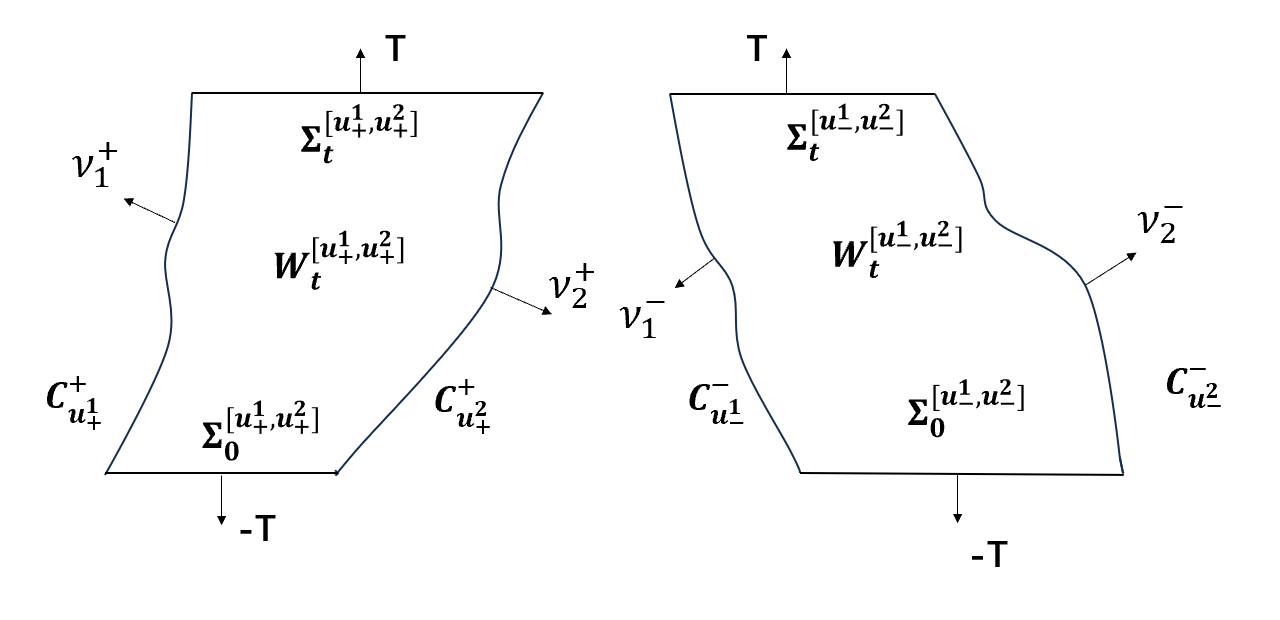}

 The unit outward normals to $\Sigma_{0}$ and $\Sigma_{t}$ are $-T$ and $T$, respectively. We denote by $\nu_{1}^{+}$, $\nu_{2}^{+}$, $\nu_{1}^{-}$, and $\nu_{2}^{-}$ the unit outward normals to $C_{t, u_{+}^{1}}^{+}$, $C_{t, u_{+}^{2}}^{+}$, $C_{t, u_{-}^{1}}^{-}$, and $C_{t, u_{-}^{2}}^{-}$, respectively. Then we have
\beno
\begin{array}{c}
\nu_{1}^{+}=-\frac{\left(\partial_{t} u^1_{+}, \nabla u^1_{+}\right)}{\sqrt{\left(\partial_{t} u^1_{+}\right)^{2}+\left|\nabla u^1_{+}\right|^{2}}}, \quad \nu_{2}^{+}=\frac{\left(\partial_{t} u^2_{+}, \nabla u^2_{+}\right)}{\sqrt{\left(\partial_{t} u^2_{+}\right)^{2}+\left|\nabla u^2_{+}\right|^{2}}},\\
\nu_{1}^{-}=-\frac{\left(\partial_{t} u^1_{-}, \nabla u^1_{-}\right)}{\sqrt{\left(\partial_{t} u^1_{-}\right)^{2}+\left|\nabla u^1_{-}\right|^{2}}}, \quad \nu_{2}^{-}=\frac{\left(\partial_{t} u^2_{-}, \nabla u^2_{-}\right)}{\sqrt{\left(\partial_{t} u^2_{-}\right)^{2}+\left|\nabla u^2_{-}\right|^{2}}}.
\end{array} 
\eeno

Intuitively, the spacetime regions $W_{t,\pm}^{[u_{\pm}^{1}, u_{\pm}^{2}]}$ are bounded by  $\Sigma_{0}$, $\Sigma_{t}$, $C_{t, u_{\pm}^{1}}^{\pm}$, and $C_{t, u_{\pm}^{2}}^{\pm}$. 

\medskip 

To conclude this subsection, we extend the characteristic coordinates  by defining the functions  $x_{i}^{\pm} = x_{i}^{\pm}(t,x)$ on $[0,T^*]\times \R^3 $  (for $i=1,2$)  through the following equations: 
\beno
L_{\pm} x_{i}^{\pm}=0, \quad x_{i}^{\pm}\big|_{t=0}=x_{i},\quad i=1,2.
\eeno
As established previously, the vector fileds $z_\pm $ admit $2L$-perodic extension on $[0,T^*]\times \mathbb{R}^3$. Consequently, the charateristic coordinates  $x_{i}^{\pm}$ (for $i=1,2,3$) preserve this periodicity in all spatial variables except their respective coordinate directions. More precisely, for all $(t,x)\in [0,T^*]\times \mathbb{R}^3$:
\begin{equation}\label{eq:g16}
\begin{aligned}
& x_{1}^{\pm}(t,x)=x_{1}^{\pm}(t,x_1,x_2+2L,x_3)=x_{1}^{\pm}(t,x_1,x_2,x_3+2L),\\
& x_{2}^{\pm}(t,x)=x_{2}^{\pm}(t,x_1+2L,x_2,x_3)=x_{2}^{\pm}(t,x_1,x_2,x_3+2L),\\
& x_{3}^{\pm}(t,x)=x_{3}^{\pm}(t,x_1+2L,x_2,x_3)=x_{3}^{\pm}(t,x_1,x_2+2L,x_3).
\end{aligned} 
\end{equation}

\subsection{Main results}

In this subsection, we will state two results of MHD system \eqref{eq:MHD}.

\subsubsection{Local well-posedness.}
Before stating the existence result, we have to define the energy functionals, diffusions and fluxes related to\eqref{eq:MHD}. 

Firstly, we introduce two weight functions $\langle w_{+}\rangle$ and $\langle w_{-}\rangle$ on $[0,T^*]\times Q_L$ as follows:
\beno
\langle w_{+}\rangle=\big(R^{2}+\left|x_{1}^{+}\right|^{2}+\left|x_{2}^{+}\right|^{2}+\left|u_{+}\right|^{2}\big)^{\frac{1}{2}}, \quad \langle w_{-}\rangle=\big(R^{2}+\left|x_{1}^{-}\right|^{2}+\left|x_{2}^{-}\right|^{2}+\left|u_{-}\right|^{2}\big)^{\frac{1}{2}},
\eeno
where $R$ is a large positive number that will be determined later on. Similarly, we also define the weights $\langle \overline{w_{+}}\rangle$ and $\langle \overline{w_{-}}\rangle$ on $[0,T^*]\times \mathbb{R}^3$ by
\beno
 \langle \overline{w_{+}}\rangle=\big(R^{2}+\left|x_{1}^{+}\right|^{2}+\left|x_{2}^{+}\right|^{2}+\left|x_{3}^{+}\right|^{2}\big)^{\frac{1}{2}}, \quad
 \langle \overline{w_{-}}\rangle=\big(R^{2}+\left|x_{1}^{-}\right|^{2}+\left|x_{2}^{-}\right|^{2}+\left|x_{3}^{-}\right|^{2}\big)^{\frac{1}{2}}.
\eeno

By virtue of \eqref{eq:g3}, \eqref{eq:g4} and \eqref{eq:g16}, we obtain
\begin{equation*}
\langle w_{+}\rangle (t,x)=
\begin{cases}
\langle \overline{w_{+}}\rangle (t,x), &\text{if }\,x_{3}^{+}(t,x)\in(-L,L],\\
\langle \overline{w_{+}}\rangle (t,x_1,x_2,x_3+2L), &\text{if }\,x_{3}^{+}(t,x)\in(-2L, -L],
\end{cases}
\end{equation*}
\begin{equation*}
\langle w_{-}\rangle (t,x)=
\begin{cases}
\langle \overline{w_{-}}\rangle (t,x), &\text{if }\,x_{3}^{-}(t,x)\in [-L,L),\\
\langle \overline{w_{-}}\rangle (t,x_1,x_2,x_3-2L), &\text{if }\,x_{3}^{-}(t,x)\in[ L,2L).
\end{cases}
\end{equation*}
Furthermore, by definitions, there hold $L_{\pm} \langle \overline{w_{\pm}}\rangle=0$ and $L_{\pm} \langle w_{\pm}\rangle=0$.

To simplify notations, for a multi-index $\alpha=\left(\alpha_{1}, \alpha_{2}, \alpha_{3}\right)$, we define $z_{\pm}^{(\alpha)}=\big(\frac{\partial}{\partial x_{1}}\big)^{\alpha_{1}}\big(\frac{\partial}{\partial x_{2}}\big)^{\alpha_{2}}\big(\frac{\partial}{\partial x_{3}}\big)^{\alpha_{3}}z_{\pm}$. For any positive integer $k$, we define  $|z_{\pm}^{(k)}|=\Big(\sum\limits_{|\alpha|=k}|z_{\pm}^{(\alpha)}|^2\Big)^{\frac{1}{2}}$. Analogously, we define $j_{\pm}^{(\alpha)}$ and $|j_{\pm}^{(k)}|$ in the same manner.

Let $N_{*} \in \mathbb{Z}_{\geq 7}$. For $t \geq 0$, we define the {\it lowest-order} and {\it higher-order} weighted energy functionals of $z_{\pm}$ as follows:
\beno
E_{\pm}(t) = \int_{\Sigma_{t}} \left(\log \langle w_{\mp}\rangle\right)^{4}\left| z_{\pm}\right|^{2} dx, 
\eeno
\beno
E_{\pm}^{(\alpha)}(t):=
\begin{cases}
\ \int_{\Sigma_{t}} \langle w_{\mp}\rangle^{2}\left(\log \langle w_{\mp}\rangle\right)^{4}|\nabla z_{\pm}^{(\alpha)}|^{2} dx, &\text{for }0\leq |\alpha| \leq N_*+1,\\
\ \int_{\Sigma_{t}} \langle w_{\mp}\rangle \left(\log \langle w_{\mp}\rangle\right)^{4}|\nabla z_{\pm}^{(\alpha)}|^{2} dx, &\text{for }|\alpha|= N_*+2,\\
\ \int_{\Sigma_{t}} \left(\log \langle w_{\mp}\rangle\right)^{4}|\nabla z_{\pm}^{(\alpha)}|^2 dx, &\text{for }|\alpha|= N_*+3,\\
\ \int_{\Sigma_{t}}|\nabla z_{\pm}^{(\alpha)}|^2 dx, &\text{for }|\alpha|= N_*+4.
\end{cases} 
\eeno
We define the {\it lowest-order, first-order,} and {\it higher-order} weighted fluxes of $z_{\pm}$ associated with the characteristic hypersurfaces as follows:
\beno
\begin{aligned}
F_\pm(t):=& \sup_{|u_{\pm}|\leq \frac{L}{4}}\int_{\overline{C}_{t,u_{\pm}}^{\pm}} \left(\log \langle w_{\mp}\rangle\right)^{4}\left|z_{\pm}\right|^{2} d\sigma_{\pm},\\
F_{\pm}^{(0)}(t):=& \sup_{|u_{\pm}|\leq \frac{L}{4}}\int_{\overline{C}_{t,u_{\pm}}^{\pm}} \langle w_{\mp}\rangle^{2}\left(\log \langle w_{\mp}\rangle\right)^{4}\left|\nabla z_{\pm}\right|^{2} d\sigma_{\pm},\\
F_{\pm}^{(\alpha)}(t):=&\sup_{|u_{\pm}|\leq \frac{L}{4}}\int_{\overline{C}_{t,u_{\pm}}^{\pm}} \langle w_{\mp}\rangle^{2}\left(\log \langle w_{\mp}\rangle\right)^{4}|j_{\pm}^{(\alpha)}|^{2} d\sigma_{\pm}, \ \text{for }1\leq |\alpha| \leq N_*,
\end{aligned} 
\eeno
where $d\sigma_\pm$ denotes the surface measure on the characteristic hypersurfaces $\overline{C}_{t,u_{\pm}}^{\pm}$. We emphasize that for all $|\alpha|\geq 1$, the flux terms are defined exclusively in terms of the vorticity components, deliberately excluding full derivatives. This technical formulation is adopted to circumvent analytical challenges associated with pressure terms in higher-order energy estimates. 
We define the {\it lowest-order} and {\it higher-order} weighted diffusions for $z_{\pm}$ as follows:
\beno
D_{\pm}(t) = \mu\int_{0}^{t} \int_{\Sigma_{\tau}} \left(\log \langle w_{\mp}\rangle\right)^{4}\left|\nabla z_{\pm}\right|^{2} dx d\tau, 
\eeno
\beno
D_{\pm}^{(\alpha)}(t)=
\begin{cases}
\ \mu\int_{0}^{t} \int_{\Sigma_{\tau}} \langle w_{\mp}\rangle^{2}\left(\log \langle w_{\mp}\rangle\right)^{4}|\nabla^2 z_{\pm}^{(\alpha)}|^{2} dx d\tau, &\text{for }0\leq |\alpha| \leq N_*+1,\\
\ \mu\int_{0}^{t} \int_{\Sigma_{\tau}} \langle w_{\mp}\rangle \left(\log \langle w_{\mp}\rangle\right)^{4}|\nabla^2 z_{\pm}^{(\alpha)}|^{2} dx d\tau, &\text{for }|\alpha|= N_*+2,\\
\ \mu\int_{0}^{t} \int_{\Sigma_{\tau}} \left(\log \langle w_{\mp}\rangle\right)^{4}|\nabla^2 z_{\pm}^{(\alpha)}|^{2} dx d\tau, &\text{for }|\alpha|= N_*+3,\\
\ \mu\int_{0}^{t} \int_{\Sigma_{\tau}}|\nabla^2 z_{\pm}^{(\alpha)}|^{2} dx d\tau, &\text{for }|\alpha|= N_*+4.
\end{cases}
\eeno

For any $k\in\N$, we define $E_\pm^{k}(t)=\sum_{|\al|=k}E_\pm^{(\al)}(t)$. The quantities $F_\pm^{k}(t)$ and $D_\pm^{k}(t)$ are defined analogously.

Given $T^* \in [0,\log L]$, we introduce the total energy norms, total flux norms, and total diffusions of each order as follows:
\beno
\begin{aligned}
& E_{\pm}= \sup\limits_{0\leq t \leq T^*}E_{\pm}(t), \quad E_{\pm}^{k}= \sup\limits_{0\leq t \leq T^*} \sum\limits_{|\alpha|=k}E_{\pm}^{(\alpha)}(t),\\
& F_{\pm}= \sup\limits_{0\leq t \leq T^*}F_{\pm}(t), \quad F_{\pm}^{k}= \sup\limits_{0\leq t \leq T^*} \sum\limits_{|\alpha|=k}F_{\pm}^{(\alpha)}(t),\\
& D_{\pm}= \sup\limits_{0\leq t \leq T^*}D_{\pm}(t), \quad D_{\pm}^{k}= \sup\limits_{0\leq t \leq T^*} \sum\limits_{|\alpha|=k}D_{\pm}^{(\alpha)}(t).
\end{aligned} 
\eeno

We also define the total energy and diffusion as follows:
\beno
\begin{aligned}
\mathcal{E}_{\pm}^{w}(t) & := E_{\pm}(t)+ \sum_{k=0}^{N_{*}} E_{\pm}^k(t)+ \mu E_{\pm}^{N_*+1}(t)+ \frac{\mu}{\log L}E_{\pm}^{N_*+2}(t) + \frac{\mu}{\left(\log L\right)^2}E_{\pm}^{N_*+3}(t)+\f{\mu}{L}  E_{\pm}^{N_{*}+4}(t),\\ 
\mathcal{D}_{\pm}^{w}(t) & := D_{\pm}(t)+ \sum_{k=0}^{N_{*}} D_{\pm}^k(t)+ \mu D_{\pm}^{N_*+1}(t)+ \frac{\mu}{\log L}D_{\pm}^{N_*+2}(t) + \frac{\mu}{\left(\log L\right)^2}D_{\pm}^{N_*+3}(t)+\f{\mu}{L}  D_{\pm}^{N_{*}+4}(t).
\end{aligned} 
\eeno

We now present our main results concerning the incompressible MHD system \eqref{eq:MHD} with initial conditions
\beq\label{initial}
z_+|_{t=0}=z_{+,0}(x),\quad z_-|_{t=0}=z_{-,0}(x),\quad\forall\, x\in Q_L
\eeq
in the presence of a strong background magnetic field. And we denote the initial total energy by
\beno 
\begin{aligned}
\mathcal{E}^{w}(0)&:= \sum_{+,-}\Bigl(\left\|\bigl(\log \langle x\rangle\bigr)^2 z_{\pm,0}\right\|_{L^2(Q_L)}^2+\sum_{k=0}^{N_{*}}\left\|\langle x\rangle\bigl(\log \langle x\rangle\bigr)^2\nabla^{k+1} z_{\pm,0}\right\|_{L^2(Q_L)}^2  \\
&\qquad+\mu\left\|\langle x\rangle\bigl(\log \langle x\rangle\bigr)^{2} \nabla^{N_*+2} z_{\pm,0}\right\|_{L^2(Q_L)}^2
+\frac{\mu}{\log L }\left\|\langle x\rangle^{\f12}\bigl(\log \langle x\rangle\bigr)^{2} \nabla^{N_*+3} z_{\pm,0}\right\|_{L^2(Q_L)}^{2}\\
&\qquad +\frac{\mu}{(\log L)^2} \left\|\bigl(\log \langle x\rangle\bigr)^{2} \nabla^{N_*+4} z_{\pm,0}\right\|_{L^2(Q_L)}^2+ \frac{\mu}{L} \left\|\nabla^{N_{*}+5} z_{\pm,0}\right\|_{L^2(Q_L)}^2 \Bigr),
\end{aligned}
\eeno
where  $\langle x\rangle:=(R^2+|x|^2)^{\f12}$.

\medskip

Our first theorem establishes the local existence of solutions to \eqref{eq:MHD} over the time interval $[0,\log L]$:

\begin{theorem}[Local existence of the solution] \label{thm:a1}
Let $B_0 = (0,0,1)$, $N^* \in \mathbb{Z}_{\geq 7}$, $0<\mu\ll1$, $R\gg 1$ and $L\geq e^{\frac{1}{\mu}}$. There exists a universal constant $\varepsilon_0\in (0,1)$, independent of both the viscosity coefficient $\mu$ and the domain scale $L$, such that if the initial divergence-free vector fields  $(z_{+,0}(x),z_{-,0}(x))$ satisfy the periodic boundary conditions and 
\beq\label{initial energy}
\mathcal{E}^{w}(0) \leq \varepsilon_{0}^{2} ,
\eeq
 then \eqref{eq:MHD}-\eqref{initial} admits a unique solution $\big(z_{+}(t,x),z_{-}(t,x)\big)$ on $[0,\log L]\times Q_L$ satisfying the periodic boundary conditions and
\begin{equation}\label{eq:thm2}
\sup_{0\leq t\leq\log L}\mathcal{E}_{\pm}^{w}(t)
+\mathcal{D}_{\pm}^{w}|_{t=\log L}+\bigl(F_{\pm}+ \sum_{k=0}^{N_{*}}F_{\pm}^{k}\bigr)\bigr|_{T^*=\log L}\leq C \mathcal{E}^w(0),
\end{equation}
where $C>0$ is a universal constant independent of  both $\mu$ and $L$. 
\end{theorem}
\begin{remark}
(1). The uniform energy estimate \eqref{eq:thm2} with respect to the domain size $L$ of $Q_L$ is essential for  passing to the limit from the large periodic box $Q_L$ to the whole space $\R^3$.

(2). As $L\rightarrow+\infty$, the box $Q_L$ exhausts $\R^3$,  and the MHD system \eqref{eq:MHD} on $Q_L$ formally converges to the MHD system on $\R^3$. According to this theorem, the existence time $\log L$ of the solution $(z_+,z_-)$ tends to the infinity as $L\rightarrow +\infty$. Moreover, the uniform energy estimates \eqref{eq:thm2} shows that  the contributions from the higher-order terms (of orders $N_*+2$, $N_*+3$ and $N_*+4$) in the total energy functional $\mathcal{E}_{\pm}^{w}(t)$ vanish in the limit. Consequently, the limiting energy estimate as $L\rightarrow+\infty$  coincides with estimate (1.13) in Theorem 1.2 of \cite{He-Xu-Yu}. This suggests that the solution of \eqref{eq:MHD} formally converges to the solution of the MHD system on $\R^3$ as $L\rightarrow+\infty$. A rigorous justification of this limit, along with a convergence rate analysis, will be addressed in future work.
\end{remark}

\subsubsection{Global well-posedness and dynamics}

Our second main theorem establishes the global existence and  dynamics behavior of solutions to \eqref{eq:MHD} with initial condition \eqref{initial}. Before stating the result, we introduce the following two concepts:

{\it (i). The classical $\mu$-small-data parabolic regime.} For a fixed $\mu>0$, the MHD system \eqref{eq:MHD} can be treated—in a manner similar to the Navier–Stokes equations—as a {\it semilinear} heat equations rather than a {\it quasilinear} system. Then the classical approach for the Navier-Stokes equations shows that, there exists a constant $\e_\mu>0$, such that if the $H^2$-norm of the initial data is bounded above by $\e_\mu$, \eqref{eq:MHD} admits a unique global solution in $H^2$ that converges to the steady state of the system. Here $\e_\mu=O(\mu)$ and satisfies $\e_\mu\ll\mu$. We say the solution enters {\it the classical $\mu$-small-data parabolic regime} when its $H^2$-norm is bounded by $\varepsilon_\mu$.

\smallskip 

{\it (ii). The unweighted energy functional.} To describle the dynamics of the solutions, we define the following energy functional:
\beq\label{total general energy}
\mathcal{E}(t)= \sum_{+,-} \bigl(\sum_{|\alpha|\leq N_*+1} \|z_{\pm}^{(\alpha)}\|_{L^2(\Sigma_t)}^2+ \mu \sum_{|\alpha|=N_*+2} \|z_{\pm}^{(\alpha)}\|_{L^2(\Sigma_t)}^2\bigr) .
\eeq

\begin{theorem}[Global existence and dynamics]\label{thm:d1}
Let $B_0 = (0,0,1)$, $N^* \in \mathbb{Z}_{\geq 7}$, $0<\mu\ll1$, $R\gg 1$, $L\geq e^{\f{1}{\mu}}$ and assume that
\beno 
\mathcal{E}^{w}(0)\leq\e_0^2,
\eeno
where $\mathcal{E}^{w}(0)$ and $\e_0$ are given by Theorem \ref{thm:a1}. Assume that $(z_+,z_-)$ is the solution of \eqref{eq:MHD}-\eqref{initial} obtained in Theorem \ref{thm:a1}. Then there exists a minimal domain scale $L_{\mu}\geq e^{\frac{1}{\mu}}$, and  for any $L\geq L_{\mu}$, there exists a constant $\e_1\in(0,\varepsilon_0]$ (independent of $\mu$ and $L$), such that if
\beq\label{initial energy 2}
 \mathcal{E}^{w}(0)\leq \e_1^2,
\eeq   
\begin{enumerate}
\item {\bf (Decay to the small-data parabolic regime)} there exist a universal constant $C_1\geq 1$ (independent of $\mu$ and $L$), and an integer $n_0$, such that for any $1\leq n\leq n_0$,
\begin{equation}\label{decay estimate}
\mathcal{E}(t_n) \leq \left(C_{1}\varepsilon_{0}\right)^{n+2},\quad\text{with }\quad t_n=n T,\quad T:=\f{\log L}{n_0}. 
\end{equation}
In particular, 
\beq\label{decay estimate 1}
\mathcal{E}(t_{n_0}) \leq \left(C_{1}\varepsilon_{0}\right)^{n_{0}+2} < \varepsilon_{\mu}^2. 
\eeq
This indicates that the solution enters the classical small-data parabolic regime at time $\log L$. 
\item {\bf (Global extension)} the local solution $(z_+,z_-)$  on $[0,\log L]$ can be extended to be a unqiue global solution of \eqref{eq:MHD}-\eqref{initial} on $[0,\infty)\times Q_L$ such that
\beq\label{eq:thm 3}
\sup_{t>\log L}\mathcal{E}(t)+\mu\int_{\log L}^\infty\bigl(\sum_{|\alpha|\leq N_*+1} \|\na z_{\pm}^{(\alpha)}\|_{L^2(\Sigma_t)}^2+ \mu \sum_{|\alpha|=N_*+2} \|\na z_{\pm}^{(\alpha)}\|_{L^2(\Sigma_t)}^2\bigr)dt\leq C\e_\mu^2.
\eeq
\end{enumerate}
\end{theorem}
\begin{remark}
(1). The global existence follows directly from the local existence theory established in Therorem \ref{thm:a1} combined with the decay estimates in \eqref{decay estimate}. 

(2). In our proof of this theorem (presented in the final section 3), we will iteratively apply the following key decay estimate:
\begin{equation}\label{linear driven decay estimates}
\begin{aligned}
\mathcal{E}(t_n)
&\leq C_0\Bigl(\frac{\log\big(\log(\mu T+e)+e\big)}{\log\big(\mu T+e\big)}+\f{2^{n}}{\log\big(\log(\mu T+e)+e\big)}\Bigr)\cdot\mathcal{E}^w(0)+C_0\bigl(\mathcal{E}^w(0)\bigr)^{\f12}\cdot \mathcal{E}(t_{n-1}),
\end{aligned} 
\end{equation}
for $1\leq n\leq n_0$.
This  decay estimate demonstrates the linear-driving decay mechanism of the solutions, which ultimately drives them  into the $\mu$-small-data parabolic regime, guaranteeing the global existence.

(3). We prove the existence of a domain-scale threshold $L_\mu$, leaving its optimality for future study.
\end{remark}

Let us complete this section with the notations we shall use in this context.

\medskip

\noindent{\bf Notations:} For $p\in[1,\infty]$, without confusion of domain, the norms $\|f \|_{L^p_x},\,$ $\|f\|_{L^p_{(x_1,x_2)}},\,$ $\|f\|_{L^p_{(x_1,x_3)}},\,$ $\|f\|_{L^p_{(x_2,x_3)}},\,$ $\|f\|_{L^p_t}$, mean $\|f(x)\|_{L^p(Q_L)},\,$ $\|f(x_1,x_2)\|_{L^p([-L,L]^2)},\,$  $\|f(x_1,x_3)\|_{L^p([-L,L]^2)},\,$ $\|f(x_2,x_3)\|_{L^p([-L,L]^2)},\,$ $\|f(\tau)\|_{L^p([0,t])}$ respectively. Similarly, we shall also use the notations $\|\cdot\|_{L^p_{(t,x)}},\, \|\cdot\|_{L^p_{(t,x_1,x_2)}},\, \|\cdot\|_{L^p_t(L^q_x)},\,\|\cdot\|_{L^p_t(L^q_{(x_1,x_2)})},\, \|\cdot\|_{L^p_t(L^q_{(x_1,x_3)})} $ etc. for any $p,q\in[1,\infty]$. The notation $f\lesssim g$ means that there exists an universal constant $C>0$ such that $f\leq Cg$. While the notation $g\gtrsim f$ means $f\lesssim g$ and  $f\sim g$ represents  both $f\lesssim g$ and $g\lesssim f$.  The constant $C>0$ denotes a universal constant which may change from line to line. 

\section{Main A Priori Estimates and Proof of Theorem \ref{thm:a1}}
\subsection{Ansatz for the method of continuity}
We prove Theorem \ref{thm:a1} by the standard method of continuity. Fix an  integer $N_*\geq 7$ and consider the initial data of $z_{\pm}$ with energy $\mathcal{E}^w(0)=\varepsilon^2$, where the small constant $\varepsilon>0$  will be determined later. Our analysis relies on three sets of assumptions concerning the underlying geometry and the energy of the waves. 

The first ansatz is on the underlying geometry:
\begin{equation}\label{eq:a1}
\Bigl|\Bigl(\frac{\partial x_{i}^{\pm}}{\partial x_{j}}\Bigr)-I\Bigr|\leq 2C_{0}\varepsilon\leq \frac{1}{100}, \quad \Bigl|\nabla \Bigl(\frac{\partial x_{i}^{\pm}}{\partial x_{j}}\Bigr)\Bigr|\leq 2C_{0}\varepsilon\leq \frac{1}{100},\quad \forall\,  (t,x)\in [0, T^{*}]\times \mathbb{R}^3 , 
\end{equation}
where $I$ is the $3\times3$ identity matrix and $C_{0}>0$ is a universal constant which will be determined at the end of the proof.

The second ansatz is about the amplitude of $z_{\pm}$:
\begin{equation}\label{eq:a2}
\|z_{\pm}\|_{L^{\infty}_x} \leq \frac{1}{2}, \quad \forall\, t\in [0, T^{*}]. 
\end{equation}

The third ansatz is concerning the total weighted energies, fluxes and diffusions:
\begin{equation}\label{eq:a3}
\sup_{0\leq t\leq T^*}\mathcal{E}_{\pm}^{w}(t)
+\mathcal{D}_{\pm}^{w}|_{t= T^*}+F_{\pm}+ \sum_{k=0}^{N_{*}}F_{\pm}^{k}\leq 2C_{1}\varepsilon^2,
\end{equation}
where $C_1\geq 1$ is a universal constant which will be determined by the energy estimates.

Since \eqref{eq:a1} and \eqref{eq:a3} hold for the initial data, they remain valid for a short time, say $[0,t_{\text{max}}]$, where $t_{\text{max}}$ is the maximal possible time so that the three sets of ansatz remain valid. Without loss of generality, we take $t_{\text{max}}=T^*$. To close the continuity argument, we shall show that there exists a constant $\varepsilon_{0}>0$, such that for all $\varepsilon\leq \varepsilon_{0}$ and $T^*=\log L$, the constant 2 in \eqref{eq:a1} and \eqref{eq:a3} can be improved to 1, i.e.,
\begin{equation*}
\Bigl|\Bigl(\frac{\partial x_{i}^{\pm}}{\partial x_{j}}\Bigr)-I\Bigr|\leq C_{0}\varepsilon, \quad \Bigl|\nabla \Bigl(\frac{\partial x_{i}^{\pm}}{\partial x_{j}}\Bigr)\Bigr|\leq C_{0}\varepsilon,\quad \forall\,  (t,x)\in [0, T^{*}]\times \mathbb{R}^3 , 
\end{equation*}
\begin{equation*}
\sup_{0\leq t\leq T^*}\mathcal{E}_{\pm}^{w}(t)
+\mathcal{D}_{\pm}^{w}|_{t=T^*}+F_{\pm}+ \sum_{k=0}^{N_{*}}F_{\pm}^{k}\leq C_{1}\varepsilon^2.
\end{equation*}
Once we complete these improvements, the continuity argument hints that the total weighted energy remains uniformly bounded for $t\leq \log L$. Consequently, the local existence result yields that the solution exists beyond $T^*=\log L$. We emphasize that the smallness of $\e_0$ is independent of $\mu$ and $L$.

\subsection{Preliminary estimates}
In this subsection, we assume that the bootstrap assumptions \eqref{eq:a1}-\eqref{eq:a3} hold and $L\geq e^{\f{1}{\mu}}\geq R\gg 1$.

\subsubsection{Analysis of the weights}

Let $\psi_{\pm}(t,y)=\big(\psi_{\pm}^{1}(t,y),\psi_{\pm}^{2}(t,y),\psi_{\pm}^{3}(t,y)\big)$ be the flow generated by $Z_{\pm}$, i.e.
\begin{equation}
\frac{d}{dt}\psi_{\pm}(t,y)= Z_{\pm}\big(t, \psi_{\pm}(t,y)\big), \quad \psi_{\pm}(0,y)=y,
\label{eq:e2star}
\end{equation}
where $y\in \mathbb{R}^3$. Since $z_{\pm}=Z_{\pm}\mp B_{0}$ ($B_0=(0,0,1)$), after integrating the time variable, we have
\begin{equation}
\psi_{\pm}(t,y)= y+ \int_{0}^{t} Z_{\pm}\big(\tau, \psi_{\pm}(\tau,y)\big) d\tau= y\pm tB_{0}+ \int_{0}^{t} z_{\pm}\big(\tau, \psi_{\pm}(\tau,y)\big) d\tau. \label{eq:e2}
\end{equation}
Let $\frac{\partial \psi_{\pm}(t,y)}{\partial y}$ be the differential of $\psi_{\pm}(t,y)$ (a $3\times 3$ Jacobian matrix) at y. By definition, we have $x^{\pm}\big(t, \psi_{\pm}(t,y)\big)=y$, which further implies that
\begin{equation*}
\begin{aligned}
\frac{\partial x^{\pm}}{\partial x}\Big|_{x=\psi_{\pm}(t,y)}&= \Bigl(\frac{\partial \psi_{\pm}(t,y)}{\partial y}\Bigr)^{-1},\\
\nabla_x \Bigl(\frac{\partial x^{\pm}}{\partial x}\Bigr)\Big|_{x=\psi_{\pm}(t,y)}&= \nabla_y \Bigl(\Bigl(\frac{\partial \psi_{\pm}(t,y)}{\partial y}\Bigr)^{-1}\Bigr) \Bigl(\frac{\partial \psi_{\pm}(t,y)}{\partial y}\Bigr)^{-1} .
\end{aligned}
\end{equation*}
Therefore, the geometric ansatz \eqref{eq:a1} can be rephrased as
\begin{equation}
\Bigl|\frac{\partial \psi_{\pm}(t,y)}{\partial y}-I\Bigr|\leq 2C_{0}^{\prime}\varepsilon\leq \frac{1}{10}, \quad \Bigl|\nabla_y \Bigl(\frac{\partial \psi_{\pm}(t,y)}{\partial y}\Bigr)\Bigr|\leq 2C_{0}^{\prime}\varepsilon\leq \frac{1}{10},\quad \forall\, (t,y)\in [0, T^{*}]\times \mathbb{R}^3 . \label{eq:e4}
\end{equation}

Under the ansatz \eqref{eq:a1}, the direct calculation gives rise to the following bounds on the weight functions:
\begin{lemma} \label{lem:e1}
 We have 
\begin{equation}
\left|\nabla^{i} \langle \overline{w_{\pm}}\rangle\right|\leq 2 \quad \text{for }i=1,2. \label{eq:e5}
\end{equation}
Moreover, for all $s_1,s_2 \in \mathbb{R}$, we have for $i=1,2$ that
\beq\label{eq:e6}
\begin{aligned}
&\left|\nabla^{i} \langle \overline{w_{\pm}}\rangle^{s_1}\right|\lesssim \langle \overline{w_{\pm}}\rangle^{s_1-1}, \quad
\left|\nabla^{i} \left(\log \langle \overline{w_{\pm}}\rangle\right)^{s_1}\right|\lesssim \frac{\left(\log \langle \overline{w_{\pm}}\rangle\right)^{s_1-1}}{\langle \overline{w_{\pm}}\rangle},\\
&\left|\nabla^{i} \big(\langle \overline{w_{+}}\rangle^{s_1} \langle \overline{w_{-}}\rangle^{s_2}\big)\right|\lesssim \frac{\langle \overline{w_{+}}\rangle^{s_1} \langle \overline{w_{-}}\rangle^{s_2}}{R},\\
&\left|\nabla^{i} \big(\langle \overline{w_{\pm}}\rangle^{s_1} \left(\log \langle \overline{w_{\pm}}\rangle\right)^{s_2}\big)\right|\lesssim \langle \overline{w_{\pm}}\rangle^{s_1-1} \left(\log \langle \overline{w_{\pm}}\rangle\right)^{s_2}.
\end{aligned} 
\eeq
\end{lemma}

\begin{lemma} \label{lem:e2}
For all $(t,x)\in [0,T^*]\times Q_L$, there hold
\begin{itemize}
\item[(1).] for $\f{L}{3}\leq |x_3|\leq L$,  $\frac{L}{4}< |u_{\pm}(t,x)|\leq \frac{9L}{8}$;
\item[(2).] for $|x_3|\leq \frac{5L}{6}$,  $u_{\pm}(t,x)=x_{3}^{\pm}(t,x)$ and $|u_{\pm}(t,x)|\leq \frac{7L}{8}$; 
\item[(3).] 
$\langle w_{\pm}\rangle(t,x) \sim \langle \overline{w_{\pm}}\rangle(t,x)$ and $\langle \overline{w_{\pm}}\rangle(t,x)\leq 3 L$.
\end{itemize}
\end{lemma}
\begin{proof}
It suffices to give the estimates for $u_{-}$ and $\langle w_{-}\rangle$, the estimates for $u_{+}$ and $\langle w_{+}\rangle$ can be derived in a similar way. 

{\bf Step 1. Estimate of $x_3^-(t,x)-x_3$.} Since $L_{-} x_{3}^{-}=0$ and $L_{-}=\partial_{t}+ Z_{-}\cdot \nabla$, we have
\begin{equation*}
(\partial_t-\partial_3)x_{3}^{-} + z_{-} \cdot \nabla x_{3}^{-}=0, \quad x_{3}^{-}\big|_{t=0} = x_{3}.
\end{equation*}
By taking $(t,x_1,x_2,x_3)=(t,y_1,y_2,y_3-t)$, we get
\begin{equation*}
\frac{d}{dt}\left(x_{3}^{-}(t,y_1,y_2,y_3-t)\right)+ \left(z_{-} \cdot \nabla x_{3}^{-}\right)(t,y_1,y_2,y_3-t)= 0,\quad x_{3}^{-}(0,y_1,y_2,y_3)= y_3,
\end{equation*}
which implies
\begin{equation}\label{eq:e7}
x_{3}^{-}(t,y_1,y_2,y_3-t)+ \int_{0}^{t} \left(z_{-} \cdot \nabla x_{3}^{-}\right)(\tau,y_1,y_2,y_3-\tau)\,d\tau = x_{3}^{-}(0,y_1,y_2,y_3)= y_3. 
\end{equation}
Thus, for all $(t,x)\in [0, T^*]\times\R^3$, by setting $(y_1,y_2,y_3)=(x_1,x_2,x_3+t)$, we deduce from \eqref{eq:e7} that
\beno\begin{aligned}
\left|x_{3}^{-}(t,x)-(x_{3}+t)\right| &\leq \int_{0}^{t} \left|z_{-} \cdot \nabla x_{3}^{-}(\tau,x_1,x_2,x_3+t-\tau)\right|\,d\tau\\ 
&
\overset{\eqref{eq:a1}}{\leq} \int_0^t|z_-(\tau,x_1,x_2,x_3+t-\tau)|\,d\tau\cdot(1+2C_0\e),
\end{aligned}\eeno
which along with \eqref{eq:a2} and $\e\ll1$ implies 
\beq\label{eq:e8}
\left|x_{3}^{-}(t,x)-(x_{3}+t)\right|\leq\f32\int_0^t|z_-(\tau,x_1,x_2,x_3+t-\tau)|\,d\tau \leq \f34t.
\eeq
Then for all $(t,x)\in [0, T^*]\times\R^3$,
\begin{equation}\label{eq:e9}
\left|x_{3}^{-}(t,x)-(x_{3}+t)\right| \leq \frac{3}{4}t,
\quad\text{i.e.,}\quad x_{3}+\frac{1}{4}t \leq x_{3}^{-}(t,x) \leq x_{3}+\frac{7}{4}t. 
\end{equation}

{\bf Step 2. Proof of (1).}
For any $(t,x)\in [0,T^*]\times Q_L$, using  \eqref{eq:e9} and $T^* =\log L \ll L$ , we have

\smallskip 

{\bf Case (i).} when $-L\leq x_{3}\leq -\frac{L}{3}$,  $-L \leq x_{3}^{-}(t,x) \leq -\frac{L}{4}$, 
which yields 
\beno
u_{-}(t,x)=x_{3}^{-}(t,x)\quad\text{ and}\quad |u_{-}(t,x)|> \frac{L}{4};
\eeno 

{\bf Case (ii). }when $\frac{L}{3}\leq x_{3}\leq L$, $\frac{L}{3} \leq x_{3}^{-}(t,x) \leq \frac{9L}{8}$, which hints
\beno\begin{aligned}
\text{either}\quad &u_{-}(t,x)=x_{3}^{-}(t,x)\in[\f{L}{3},L), \\ 
\text{ or }\quad &u_{-}(t,x)=x_{3}^{-}(t,x_1,x_2,x_3-2L),\quad\text{if}\quad x_{3}^{-}(t,x)\in [L,x_{3}+\frac{7}{4}t].
\end{aligned}\eeno
For the latter case where $L \leq x_{3}^{-}(t,x) \leq x_{3}+\frac{7}{4}t $, we derive
\begin{equation*}
-L-\frac{7}{4}t \leq x_{3}-2L \leq -L,
\end{equation*}
which along with \eqref{eq:e9} implies  
\begin{equation*}
-\frac{9L}{8}\leq -L-\f{3}{2}t \leq u_{-}(t,x)=x_{3}^{-}(t,x_1,x_2,x_3-2L)  \leq -L+\f{7}{4}t\leq-\frac{7L}{8}.
\end{equation*}

Based on the above analysis, we conclude that 
\beno
\frac{L}{4}< |u_{-}(t,x)|\leq \frac{9L}{8},\quad\forall\, (t,x)\in [0,T^*]\times Q_L,\, \f{L}{3}\leq |x_3|\leq L. 
\eeno 
This completes the proof of (1).

\smallskip 

{\bf Step 3. Proof of (2).} For any $(t,x)\in [0,T^*]\times Q_L$ with $|x_3|\leq \frac{5L}{6}$, using \eqref{eq:e9} and $T^*=\log L \ll L$, we get 
\begin{equation*}
-\frac{5L}{6} \leq x_{3} \leq x_{3}^{-}(t,x) \leq x_{3}+\frac{7}{4}\log L \leq \frac{7L}{8},
\end{equation*}
which yields  $u_{-}(t,x)=x_{3}^{-}(t,x)$ and $|u_{-}(t,x)|\leq \frac{7L}{8}$. This closes the proof of (2).

\smallskip 

{\bf Step 4. Proof of (3).}
For $i=1,2$, similar derivation as \eqref{eq:e8} leads to 
\beq\label{eq:e9a}
|x_i^{\pm}(t,x)-x_i|\leq\f32\int_0^t|z_\pm (\tau,x_1,x_2,x_3+t-\tau)|\,d\tau\leq \f34t\leq\f{L}{20},\quad\forall\,(t,x)\in[0,T^*]\times\R^3,
\eeq
which gives rise to
\beq\label{eq:e9b}
|x_i^{\pm}(t,x)|\leq L+\f{L}{20}=\f{21}{20}L,\quad\forall\,(t,x)\in[0,T^*]\times Q_L.
\eeq 

Combining the results in (1),(2) and \eqref{eq:e9b}, for any $(t,x)\in[0,T^*]\times Q_L$, we get either $\langle w_{\pm}\rangle= \langle \overline{w_{\pm}}\rangle$ or
\begin{equation*}
\begin{aligned}
&\f{L}{4}\leq\langle w_{\pm}\rangle(t,x)=\big(R^{2}+\left|x_{1}^{\pm}\right|^{2}+\left|x_{2}^{\pm}\right|^{2}+\left|u_{\pm}\right|^{2}\big)^{\frac{1}{2}}(t,x)\overset{R\leq L}{\leq}3L,\\
\text{and}\quad & L\leq\langle \overline{w_{\pm}}\rangle(t,x)=\big(R^{2}+\left|x^{\pm}\right|^{2}\big)^{\frac{1}{2}}(t,x)\overset{R\leq L}{\leq}3L,
\end{aligned}
\end{equation*}
which implies that $\langle w_{\pm}\rangle\sim \langle \overline{w_{\pm}}\rangle$ on $[0,T^*]\times Q_L$. 
This ends the proof of the lemma.
\end{proof}
\begin{remark}\label{rmk for u pm}
(1). As a consequence of \eqref{eq:e9} and \eqref{eq:e9a}, for any $(t,x)\in[0,T^*]\times Q_L$, there holds 
\beq\label{eq:e9d}
|x_i^\pm(t,x)-x_i|\leq\f{L}{20},\quad\text{for}\,\, i=1,2,3. 
\eeq 
(2). Thanks to the $2L$-periodicity of $Z_-$ and the definition of characteristic coordinates in \eqref{def of x 3 pm}, the solution for $x_3^-(t,x)$ satisfies
\beno
x_3^-(t,x)|_{x=\psi_-(t,y)}=y_3,\quad\forall\, (t,y)\in[0,T^*]\times\R^3,
\eeno 
where $\psi_-(t,y)$ is the flow map generated by $Z_-$ (defined in \eqref{eq:e2star}). Moreover, there holds
\beq\label{eq:e9c}\begin{aligned}
&x_3^-(t,x-2Le_3)-x_3^-(t,x)|_{x=\psi_-(t,y)}=-2L,\\
\text{i.e.,}\quad& x_3^-(t,x-2L e_3)-x_3^-(t,x)=-2L,\quad\forall\, (t,x)\in[0,T^*]\times\R^3,
\end{aligned}\eeq
where $e_3=(0,0,1)$.

For $(t,x)\in[0,T^*]\times Q_L$ with $L\gg1$, the estimate \eqref{eq:e9} yields to
\beno 
x_3^-(t,x)\in[-L,2L),\quad x_3^-(t,x)|_{x_3=-L}\in[-L,0),\quad 
x_3^-(t,x)|_{x_3=L}\in[L,2L),
\eeno
which along with \eqref{eq:e9c} and the definition of $u_-(t,x)$ in \eqref{eq:g4} implies that
\beno 
u_-(t,x)\in[-L,L)\quad\text{and}\quad u_-|_{x_3=L}=u_-|_{x_3=-L}.
\eeno
Similar argument also holds for $u_+$. Thus, \eqref{prop for u pm} holds for both $u_+$ and $u_-$.
\end{remark}

As an application of the above two lemmas, we have the following weighted Sobolev inequalities:
\begin{lemma}[weighted Sobolev inequality] \label{lem:e3}
For all integers $0\leq k\leq N_{*}-2$ and multi-indices $\alpha$ with $|\alpha|=k$, we have, for all $(t,x)\in [0,T^*] \times Q_L$, 
\begin{equation}\label{eq:e11}
\left|z_{\pm}(t,x)\right| \lesssim \frac{\left(E_{\pm}+E_{\pm}^0+ E_{\pm}^1\right)^{\frac{1}{2}}}{\left(\log \langle w_{\mp}\rangle\right)^2(t,x)} , 
\end{equation}
\begin{equation}\label{eq:e12}
\bigl|\nabla z_{\pm}^{(\alpha)}(t,x)\bigr| \lesssim \frac{\left(E_{\pm}^{k}+ E_{\pm}^{k+1}+ E_{\pm}^{k+2}\right)^{\frac{1}{2}}}{\big(\langle w_{\mp}\rangle\left(\log \langle w_{\mp}\rangle\right)^2\big)(t,x)}. 
\end{equation}
\end{lemma}
\begin{proof}
We only verify the inequalities for $z_{+}$, as those for $z_{-}$ can be derived in the same manner. By the standard Sobolev inequality $\|f\|_{L^\infty(Q_L)}\lesssim\|f\|_{H^2(Q_L)}$  with $L\gg 1$, for all $(t,x)\in [0,T^*] \times Q_L$, we have
\begin{equation*}
\bigl|\left(\log \langle \overline{w_{-}}\rangle\right)^2z_{+}(t,x)\bigr|^2 \lesssim \sum_{|\beta|\leq 2} \bigl\|\partial^{\beta}\big(\left(\log \langle \overline{w_{-}}\rangle\right)^2z_{+}\big)\bigr\|_{L^2(\Sigma_{t})}^2.
\end{equation*}

Using Lemma \ref{lem:e1}, we get
\begin{equation*}
\begin{aligned}
\bigl|\partial^{\beta}\big(\left(\log \langle \overline{w_{-}}\rangle\right)^2z_{+}\big)\bigr|& \lesssim \sum_{\gamma\leq \beta} \bigl|\partial^{\gamma}\big(\left(\log \langle \overline{w_{-}}\rangle\right)^2\big) z_{+}^{(\beta-\gamma)}\bigl|\lesssim \sum_{\gamma\leq \beta} \bigl|\left(\log \langle \overline{w_{-}}\rangle\right)^2 z_{+}^{(\beta-\gamma)}\bigr|.
\end{aligned}
\end{equation*}
Thanks to Lemma \ref{lem:e2}, we obtain $\langle w_{-}\rangle\sim \langle \overline{w_{-}}\rangle$ and
\begin{equation*}
\bigl|\left(\log \langle w_{-}\rangle\right)^2z_{+}(t,x)\bigr|^2 \lesssim \sum_{|\alpha|\leq 2} \bigl\|\left(\log \langle w_{-}\rangle\right)^2 z_{+}^{(\alpha)}\bigr\|_{L^{2}(\Sigma_{t})}^2 \lesssim E_{+}+E_{+}^0+ E_{+}^1 .
\end{equation*}
This verifies the bound for $z_+$ in \eqref{eq:e11}.

Similarly, for higher order derivatives with $|\alpha|=k$, we have
\beno
\bigl|\langle \overline{w_{-}}\rangle \left(\log \langle \overline{w_{-}}\rangle\right)^2\nabla z_{+}^{(\alpha)}(t,x)\bigr|^2\lesssim \sum_{|\beta|\leq 2} \bigl\|\partial^{\beta}\big(\langle \overline{w_{-}}\rangle (\log \langle \overline{w_{-}}\rangle)^2\nabla z_{+}^{(\alpha)}\big)\bigr\|_{L^{2}(\Sigma_{t})}^2,
\eeno
which along with Lemma \ref{lem:e1} and the fact $\langle w_{-}\rangle\sim \langle \overline{w_{-}}\rangle$ (stated in Lemma \ref{lem:e2}) implies
\beno\begin{aligned}
\bigl|\langle w_{-}\rangle \left(\log \langle w_{-}\rangle\right)^2\nabla z_{+}^{(\alpha)}(t,x)\bigr|^2
\lesssim \sum_{k\leq |\beta|\leq k+2} \bigl\|\langle w_{-}\rangle \left(\log \langle w_{-}\rangle\right)^2\nabla z_{+}^{(\beta)}\bigr\|_{L^{2}(\Sigma_{t})}^2\lesssim E_{+}^{k}+ E_{+}^{k+1}+ E_{+}^{k+2}. 
\end{aligned}\eeno
This proves the bound for $z_+^{(\al)}$ in \eqref{eq:e12}.
It completes the proof of the lemma.
\end{proof}
The following lemma, combined with weighted energy estimates, demonstrates the separation property of the left-traveling and right-traveling waves.
\begin{lemma} \label{lem:e4}
Assume that $R>10$ and $\|z_\pm\|_{L^\infty_x}\leq\f12$ for any $t\in[0,T^*]$ , we have
\begin{equation}
\begin{aligned}
\langle w_{+}\rangle \langle w_{-}\rangle &\geq \big(R^{2}+\left|u_{+}\right|^{2}\big)^{\frac{1}{2}} \big(R^{2}+\left|u_{-}\right|^{2}\big)^{\frac{1}{2}} \geq \frac{R}{2} \left(R^2+t^2\right)^{\frac{1}{2}},\\
\log \langle w_{+}\rangle \log \langle w_{-}\rangle &\geq \log \big(R^{2}+\left|u_{+}\right|^{2}\big)^{\frac{1}{2}} \log \big(R^{2}+\left|u_{-}\right|^{2}\big)^{\frac{1}{2}} \geq \frac{\log R}{2} \log \big(R^2+t^2\big)^{\frac{1}{2}} .
\end{aligned} \label{eq:e13}
\end{equation}
\end{lemma}
\begin{proof}
By virtue of $\psi_\pm$ (defined in \eqref{eq:e2}), we solve $x_{3}^{\pm}$ from $L_\pm x_{3}^{\pm}=0$ as follows
\begin{equation*}
x_{3}^{\pm}\big(t, \psi_{\pm}(t,y)\big)=y_{3}= \psi_{\pm}^{3}(t,y)\mp t- \int_{0}^{t} z_{\pm}^{3}\big(\tau, \psi_{\pm}(\tau,y)\big) d\tau.
\end{equation*}
For any $(t,x)\in [0,T^*] \times Q_L$,  using $x=\psi_{\pm}(t,y)$ with $y\in \mathbb{R}^3$,  we get
\begin{equation*}
x_{3}^{\pm}(t,x)=x_{3}\mp t- \int_{0}^{t} z_{\pm}^{3}\big(\tau, \psi_{\pm}\big(\tau, \psi_{\pm}^{-1}(t,x)\big)\big) d\tau.
\end{equation*}

According to Lemma \ref{lem:e2}, if $\frac{L}{3}\leq |x_3|\leq L$, then $|u_{\pm}(t,x)|>\frac{L}{4}$ and \eqref{eq:e13} is obviously valid; if $|x_3|\leq \frac{L}{3}$, then $u_{\pm}(t,x)=x_{3}^{\pm}(t,x)$, thereby 
\begin{equation*}
\left|(u_{-}-u_{+})-2t\right|\leq \int_{0}^{t} \big(\left\|z_{+}^{3}\right\|_{L^{\infty}_x}+ \left\|z_{-}^{3}\right\|_{L^{\infty}_x}\big)\, d\tau\stackrel{\|z_\pm\|_{L^\infty_x}\leq\f12}{\leq} t,
\end{equation*}
which gives rise to
\begin{equation*}
t\leq \left|u_{-}-u_{+}\right|\leq 3t.
\end{equation*}
This implies that either $|u_{+}|\geq \frac{t}{2}$ or $|u_{-}|\geq \frac{t}{2}$, then \eqref{eq:e13} also holds for $|x_3|\leq \frac{L}{3}$. The lemma is proved.
\end{proof}

\subsubsection{Technical lemmas involving characteristic hypersurfaces.}
We now present a lemma to control the normal derivatives of the characteristic hypersurfaces $\overline{C}_{a}^\pm$ and a trace theorem for function restrictions to these hypersurfaces.
\begin{lemma} \label{lem:e5}
Under assumptions \eqref{eq:a1} and \eqref{eq:a2}, for all $a\in \mathbb{R}$, we have
\begin{equation}
\frac{7}{16} \leq \langle L_{-} , \nu_{+}\rangle\big{|}_{\overline{C}_{a}^{+}} \leq 4 , \quad \frac{7}{16} \leq \langle L_{+} , \nu_{-}\rangle\big{|}_{\overline{C}_{a}^{-}} \leq 4, \label{eq:e15}
\end{equation}
where $\nu_{\pm}$ is the unit normal vector field of $\overline{C}_{a}^{\pm}$.
\end{lemma}

\begin{lemma}[Trace] \label{lem:e6}
Under assumptions \eqref{eq:a1} and \eqref{eq:a2}, for all $f(t, x)\in L^{2}\big([0, T^{*}]; H^{1}(Q_{L})\big)$, the restriction of $f$ to $\overline{C}_{a}^{\pm}$ belongs to $L^{2}\big(\overline{C}_{a}^{\pm}\big)$. In fact, we have
\begin{equation}
\|f\|_{L^{2}(\overline{C}_{a}^{\pm})} \lesssim \|f\|_{L^{2}\big([0, T^{*}]; H^{1}(Q_{L})\big)} , \label{eq:e16}
\end{equation}
and
\begin{equation}
d\sigma_{\pm}=\big(\sqrt{2}+O(\varepsilon)\big) dx_{1} dx_{2} dt, \label{eq:e17}
\end{equation}
where $d\sigma_\pm$ is the surface measure on $\overline{C}_{a}^{\pm}$.
\end{lemma}
The detailed proofs of Lemmas \ref{lem:e5} and \ref{lem:e6} are similar to those of Lemmas 2.5 and 2.9 in \cite{He-Xu-Yu}. 
\begin{remark}
For all $(t,x)\in [0,T^*]\times Q_L$ satisfying $-\frac{L}{4}\leq u_{\pm}(t,x)\leq \frac{L}{4}$, Lemma \ref{lem:e2}(1) yields  $|x_3|< \frac{L}{3}$ and consequently $u_{\pm}(t,x)=x_{3}^{\pm}(t,x)$. This implies that
\beno
C_{u_\pm}^{\pm}=\overline{C}_{u_\pm}^{\pm}, \quad\text{for}\,\, -\frac{L}{4}\leq u_{\pm}\leq \frac{L}{4}.
\eeno
As a result, the energy fluxes of $z_{\pm}$ can be equivalently defined as 
\begin{equation}
\begin{aligned}
F_{\pm}(t) & = \sup_{|u_{\pm}|\leq \frac{L}{4}}\int_{C_{t,u_{\pm}}^{\pm}} \left(\log \langle w_{\mp}\rangle\right)^{4}\left|z_{\pm}\right|^{2} d\sigma_{\pm},\\
F_{\pm}^{(0)}(t) & = \sup_{|u_{\pm}|\leq \frac{L}{4}}\int_{C_{t,u_{\pm}}^{\pm}} \langle w_{\mp}\rangle^{2}\left(\log \langle w_{\mp}\rangle\right)^{4}\left|\nabla z_{\pm}\right|^{2} d\sigma_{\pm},\\
F_{\pm}^{(\alpha)}(t) & = \sup_{|u_{\pm}|\leq \frac{L}{4}}\int_{C_{t,u_{\pm}}^{\pm}} \langle w_{\mp}\rangle^{2}\left(\log \langle w_{\mp}\rangle\right)^{4}\bigl|j_{\pm}^{(\alpha)}\bigr|^{2} d\sigma_{\pm}, \ \text{for }1\leq |\alpha| \leq N_*.
\end{aligned} \label{eq:e18}
\end{equation}
Moreover, Lemma \ref{lem:e6} shows that for $-\frac{L}{4}\leq u_{\pm}\leq \frac{L}{4}$, there holds
\begin{equation}\label{eq:e19}
\|f\|_{L^{2}(C_{u_\pm}^{\pm})} \lesssim \|f\|_{L^2([0, T^*]; H^{1}(Q_{L}))} . 
\end{equation}
\end{remark}

\subsubsection{Estimates of the weights involving boundaries.} Since the weight functions are piecewise smooth and non-periodic in the $(x_1,x_2)$-directions, integration by parts in the energy estimates yields boundary terms. To control these boundary contributions, we establish estimates for certain difference terms of the weight functions at the boundaries. The weight functions for $z_\pm$ on $[0,T^*]\times Q_L$ are constructed as follows:
\beq\label{weight:w1}
\lambda_\pm(t,x):=\Lambda(\langle w_{\mp}\rangle),\quad\text{with}\quad \Lambda(r)=(\log r)^2,\,\,(\log r)^4,\,\,  r^2(\log r)^4\quad  \text{or}\quad r(\log r)^4,\quad\forall\,r>0.
\eeq
The property \eqref{prop for u pm} of $u_\pm$ together with the periodicity property \eqref{eq:g16} shows that
\begin{equation}\label{eq:l5}
\lambda_{+}(t, x_{1}, x_{2}, -L)=\lambda_{+}(t, x_{1}, x_{2}, L), \quad \lambda_{-}(t, x_{1}, x_{2}, -L)=\lambda_{-}(t, x_{1}, x_{2}, L). 
\end{equation}

We introduce the smooth positive functions on $\left[0, T^{*}\right] \times \mathbb{R}^{3}$ as follows:
\beno
\lambda_{\pm}^{1}(t,x)=\Lambda(\langle\overline{w_\mp}\rangle)(t,x),
\quad \lambda_{\pm}^{2}(t, x)=\Lambda(\langle\overline{w_\mp}\rangle)(t,x\mp 2Le_3).
\eeno 
The definitions of $u_\pm$ yield that 
\begin{equation}\label{eq:l2}
\lambda_{+}(t,x)=
\begin{cases}
 \Lambda(\langle\overline{w_-}\rangle)(t,x)=\lambda_{+}^{1}(t,x), \quad \text{if }\,x_{3}^{-}(t,x)< L,\\
 \Lambda(\langle\overline{w_-}\rangle)(t,x-2Le_3)=\lambda_{+}^{2}(t,x), \quad \text{if }\,x_{3}^{-}(t,x)\geq L,
\end{cases} 
\end{equation}
\begin{equation}\label{eq:l3}
\lambda_{-}(t,x)=
\begin{cases}
 \Lambda(\langle\overline{w_+}\rangle)(t,x)=\lambda_{-}^{1}(t,x), \quad \text{if }\,x_{3}^{+}(t,x)> -L,\\
 \Lambda(\langle\overline{w_+}\rangle)(t,x+2Le_3)=\lambda_{-}^{2}(t,x), \quad \text{if }\,x_{3}^{+}(t,x)\leq -L.
\end{cases} 
\end{equation}
The definitions of $x^\pm(t,x)$ and Lemma \ref{lem:e1} yields that for $i=1,2$,
\begin{equation}\label{eq:l4}
L_{-} \lambda_{+}^{i}(t,x)=0, \quad L_{+} \lambda_{-}^{i}(t,x)=0,\quad\forall(t,x)\in[0,T^*]\times\R^3,  
\end{equation}
\begin{equation}\label{eq:l6}
 \left|\nabla \lambda_\pm^{i}(t, x)\right|\lesssim \left|\lambda_\pm^{i}(t, x)\right| \quad\text{and}\quad   \lambda_\pm(t, x)\sim \lambda_\pm^{1}(t, x),\quad\forall (t, x)\in \left[0, T^{*}\right] \times Q_{L}.
\end{equation}

With above notations and properties, we obtain the following estimates:

\begin{lemma}\label{prop:b1}
Under the assumptions \eqref{eq:a1}--\eqref{eq:a3}, we have for all $t\in[0,T^*]$,
\begin{equation}\label{eq:b1}
\begin{aligned}
&\Bigl\|\frac{\lambda_{\pm}|_{x_1=L}-\lambda_{\pm}|_{x_1=-L}}{\lambda_{\pm}|_{x_1=L}}\Bigr\|_{L^\infty_{(t, x_{2}, x_{3})}}
+\Bigl\|\frac{\lambda_{\pm}|_{x_2=L}-\lambda_{\pm}|_{x_2=-L}}{\lambda_{\pm}|_{x_2=L}}\Bigr\|_{L^\infty_{(t, x_{1}, x_{3})}}+ \delta_{\pm}(\lambda_{\pm}^{1}-\lambda_{\pm}^{2})\lesssim F(L)\cdot\e,\\
&\qquad\text{with}\quad F(L):=\left\{\begin{aligned}
&\frac{1}{L\left(\log L\right)^2},\quad\text{if}\quad\lambda_\pm(t,x)= \left(\log \langle w_\mp\rangle\right)^2\,\,\text{or}\,\,(\log \langle w_{\mp}\rangle)^4,\\ 
& \frac{1}{L\log L},\quad\text{if}\quad\lambda_\pm(t,x)=\langle w_{\mp}\rangle^2(\log \langle w_{\mp}\rangle)^4\,\,\text{or}\,\, \langle w_{\mp}\rangle(\log \langle w_{\mp}\rangle)^4,
\end{aligned}\right.
\end{aligned} 
\end{equation}
where $\delta_\pm(\lambda_\pm^{1}-\lambda_\pm^{2})$ are defined by
\begin{equation}\label{eq:l8}
\delta_{+}(\lambda_{+}^{1}-\lambda_{+}^{2})=\Bigl\|\frac{\lambda_{+}^{1}-\lambda_{+}^{2}}{\lambda_{+}^{2}}\Bigr\|_{L^{\infty}(\overline{C}_{L}^{-})}, \quad \delta_{-}(\lambda_{-}^{1}-\lambda_{-}^{2})=\Bigl\|\frac{\lambda_{-}^{1}-\lambda_{-}^{2}}{\lambda_{-}^{2}}\Bigr\|_{L^{\infty}(\overline{C}_{-L}^{+})}.
\end{equation} 
\end{lemma}
The proof of Lemma \ref{prop:b1} relies on Lemma \ref{lem:e1}, Lemma \ref{lem:e2}, and the properties of $\lambda$ stated above. We postpone the proof to the Appendix.

\smallskip

We also need to handle the boundary terms arising from the weight differences in the derivation of the weighted div-curl lemma. The corresponding weights on $[0,T^*]\times Q_L$ are defined by
\beq\label{weight:w2}
\lambda(t,x):=\Phi(\langle w_+\rangle,\langle w_-\rangle)\quad
\text{with}\quad \Phi(r_+,r_-)=\f{r_-^2(\log r_-)^4}{r_+(\log r_+)^2},\,\,  \f{r_+^2(\log r_+)^4}{r_-(\log r_-)^2},\quad\forall\,r_\pm>0.
\eeq
The property \eqref{prop for u pm} of $u_\pm$ combined with the periodicity property \eqref{eq:g16} implies that
\beq\label{eq:dc 1}
\lambda(t,x_1,x_2,L)=\lambda(t,x_1,x_2,-L).
\eeq

We also introduce the smooth positive functions on $\left[0, T^{*}\right] \times \mathbb{R}^{3}$ as follows:
\beq\label{eq:dc 2}
\begin{aligned}
&\lambda_{1}(t,x):=\Phi(\langle\overline{w_+}\rangle(t,x),\langle\overline{w_-}\rangle(t,x)),\\
&\lambda_{2}(t,x):=\Phi(\langle\overline{w_+}\rangle(t,x),\langle\overline{w_-}\rangle(t,x-2Le_3)),\\
&\lambda_{3}(t,x):=\Phi(\langle\overline{w_+}\rangle(t,x+2Le_3),\langle\overline{w_-}\rangle(t,x)).
\end{aligned}
\eeq 
The defintions of $u_\pm$ imply that for all $(t,x)\in[0,T^*]\times Q_L$,
\beq\label{eq:dc 2a} 
\lambda(t,x)=
\left\{\begin{aligned}
& \lambda_{1}(t,x), \quad \text{if }x_{3}^{+}(t,x)>-L \text{ and } x_{3}^{-}(t,x)< L,\\
& \lambda_{2}(t,x), \quad \text{if }\,x_{3}^{-}(t,x)\geq L,\\
& \lambda_{3}(t,x), \quad \text{if }\,x_{3}^{+}(t,x)\leq -L,
\end{aligned}\right.
\eeq
which together with Lemmas \ref{lem:e1} and \ref{lem:e2} shows that for all $(t,x)\in[0,T^*]\times Q_L$,
\beq\label{eq:dc 3}
\left|\nabla \lambda_{i}(t, x)\right|\lesssim\left|\lambda_{i}(t, x)\right|\,\, \text{for}\,\,  i=1,2,3,\quad\text{and}\quad  \lambda(t, x)\sim\lambda_{1}(t, x).
\eeq 

\begin{lemma}\label{prop:div-curl-apply}
Let $\lambda(t,x)=\f{\langle w_-\rangle^2(\log \langle w_-\rangle)^4}{\langle w_+\rangle(\log \langle w_+\rangle)^2}$ or $\f{\langle w_+\rangle^2(\log \langle w_+\rangle)^4}{\langle w_-\rangle(\log \langle w_-\rangle)^2}$. Then under the assumptions \eqref{eq:a1}--\eqref{eq:a3},  we have for any $t\in[0,T^*]$,
\begin{equation}\label{eq:div-curl *}
\begin{aligned}
& \Bigl\|\frac{\lambda|_{x_1=L}-\lambda|_{x_1=-L}}{\lambda|_{x_1=L}}\Bigr\|_{L^\infty_{(t, x_{2}, x_{3})}} + \Bigl\|\frac{\lambda|_{x_2=L}-\lambda|_{x_2=-L}}{\lambda|_{x_2=L}}\Bigr\|_{L^\infty_{(t, x_{1}, x_{3})}} \\
& \quad\ +\Bigl\|\frac{\lambda_{1}-\lambda_{2}}{\lambda_{2}}\Bigr\|_{L^{\infty}(\overline{C}_{L}^{-})}+ \Bigl\|\frac{\lambda_{1}-\lambda_{3}}{\lambda_{3}}\Bigr\|_{L^{\infty}(\overline{C}_{-L}^{+})}\lesssim\frac{\varepsilon}{L\log L} ,
\end{aligned} 
\end{equation}
where $\lambda_1(t,x),\,\lambda_2(t,x),\,\lambda_3(t,x)$ are defined in \eqref{eq:dc 2}.
\end{lemma}

Lemma \ref{prop:div-curl-apply} is proved in a similar way to Lemma \ref{prop:b1}. We also postpone the proof to the Appendix.

\subsubsection{Div-curl Lemma} To derive the higher-order energy estimates, we need a weighted div-curl lemma.

\begin{lemma}[div-curl lemma]\label{lem:div-cur}
 Let $\lambda(t,x)$ be a weight function defined in either \eqref{weight:w1} or \eqref{weight:w2}. Then under the assumptions \eqref{eq:a1}--\eqref{eq:a3}, for all vector fields $\vv v(t, x)\in L^{\infty}\big([0, T^{*}]; H^{1}(Q_{L})\big)$ that are $2L$-periodic in each spatial direction and satisfy
\begin{equation*}
\f{|\na\lambda|}{\sqrt{\lambda}}\vv v\in L^{2}\left(\Sigma_{t}\right) ,\quad\sqrt{\lambda} \nabla \vv v\in L^{2}\left(\Sigma_{t}\right), 
\end{equation*}
we have for all $t \in [0,T^*]$,
\begin{equation}\label{eq:div-curl lemma}
\begin{aligned}
\|\sqrt{\lambda} \nabla\vv v\|_{L^{2}(\Sigma_{t})}^{2} &\lesssim \|\sqrt{\lambda} \operatorname{div}\vv v\|_{L^{2}(\Sigma_{t})}^{2}+ \|\sqrt{\lambda} \operatorname{curl}\vv v\|_{L^{2}(\Sigma_{t})}^{2}
+ \Bigl\|\f{|\na\lambda|}{\sqrt{\lambda}} \vv v\Bigr\|_{L^{2}(\Sigma_{t})}^{2}\\ 
&\qquad
+F(L)\e \bigl(\|\sqrt{\lambda} \vv v\|_{L^{2}(\Sigma_{t})}^{2}+\|\sqrt{\lambda} \nabla^2 \vv v\|_{L^2(\Sigma_t)}^{2}\bigr),
\end{aligned} 
\end{equation}
where $F(L)$ is given by
\beq\label{eq:def of F}
F(L):=\left\{\begin{aligned}
&\frac{1}{L\left(\log L\right)^2},\quad\text{for}\quad\lambda_\pm= \left(\log \langle w_\mp\rangle\right)^2\,\,\text{or}\,\,(\log \langle w_{\mp}\rangle)^4,\\ 
& \frac{1}{L\log L},\quad\text{for the other cases}.
\end{aligned}\right.
\eeq
Here we remark that both $\lambda$ and  $\na\lambda$ are piecewise smooth on $[0,T^*]\times Q_L$. 
\end{lemma}
\begin{proof}
We derive \eqref{eq:div-curl lemma} only for the case where $\lambda(t,x)=\f{\langle w_-\rangle^2(\log \langle w_-\rangle)^4}{\langle w_+\rangle(\log \langle w_+\rangle)^2}$. For $\lambda(t,x)$ defined by either \eqref{weight:w1} or \eqref{weight:w2}, the proof of \eqref{eq:div-curl lemma} follows analogously.
For $\lambda(t,x)=\f{\langle w_-\rangle^2(\log \langle w_-\rangle)^4}{\langle w_+\rangle(\log \langle w_+\rangle)^2}$, it satisfies the properties \eqref{eq:dc 1}, \eqref{eq:dc 2}, \eqref{eq:dc 2a} and \eqref{eq:dc 3}. We divide the proof into several steps.

{\bf Step 1. Estimate of $\sqrt\lambda\na\vv v$.} For any $2L$-periodic vector field $\vv v\in L^{\infty}\big([0, T^{*}]; H^{1}(Q_{L})\big)$, there holds
\begin{equation}\label{eq:div-curl 2}
-\Delta\vv v =\operatorname{curl} \operatorname{curl}\vv v-\nabla \operatorname{div} \vv v. 
\end{equation}
Direct calculations shows that
\beno
\begin{aligned}
-\Delta\vv v\cdot\lambda\vv v&=-\sum_{i=1}^3\p_i(\p_i\vv v\cdot\lambda\vv v)
+\lambda|\na\vv v|^2+(\na\lambda\cdot\na)\vv v\cdot\vv v,\\ 
\operatorname{curl} \operatorname{curl}\vv v \cdot \lambda\vv v 
& =\operatorname{div}\big(\operatorname{curl} \vv v \times (\lambda \vv v)\big)+ \lambda|\operatorname{curl} \vv v|^{2}+ \operatorname{curl}\vv v\cdot(\nabla \lambda \times \vv v),\\
\nabla \operatorname{div}\vv v\cdot \lambda\vv v
& =\operatorname{div}(\operatorname{div}\vv v \,\lambda \vv v)-\lambda|\operatorname{div}\vv v|^{2}-\nabla \lambda\cdot\vv v\operatorname{div}\vv v.
\end{aligned}
\eeno
from which and \eqref{eq:div-curl 2}, we deduce that 
\beq\label{eq:dc 4}
\int_{Q_L}\lambda|\na\vv v|^2dx=\int_{Q_L}\lambda|\dive\vv v|^2dx
+\int_{Q_L}\lambda|\curl\vv v|^2dx+\sum_{k=1}^4R_k,
\eeq
where
\beno\begin{aligned}
& R_1:=\int_{Q_L}\sum_{i=1}^3\p_i(\p_i\vv v\cdot\lambda\vv v)dx=
\f12 \int_{Q_L}\dive[\lambda\na(|\vv v|^2)]dx,
\\ 
&R_2:=\int_{Q_L}\operatorname{div}\big(\operatorname{curl} \vv v \times (\lambda \vv v)\big)dx,\quad  R_3:=-\int_{Q_L}\operatorname{div}(\operatorname{div}\vv v \,\lambda \vv v)dx,\\ 
& R_4:=\int_{Q_L}\Bigl(-(\na\lambda\cdot\na)\vv v\cdot\vv v+\operatorname{curl}\vv v\cdot(\nabla \lambda \times \vv v)+\nabla \lambda\cdot\vv v\operatorname{div}\vv v\Bigr)dx.
\end{aligned}\eeno

\smallskip

{\bf Step 2. Estimates of $\{R_k\}_{k=1}^4$.} We proceed to estimate terms $\{R_k\}_{k=1}^4$ individually. 

{\it Step 2.1. Estimate of $R_4$.} It is easy to get 
\beq\label{eq:dc 5}
|R_4|\leq C\|\sqrt\lambda\na\vv v\|_{L^2(\Sigma_t)}\cdot\Bigl\|\f{|\na\lambda|}{\sqrt\lambda}\vv v\Bigr\|_{L^2(\Sigma_t)}\leq\f18\|\sqrt\lambda\na\vv v\|_{L^2(\Sigma_t)}^2+C\Bigl\|\f{|\na\lambda|}{\sqrt\lambda}\vv v\Bigr\|_{L^2(\Sigma_t)}^2.
\eeq

{\it Step 2.2. Estimate of $R_1$.} By \eqref{eq:dc 2} and \eqref{eq:dc 2a},  $\lambda(t,x)$ is piecewise smooth on $Q_L$ (or $\Sigma_t$), taking the form of smooth functions $\lambda_{1}$, $\lambda_{2}$, and $\lambda_{3}$ in the respective regions:
\beno 
\overline{\Sigma}_{t,+}^{>-L} \bigcap \overline{\Sigma}_{t,-}^{< L},\quad  \overline{\Sigma}_{t,-}^{\geq L} \quad\text{and}\quad \overline{\Sigma}_{t,+}^{\leq -L}.
\eeno 
Note that these three regions form  a complete partition of $Q_L$ with two interior interfaces:
\beno
\overline{S}_{t,L}^-=\p\bigl(\overline{\Sigma}_{t,+}^{>-L} \bigcap \overline{\Sigma}_{t,-}^{< L}\bigr)\cap\p\bigl(\overline{\Sigma}_{t,-}^{\geq L}\bigr),\quad 
\overline{S}_{t,-L}^+=\p\bigl(\overline{\Sigma}_{t,+}^{>-L} \bigcap \overline{\Sigma}_{t,-}^{< L}\bigr)\cap\p\bigl(\overline{\Sigma}_{t,+}^{\leq -L}\bigr).
\eeno 
 By applying the divergence theorem for piecewise smooth vector fields (Lemma \ref{lem:divergence}), we obtain
\beq\label{eq:dc 6}\begin{aligned}
R_1=&\f12 \int_{\p Q_L}\lambda\na(|\vv v|^2)\cdot\vv ndS
+\f12 \int_{\overline{S}_{t,L}^-}(\lambda_1-\lambda_2)\na(|\vv v|^2)\cdot\vv n_-dS_-\\ 
&\quad 
+\f12 \int_{\overline{S}_{t,-L}^+}(\lambda_1-\lambda_3)\na(|\vv v|^2)\cdot\vv n_+dS_+,
\end{aligned}\eeq
where $\vv n$ is the unit outward normal vector on $\p Q_L$, 
 $\vv n_-$ is the unit normal vector on $\overline{S}_{t,L}^-$ pointing from $\overline{\Sigma}_{t,+}^{>-L} \bigcap \overline{\Sigma}_{t,-}^{< L}$ to $\overline{\Sigma}_{t,-}^{\geq L}$, and $\vv n_+$ is the unit normal vector on $\overline{S}_{t,-L}^+$ pointing from $\overline{\Sigma}_{t,+}^{>-L} \bigcap \overline{\Sigma}_{t,-}^{< L}$ to $\overline{\Sigma}_{t,+}^{\leq -L}$.

\smallskip 

{\it (i). For the first term} on the right hand side(RHS) of \eqref{eq:dc 6}, using  $\lambda(t,x_1,x_2,L)=\lambda(t,x_1,x_2,-L)$ (see \eqref{eq:dc 1}), the periodic properties of $x^\pm_i$ in \eqref{eq:g16} and the periodic propeties of $\vv v$, we get
\beq\label{eq:dc 7}\begin{aligned}
\f12\Bigl|\int_{\p Q_L}\lambda\na(|\vv v|^2)\cdot\vv ndS\Bigr|&
=\Bigl|\sum_{i=1}^2\int_{\{x_i=L\}}\bigl(\lambda|_{x_i=L}-\lambda|_{x_i=-L}\bigr)\p_i\vv v\cdot\vv vd{\widehat{x_i}}\Bigr|\\
&\lesssim\sum_{i=1}^2\Bigl\|\f{\lambda|_{x_i=L}-\lambda|_{x_i=-L}}{\lambda|_{x_i=L}}\Bigr\|_{L^\infty_{(t,\widehat{x}_i)}}\cdot\int_{\{x_i=L\}} \bigl(\lambda|\na\vv v|^2+\lambda|\vv v|^2\bigr)d{\widehat{x_i}},
\end{aligned}\eeq
where  $\{x_i=L\}:=\{x\in Q_L\,|\, x_i=L\}$, 
$\widehat{x_1}:=(x_2,x_3)$ and $\widehat{x_2}:=(x_1,x_3)$. 

By virtue of \eqref{eq:dc 3}) and the trace theorem, we obtain
\beno\begin{aligned}
\int_{\{x_i=L\}} &\Bigl(\lambda|\na\vv v|^2+\lambda|\vv v|^2\Bigr)d{\widehat{x_i}}\lesssim\int_{\{x_i=L\}} \Bigl(\lambda_1|\na\vv v|^2+\lambda_1|\vv v|^2\Bigr)d{\widehat{x_i}}\\ 
&\lesssim\|\sqrt\lambda_1\na\vv v\|_{H^1(Q_L)}^2+\|\sqrt\lambda_1\vv v\|_{H^1(Q_L)}^2\\ 
&\lesssim\sum_{k=0}^2\|\sqrt\lambda_1\na^k\vv v\|_{L^2(\Sigma_t)}^2\lesssim\sum_{k=0}^2\|\sqrt\lambda\na^k\vv v\|_{L^2(\Sigma_t)}^2,
\end{aligned}\eeno
from which, \eqref{eq:div-curl *} (Lemma \ref{prop:div-curl-apply}) and \eqref{eq:dc 7}, we deduce that
\beq\label{eq:dc 8}
\f12\Bigl|\int_{\p Q_L}\lambda\na(|\vv v|^2)\cdot\vv ndS\Bigr|
\lesssim\f{\e}{L\log L}\cdot\sum_{k=0}^2\|\sqrt\lambda\na^k\vv v\|_{L^2(\Sigma_t)}^2.
\eeq

\smallskip

{\it (ii). For the second term} on the RHS  of \eqref{eq:dc 6}, using the fact that $\overline{S}_{t,L}^-=\overline{C}_L^-\cap\Sigma_t$, we have
\beq\label{eq:dc 9}
\f12\Bigl|\int_{\overline{S}_{t,L}^-}(\lambda_1-\lambda_2)\na(|\vv v|^2)\cdot\vv n_-dS_-\Bigr|\lesssim\bigl\|\f{\lambda_1-\lambda_2}{\lambda_2}\bigr\|_{L^\infty(\overline{C}_L^-)}\cdot\int_{\overline{S}_{t,L}^{-}} \left(\lambda_2 |\na\vv v|^2 + \lambda_2 |\vv v|^2 \right) dS_{-}.
\eeq
By virtue of the trace theorem and \eqref{eq:dc 3}), we get
\beno\begin{aligned}
\int_{\overline{S}_{t,L}^{-}} \left(\lambda_2 |\na\vv v|^2 + \lambda_2 |\vv v|^2 \right) dS_{-}
&\lesssim\|\sqrt\lambda_2\na\vv v\|_{H^1(\Sigma_t)}^2+\|\sqrt\lambda_2\vv v\|_{H^1(\Sigma_t)}^2\\ 
&\lesssim\sum_{k=0}^2\|\sqrt\lambda_2\na^k\vv v\|_{L^2(\Sigma_t)}^2
\lesssim\sum_{k=0}^2\|\sqrt\lambda\na^k\vv v\|_{L^2(\Sigma_t)}^2,
\end{aligned}\eeno 
from which, \eqref{eq:div-curl *} (Lemma \ref{prop:div-curl-apply}) and \eqref{eq:dc 9}, we derive that
\beq\label{eq:dc 10}
\f12\Bigl|\int_{\overline{S}_{t,L}^-}(\lambda_1-\lambda_2)\na(|\vv v|^2)\cdot\vv n_-dS_-\Bigr|\lesssim\f{\e}{L\log L}\cdot\sum_{k=0}^2\|\sqrt\lambda\na^k\vv v\|_{L^2(\Sigma_t)}^2.
\eeq 
The same estimate also holds for the last term on the RHS of \eqref{eq:dc 6}.

\smallskip 

Thanks to  \eqref{eq:dc 8} and \eqref{eq:dc 10}, we deduce from \eqref{eq:dc 6} that
\beq\label{eq:dc 11}
|R_1|\lesssim\f{\e}{L\log L}\cdot\sum_{k=0}^2\|\sqrt\lambda\na^k\vv v\|_{L^2(\Sigma_t)}^2.
\eeq 
The same estimates also hold for $R_2$ and $R_3$.

\smallskip 

{\it Step 2.3. Estimates of $\{R_k\}_{k=1}^4$.} Due to the results in Step 2.1 and Step 2.2, we get
\beq\label{eq:dc 12}
\sum_{k=1}^4|R_k|\leq\f18\|\sqrt\lambda\na\vv v\|_{L^2(\Sigma_t)}^2+C\Bigl\|\f{|\na\lambda|}{\sqrt\lambda}\vv v\Bigr\|_{L^2(\Sigma_t)}^2+\f{C\e}{L\log L}\cdot\sum_{k=0}^2\|\sqrt\lambda\na^k\vv v\|_{L^2(\Sigma_t)}^2.
\eeq

\smallskip 

{\bf Step 3. Modified div-curl lemma.} Using \eqref{eq:dc 12} and $L\gg 1$, we derive \eqref{eq:div-curl lemma} from \eqref{eq:dc 4} for $\lambda(t,x)=\f{\langle w_-\rangle^2(\log \langle w_-\rangle)^4}{\langle w_+\rangle(\log \langle w_+\rangle)^2}$. The same conclusion \eqref{eq:div-curl lemma} holds for the symmetric case $\lambda(t,x)=\f{\langle w_+\rangle^2(\log \langle w_+\rangle)^4}{\langle w_-\rangle(\log \langle w_-\rangle)^2}$.

\smallskip 

For $\lambda=\lambda_+=\Lambda(\langle w_-\rangle)$ defined in \eqref{weight:w1}, note that $\lambda$ is piecewise smooth on $\Sigma_t$, taking the form of smooth functions $\lambda_+^1$ and  $\lambda_+^2$ in the regions $\overline{\Sigma}_{t,-}^{<L}$ and $\overline{\Sigma}_{t,-}^{\geq L}$, respectively. By an analogous derivation to the case above, \eqref{eq:div-curl lemma} remains valid. Here, we replace the use of \eqref{eq:div-curl *} (Lemma \ref{prop:div-curl-apply}) with \eqref{eq:b1} (Lemma \ref{prop:b1}). The case $\lambda=\lambda_-=\Lambda(\langle w_+\rangle)$ defined in \eqref{weight:w1} follows similarly.

This completes the proof of the lemma.
\end{proof}

\begin{remark}
Since $\operatorname{div} z_{\pm}=0$, the div-curl lemma allows us to obtain energy estimates for $\nabla z_{\pm}^{(\gamma)}$ by investigating the energy estimates in terms of $j_{\pm}^{(\gamma)}$, and thus use the vorticity formulation \eqref{eq:j} of the MHD system. 
\end{remark}
\begin{remark}
For $\lambda$ defined in \eqref{weight:w1} and \eqref{weight:w2}, we derive from  Lemma \ref{lem:e1} that 
\beno 
|\na\lambda|\lesssim\f{\lambda}{R}. 
\eeno 
Now for any multi-index $\gamma$ with $1\leq |\gamma|\leq N_{*}$, we
take 

i). $\vv v=z_{+}^{(\gamma)}:=\p^\gamma z_+$ with either
$ 
\lambda= \frac{\langle w_{-}\rangle^2\left(\log \langle w_{-}\rangle\right)^4}{\langle w_{+}\rangle \left(\log \langle w_{+}\rangle\right)^2}$ or $\lambda=\lambda_+=\Lambda(\langle w_{-}\rangle),$

ii). or $\vv v=z_{-}^{(\gamma)}$ with either $
 \lambda= \frac{\langle w_{+}\rangle^2\left(\log \langle w_{+}\rangle\right)^4}{\langle w_{-}\rangle \left(\log \langle w_{-}\rangle\right)^2}$ or $\lambda=\lambda_-=\Lambda(\langle w_{+}\rangle)$.

Applying these choices in Lemma \ref{lem:div-cur}, thanks to \eqref{eq:div-curl lemma} and the obtained fact that $|\na\lambda|\lesssim\f{\lambda}{R}$, we get
\begin{equation}
\bigl\|\sqrt{\lambda} \nabla z_{\pm}^{(\gamma)}\bigr\|_{L^{2}(\Sigma_{t})}^{2} \lesssim \bigl\|\sqrt{\lambda} j_{\pm}^{(|\gamma|)}\bigr\|_{L^{2}(\Sigma_{t})}^{2}+ \bigl\|\sqrt{\lambda} \nabla z_{\pm}^{(|\gamma|-1)}\bigr\|_{L^{2}(\Sigma_{t})}^{2}+ \frac{\varepsilon}{L\log L} \bigl\|\sqrt{\lambda} \nabla z_{\pm}^{(|\gamma|+1)}\bigr\|_{L^{2}(\Sigma_{t})}^{2}.
\end{equation}
For $1\leq |\gamma|\leq N_{*}$, we can iterate this estimate to get
\begin{equation}
\bigl\|\sqrt{\lambda} \nabla z_{\pm}^{(\gamma)}\bigr\|_{L^{2}(\Sigma_{t})}^{2} \lesssim \sum_{k=1}^{|\gamma|} \bigl\|\sqrt{\lambda} j_{\pm}^{(k)}\bigr\|_{L^{2}(\Sigma_{t})}^{2}+ \bigl\|\sqrt{\lambda} \nabla z_{\pm}\bigr\|_{L^{2}(\Sigma_{t})}^{2}+ \underbrace{\frac{\varepsilon}{L\log L} \bigl\|\sqrt{\lambda} \nabla^2 z_{\pm}^{(|\gamma|)}\bigr\|_{L^{2}(\Sigma_{t})}^{2}}_{\lesssim \frac{\mu\varepsilon}{L}\bigl\|\sqrt{\lambda} \nabla^2 z_{\pm}^{(|\gamma|)}\bigr\|_{L^{2}(\Sigma_{t})}^{2},\text{ since }\log L \geq \frac{1}{\mu}} . \label{eq:div-curl iterate}
\end{equation}
\end{remark}

\subsection{Energy estimates for viscous linear equations}
Under the assumptions \eqref{eq:a1}--\eqref{eq:a3}, we derive the energy estimates for the following viscous linear system on $[0,T^*]\times Q_L$:
\begin{equation}
\left\{\begin{array}{l}
\partial_{t} f_{+}+ Z_{-}\cdot \nabla f_{+}- \mu\Delta f_{+}= \rho_{+},
\\
\partial_{t} f_{-}+ Z_{+}\cdot \nabla f_{-}- \mu\Delta f_{-}= \rho_{-},
\end{array}\right. \label{eq:l1}
\end{equation}
where $Z_{+}$ and $Z_{-}$ are divergence-free vector fields with 2L-periodicity in all spatial directions.


\begin{proposition} \label{prop:l1}
Let $T^*=\log L$, $\mu\in(0,1)$ and $L\geq e^{\f{1}{\mu}}$. Assume that the weights $\lambda_{\pm}$ are defined as one of the following forms from \eqref{weight:w1}: 
\beno
(\log \langle w_{\mp}\rangle)^2,\quad(\log \langle w_{\mp}\rangle)^4,\quad \langle w_{\mp}\rangle^2(\log \langle w_{\mp}\rangle)^4,\quad \langle w_{\mp}\rangle(\log \langle w_{\mp}\rangle)^4. 
\eeno 
Then for any sufficiently smooth $2L$-periodic vector fields $f_\pm(t, x)$ satisfying \eqref{eq:l1}, we have
\begin{equation}\label{eq:l7}
\begin{aligned}
& \quad\sup_{0\leq\tau\leq t}\|\sqrt{\lambda_\pm}f_\pm\|_{L^2(\Sigma_\tau)}^2+\mu\|\sqrt{\lambda_\pm}\na f_\pm\|_{L^2_t(L^2_x)}^2+ \sup\limits_{|u_{\pm}|\leq \frac{L}{4}} \int_{C^{\pm}_{t,u_{\pm}}} \lambda_{\pm}\left|f_{\pm}\right|^2 d\sigma_{\pm}\\
& \leq 10\|\sqrt{\lambda_\pm}f_\pm\|_{L^2(\Sigma_0)}^2+14\mu \Bigl\|\f{|\na\lambda_\pm|}{\sqrt{\lambda_\pm}} f_\pm\Bigr\|_{L^2_t(L^2_x)}^2+C\mu^2\|\sqrt{\lambda_\pm}\na^2 f_\pm\|_{L_t^2(L^2_x)}^2\\ 
&\qquad+ C\int_{0}^{t} \int_{\Sigma_{\tau}} \lambda_\pm \left|f_\pm\right|\cdot |\rho_\pm|\, dxd\tau.
\end{aligned} 
\end{equation}
\end{proposition}
\begin{proof}
We only give the estimates for $f_{+}$ which corresponds to the left-traveling Alfv\'en wave $z_{+}$ and the estimates for $f_{-}$ follow analogously.

{\bf Step 1. Energy estimate on $W_t:=[0,t]\times Q_L$.}
By the definition of $\lambda_\pm$, there hold \eqref{eq:l5}, \eqref{eq:l4} and \eqref{eq:l6}.  Observing that $L_{-}=T+Z_{-}$ and $L_-\lambda_+=0$ (since $L_{-} \lambda_{+}^{i}=0$ in \eqref{eq:l4}), we take the inner product of the first equation in \eqref{eq:l1} with  $\lambda_{+} f_{+}$ to obtain
\begin{equation}\label{eq:l9}
\frac{1}{2} L_{-}\big(\lambda_{+} \left|f_{+}\right|^2\big)- \mu\Delta f_{+}\cdot \lambda_{+} f_{+}= \lambda_{+} f_{+}\cdot \rho_{+}, \quad\forall\,(t,x)\in[0,T^*]\times Q_L.
\end{equation}

We use $\widetilde{\operatorname{div}}$ to denote the divergence of $\mathbb{R}^4=\mathbb{R}_{t} \times \mathbb{R}_{x}^3$ with respect to the standard Euclidean metric. Since $\operatorname{div} Z_{\pm}=0$, we have $\widetilde{\operatorname{div}}\, L_{\pm}=0$. We integrate \eqref{eq:l9} on $W_t$
 to get
\begin{equation}\label{eq:l10}
\frac{1}{2}\iint_{W_t}\wt{\dive}\big(\lambda_{+} \left|f_{+}\right|^2 L_-\big) dx d\tau-\mu\iint_{W_t} \Delta f_{+}\cdot \lambda_{+} f_{+}\, dx d\tau= \iint_{W_t}\lambda_{+} f_{+}\cdot \rho_{+}\, dx d\tau .
\end{equation}

We now analyze and estimate each term in \eqref{eq:l10} sequentially.

{\it Step 1.1. Estimate of $\frac{1}{2}\iint_{W_t}\wt{\dive}\big(\lambda_{+} \left|f_{+}\right|^2 L_-\big) dx d\tau.$} First, we recall that $\lambda_+$ is piecewise smooth on $W_t$, taking the form of smooth functions $\lambda_+^1$ and $\lambda_+^2$ on the respective regions:
\beno 
\overline{W}_{t,-}^{<L}:= \bigcup\limits_{a<L} \overline{C}_{t,a}^{-},\quad 
\overline{W}_{t,-}^{\geq L}:= \bigcup\limits_{a\geq L} \overline{C}_{t,a}^{-}.
\eeno
Note that $W_t=\overline{W}_{t,-}^{<L}\bigcup \overline{W}_{t,-}^{\geq L}$ and $\p\bigl(\overline{W}_{t,-}^{<L}\bigr)\bigcap\p\bigl(\overline{W}_{t,-}^{\geq L}\bigr)=\overline{C}_{t,L}^{-}$ is the interior interface. 

Using divergence theorem for piecewise smooth vector fields (Lemma \ref{lem:divergence}), we obtain 
\beq\label{eq:l11}
\iint_{W_t}\wt{\dive}\big(\lambda_{+} \left|f_{+}\right|^2 L_-\big) dx d\tau=\int_{\p W_t}\lambda_{+} \left|f_{+}\right|^2\langle L_-,\vv\nu\rangle d\sigma+\underbrace{\int_{\overline{C}^-_{t,L}}(\lambda_+^1-\lambda_+^2)|f_+|^2\langle L_-,\vv\nu_-\rangle d\sigma_-}_{=0,\,\text{since}\, L_-\,\text{is tangential to }\, \overline{C}^-_{t,L}},
\eeq
where $\vv\nu$ is the unit outward normal vector on $\p W_t$ and $\vv\nu_-$ is the unit normal vector on $\overline{C}^-_{t,L}$ pointing from $\overline{W}_{t,-}^{<L}$ to $\overline{W}_{t,-}^{\geq L}$.

For the first term on the RHS of \eqref{eq:l11}, noticing that $L_-=(1,Z_-)$, we have
\beq\label{eq:l12}
\int_{\p W_t}\lambda_{+} \left|f_{+}\right|^2\langle L_-,\vv\nu\rangle d\sigma
=\int_{\Sigma_t}\lambda_{+} \left|f_{+}\right|^2dx-\int_{\Sigma_0}\lambda_{+} \left|f_{+}\right|^2dx+\int_0^t\int_{\p Q_L}\lambda_{+} \left|f_{+}\right|^2(Z_-\cdot\vv n)dSd\tau,
\eeq
where $\vv n$ is the unit outward normal vector on $\p Q_L$. By virtue of $\lambda_+|_{x_3=L}=\lambda_+|_{x_3=-L}$ (see \eqref{eq:l5}) and the periodic properties of $f_+$,  we obtain
\beno\begin{aligned}
\Bigl|\int_0^t\int_{\p Q_L}&\lambda_{+} \left|f_{+}\right|^2(Z_-\cdot\vv n)dSd\tau\Bigr|
=\Bigl|\sum_{i=1}^2\int_0^t\int_{\{x_i=L\}}(\lambda_{+}|_{x_i=L}-\lambda_{+}|_{x_i=-L})\left|f_{+}\right|^2\cdot z_-^id\widehat{x}_id\tau\Bigr|\\ 
&\lesssim\sum_{i=1}^2\Bigl\|\f{\lambda_{+}|_{x_i=L}-\lambda_{+}|_{x_i=-L}}{\lambda_{+}|_{x_i=L}}\Bigr\|_{L^\infty_{(t,\widehat{x}_i)}}\cdot\|z_-^i\|_{L^\infty_{(t,x)}}\cdot\int_0^t\int_{\{x_i=L\}}\lambda_{+}|f_+|^2\,d\widehat{x}_id\tau\\ 
&\lesssim F(L)\e\cdot\sum_{i=1}^2\int_0^t\int_{\{x_i=L\}}\lambda_{+}^1|f_+|^2\,d\widehat{x}_id\tau,
\end{aligned}\eeno
where we used \eqref{eq:b1} (Lemma \ref{prop:b1}), the assumption that $\|z_-^i\|_{L^\infty_{(t,x)}}\leq\f12$ and $\lambda_+(t,x)\sim\lambda_+^1(t,x)$ for all $(t,x)\in[0,T^*]\times Q_L$ (see \eqref{eq:l6}) in the last inequality. Here $F(L)$ is defined in \eqref{eq:b1}. 
Using the trace theorem, and \eqref{eq:l6} again, we get
\beno\begin{aligned}
\int_{\{x_i=L\}}\lambda_{+}^1|f_+|^2\,d\widehat{x}_i
&\lesssim\|\sqrt{\lambda_{+}^1}f_+\|_{H^1(Q_L)}^2
\lesssim\|\sqrt{\lambda_{+}^1}f_+\|_{L^2_x}^2+\|\sqrt{\lambda_{+}^1}\na f_+\|_{L^2_x}^2\\ 
&
\lesssim\|\sqrt{\lambda_{+}}f_+\|_{L^2_x}^2+\|\sqrt{\lambda_{+}}\na f_+\|_{L^2_x}^2,
\end{aligned}\eeno 
which implies 
\beq\label{eq:l13}
\Bigl|\int_0^t\int_{\p Q_L}\lambda_{+} \left|f_{+}\right|^2(Z_-\cdot\vv n)dSd\tau\Bigr|\lesssim F(L)\e\cdot\bigl(T^*\cdot\sup_{0\leq\tau\leq t}\|\sqrt{\lambda_{+}}f_+\|_{L^2_x}^2+\|\sqrt{\lambda_{+}}\na f_+\|_{L^2_t(L^2_x)}^2\bigr).
\eeq 

Thanks to \eqref{eq:l12} and \eqref{eq:l13}, we deduce from \eqref{eq:l11} that
\beq\label{eq:l15}\begin{aligned}
\f12 \iint_{W_t}\wt{\dive}\big(\lambda_{+} \left|f_{+}\right|^2 L_-\big) dx d\tau\geq&\f12 \|\sqrt{\lambda_+}f_+\|_{L^2(\Sigma_t)}^2-\f12 \|\sqrt{\lambda_+}f_+\|_{L^2(\Sigma_0)}^2\\
&- CF(L)\e\cdot\bigl(T^*\cdot\sup_{0\leq\tau\leq t}\|\sqrt{\lambda_{+}}f_+\|_{L^2_x}^2+\|\sqrt{\lambda_+}\na f_+\|_{L^2_t(L^2_x)}^2\bigr).
\end{aligned}\eeq

{\it Step 1.2. Estimate of the diffusion term $-\mu\iint_{W_t} \Delta f_{+}\cdot \lambda_{+} f_{+}\, dx d\tau$.} Since 
\beq\label{eq:l15a}
-\Delta f_{+}\cdot \lambda_{+} f_{+}=-\f12\na\cdot\bigl[\lambda_+\na (|f_+|^2)\bigr]+(\na\lambda_+\cdot\na)f_+\cdot f_++\lambda_+|\na f_+|^2,
\eeq
we have
\beq\label{eq:l16}
-\mu\iint_{W_t} \Delta f_{+}\cdot \lambda_{+} f_{+}\, dx d\tau
=\mu\int_0^t\|\sqrt{\lambda_+}\na  f_+\|_{L^2(\Sigma_\tau)}^2d\tau+I_1+I_2,
\eeq 
where
\beno
I_1=-\f{\mu}{2}\int_0^t\int_{Q_L}\na\cdot\bigl[\lambda_+\na (|f_+|^2)\bigr]dxd\tau,\quad 
I_2=\mu\iint_{W_t}(\na\lambda_+\cdot\na)f_+\cdot f_+ dxd\tau,
\eeno

For $I_1$, we use the divergence theorem for piecewise smooth vector fields (Lemma \ref{lem:divergence}) to get
\beno 
I_1=-\f{\mu}{2}\int_0^t\int_{\p Q_L}\lambda_+\na (|f_+|^2)\cdot \vv n dSd\tau-\f{\mu}{2}\int_0^t\int_{\overline{S}^-_{\tau,L}}(\lambda_+^1-\lambda_+^2)\na (|f_+|^2)\cdot \vv n_- dS_-d\tau,
\eeno 
where $\vv n$ is the unit outward normal vector on $\p Q_L$ and $\vv n_-$ is the unit normal vector on $\overline{S}^-_{\tau,L}=\p(\overline{\Sigma}_{\tau,-}^{<L})\bigcap\p(\overline{\Sigma}_{\tau,-}^{\geq L}) $ pointing from $\overline{\Sigma}_{\tau,-}^{<L}$ to $\overline{\Sigma}_{\tau,-}^{\geq L}$. Similar derivation as that of \eqref{eq:dc 11} yields to
\beq\label{eq:l17}
|I_1|\lesssim F(L)\e\cdot\mu\sum_{k=0}^2\int_0^t\|\sqrt{\lambda_+}\na^kf_+\|_{L^2(\Sigma_\tau)}^2d\tau.
\eeq

For $I_2$, it is easy to get
\beq\label{eq:l18}
|I_2|\leq\f{\mu}{4}\|\sqrt{\lambda_+}\na f_+\|_{L^2_t(L^2_x)}^2+\mu \Bigl\|\f{|\na\lambda_+|}{\sqrt{\lambda_+}} f_+\Bigr\|_{L^2_t(L^2_x)}^2.
\eeq

Thanks to \eqref{eq:l17}, \eqref{eq:l18}, $F(L)\leq\f{1}{L\log L}$ (see \eqref{eq:b1}) and $L\gg 1$, we deduce from \eqref{eq:l16} that
\beq\label{eq:l19}\begin{aligned}
-\mu\iint_{W_t} \Delta f_{+}\cdot \lambda_{+} f_{+}\, dx d\tau
&\geq\f\mu2\|\sqrt{\lambda_+}\na f_+\|_{L^2_t(L^2_x)}^2
-\mu \Bigl\|\f{|\na\lambda_+|}{\sqrt{\lambda_+}} f_+\Bigr\|_{L^2_t(L^2_x)}^2\\ 
&\qquad- CF(L)\e\cdot\mu\bigl(T^*\cdot\sup_{0\leq\tau\leq t}\|\sqrt{\lambda_+}f_+\|_{L^2_x}^2+\|\sqrt{\lambda_+}\na^2 f_+\|_{L_t^2(L^2_x)}^2\bigr).
\end{aligned}\eeq

\smallskip 

{\it Step 1.3. Energy estimate on $W_t$. } By virtue of \eqref{eq:l15} and \eqref{eq:l19}, we derive from \eqref{eq:l10} that
\beno\begin{aligned}
\|\sqrt{\lambda_+}f_+&\|_{L^2(\Sigma_t)}^2+\mu\|\sqrt{\lambda_+}\na f_+\|_{L^2_t(L^2_x)}^2
\leq\|\sqrt{\lambda_+}f_+\|_{L^2(\Sigma_0)}^2+2\mu \Bigl\|\f{|\na\lambda_+|}{\sqrt{\lambda_+}} f_+\Bigr\|_{L^2_t(L^2_x)}^2\\ 
&+CF(L)\e\cdot T^*\cdot\sup_{0\leq\tau\leq t}\|\sqrt{\lambda_+}f_+\|_{L^2_x}^2+CF(L)\e\cdot \|\sqrt{\lambda_+}\na f_+\|_{L_t^2(L^2_x)}^2\\
&+CF(L)\e\cdot\mu \|\sqrt{\lambda_+}\na^2 f_+\|_{L_t^2(L^2_x)}^2+2\int_{0}^{t} \int_{\Sigma_{\tau}} \lambda_+ \left|f_+\right| |\rho_+|\, dx d\tau.
\end{aligned}\eeno 
Since $T^*=\log L$, $L\geq e^{\f{1}{\mu}}\gg1$ and $F(L)\leq\f{1}{L\log L}$, we have
\beq\label{eq:l20a}
F(L)\cdot T^*\leq\f{1}{L}\ll 1,\quad F(L)\leq\f{1}{L}\cdot\mu\ll\mu .
\eeq
Consequently, we obtain
\beq\label{eq:l20}\begin{aligned}
\sup_{0\leq\tau\leq t}\|\sqrt{\lambda_+}f_+&\|_{L^2(\Sigma_\tau)}^2+\mu\|\sqrt{\lambda_+}\na f_+\|_{L^2_t(L^2_x)}^2
\leq2\|\sqrt{\lambda_+}f_+\|_{L^2(\Sigma_0)}^2+4\mu \Bigl\|\f{|\na\lambda_+|}{\sqrt{\lambda_+}} f_+\Bigr\|_{L^2_t(L^2_x)}^2\\ 
&+C\underbrace{F(L)\e\cdot\mu}_{\lesssim{\mu^2\e}/L}\|\sqrt{\lambda_+}\na^2 f_+\|_{L_t^2(L^2_x)}^2+4\int_{0}^{t} \int_{\Sigma_{\tau}} \lambda_+ \left|f_+\right|\cdot |\rho_+|\, dx d\tau.
\end{aligned}\eeq 

\medskip 

{\bf Step 2. Local energy estimates on $W_{t,+}^{\geq u_+}\cap\overline{W}_{t,-}^{< L}$.} This step aims to bound the flux term $\sup\limits_{|u_+|\leq \frac{L}{4}} \int_{C^{+}_{t,u_+}} \lambda_{+}\left|f_+\right|^2 d\sigma_+$ appearing on the LHS of \eqref{eq:l7}. 

{\it Step 2.1. Cut-off equation.}  First, we take a smooth cut-off function $\varphi(r)\in C_{c}^{\infty}(\mathbb{R})$ such that 
\begin{equation*}
\varphi(r)=
\begin{cases}
& 1, \quad \text{for }|r|\leq \frac{1}{3},\\
& 0, \quad \text{for }|r|\geq\f12,
\end{cases}
\quad\text{with}\quad 0\leq\varphi(r)\leq 1,
\end{equation*}
and we define
\begin{equation*}
\varphi_{L}(r)=\varphi\Big(\frac{r}{L}\Big)=
\begin{cases}
& 1, \quad \text{for }|r|\leq \frac{L}{3},\\
& 0, \quad \text{for }|r|\geq \frac{L}{2}.
\end{cases}
\end{equation*}

Setting $\widetilde{f}_{+}=\varphi_{L}(x_{3})f_{+}$, due to $\partial_{t} f_{+}+ Z_{-}\cdot \nabla f_{+}- \mu\Delta f_{+}= \rho_{+}$, we observe that $\widetilde{f}_{+}$ satisfies
\beq\label{eq:l21}
\quad\ \partial_{t} \widetilde{f}_{+}+ Z_{-}\cdot \nabla \widetilde{f}_{+}- \mu\Delta \widetilde{f}_{+}=\widetilde{\rho}_{+},
\eeq
where 
\beno 
\widetilde{\rho}_{+}=\varphi_{L}(x_{3})\rho_{+} + \bigl(Z_{-}^{3}\varphi_{L}^{\prime}(x_{3})-\mu\varphi_{L}^{\prime\prime}(x_{3})\bigr)f_{+} - 2\mu\varphi_{L}^{\prime}(x_{3})\partial_{3}f_{+}.
\eeno 
By the definition of $\varphi_L$ and the ansatz $\|z_-\|_{L^\infty_{t,x}}\leq\f12$, we obtain 
\beq\label{eq:l22}
|\widetilde\rho_+|\lesssim|\rho_+|+\f{1}{L}|f_+|+\f{\mu}{L}|\p_3f_+|.
\eeq

By Lemma \ref{lem:e2}, for all $(t,x)\in [0,T^*]\times Q_L$, the following holds:
\begin{itemize}
\item[(i).] if $|u_{+}(t,x)|\leq \frac{L}{4}$, then
\beq\label{eq:l23}
|x_3|< \frac{L}{3}\Rightarrow \varphi_{L}(x_{3})=1\Rightarrow\widetilde{f}_{+}(t,x)=f_{+}(t,x)\quad\text{and}\quad C_{u_+}^{+}=\overline{C}_{u_+}^{+}.
\eeq 
\item[(ii).] if $x_{3}^{-}(t,x)=L$, then
\beq\label{eq:l24}
|x_3|> \frac{L}{2}\Rightarrow\varphi_{L}(x_{3})=0\Rightarrow\widetilde{f}_{+}(t,x)=0,\quad\forall\, (t,x)\in\overline{C}^-_L.
\eeq 
\end{itemize}

\smallskip

{\it Step 2.2. Local energy estimate.}
Since $L_{-}=T+Z_{-}$ and $L_{-} \lambda_{+}=0$, we take inner product of \eqref{eq:l21} with $\lambda_{+} \widetilde{f}_{+}$ to obtain
\beno
\frac{1}{2} L_{-}\big(\lambda_{+} \big|\widetilde{f}_{+}\big|^2\big)- \mu\Delta \widetilde{f}_{+}\cdot \lambda_{+} \widetilde{f}_{+}= \lambda_{+} \widetilde{f}_{+}\cdot \widetilde{\rho}_{+}, \quad\forall\,(t,x)\in[0,T^*]\times Q_L.
\eeno

Fix an arbitrary $u_{+}\in [-\frac{L}{4},\frac{L}{4}]$, we integrate the above equality on $W_{t,+}^{\geq u_{+}}\bigcap \overline{W}_{t,-}^{<L}$ to get
\begin{equation}\label{eq:l25}\begin{aligned}
&\frac{1}{2}\iint_{W_{t,+}^{\geq u_{+}}\bigcap \overline{W}_{t,-}^{< L}}\wt{\dive}\big(\lambda_{+} |\wt f_{+}|^2 L_-\big) dx d\tau-\mu\iint_{W_{t,+}^{\geq u_{+}}\bigcap \overline{W}_{t,-}^{<L}} \Delta\wt f_{+}\cdot \lambda_{+}\wt f_{+}\, dx d\tau\\ 
&\qquad=\iint_{W_{t,+}^{\geq u_{+}}\bigcap \overline{W}_{t,-}^{< L}}\lambda_{+}\wt f_{+}\cdot\wt \rho_{+}\, dx d\tau,
\end{aligned}\end{equation}
where we used the notation $L_-=(1,Z_-)$ for a vector field and the property $\wt\dive L_-=0$. Here we remark that $\lambda_+=\lambda_+^1$ is smooth on $W_{t,+}^{\geq u_{+}}\bigcap \overline{W}_{t,-}^{<L}$.

 {\bf (a). Estimate of $\frac{1}{2}\iint_{W_{t,+}^{\geq u_{+}}\bigcap \overline{W}_{t,-}^{< L}}\wt{\dive}\big(\lambda_{+} |\wt f_{+}|^2 L_-\big) dx d\tau$.}  By virtue of divergence theorem, we have
\beno
\iint_{W_{t,+}^{\geq u_{+}}\bigcap \overline{W}_{t,-}^{< L}}\wt{\dive}\big(\lambda_{+} |\wt f_{+}|^2 L_-\big) dx d\tau
=\int_{\p(W_{t,+}^{\geq u_{+}}\bigcap \overline{W}_{t,-}^{< L})}\lambda_{+} |\wt f_{+}|^2 \langle L_-,\vv\nu\rangle d\sigma,
\eeno 
where $\vv\nu$ is the unit outward normal vector on $W_{t,+}^{\geq u_{+}}\bigcap \overline{W}_{t,-}^{< L}$.
Note that 
\beno 
\p(W_{t,+}^{\geq u_{+}}\bigcap \overline{W}_{t,-}^{< L})
=\Bigl(\p(W_{t,+}^{\geq u_{+}}\bigcap \overline{W}_{t,-}^{< L})\bigcap\p([0,t]\times Q_L)\Bigr)\bigcup C_{t,u_+}^{+}\bigcup\overline{C}_{t,L}^{-}.
\eeno 

Following similar derivation as in Step 1.1., we get
\begin{equation}\label{eq:l26}
\begin{aligned}
\iint_{W_{t,+}^{\geq u_{+}}\bigcap \overline{W}_{t,-}^{< L}}&\wt{\dive}\big(\lambda_{+} |\wt f_{+}|^2 L_-\big) dx d\tau
 = \int_{\Sigma_{t,+}^{\geq u_{+}}\bigcap\overline{\Sigma}_{t,-}^{< L}} \lambda_{+} \big|\widetilde{f}_{+}\big|^2 dx- \int_{\Sigma_{0,+}^{\geq u_{+}}\bigcap\overline{\Sigma}_{0,-}^{< L}} \lambda_{+} \big|\widetilde{f}_{+}\big|^2 dx\\
& \quad\ +\sum_{i=1}^2\underbrace{\int_0^t\int_{\Sigma_{\tau,+}^{\geq u_{+}}\bigcap \overline{\Sigma}_{\tau,-}^{< L}\bigcap \{x_i=L\}}(\lambda_{+}|_{x_i=L}-\lambda_{+}|_{x_i=-L}) \big|\widetilde{f}_{+}\big|^2 z_{-}^{i}\, d\widehat{x}_i d\tau}_{J_i}\\ 
& \quad\ +\underbrace{\int_{\overline{C}_{t,L}^{-}} \lambda_{+} \big|\widetilde{f}_{+}\big|^2\left\langle L_{-}, \nu_{-}\right\rangle\, d\sigma_{-}}_{=0,\ \text{ since }L_{-}\text{ is tangential to }\overline{C}_{t,L}^{-}}+ \underbrace{\int_{C_{t,u_{+}}^{+}} \lambda_{+} \big|\widetilde{f}_{+}\big|^2\left\langle L_{-}, \nu_{+}\right\rangle\, d\sigma_{+}}_{J_{3}},
\end{aligned} 
\end{equation}
and
\beq\label{eq:l27} \begin{aligned}
|J_i|&\overset{|\wt f_+|\leq|f_+|}{\leq}\int_0^t\int_{\{x_i=L\}}\bigl|\lambda_{+}|_{x_i=L}-\lambda_{+}|_{x_i=-L}\bigr|\cdot \big|f_{+}\big|^2\cdot |z_{-}^{i}|\, d\widehat{x}_i d\tau\\
&\lesssim F(L)\e\cdot\bigl(T^*\cdot\sup_{0\leq\tau\leq t}\|\sqrt{\lambda_{+}} f_+\|_{L^2_x}^2+\|\sqrt{\lambda_{+}}\na f_+\|_{L^2_t(L^2_x)}^2\bigr),\quad i=1,2.
\end{aligned}\eeq

For $J_3$, by the facts that $\langle L_{-}, \nu_{+}\rangle\in[\f{7}{16},\, 4]$  (Lemma \ref{lem:e5}) and $\wt f_+|_{C_{u_{+}}^{+}}=f_+|_{C_{u_{+}}^{+}}$ (see \eqref{eq:l23}), we have
\beno
J_{3}= \int_{C_{t,u_{+}}^{+}} \lambda_{+} \left|f_{+}\right|^2\left\langle L_{-}, \nu_{+}\right\rangle\, d\sigma_{+} \geq \frac{7}{16}\int_{C_{t,u_{+}}^{+}} \lambda_{+} \left|f_{+}\right|^2 d\sigma_{+},
\eeno
from which, \eqref{eq:l27} and \eqref{eq:l26}, we deduce that
\beno
\begin{aligned}
\frac{1}{2}&\iint_{W_{t,+}^{\geq u_{+}}\bigcap \overline{W}_{t,-}^{< L}}\wt{\dive}\big(\lambda_{+} |\wt f_{+}|^2 L_-\big) dx d\tau
\geq \frac{1}{2}\int_{\Sigma_{t,+}^{\geq u_{+}}\bigcap\overline{\Sigma}_{t,-}^{< L}} \lambda_{+} \big|\widetilde{f}_{+}\big|^2 dx- \frac{1}{2}\int_{\Sigma_{0,+}^{\geq u_{+}}\bigcap\overline{\Sigma}_{0,-}^{< L}} \lambda_{+} \big|\widetilde{f}_{+}\big|^2 dx\\ 
&\quad + \frac{7}{32}\int_{C_{t,u_{+}}^{+}} \lambda_{+} \left|f_{+}\right|^2 d\sigma_{+}- CF(L)\e\cdot\bigl(T^*\cdot\sup_{0\leq\tau\leq t}\|\sqrt{\lambda_{+}} f_+\|_{L^2_x}^2+\|\sqrt{\lambda_{+}}\na f_+\|_{L^2_t(L^2_x)}^2\bigr).
\end{aligned} 
\eeno
Thus, using the facts that $F(L)T^*\leq\f{1}{L}$ and $F(L)\leq\f{\mu}{L}$ (see \eqref{eq:l20a}), we obtain
\begin{equation}\label{eq:l28}
\begin{aligned}
\frac{1}{2}\iint_{W_{t,+}^{\geq u_{+}}\bigcap \overline{W}_{t,-}^{< L}}&\wt{\dive}\big(\lambda_{+} |\wt f_{+}|^2 L_-\big) dx d\tau
\geq \frac{7}{32}\int_{C_{t,u_{+}}^{+}} \lambda_{+} \left|f_{+}\right|^2 d\sigma_{+}- \frac{1}{2}\int_{\Sigma_{0}} \lambda_{+} \big|f_{+}\big|^2 dx\\ 
&\quad - \f{C\e}{L}\cdot\bigl(\sup_{0\leq\tau\leq t}\|\sqrt{\lambda_{+}}f_+\|_{L^2_x}^2+\mu\|\sqrt{\lambda_{+}}\na f_+\|_{L^2_t(L^2_x)}^2\bigr).
\end{aligned} 
\end{equation}

{\bf (b). Estimate of $-\mu\iint_{W_{t,+}^{\geq u_{+}}\bigcap \overline{W}_{t,-}^{<L}} \Delta\wt f_{+}\cdot \lambda_{+}\wt f_{+}\, dx d\tau$. }
Using \eqref{eq:l15a}, we have
\beq\label{eq:l29}
-\mu\iint_{W_{t,+}^{\geq u_{+}}\bigcap \overline{W}_{t,-}^{<L}} \Delta\wt f_{+}\cdot \lambda_{+}\wt f_{+}\, dx d\tau
=\mu\iint_{W_{t,+}^{\geq u_{+}}\bigcap \overline{W}_{t,-}^{<L}}\lambda_+|\na\wt{f}_+|^2dxd\tau+K_1+K_2,
\eeq 
where
\beno\begin{aligned}
&K_1=-\f{\mu}{2}\iint_{W_{t,+}^{\geq u_{+}}\bigcap \overline{W}_{t,-}^{<L}}\na\cdot\bigl[\lambda_+\na (|\wt f_+|^2)\bigr]dxd\tau,\\
&K_2=\mu\iint_{W_{t,+}^{\geq u_{+}}\bigcap \overline{W}_{t,-}^{<L}}(\na\lambda_+\cdot\na)\wt f_+\cdot\wt f_+ dxd\tau.
\end{aligned}\eeno 

For $K_2$,  we have
\beq\label{eq:l30}\begin{aligned}
|K_2|&\leq \frac{\mu}{2}\iint_{W_{t,+}^{\geq u_{+}}\bigcap \overline{W}_{t,-}^{< L}} \lambda_{+} \big|\nabla \widetilde{f}_{+}\big|^2 dx d\tau
+\frac{\mu}{2}\iint_{W_{t,+}^{\geq u_{+}}\bigcap \overline{W}_{t,-}^{< L}}\f{|\na\lambda_{+}|^2}{\lambda_{+}} \big|\widetilde{f}_{+}\big|^2 dx d\tau\\ 
&\leq \frac{\mu}{2}\iint_{W_{t,+}^{\geq u_{+}}\bigcap \overline{W}_{t,-}^{< L}} \lambda_{+} \big|\nabla \widetilde{f}_{+}\big|^2 dx d\tau
+\frac{\mu}{2}\Bigl\|\f{|\na\lambda_{+}|}{\sqrt{\lambda_{+}}} f_{+}\Bigr\|_{L^2_t(L^2_x)}^2, 
\end{aligned}\eeq
where we used $|\wt f_+|\leq|f_+|$ in the last inequality.

For $K_1$, applying the divergence theorem, similar to \eqref{eq:l26},  we obtain
\beno\begin{aligned}
K_1&=-\underbrace{\f{\mu}{2}\sum_{i=1}^2\int_0^t\int_{\Sigma_{\tau,+}^{\geq u_{+}}\bigcap \overline{\Sigma}_{\tau,-}^{< L}\bigcap \{x_i=L\}}(\lambda_{+}|_{x_i=L}-\lambda_{+}|_{x_i=-L})\p_i(|\wt f_+|^2)\, d\widehat{x}_i d\tau}_{K_{11}}\\ 
& \quad-\underbrace{\f{\mu}{2}\int_{\overline{C}_{t,L}^{-}} \lambda_+\na (|\wt f_+|^2)\cdot\vv n_{-}\, d\sigma_{-}}_{=0,\ \text{ since }\wt{f}_+|_{\overline{C}_{t,L}^-}=0}-\underbrace{\f{\mu}{2}\int_{C_{t,u_{+}}^{+}} \lambda_+\na (|\wt f_+|^2)\cdot\vv n_{+}\, d\sigma_{+}}_{K_{12}},
\end{aligned}\eeno
where $\overline{C}_{t,L}^-=\bigcup_{0\leq\tau\leq t}\overline{S}^-_{\tau,L}$ with the normal vector on $\overline{S}^-_{\tau,L}$ given by $\vv n_-=\f{\na_xx_3^-(\tau,x)}{|\na_{\tau,x}x_3^-(\tau,x)|}$, and $\overline{C}_{t,u_+}^+=\bigcup_{0\leq\tau\leq t}\overline{S}^+_{\tau,u_+}$ with the  normal vector on $\overline{S}^+_{\tau,u_+}$ given by $\vv n_+=-\f{\na_xu_+(\tau,x)}{|\na_{\tau,x}u_+(\tau,x)|}$.

Since 
\beno 
|\p_i(|\wt f_+|^2)|=2|\varphi_L(x_3)|^2\cdot |f_+\cdot\p_i f_+|\leq 2|f_+|\cdot|\p_i f_+|,\quad\text{for}\,\, i=1,2,
\eeno
similarly as $J_i$, we get by using the facts $F(L)T^*\leq\f{1}{L}$, $F(L)\leq\f{\mu}{L}$ (see \eqref{eq:l20a}) that
\beno\begin{aligned} 
|K_{11}|&\lesssim\mu\cdot F(L)\e\cdot\bigl(T^*\cdot\sup_{0\leq\tau\leq t}\|\sqrt{\lambda_{+}} f_+\|_{L^2_x}^2+\|\sqrt{\lambda_{+}}\na f_+\|_{L^2_t(L^2_x)}^2+\|\sqrt{\lambda_{+}}\na^2 f_+\|_{L^2_t(L^2_x)}^2\bigr)\\ 
&\lesssim\f{\mu\e}{L}\cdot\sup_{0\leq\tau\leq t}\|\sqrt{\lambda_{+}} f_+\|_{L^2_x}^2+\f{\mu\e}{L}\cdot\mu\|\sqrt{\lambda_{+}}\na f_+\|_{L^2_t(L^2_x)}^2+\f{\e}{L}\cdot\mu^2\|\sqrt{\lambda_{+}}\na^2 f_+\|_{L^2_t(L^2_x)}^2.
\end{aligned}\eeno 

For $K_{12}$, noting that $\wt f_+|_{C_{u_+}^{+}}=f_+|_{C_{u_+}^{+}}$ (see \eqref{eq:l23}), we derive
\beno\begin{aligned}
|K_{12}| &=\mu\bigl|\int_{C_{t,u_{+}}^{+}} \lambda_{+}f_{+}\cdot(\vv n_+\cdot\na)f_+\,d\sigma_{+}\bigr|\leq \frac{1}{8}\int_{C_{t,u_{+}}^{+}} \lambda_{+} \left|f_{+}\right|^2 d\sigma_{+}+ 2\mu^2\int_{C_{t,u_{+}}^{+}} \lambda_{+} \left|\nabla f_{+}\right|^2 d\sigma_{+} .
\end{aligned}\eeno
By Lemma \ref{lem:e6} (trace lemma) and property \eqref{eq:l6} of $\lambda_{+}$, we have
\beno
\int_{C_{t,u_{+}}^{+}} \lambda_{+} \left|\nabla f_{+}\right|^2 d\sigma_{+}\lesssim\int_0^t\|\sqrt{\lambda_{+}}\nabla f_{+}\|_{H^1(\Sigma_\tau)}^2d\tau\lesssim\int_{0}^{t} \int_{\Sigma_{\tau}} \lambda_{+}\big(\left|\nabla f_{+}\right|^2+\left|\nabla^2 f_{+}\right|^2\big)\, dx d\tau. 
\eeno

Combining these estimates, we conclude that
\beq\label{eq:l31}\begin{aligned}
|K_1|&\leq \frac{1}{8}\int_{C_{t,u_{+}}^{+}} \lambda_{+} \left|f_{+}\right|^2 d\sigma_{+}+C\mu\cdot\bigl(\sup_{0\leq\tau\leq t}\|\sqrt{\lambda_{+}} f_+\|_{L^2_x}^2+\mu\|\sqrt{\lambda_{+}}\na f_+\|_{L^2_t(L^2_x)}^2\bigr)\\ 
&\qquad\quad+C\mu^2\|\sqrt{\lambda_{+}}\na^2 f_+\|_{L^2_t(L^2_x)}^2.
\end{aligned}\eeq

Therefore, thanks to \eqref{eq:l30} and \eqref{eq:l31}, we deduce from \eqref{eq:l29} that 
\beq\label{eq:l32}
\begin{aligned}
& -\mu\iint_{W_{t,+}^{\geq u_{+}}\bigcap \overline{W}_{t,-}^{< L}} \Delta \widetilde{f}_{+}\cdot \lambda_{+} \widetilde{f}_{+}\, dx d\tau\\ 
&\quad\geq \frac{\mu}{2}\iint_{W_{t,+}^{\geq u_{+}}\bigcap \overline{W}_{t,-}^{<L}} \lambda_{+} \big|\nabla \widetilde{f}_{+}\big|^2 dx d\tau-\frac{\mu}{2}\Bigl\|\f{|\na\lambda_{+}|}{\sqrt{\lambda_{+}}} f_{+}\Bigr\|_{L^2_t(L^2_x)}^2-C\mu^2\|\sqrt{\lambda_{+}}\na^2 f_+\|_{L^2_t(L^2_x)}^2\\ 
&\qquad -C\mu\cdot\bigl(\sup_{0\leq\tau\leq t}\|\sqrt{\lambda_{+}} f_+\|_{L^2_x}^2+\mu\|\sqrt{\lambda_{+}}\na f_+\|_{L^2_t(L^2_x)}^2\bigr)
 - \frac{1}{8}\int_{C_{t,u_{+}}^{+}} \lambda_{+} \left|f_{+}\right|^2 d\sigma_{+}.
\end{aligned}
\eeq
We remark that the last term in \eqref{eq:l32} can be obsorbed by the first term $\frac{7}{32}\int_{C_{t,u_{+}}^{+}} \lambda_{+} \left|f_{+}\right|^2 d\sigma_{+}$ on the RHS of \eqref{eq:l28}, since $\f18<\frac{7}{32}$. Furthermore, we may discard the non-negative term $\frac{\mu}{2}\iint_{W_{t,+}^{\geq u_{+}}\bigcap \overline{W}_{t,-}^{<L}} \lambda_{+} \big|\nabla \widetilde{f}_{+}\big|^2 dx d\tau$  appearing on the RHS of \eqref{eq:l32}.

{\bf (c). Estimate of $\iint_{W_{t,+}^{\geq u_{+}}\bigcap \overline{W}_{t,-}^{< L}}\lambda_{+}\wt f_{+}\cdot\wt \rho_{+}\, dx d\tau$.}
By virtue of \eqref{eq:l22} and $|\wt f_+|\leq |f_+|$, we have
\begin{equation}\label{eq:l33}
\begin{aligned}
&\iint_{W_{t,+}^{\geq u_{+}}\bigcap \overline{W}_{t,-}^{<L}}\lambda_{+} \widetilde{f}_{+}\cdot \widetilde{\rho}_{+}\, dx d\tau
\lesssim \int_{0}^{t} \int_{\Sigma_{\tau}} \lambda_{+} \left|f_{+}\right|\cdot |\rho_{+}|\, dx d\tau+\frac{1}{L}\int_{0}^{t} \int_{\Sigma_{\tau}} \lambda_{+} \left|f_{+}\right|^2 dx d\tau\\ 
&\qquad\qquad\qquad\qquad\qquad+ \frac{\mu}{L}\int_{0}^{t} \int_{\Sigma_{\tau}} \lambda_{+}\left|f_{+}\right| \cdot\left|\nabla f_{+}\right|\, dx d\tau\\
&\quad \lesssim \int_{0}^{t} \int_{\Sigma_{\tau}} \lambda_{+} \left|f_{+}\right|\cdot |\rho_{+}|\, dx d\tau+ \frac{T^*}{L}\sup_{0\leq\tau\leq t}\|\sqrt{\lambda_{+}}f_{+}\|_{L^2_x}^2+ \frac{\mu}{L}\|\sqrt{\lambda_{+}} \nabla f_{+}\|_{L^2_t(L^2_x)}^2\\
& \overset{T^*=\log L}{\lesssim} \int_{0}^{t} \int_{\Sigma_{\tau}} \lambda_{+} \left|f_{+}\right| |\rho_{+}|\, dx d\tau+ \frac{\log L}{L}\sup\limits_{0\leq \tau \leq t} \|\sqrt{\lambda_{+}} f_{+}\|^2_{L^2(\Sigma_\tau)}+ \frac{\mu}{L}\|\sqrt{\lambda_{+}} \nabla f_{+}\|_{L^2_t(L^2_x)}^2.
\end{aligned} 
\end{equation}

\smallskip 

Combining \eqref{eq:l28}, \eqref{eq:l32} and \eqref{eq:l33}, we deduce from \eqref{eq:l25} that
\beq\label{eq:l34}\begin{aligned}
&\frac{3}{32}\int_{C_{t,u_{+}}^{+}} \lambda_{+} \left|f_{+}\right|^2 d\sigma_{+}\leq\frac{1}{2}\int_{\Sigma_{0}} \lambda_{+} \big|f_{+}\big|^2 dx+\frac{\mu}{2}\Bigl\|\f{|\na\lambda_{+}|}{\sqrt{\lambda_{+}}} f_{+}\Bigr\|_{L^2_t(L^2_x)}^2+C\mu^2\|\sqrt{\lambda_{+}}\na^2 f_+\|_{L^2_t(L^2_x)}^2\\ 
&\quad+C\int_{0}^{t} \int_{\Sigma_{\tau}} \lambda_{+} \left|f_{+}\right|\cdot |\rho_{+}|\, dx d\tau+\underbrace{C\bigl(\f{\log L}{L}+\mu\bigr)\cdot\bigl(\sup_{0\leq\tau\leq t}\|\sqrt{\lambda_{+}}f_+\|_{L^2_x}^2+\mu\|\sqrt{\lambda_{+}}\na f_+\|_{L^2_t(L^2_x)}^2\bigr)}_{\text{can be obsorbed into the LHS terms in \eqref{eq:l20} due to $L\gg 1$ and $\mu\ll1$} }.
\end{aligned}\eeq

\medskip 

{\bf Step 3. Final energy estimates.}
By $L\gg 1$ and $\mu\ll1$, combined with \eqref{eq:l20} and \eqref{eq:l34}, we finally get 
\begin{equation} \label{eq:l35}
\begin{aligned}
& \quad\sup_{0\leq\tau\leq t}\|\sqrt{\lambda_+}f_+\|_{L^2(\Sigma_\tau)}^2+\mu\|\sqrt{\lambda_+}\na f_+\|_{L^2_t(L^2_x)}^2+ \sup\limits_{|u_{+}|\leq \frac{L}{4}} \int_{C^{+}_{t,u_{+}}} \lambda_{+}\left|f_{+}\right|^2 d\sigma_{+}\\
& \leq 10\|\sqrt{\lambda_+}f_+\|_{L^2(\Sigma_0)}^2+14\mu \Bigl\|\f{|\na\lambda_+|}{\sqrt{\lambda_+}} f_+\Bigr\|_{L^2_t(L^2_x)}^2+C\mu^2\|\sqrt{\lambda_+}\na^2 f_+\|_{L_t^2(L^2_x)}^2\\ 
&\qquad+ C\int_{0}^{t} \int_{\Sigma_{\tau}} \lambda_+ \left|f_+\right|\cdot |\rho_+|\, dx d\tau.
\end{aligned}
\end{equation}
This is the desired estimate \eqref{eq:l7} for $f_+$. This completes the proof of the proposition.
\end{proof}
A byproduct of the proof is the energy inequality \eqref{eq:l20}. Since it will play a crucial role in controlling the viscosity terms later on, we restate it as the following proposition.

\begin{proposition}\label{prop:l2}
Under the assumptions of Proposition \ref{prop:l1}, for any sufficiently smooth $2L$-periodic vector fields $f_\pm(t,x)$ satisfying \eqref{eq:l1}, there holds
\beq\label{eq:l36}\begin{aligned}
\sup_{0\leq\tau\leq t}\|\sqrt{\lambda_\pm}f_\pm&\|_{L^2(\Sigma_\tau)}^2+\mu\|\sqrt{\lambda_\pm}\na f_\pm \|_{L^2_t(L^2_x)}^2
\leq2\|\sqrt{\lambda_\pm }f_\pm \|_{L^2(\Sigma_0)}^2+4\mu \Bigl\|\f{|\na\lambda_\pm |}{\sqrt{\lambda_\pm }} f_\pm \Bigr\|_{L^2_t(L^2_x)}^2\\ 
&+C\e\mu F(L)\cdot\|\sqrt{\lambda_\pm }\na^2 f_\pm \|_{L_t^2(L^2_x)}^2+ 4\int_{0}^{t} \int_{\Sigma_{\tau}} \lambda_\pm  \left|f_\pm \right|\cdot |\rho_\pm |\, dx d\tau,
\end{aligned}\eeq 
where $F(L)$ is defined in \eqref{eq:def of F}.
\end{proposition}

\subsection{Energy estimates for the lowest order terms}
In this section,  we will apply  Proposition \ref{prop:l1} to the following MHD system (\eqref{eq:MHD}):
\beq\label{MHD}
\left\{\begin{array}{l}
\partial_{t} z_{+}+ Z_{-}\cdot \nabla z_{+}- \mu\Delta z_{+}= -\nabla p, \\
\partial_{t} z_{-}+ Z_{+}\cdot \nabla z_{-}- \mu\Delta z_{-}= -\nabla p, \\
\operatorname{div} z_+ = \operatorname{div} z_- = 0,
\end{array}\right. 
\eeq
with the weight functions $\lambda_{\pm}=\left(\log \langle w_{\mp}\rangle\right)^4$. For the case $\lambda_{\pm}=1$,  the divergence-free condition $\operatorname{div} Z_{\pm}=0$ yields the basic energy identity:
\begin{equation} \label{eq:basic}
\int_{\Sigma_{t}} \left|z_{\pm}\right|^2 dx+ 2\mu\int_{0}^{t} \int_{\Sigma_{\tau}} \left|\nabla z_{\pm}\right|^2 dx d\tau = \int_{\Sigma_{0}} \left|z_{\pm}\right|^2 dx .
\end{equation}
In particular, this implies
\begin{equation}
\mu\int_{0}^{t} \int_{\Sigma_{\tau}} \left|\nabla z_{\pm}\right|^2 dx d\tau \leq \frac{1}{2}\int_{\Sigma_{0}} \left|z_{\pm}\right|^2 dx . 
\end{equation}

The main result of this section is the following proposition concerning the lowest order energy estimate.
\begin{proposition}\label{prop:estimates on lowest order terms}
Assume that R is very large, under the assumptions \eqref{eq:a1}--\eqref{eq:a3}, there holds
\begin{equation}
\begin{aligned}
&\ E_{\pm}+ F_{\pm}+ D_{\pm}\lesssim E_{\pm}(0)+ \Big(E_{\mp}+ \sum_{k=0}^2 E_{\mp}^k \Big)^{\frac{1}{2}} \left(E_{\pm}^{0}+ F_{\pm}^{0}\right)+ \mu D_{\pm}^{0} .
\end{aligned} \label{eq:estimates on lowest order terms}
\end{equation}
\end{proposition}

\subsubsection{Estimates on the pressure}
In the current subsection, we aim to derive the estimates concerning the pressure term $\nabla p$. We first derive the expression of $\nabla p$ on $Q_L$ as follows: 

\begin{lemma}[Expression of the pressure]\label{lem:pressure}
Assume that $(z_+,z_-)$ are sufficiently smooth $2L$-periodic solutions to \eqref{MHD}. Denoting by  $\vv k=(k_1,k_2,k_3)\in\mathbb{Z}^3$, and
\beno 
G_{\vv k}(x,y):=\f{1}{4\pi}\cdot\f{1}{|x-y_{\vv k}|},\quad 
y_{\vv k}:=y+2\vv k L,\quad\forall\, x,y\in Q_L,
\eeno 
we have for any $(t,x)\in[0,T^*]\times Q_L$,
\beq\label{eq:pressure}
\begin{aligned}
\nabla p(t,x) = &\sum_{k_1,k_2,k_3=-1}^1\sum_{i,j=1}^3 \int_{Q_L} \nabla_x G_{\vv k}(x,y) \bigl(\partial_{i} z_{+}^{j} \partial_{j} z_{-}^{i}\bigr)(t, y) dy \\ 
&\quad+\na p_{\geq 2}(t,x)+\vv B_1(t,x)+\vv B_2(t,x).
\end{aligned}
\eeq
Here $\na p_{\geq 2}(t,x)$ is the integral term with the kernel far  from zero and satisfying
\beq\label{eq:est for p geq 2}
|\na p_{\geq 2}(t,x)|\lesssim\f{1}{L^4}\int_{Q_L}|z_-(t,y)|\cdot|z_+(t,y)|dy,
\eeq
while 
 $\vv B_1(t,x)$ and $\vv B_2(t,x)$ are boundary integral terms satisfying
\beq\label{eq:est for B 1 2}\begin{aligned}
&\bigl|\vv B_1(t,x)\bigr|\lesssim \f{1}{L^2}\cdot\sum_{i=1}^3\int_{\{y_i=L\}}|z_\mp(t,y)|\cdot|\na z_\pm(t,y)|d\widehat{y_i},\\
&\bigl|\vv B_2(t,x)\bigr|
\lesssim\f{1}{L^3}\cdot\sum_{j=1}^3\int_{\{y_j=L\}}|z_-(t,y)|\cdot| z_+(t,y)|d\widehat{y_j}. 
\end{aligned}\eeq
\end{lemma}
\begin{proof} We divide the proof into several steps.

{\bf Step 1. Formal derivation of the pressure.} Since $\dive\,z_\pm=0$, taking the divergence  of the first equation in \eqref{MHD} yields
\beq\label{eq:P 1}
-\Delta p=\dive(z_+\cdot\na z_-)=:f,\quad\forall\,(t,x)\in(0,T^*]\times Q_L,
\eeq
where 
\beno 
f=\sum_{i,j=1}^3\p_j(z_+^i\p_iz_-^j)=\sum_{i,j=1}^3\p_j\p_i(z_+^iz_-^j)=\sum_{i,j=1}^3\p_jz_+^i\p_iz_-^j.
\eeno

We extend $f(t,\cdot)$ from $Q_L$ to $\R^3$ periodically by defining
\beno
\wt f(t,x)=f(t,x-2\vv k L),\quad\text{for}\quad x_i\in\bigl[(2k_i-1)L,\,(2k_i+1)L\bigr),\,\, i=1,2,3,
\eeno
where $\vv k=(k_1,k_2,k_3)\in\mathbb{Z}^3$. For $f\in C([0,T^*]\times Q_L)\cap L^\infty\bigl([0,T^*]\times Q_L\bigr)$, it follows that $\wt f\in C([0,T^*]\times\R^3)\cap L^\infty\bigl([0,T^*]\times\R^3\bigr)$.

We now consider the solution $\wt{p}(t,x)$ to the equation
\beno 
-\Delta\wt p=\wt f,\quad x\in\R^3.
\eeno
Then the restriction $p=\wt p|_{Q_L}$ solves \eqref{eq:P 1} on $Q_L$. By the theory of the Laplace equation, there exists a unique bounded solution $\wt p$ (up to a constant) such that
\beq\label{eq:P 2}
\na p(t,x)=\na \wt p(t,x)=\f{1}{4\pi}\int_{\R^3}\na\bigl(\f{1}{|x-y|}\bigr)\wt f(t,y)dy,\quad\forall\, x\in Q_L.
\eeq
This is the formal expression for $\na p$. The finiteness of the integral on the RHS of \eqref{eq:P 2} must be verified.

\medskip

{\bf Step 2. The explicit expression of the pressure.} In this step, we will derive a rigorous explicit expression for the pressure.

{\it Step 2.1. Formal expression in terms of $f$.}  By the definition of $\wt f$, we have
\beq\label{eq:P 3}\begin{aligned}
\na p(t,x)&=\f{1}{4\pi}\sum_{\vv k\in\mathbb{Z}^3}\int_{(2k_3-1)L}^{(2k_3+1)L}\int_{(2k_2-1)L}^{(2k_2+1)L}\int_{(2k_1-1)L}^{(2k_1+1)L}
\na\bigl(\f{1}{|x-y|}\bigr)f(t,y-2\vv k L)dy_1dy_2dy_3\\ 
&=\f{1}{4\pi}\sum_{\vv k\in\mathbb{Z}^3}\int_{Q_L}\na\bigl(\f{1}{|x-y_{\vv k}|}\bigr)f(t,y)dy,\quad\forall\, x\in Q_L,
\end{aligned}\eeq
where $y_{\vv k}=y+2\vv k L$. Note that on $Q_L$,
\beno 
-\Delta\Bigl( \f{1}{4\pi}\int_{Q_L}\f{1}{|x-y_{\vv k}|}f(t,y)dy\Bigr)=\left\{\begin{aligned}
&f(t,x),\quad\text{if }\,\, \vv k=\vv 0,\\ 
& 0,\quad\text{if }\,\, \vv k\neq\vv 0.
\end{aligned}\right.
\eeno
This shows that to verify that $\na p$  defined in \eqref{eq:P 2} solves \eqref{eq:P 1}, we only need to check that the series on the RHS of  \eqref{eq:P 3} is convergent.

In view of the sigularity of $\f{1}{|x-y_{\vv k}|}$, we specify $\mathbb{Z}^3$ into three regions:
\beno
\mathbb{Z}^3=\bigl\{\vv k=\vv 0\bigr\}\bigcup\bigl\{\vv k\in\mathbb{Z}^3\,|\, \max\{|k_1|,\,|k_2|,\,|k_3|\}=1\bigr\}\bigcup\underbrace{\bigl\{\vv k\in\mathbb{Z}^3\,|\, \max\{|k_1|,\,|k_2|,\,|k_3|\}\geq 2\bigr\}}_{K_{\geq 2}}.
\eeno
For simplicity, we use the notation $G_{\vv k}(x,y)$ and denote by
\beno
p_{\vv k}(t,x):=\int_{Q_L}G_{\vv k}(x,y)f(t,y)dy,\quad\forall\, x\in Q_L,
\eeno
then
\beq\label{eq:P 3a} 
\na p=\sum_{k_1,k_2,k_3=-1}^1\na p_{\vv k}+\sum_{\vv k\in K_{\geq 2}}\na p_{\vv k}.
\eeq

{\it Step 2.2. An equivalent expression of  $\na p_{\vv k}$  for  $\vv k\in K_{\geq 2}$.} In this case, for any $x,y\in Q_L$, there holds
\beno
|x-y_{\vv k}|=|x-y-2\vv kL|\geq 2L,
\eeno
which gives rise to
\beq\label{eq:P 4}
\bigl|\na_x^m G_{\vv k}(x,y)\bigr|\lesssim\bigl|x-y_{\vv k}\bigr|^{-m-1}\lesssim \f{1}{L^{m+1}},\quad\forall\, m\in\mathbb{Z}_{\geq 0}.
\eeq

Since  $f=\sum_{i,j=1}^3\p_i\p_j(z_+^iz_-^j)$ is $2L$-periodic, integration by parts yields 
\beno\begin{aligned}
\na p_{\vv k}(t,x)&=-\sum_{i,j=1}^3\int_{Q_L}\p_{y_i}\na_xG_{\vv k}(x,y)\p_j(z_+^iz_-^j)(t,y)dy+\vv B_{1,\vv k}(t,x)\\ 
&=\sum_{i,j=1}^3\int_{Q_L}\p_{x_i}\na_xG_{\vv k}(x,y)\p_j(z_+^iz_-^j)(t,y)dy+\vv B_{1,\vv k}(t,x),
\end{aligned}\eeno
where the boundary term  $\vv B_{1,\vv k}(t,x)$ is given by
\beno 
\vv B_{1,\vv k}(t,x):=\sum_{i=1}^3\int_{\{y_i=L\}}\Bigl(\underbrace{\na_xG_{\vv k}(x,y)\big|_{y_i=L}-\na_xG_{\vv k}(x,y)\big|_{y_i=-L}}_{q_{1,\vv k}(x,\widehat{y_i})}\Bigr)(z_-\cdot\na z_+^i)(t,y)d\widehat{y_i}.
\eeno
We remark here, there also holds
\beno
\na p_{\vv k}(t,x)=\sum_{i,j=1}^3\int_{Q_L}\p_{x_j}\na_xG_{\vv k}(x,y)\p_i(z_+^iz_-^j)(t,y)dy+\vv B_{1,\vv k}(t,x),
\eeno
with
\beno 
\vv B_{1,\vv k}(t,x)=\sum_{j=1}^3\int_{\{y_j=L\}}q_{1,\vv k}(x,\widehat{y_j})(z_+\cdot\na z_-^j)(t,y)d\widehat{y_j}.
\eeno 
Using integration by parts again, we have
\beq\label{eq:P 5}
\na p_{\vv k}(t,x)=\sum_{i,j=1}^3\int_{Q_L}\p_{x_j}\p_{x_i}\na_xG_{\vv k}(x,y)(z_+^iz_-^j)(t,y)dy+\vv B_{1,\vv k}(t,x)+\vv B_{2,\vv k}(t,x),
\eeq
where the boundary term  $\vv B_{2,\vv k}(t,x)$ is given by
\beno 
\vv B_{2,\vv k}(t,x):=\sum_{i,j=1}^3\int_{\{y_j=L\}}\Bigl(\underbrace{\p_{x_i}\na_xG_{\vv k}(x,y)\big|_{y_j=L}-\p_{x_i}\na_xG_{\vv k}(x,y)\big|_{y_j=-L}}_{q^i_{2,\vv k}(x,\widehat{y_j})}\Bigr)(z_+^iz_-^j)(t,y)d\widehat{y_j}.
\eeno

{\it Step 2.3. Estimates of $\sum_{\vv k\in K_{\geq 2}}\vv B_{1,\vv k}$ and $\sum_{\vv k\in K_{\geq 2}}\vv B_{2,\vv k}$.} To do this, we only need to estimate the summations of the kernels as follows:
\beno 
\sum_{\vv k\in K_{\geq 2}}q_{1,\vv k}(x,\widehat{y_i}),
\quad \sum_{\vv k\in K_{\geq 2}}q^i_{2,\vv k}(x,\widehat{y_j}).
\eeno

For $i=1$, since $y_1+2k_1L|_{y_1=-L}=y_1+2(k_1-1)L|_{y_1=L}$, there holds
\beno 
G_{(k_1,k_2,k_3)}(x,y)\big|_{y_1=-L}=G_{(k_1-1,k_2,k_3)}(x,y)\big|_{y_1=L},
\eeno
which gives rise to
\beno 
q_{1,\vv k}(x,\widehat{y_1})=\na_xG_{(k_1,k_2,k_3)}(x,y)\big|_{y_1=L}-\na_xG_{(k_1-1,k_2,k_3)}(x,y)\big|_{y_1=L}.
\eeno 
Since 
\beno 
K_{\geq 2}=\bigl\{\vv k\in\mathbb{Z}^3\,|\,|k_1|\geq 2,\, \max\{|k_2|,\,|k_3|\}\leq 1 \bigr\}\bigcup\bigl\{\vv k\in\mathbb{Z}^3\,|\,k_1\in\mathbb{Z},\,\max\{|k_2|,\,|k_3|\}\geq 2\bigr\},
\eeno
we have
\beq\label{eq:P 6}
\sum_{\vv k\in K_{\geq 2}}q_{1,\vv k}(x,\widehat{y_1})
=\sum_{k_2,k_3=-1}^1\Bigl(\na_x G_{(-2,k_2,k_3)}(x,y)\big|_{y_1=L}-\na_x G_{(1,k_2,k_3)}(x,y)\big|_{y_1=L}\Bigr).
\eeq
Similar expressions hold for $\sum_{\vv k\in K_{\geq 2}}q_{1,\vv k}(x,\widehat{y_i})$ with $i=2,3$. Using \eqref{eq:P 4}, we have
\beno 
\bigl|\sum_{\vv k\in K_{\geq 2}}q_{1,\vv k}(x,\widehat{y_1})\bigr|\lesssim\f{1}{L^2}.
\eeno 
The same estimates hold for $\sum_{\vv k\in K_{\geq 2}}q_{1,\vv k}(x,\widehat{y_i})$ with $i=2,3$. Consequently, denoting by 
$
\vv B_1(t,x):=\sum_{\vv k\in K_{\geq 2}}\vv B_{1,\vv k}(t,x),
$
we obtain
\beq\label{eq:P 7}
\bigl|\vv B_1(t,x)\bigr|
\lesssim\f{1}{L^2}\cdot\sum_{i=1}^3\int_{\{y_i=L\}}|z_\mp(t,y)|\cdot|\na z_\pm(t,y)|d\widehat{y_i}.
\eeq 

Similarly, denoting by 
$
\vv B_2(t,x):=\sum_{\vv k\in K_{\geq 2}}\vv B_{2,\vv k}(t,x),
$
we get
\beq\label{eq:P 8}
\bigl|\vv B_2(t,x)\bigr|
\lesssim\f{1}{L^3}\cdot\sum_{j=1}^3\int_{\{y_j=L\}}|z_-(t,y)|\cdot| z_+(t,y)|d\widehat{y_j}.
\eeq 

{\it Step 2.4. Estimate of $\sum_{\vv k\in K_{\geq 2}}\na p_{\vv k}$.} Due to \eqref{eq:P 7}, \eqref{eq:P 8} and \eqref{eq:P 5}, it remains to estimate the following term:
\beno 
\na p_{\geq 2}(t,x):=\sum_{\vv k\in K_{\geq 2}}\sum_{i,j=1}^3\int_{Q_L}\p_{x_j}\p_{x_i}\na_xG_{\vv k}(x,y)(z_+^iz_-^j)(t,y)dy.
\eeno 

Using \eqref{eq:P 4}, we have
\beno 
\sum_{\vv k\in K_{\geq 2}}\bigl|\p_{x_j}\p_{x_i}\na_xG_{\vv k}(x,y)\bigr|\lesssim\sum_{\vv k\in K_{\geq 2}}\f{1}{|x-y-2\vv kL|^4}
\lesssim\f{1}{L^4}\sum_{\vv k\in K_{\geq 2}}\f{1}{|\f{x-y}{L}-2\vv k|^4},
\eeno
which along with the fact that $|x-y-2\vv kL|\geq 2L$ yields 
\beno
\sum_{\vv k\in K_{\geq 2}}\bigl|\p_{x_j}\p_{x_i}\na_xG_{\vv k}(x,y)\bigr|\lesssim\f{1}{L^4},\quad\forall\, x,y\in Q_L.
\eeno
Thus, we obtain
\beq\label{eq:P 9}
\bigl|\na p_{\geq 2}(t,x)\bigr|\lesssim\f{1}{L^4}\int_{Q_L}|z_+(t,y)|\cdot|z_-(t,y)|dy,\quad\forall\, x\in Q_L.
\eeq 

{\it Step 2.5. The explicit expression of the pressure.} Thanks to Steps 2.1-2.4, we obtain the explicit expression for $\na p$ given in \eqref{eq:pressure}. Moreover,  the estimates  \eqref{eq:est for p geq 2} and \eqref{eq:est for B 1 2} hold.

On the other side,  if $\na p(t,x)$ is given by \eqref{eq:pressure}, one colud also verify that equation \eqref{eq:P 1} for $p$ is satisfied.
The lemma is proved.
\end{proof}

\begin{remark}
In \eqref{eq:pressure}, $\na p_{\geq 2}$  is merely a notation and does not represent the actual gradient of a function $p_{\geq 2}$, even though $p_{\geq 2}$ is formally given by
\beno 
\sum_{\vv k\in K_{\geq 2}}\sum_{i,j=1}^3\int_{Q_L}\p_{x_j}\p_{x_i}G_{\vv k}(x,y)(z_+^iz_-^j)(t,y)dy,
\eeno
which may not be a convergent series.
\end{remark}

As a consequence of the proof, we are able to derive both an explicit expression and the corresponding estimates for $\p\na p$.

\begin{corollary}\label{cor:pressure}
Under the assumptions of Lemma \ref{lem:pressure}, we have
\beq\label{eq:pressure 2}\begin{aligned}
\p\na p(t,x)=&\sum_{k_1,k_2,k_3=-1}^1\sum_{i,j=1}^3 \int_{Q_L}\nabla_x G_{\vv k}(x,y) \p\bigl(\partial_{i} z_{+}^{j} \partial_{j} z_{-}^{i}\bigr)(t, y) dy \\ 
&\quad+\p\na\wt{p}_{\geq 2}(t,x)+\wt{\vv B_1}(t,x)+\wt{\vv B_2}(t,x),
\end{aligned}\eeq
where $\p\na\wt{p}_{\geq 2}$ is the integral term with the kernel far  from zero and satisfying  
\beq\label{eq:est for wt p geq 2}
|\p\na\wt{p}_{\geq 2}(t,x)|\lesssim\f{1}{L^4}\int_{Q_L}|z_\mp(t,y)|\cdot|\na z_\pm(t,y)|dy,
\eeq
$\wt{\vv B_1}(t,x)$ and $\wt{\vv B_2}(t,x)$ are boundary integral terms such that
\beq\label{eq:est for wt B 1 2}\begin{aligned}
&\bigl|\wt{\vv B_1}(t,x)\bigr|\lesssim \f{1}{L^3}\cdot\sum_{i=1}^3\int_{\{y_i=L\}}|z_\mp(t,y)|\cdot|\na z_\pm(t,y)|d\widehat{y_i},\\ 
&\bigl|\wt{\vv B_2}(t,x)\bigr|\lesssim \f{1}{L^2}\cdot\sum_{i=1}^3\int_{\{y_i=L\}}|\na z_-(t,y)|\cdot|\na z_+(t,y)|d\widehat{y_i}.
\end{aligned}\eeq
\end{corollary}
\begin{proof} We adopt the notations from the proof of Lemma \ref{lem:pressure} and provide only a sketch of the proof for this corollary.

Due to \eqref{eq:P 3a}, we first have
\beno 
\p\na p(t,x)=\sum_{k_1,k_2,k_3=-1}^1\p\na p_{\vv k}+\sum_{\vv k\in K_{\geq 2}}\p\na p_{\vv k}.
\eeno 

 Following a similar argument to Steps 2.2-2.4 in the proof of Lemma \ref{lem:pressure}, for the term $\sum_{\vv k\in K_{\geq 2}}\p\na p_{\vv k}$, we have
\beno
\sum_{\vv k\in K_{\geq 2}}\p\na p_{\vv k}(t,x)=\p\na\wt{p}_{\geq 2}(t,x)
+\wt{\vv B_1}(t,x),
\eeno 
where $\na\wt{p}_{\geq 2}$ is a notational convention, analogous to $\na p_{\geq 2}$, and satisfies
\beno
\p\na\wt{p}_{\geq 2}(t,x)=\sum_{\vv k\in K_{\geq 2}}\sum_{i=1}^3\int_{Q_L}\p_x\p_{x_i}\na_x G_{\vv k}(x,y)\bigl(z_\mp\cdot\na z_\pm^{i}\bigr)(t,y)dy.
\eeno
Meanwhile, $\wt{\vv B_1}(t,x)$ is a boundary integral term with an expression similar to $\vv B_1(t,x)$ and satisfies
\beno
\bigl|\wt{\vv B_1}(t,x)\bigr|\lesssim\f{1}{L^3}\cdot\sum_{i=1}^3\int_{\{y_i=L\}}|z_\mp(t,y)|\cdot|\na z_\pm(t,y)|d\widehat{y_i}.
\eeno
This establishes the first inequality in \eqref{eq:est for wt B 1 2}.

When $k_1,\, k_2,\,k_3\in\{-1,0,1\}$, integration by parts yields 
\beno\begin{aligned}
\p_1\na p_{\vv k}(t,x)&=-\sum_{i,j=1}^3\int_{Q_L}\p_{y_1}\na_xG_{\vv k}(x,y)(\p_jz_+^i\p_iz_-^j)(t,y)dy\\ 
&=\sum_{i,j=1}^3\int_{Q_L}\na_xG_{\vv k}(x,y)\p_{y_1}(\p_jz_+^i\p_iz_-^j)(t,y)dy+\wt{\vv B}^1_{2,\vv k}(t,x),
\end{aligned}\eeno
where $\wt{\vv B}^1_{2,\vv k}(t,x)$ is a boundary integral term given by
\beno\begin{aligned}
\wt{\vv B}^1_{2,\vv k}(t,x)&=-\sum_{i,j=1}^3\int_{\{y_1=L\}}\bigl(\underbrace{\na_xG_{\vv k}(x,y)|_{y_1=L}-\na_xG_{\vv k}(x,y)|_{y_1=-L}}_{q_{1,\vv k}(x,\widehat{y_1})}\bigr)\cdot(\p_jz_+^i\p_iz_-^j)(t,y)dy_2dy_3.
\end{aligned}\eeno 
A similar argument to the one used in Step 2.3 of the proof of Lemma \ref{lem:pressure} yields
\beno 
\sum_{k_1,k_2,k_3=-1}^1q_{1,\vv k}(x,\widehat{y_1})=-\sum_{k_2,k_3=-1}^1
\bigl(\na_x G_{(-2,k_2,k_3)}(x,y)-\na_x G_{(1,k_2,k_3)}(x,y\bigr)\big|_{y_1=L}.
\eeno

For $k_1=-2$ and $1$, we have for any $x_1\in[-L,L]$ and $y_1=L$,
\beno 
|x_1-y_1-2k_1L|=|x_1-(2k_1+1)L|\geq 2L,
\eeno
which shows that
\beno 
\bigl|\sum_{k_1,k_2,k_3=-1}^1q_{1,\vv k}(x,\widehat{y_1})\bigr|\lesssim\f{1}{L^2}. 
\eeno 

Thus, for $\p=\p_1$, denoting $\wt{\vv B_2}=\sum_{k_1,k_2,k_3=-1}^1\wt{\vv B}^1_{2,\vv k}$, we obtain the second inequality in \eqref{eq:est for wt B 1 2}. The same estimates hold for $\p=\p_2$ and $ \p_3$. This completes the proof of the corollary.
\end{proof}

\medskip

We are now in a position to present the estimates for the pressure term that arise in the lowest-order energy estimates.

\begin{proposition}\label{prop:p1}
Let $R\geq 100$. Then under the assumptions \eqref{eq:a1}--\eqref{eq:a3},   for all $t\in [0, T^*]$, we have
\begin{equation}\label{eq:p0}
\begin{aligned}
&\int_{0}^{t} \int_{\Sigma_{\tau}}\left(\log \langle w_{\mp}\rangle\right)^4\left|z_{\pm}\right|\cdot \left|\nabla p\right| dx d\tau 
 \lesssim \Big(E_{\mp}+ \sum_{l=0}^2 E_{\mp}^l \Big)^{\frac{1}{2}} \left(E_{\pm}+ E_{\pm}^{0}+ F_{\pm}(t)+ F_{\pm}^{0}(t) \right).
\end{aligned} 
\end{equation}
\end{proposition}
\begin{proof} We prove \eqref{eq:p0} only for the term $\int_{0}^{t} \int_{\Sigma_{\tau}}\left(\log \langle w_-\rangle\right)^4\left|z_+\right|\cdot \left|\nabla p\right| dx d\tau $, and the remaining part of \eqref{eq:p0} follows similarly.   We first use \eqref{eq:pressure} to decompose $\nabla p$ as
\beq\label{eq:p2}
\na p=\na p_0+\na p_1+\na p_{\geq 2}+\vv B_1+\vv B_2,
\eeq
where $\na p_{\geq 2}$, $\vv B_1$, and $\vv B_2$ are as defined in Lemma \ref{lem:pressure}, and 
\beno\begin{aligned} 
&\na p_0(t,x):=\frac{1}{4\pi}\sum_{i,j=1}^3 \int_{Q_L} \nabla_x \Big(\frac{1}{|x-y|}\Big)\left(\partial_{i} z_{+}^{j} \partial_{j} z_{-}^{i}\right)(\tau, y) dy,\\ 
&\na p_1(t,x):=\sum_{\substack{k_1,k_2,k_3 =-1 \\ \max\{ |k_1|,|k_2|,|k_3|\}=1}}^1\sum_{i,j=1}^3 \int_{Q_L} \nabla_x G_{\vv k}(x,y) \left(\partial_{i} z_{+}^{j} \partial_{j} z_{-}^{i}\right)(\tau, y) dy.
\end{aligned}\eeno

We shall divide the estimates for $\int_{0}^{t} \int_{\Sigma_{\tau}}\left(\log \langle w_-\rangle\right)^4\left|z_+\right|\cdot \left|\nabla p\right| dx d\tau $ into five parts corresponding to the decomposition of $\na p$.

\medskip 

{\bf Step 1. Estimate for  $\int_{0}^{t} \int_{\Sigma_{\tau}}\left(\log \langle w_-\rangle\right)^4\left|z_+\right|\cdot \left|\nabla p_0\right| dx d\tau $.} By virtue of H\"older's inequality, we have
\beq\label{eq:p1}\begin{aligned}
&\int_{0}^{t} \int_{\Sigma_{\tau}}\left(\log \langle w_-\rangle\right)^4\left|z_+\right|\cdot \left|\nabla p_0\right| dx d\tau\\ 
\leq
&\Bigl\|\f{\left(\log \langle w_-\rangle\right)^2}{\langle w_+\rangle^{\f12}\log \langle w_+\rangle}z_+\Bigr\|_{L^2_t(L^2_x)}\cdot\|\langle w_+\rangle^{\f12}\log \langle w_+\rangle\cdot \left(\log \langle w_-\rangle\right)^2\na p_0\|_{L^2_t(L^2_x)}.
\end{aligned}\eeq
We now estimate the terms on the RHS of \eqref{eq:p1}.

\smallskip 

{\it Step 1.1. Decomposition of $\na p_0$.} Introducing the cut-off function $\theta(r)\in C_c^\infty(\R)$ such that 
\begin{equation*}
0\leq \theta(r)\leq 1\quad\text{and}\quad  \theta(r)=
\begin{cases}
& 1, \quad \text{for }|r|\leq 1,\\
& 0, \quad \text{for }|r|\geq 2,
\end{cases}
\end{equation*}
we decompose $\na p_0$ into two parts as follows:
\beq\label{eq:p3}\begin{aligned}
\na p_0(t,x)&=\frac{1}{4\pi}\sum_{i,j=1}^3\int_{Q_L} \nabla_x\Big(\frac{1}{|x-y|}\Big)\cdot \theta(|x-y|)\cdot \left(\partial_{i} z_{+}^{j} \partial_{j} z_{-}^{i}\right)(t, y) dy\\ 
&\qquad+\frac{1}{4\pi}\sum_{i,j=1}^3\int_{Q_L} \nabla_x\Big(\frac{1}{|x-y|}\Big)\cdot\bigl(1- \theta(|x-y|)\bigr)\cdot \partial_{i}\left( z_{+}^{j} \partial_{j} z_{-}^{i}\right)(t, y) dy\\ 
&=:\frac{1}{4\pi}\sum_{i,j=1}^3A_{1}^{ij}(t,x)+\frac{1}{4\pi}\sum_{i,j=1}^3A_{2}^{ij}(t,x).
\end{aligned}\eeq

{\it Step 1.2. The bound of $\|\langle w_+\rangle^{\f12}\log \langle w_+\rangle\cdot \left(\log \langle w_-\rangle\right)^2A_{1}^{ij}\|_{L^2_t(L^2_x)}$.} Notice that the integral region for $A_{1}^{ij}(\tau,x)$ is $\{y\in Q_L\,|\,|x-y|\leq 2\}$.  By virtue of Lemma \ref{lem:p1}, for any $x,y\in Q_L,\, |x-y|\leq 2$, there hold
\beno
\langle w_{\pm}\rangle(\tau, x)\leq 12\langle w_{\pm}\rangle(\tau, y),\quad \log \langle w_{\pm}\rangle(\tau, x)\leq 2\log \langle w_{\pm}\rangle(\tau, y).
\eeno 
Consequently, we get
\beno\begin{aligned}
&\langle w_+\rangle^{\f12}\log \langle w_+\rangle\cdot \left(\log \langle w_-\rangle\right)^2|A_{1}^{ij}(\tau,x)|\\ 
\lesssim
&\int_{|x-y|\leq 2}\f{1}{|x-y|^2}\cdot\Bigl(\langle w_+\rangle^{\f12}\log \langle w_+\rangle\cdot \left(\log \langle w_-\rangle\right)^2\cdot|\na z_-|\cdot|\na z_+|\Bigr)(\tau,y)\cdot 1_{Q_L}(y)dy,
\end{aligned}\eeno
which along with Young's inequality implies 
\beno\begin{aligned}
&\|\langle w_+\rangle^{\f12}\log \langle w_+\rangle\cdot \left(\log \langle w_-\rangle\right)^2A_{1}^{ij}\|_{L^2_t(L^2_x)}\\ 
\lesssim&\||x|^{-2}\|_{L^1(|x|\leq 2)} \cdot\|\langle w_+\rangle^{\f12}\log \langle w_+\rangle\cdot \left(\log \langle w_-\rangle\right)^2\cdot|\na z_-|\cdot|\na z_+|\|_{L^2_t(L^2(Q_L))}\\ 
\lesssim&\|\langle w_+\rangle\bigl(\log \langle w_+\rangle\bigr)^2\na z_-\|_{L^\infty_{(t,x)}}
\cdot\Bigl\|\f{\left(\log \langle w_-\rangle\right)^2}{\langle w_+\rangle^{\f12}\log \langle w_+\rangle}\cdot\na z_+\Bigr\|_{L^2_t(L^2_x)}.
\end{aligned}\eeno

Thanks to the weighted Sobolev inequality \eqref{eq:e12}, we have
\beq\label{eq:p6a}
\|\langle w_+\rangle\bigl(\log \langle w_+\rangle\bigr)^2\na z_-\|_{L^\infty_{(t,x)}}\lesssim\Bigl(\sum_{l=0}^2 E^l_-\Bigr)^{\f12}.
\eeq
Then we obtain
\beq\label{eq:p6}
\sum_{i,j=1}^3\|\langle w_+\rangle^{\f12}\log \langle w_+\rangle\cdot \left(\log \langle w_-\rangle\right)^2A_{1}^{ij}\|_{L^2_t(L^2_x)}
\lesssim\Bigl(\sum_{l=0}^2 E^l_-\Bigr)^{\f12}
\cdot\Bigl\|\f{\left(\log \langle w_-\rangle\right)^2}{\langle w_+\rangle^{\f12}\log \langle w_+\rangle}\cdot\na z_+\Bigr\|_{L^2_t(L^2_x)}. 
\eeq

{\it Step 1.3. The bound of $\|\langle w_+\rangle^{\f12}\log \langle w_+\rangle\cdot \left(\log \langle w_-\rangle\right)^2A_{2}^{ij}\|_{L^2_t(L^2_x)}$.} We first use integration by parts to get
\beno\begin{aligned}
A_{2}^{ij}(t,x)&=\int_{Q_L}\p_{x_i}\Bigl(\nabla_x\Big(\frac{1}{|x-y|}\Big)\cdot\bigl(1- \theta(|x-y|)\bigr)\Bigr)\cdot\bigl( z_{+}^{j} \partial_{j} z_{-}^{i}\bigr)(t, y) dy\\ 
&\qquad
+\int_{\{y_i=L\}}\nabla_x\Big(\frac{1}{|x-y|}\Big)\cdot\bigl(1- \theta(|x-y|)\bigr)\cdot\bigl( z_{+}^{j} \partial_{j} z_{-}^{i}\bigr)(t, y)d\widehat{y_i}\\ 
&\qquad-\int_{\{y_i=-L\}}\nabla_x\Big(\frac{1}{|x-y|}\Big)\cdot\bigl(1- \theta(|x-y|)\bigr)\cdot\bigl( z_{+}^{j} \partial_{j} z_{-}^{i}\bigr)(t, y)d\widehat{y_i}\\
&=:A_{21}^{ij}(t,x)+A_{22}^{ij}(t,x)+A_{23}^{ij}(t,x),
\end{aligned}\eeno
where $\{y_i=\pm L\}:=\{y\in Q_L\,|\, y_i=\pm L\}$, $\widehat{y_1}=(y_2,y_3)$, $\widehat{y_2}=(y_1,y_3)$ and $\widehat{y_3}=(y_1,y_2)$.

{\underline{(a). Estimate for $A_{21}^{ij}$.}} For $A_{21}^{ij}$, we have
\beno\begin{aligned}
|A_{21}^{ij}(t,x)|&\lesssim\int_{Q_L}\f{1- \theta(|x-y|)}{|x-y|^3}\cdot\bigl(|z_+|\cdot|\na z_-|\bigr)(t,y)dy+\int_{Q_L}\f{|\theta'(|x-y|)|}{|x-y|^2}\cdot\bigl(|z_+|\cdot|\na z_-|\bigr)(t,y)dy\\ 
&=:A_{211}(t,x)+A_{212}(t,x).
\end{aligned}\eeno

Since 
\beno 
A_{212}(t,x)\lesssim\int_{1\leq |x-y|\leq 2}\f{1}{|x-y|^2}\cdot\bigl(|z_+|\cdot|\na z_-|\bigr)(t,y)\cdot 1_{Q_L}(y)dy,
\eeno
similar derivation as \eqref{eq:p6} gives rise to
\beq\label{eq:p7}
\|\langle w_+\rangle^{\f12}\log \langle w_+\rangle\cdot \left(\log \langle w_-\rangle\right)^2A_{212}\|_{L^2_t(L^2_x)}
\lesssim\Bigl(\sum_{l=0}^2 E^l_-\Bigr)^{\f12}
\cdot\Bigl\|\f{\left(\log \langle w_-\rangle\right)^2}{\langle w_+\rangle^{\f12}\log \langle w_+\rangle}\cdot z_+\Bigr\|_{L^2_t(L^2_x)}. 
\eeq

While for $A_{211}$, there holds
\beno 
A_{211}(t,x)\leq \int_{|x-y|\geq 1}\f{1}{|x-y|^3}\cdot\bigl(|z_+|\cdot|\na z_-|\bigr)(t,y)\cdot 1_{Q_L}(y)dy.
\eeno 
Due to Lemma \ref{lem:p1}, for any $x,y\in Q_L,\, |x-y|\geq 1$, there hold
\beno
\langle w_{\pm}\rangle(\tau, x)\leq 13|x-y|\cdot\langle w_{\pm}\rangle(\tau, y),\quad
\log \langle w_{\pm}\rangle(\tau, x) \leq 2\log (13|x-y|)\cdot \log \langle w_{\pm}\rangle(\tau, y),
\eeno
which shows that
\beno\begin{aligned}
&\langle w_+\rangle^{\f12}\log \langle w_+\rangle\cdot \left(\log \langle w_-\rangle\right)^2|A_{211}(\tau ,x)|\\ 
\lesssim
&\int_{|x-y|\geq 1}\f{\bigl(\log (13|x-y|)\bigr)^3}{|x-y|^{\f52}}\cdot\Bigl(\langle w_+\rangle^{\f12}\log \langle w_+\rangle\cdot \left(\log \langle w_-\rangle\right)^2\cdot|z_+|\cdot|\na z_-|\Bigr)(\tau,y)\cdot 1_{Q_L}(y)dy.
\end{aligned}\eeno

By Young's inequality, we deduce that
\beno\begin{aligned}
&\|\langle w_+\rangle^{\f12}\log \langle w_+\rangle\cdot \left(\log \langle w_-\rangle\right)^2A_{211}\|_{L^2_t(L^2_x)}\\ 
\lesssim&\Bigl\|\f{\bigl(\log (13|x|)\bigr)^3}{|x|^{\f52}}\Bigr\|_{L^2(|x|\geq 1)} \cdot\|\langle w_+\rangle^{\f12}\log \langle w_+\rangle\cdot \left(\log \langle w_-\rangle\right)^2\cdot|z_+|\cdot|\na z_-|\|_{L^2_t(L^1(Q_L))}\\ 
\lesssim&\|\langle w_+\rangle\bigl(\log \langle w_+\rangle\bigr)^2\na z_-\|_{L^\infty_t(L^2_x)}
\cdot\Bigl\|\f{\left(\log \langle w_-\rangle\right)^2}{\langle w_+\rangle^{\f12}\log \langle w_+\rangle} z_+\Bigr\|_{L^2_t(L^2_x)},
\end{aligned}\eeno
which  yields
\beq\label{eq:p8}
\|\langle w_+\rangle^{\f12}\log \langle w_+\rangle\cdot \left(\log \langle w_-\rangle\right)^2A_{211}\|_{L^2_t(L^2_x)}
\lesssim\Bigl(E^0_-\Bigr)^{\f12}
\cdot\Bigl\|\f{\left(\log \langle w_-\rangle\right)^2}{\langle w_+\rangle^{\f12}\log \langle w_+\rangle}\cdot z_+\Bigr\|_{L^2_t(L^2_x)}. 
\eeq

Thanks to \eqref{eq:p7} and \eqref{eq:p8}, we obtain
\beq\label{eq:p9}
\|\langle w_+\rangle^{\f12}\log \langle w_+\rangle\cdot \left(\log \langle w_-\rangle\right)^2A_{21}^{ij}\|_{L^2_t(L^2_x)}
\lesssim\Bigl(\sum_{l=0}^2 E^l_-\Bigr)^{\f12}
\cdot\Bigl\|\f{\left(\log \langle w_-\rangle\right)^2}{\langle w_+\rangle^{\f12}\log \langle w_+\rangle}\cdot z_+\Bigr\|_{L^2_t(L^2_x)}. 
\eeq

 {\underline{(b). Estimates for $A_{22}^{ij}$ and $A_{23}^{ij}$.}} For $A_{22}^{ij}$, we have
\beno
|A_{22}^{ij}(t,x)|\lesssim\int_{\{y_i=L\}\cap\{|x-y|\geq 1\}}\frac{1}{|x-y|^2}\cdot\bigl( |z_{+}|\cdot|\na z_{-}|\bigr)(t, y)d\widehat{y_i}.
\eeno
Since $|x^\pm_i(t,y)-y_i|\leq\f{L}{20}$ (see \eqref{eq:e9d}), we have
$
x^\pm_i(t,y)\big|_{y_i=L}\sim L,
$ 
which along  with Lemma \ref{lem:e2} shows that
\beno 
\langle w_\pm\rangle(t,y)|_{y_i=L}\gtrsim L\gtrsim \langle w_\pm\rangle(t,x).
\eeno

Thus,  denoting by 
\beno
f(t,x):=\Bigl(\frac{(\log \langle w_{-}\rangle)^2}{\langle w_{+}\rangle^{\frac{1}{2}} \log \langle w_{+}\rangle}\cdot |z_{+}|\Bigr)(t,x),
\quad g(t,x):=\bigl(\langle w_{+}\rangle (\log \langle w_{+}\rangle)^2\cdot|\nabla z_{-}|\bigr)(t,x),
\eeno 
 for $i=1$,  we obtain
\beq\label{eq:p10}\begin{aligned}
&\langle w_+\rangle^{\f12}\log \langle w_+\rangle\cdot \left(\log \langle w_-\rangle\right)^2|A_{22}^{1j}(t,x)|\\ 
\lesssim&\int_{\{y_1=L\}\cap\{|x-y|\geq 1\}}\frac{1}{|x-y|^2}\cdot
f(t,L,\widehat{y_1})\cdot g(t,L,\widehat{y_1})d\widehat{y_1}=: H(t,x).
\end{aligned}\eeq

For any $x,y\in Q_L$ such that $|x-y|\geq 1$ and $y_1=L$, we have 
\beno\begin{aligned}
\text{either}\quad &x_1\in[-L,L-\f12]\,\Rightarrow |x_1-L|\geq\f12,\quad \widehat{x_1},\,\widehat{y_1}\in[-L,L]^2,\\ 
\text{or}\quad &x_1\in(L-\f12,L]\,\Rightarrow |x_1-L|\leq\f12,\quad |\widehat{x_1}-\widehat{y_1}|\geq\f12.
\end{aligned}\eeno
Then by Young's inequality in $\R^2$ and Sobolev embedding inequality,  for $x_1\in[-L,L-\f12]$, we have
\beno
\begin{aligned}
\|H(t, x)\|_{L_{(x_2, x_3)}^{2}}&\lesssim \Bigl\|\frac{1}{|x_1-L|^2+|x_2|^2+|x_3|^2}\Bigr\|_{L_{(x_2, x_3)}^{2}(\mathbb{R}^{2})}\cdot\|f|_{y_1=L}\|_{L_{(y_2, y_3)}^{2}}\cdot \|g|_{y_1=L}\|_{L_{(y_2, y_3)}^{2}} \\
&\lesssim 
\frac{1}{|x_1-L|}\cdot\|f(t,\cdot)\|_{H^1(Q_L)}\cdot\|g(t,\cdot)\|_{H^1(Q_L)},
\end{aligned}
\eeno
while for all $x_1\in(L-\f12,L]$, there holds
\beno
|H(t,x)|\leq\int_{\{y_1=L\}}\frac{1_{|(x_2-y_2,x_3-y_3)|\geq\f12}}{|x_2-y_2|^2+|x_3-y_3|^2}\cdot f(t,L,y_2,y_3)\cdot g(t,L,y_2,y_3)dy_2dy_3,
\eeno
it follows that
\beno
\begin{aligned}
\|H(t, x)\|_{L_{(x_2, x_3)}^{2}}
&\lesssim \Bigl\|\frac{1_{|(x_2,x_3)|\geq\f12}}{|x_2|^2+|x_3|^2}\Bigr\|_{L^{2}(\mathbb{R}^{2})}\cdot\|f|_{y_1=L}\|_{L_{(y_2, y_3)}^{2}}\cdot \|g|_{y_1=L}\|_{L_{(y_2, y_3)}^{2}} \\
&\lesssim\|f(t,\cdot)\|_{H^1(Q_L)}\cdot\|g(t,\cdot)\|_{H^1(Q_L)}.
\end{aligned}
\eeno

Consequently, we have
\beno\begin{aligned}
\|H(t,\cdot)\|_{L^2_x}^2&\lesssim\bigl(\int_{-L}^{L-\f12}\f{1}{|x_1-L|^2}dx_1+\int_{L-\f12}^L 1\, dx_1\bigr)\cdot\|f(t,\cdot)\|_{H^1(Q_L)}^2\cdot\|g(t,\cdot)\|_{H^1(Q_L)}^2\\ 
&\lesssim \|f(t,\cdot)\|_{H^1(Q_L)}^2\cdot\|g(t,\cdot)\|_{H^1(Q_L)}^2,
\end{aligned}\eeno
which implies that
\beq\label{eq:p11}
\|H\|_{L^2_t(L^2_x)}\lesssim\|f\|_{L^2_t(H^1(Q_L))}\cdot\|g\|_{L^\infty_t(H^1(Q_L))}.
\eeq

Using the expressions for $f$ and $g$, along with Lemmas \ref{lem:e1}-\ref{lem:e2} yields
\beno\begin{aligned}
&\|g\|_{L^\infty_t(H^1(Q_L))}\lesssim
\|\langle w_{+}\rangle (\log \langle w_{+}\rangle)^2\nabla z_{-}\|_{L^\infty_t(L^2_x)}+\|\langle w_{+}\rangle (\log \langle w_{+}\rangle)^2\nabla^2 z_{-}\|_{L^\infty_t(L^2_x)},\\
&
\|f\|_{L^2_t(H^1(Q_L))}\lesssim \Bigl\|\frac{(\log \langle w_{-}\rangle)^2}{\langle w_{+}\rangle^{\frac{1}{2}} \log \langle w_{+}\rangle}\cdot z_{+}\Bigr\|_{L^2_t(L^2_x)}+\Bigl\|\frac{(\log \langle w_{-}\rangle)^2}{\langle w_{+}\rangle^{\frac{1}{2}} \log \langle w_{+}\rangle}\cdot\na z_{+}\Bigr\|_{L^2_t(L^2_x)},
\end{aligned}\eeno 
from which, together with\eqref{eq:p11} and \eqref{eq:p10}, we deduce that
\beq\label{eq:p12}\begin{aligned}
&\|\langle w_+\rangle^{\f12}\log \langle w_+\rangle\cdot \left(\log \langle w_-\rangle\right)^2A_{22}^{1j}|\|_{L^2_t(L^2_x)}\\ 
\lesssim 
&\bigl(E_-^0+E_-^1\bigr)^{\f12}\cdot\Bigl(\Bigl\|\frac{(\log \langle w_{-}\rangle)^2}{\langle w_{+}\rangle^{\frac{1}{2}} \log \langle w_{+}\rangle}\cdot z_{+}\Bigr\|_{L^2_t(L^2_x)}+\Bigl\|\frac{(\log \langle w_{-}\rangle)^2}{\langle w_{+}\rangle^{\frac{1}{2}} \log \langle w_{+}\rangle}\cdot\na z_{+}\Bigr\|_{L^2_t(L^2_x)}\Bigr).
\end{aligned}\eeq
The same estimate holds for $A_{22}^{ij}$ and $A_{23}^{ij}$ with $i,j=1,2,3$.

{\underline{(c). The bound for $A_{2}^{ij}$.}} Combining \eqref{eq:p9} and \eqref{eq:p12}, we derive that
\beq\label{eq:p13}\begin{aligned}
&\|\langle w_+\rangle^{\f12}\log \langle w_+\rangle\cdot \left(\log \langle w_-\rangle\right)^2A_{2}^{ij}\|_{L^2_t(L^2_x)}\\ 
\lesssim 
&\bigl(\sum_{l=0}^2E_-^l\bigr)^{\f12}\cdot\Bigl(\Bigl\|\frac{(\log \langle w_{-}\rangle)^2}{\langle w_{+}\rangle^{\frac{1}{2}} \log \langle w_{+}\rangle}\cdot z_{+}\Bigr\|_{L^2_t(L^2_x)}+\Bigl\|\frac{(\log \langle w_{-}\rangle)^2}{\langle w_{+}\rangle^{\frac{1}{2}} \log \langle w_{+}\rangle}\cdot\na z_{+}\Bigr\|_{L^2_t(L^2_x)}\Bigr).
\end{aligned}\eeq

{\it Step 1.4. Estimates for $\Bigl\|\frac{\left(\log \langle w_{-}\rangle\right)^2}{\langle w_{+}\rangle^{\frac{1}{2}} \log \langle w_{+}\rangle} z_{+}\Bigr\|_{L_t^{2}(L_{x}^{2})}$ and $\Bigl\|\frac{\left(\log \langle w_{-}\rangle\right)^2}{\langle w_{+}\rangle^{\frac{1}{2}} \log \langle w_{+}\rangle} \nabla z_{+}\Bigr\|_{L_t^{2}(L_{x}^{2})}$.}
To bound these terms, we recall two characteristic foliations for the spacetime region $[0,t] \times Q_L$ as follows:
\begin{equation*}
\Big(\bigcup\limits_{-\frac{L}{4}< u_{+}<\frac{L}{4}} C_{t,u_{+}}^{+}\Big) \bigcup W_{t,+}^{\leq -\frac{L}{4}} \bigcup W_{t,+}^{\geq \frac{L}{4}} \,\text{ and }\, \Big(\bigcup\limits_{-\frac{L}{4}< u_{-}<\frac{L}{4}} C_{t,u_{-}}^{-}\Big) \bigcup W_{t,-}^{\leq -\frac{L}{4}} \bigcup W_{t,-}^{\geq \frac{L}{4}}.
\end{equation*}
Then denoting by 
\beno
J_1:=\Bigl\|\frac{\left(\log \langle w_{-}\rangle\right)^2}{\langle w_{+}\rangle^{\frac{1}{2}} \log \langle w_{+}\rangle} z_{+}\Bigr\|_{L_t^{2}(L_{x}^{2})}^2=\iint_{[0,t]\times Q_L}\frac{\left(\log \langle w_{-}\rangle\right)^4}{\langle w_{+}\rangle\bigl(\log \langle w_{+}\rangle\bigr)^2}|z_{+}|^2dxdt,
\eeno
we have
\beq\label{eq:p37}
\begin{aligned}
&J_1=\underbrace{\int_{W_{t,+}^{\leq -\frac{L}{4}} \bigcup W_{t,+}^{\geq \frac{L}{4}}} \frac{\left(\log \langle w_{-}\rangle\right)^4\left|z_{+}(\tau, x)\right|^2}{\langle w_{+}\rangle \left(\log \langle w_{+}\rangle\right)^2} dx d\tau}_{J_{11}} + \underbrace{\int_{W_{t,+}^{[-\frac{L}{4}, \frac{L}{4}]}} \frac{\left(\log \langle w_{-}\rangle\right)^4\left|z_{+}(\tau, x)\right|^2}{\langle w_{+}\rangle \left(\log \langle w_{+}\rangle\right)^2} dx d\tau}_{J_{12}} . 
\end{aligned}\end{equation}

{\underline{(a). Estimate for $J_{11}$.}} For any $(\tau,x)\in W_{t,+}^{\leq -\frac{L}{4}} \bigcup W_{t,+}^{\geq \frac{L}{4}}$, there holds  
\beno
\left|u_{+}(\tau,x)\right|\geq \frac{L}{4}\Rightarrow\langle w_{+}\rangle(\tau,x)\geq \big(R^{2}+\left|u_{+}\right|^{2}\big)^{\frac{1}{2}}(\tau,x)\gtrsim L, 
\eeno 
which implies
\beno
\begin{aligned}
J_{11} & \lesssim \frac{1}{L \left(\log L\right)^2}\int_{W_{t,+}^{\leq -\frac{L}{4}} \bigcup W_{t,+}^{\geq \frac{L}{4}}} \left(\log \langle w_{-}\rangle\right)^4\left|z_{+}(\tau, x)\right|^2 dx d\tau\\
& \lesssim \frac{1}{L \left(\log L\right)^2}\int_{0}^{t} \int_{\Sigma_{\tau}} \left(\log \langle w_{-}\rangle\right)^4\left|z_{+}(\tau, x)\right|^2 dx d\tau.
\end{aligned} 
\eeno
Thus,  using $T^*=\log L$, we get 
\beq\label{eq:p38}
J_{11}\lesssim  \frac{T^*}{L \left(\log L\right)^2}\cdot E_{+} \lesssim \frac{E_{+}}{L\log L}.
\eeq

{\underline{(b). Estimate for $J_{12}$.}}
By Lemma\ref{lem:e2}, for all $(\tau, x)\in W_{t,+}^{[-\frac{L}{4}, \frac{L}{4}]}$, we have $|x_3|\leq \frac{L}{3}$ and thus $u_{+}(\tau,x)= x_{3}^{+}(\tau,x)$. Moreover, for $-\frac{L}{4}\leq a_{+} \leq\frac{L}{4}$, we have $C_{t,a_{+}}^{+}= \overline{C}_{t,a_{+}}^{+}$ and 
\begin{equation*}
W_{t,+}^{[-\frac{L}{4}, \frac{L}{4}]}= \bigcup\limits_{-\frac{L}{4}\leq u_{+}\leq \frac{L}{4}} C_{t,u_{+}}^{+}= \bigcup\limits_{-\frac{L}{4}\leq a_{+}\leq \frac{L}{4}} \overline{C}_{t,a_{+}}^{+}= \overline{W}_{t,+}^{[-\frac{L}{4}, \frac{L}{4}]}.
\end{equation*}

We consider the following change of coordinates:
\begin{equation*}
\begin{aligned}
\Phi:\ W_{t,+}^{[-\frac{L}{4}, \frac{L}{4}]}= \bigcup\limits_{-\frac{L}{4}\leq u_{+}\leq \frac{L}{4}} C_{t,u_{+}}^{+} & \longrightarrow  [-L, L]^{2}\times \big[-\frac{L}{4}, \frac{L}{4}\big] \times \left[0, t\right], \\
\left(x_{1}, x_{2}, x_{3}, \tau\right) & \mapsto \left(x_{1}, x_{2}, u_+, \tau\right)=\left(x_{1}, x_{2}, x_{3}^{+}(\tau, x), \tau\right) .
\end{aligned}
\end{equation*}
By virtue of the ansatz \eqref{eq:a1}, the Jacobian matrix $d\Phi$ of $\Phi$ satisfies
\begin{equation*}
\operatorname{det}\left(d\Phi\right)=\partial_{3}x_{3}^{+}=1+O(\varepsilon) .
\end{equation*}
Therefore, to compute the integral $J_{12}$, by using $\left(x_{1}, x_{2}, x_{3}^{+}, \tau\right)$ as reference coordinates and the fact that $d\sigma_{+}=\big(\sqrt{2}+O(\varepsilon)\big) dx_{1} dx_{2} dt$ (see \eqref{eq:e17}), we have
\beno\begin{aligned}
\ J_{12} 
&\lesssim\int_{-\frac{L}{4}}^{\frac{L}{4}}\left(\int_{C_{t,u_{+}}^{+}} \frac{\left(\log \langle w_{-}\rangle\right)^{4}\left|z_{+}\right|^{2}}{\langle w_{+}\rangle\left(\log \langle w_{+}\rangle\right)^{2}} d\sigma_{+}\right) d u_{+}\\ 
&\lesssim \int_{-\frac{L}{4}}^{\frac{L}{4}}\int_{C_{t,u_{+}}^{+}} \frac{\left(\log \langle w_{-}\rangle\right)^{4}\left|z_{+}\right|^{2}}{\big(R^{2}+\left|u_{+}\right|^{2}\big)^{\frac{1}{2}}\Big(\log \big(R^{2}+\left|u_{+}\right|^{2}\big)^{\frac{1}{2}}\Big)^{2}}  d\sigma_{+}d u_{+}.
\end{aligned}
\eeno
Since $u_+$ is a constant along $C_{t,u_{+}}^{+}$, we get
\beq\label{eq:p39}
\ J_{12}\lesssim\sup_{|u_{+}|\leq \frac{L}{4}} \int_{C_{t,u_{+}}^{+}} \left(\log \langle w_{-}\rangle\right)^{4}\left|z_{+}\right|^{2} d\sigma_{+}\cdot \int_{\mathbb{R}} \frac{1}{\big(R^{2}+\left|u_{+}\right|^{2}\big)^{\frac{1}{2}}\Big(\log \big(R^{2}+\left|u_{+}\right|^{2}\big)^{\frac{1}{2}}\Big)^{2}} d u_{+}\lesssim F_{+}(t) .
\eeq

Combining \eqref{eq:p38} and \eqref{eq:p39}, we obtain
\begin{equation}\label{eq:p40}
\Bigl\|\frac{\left(\log \langle w_{-}\rangle\right)^2}{\langle w_{+}\rangle^{\frac{1}{2}} \log \langle w_{+}\rangle} z_{+}\Bigr\|_{L_t^{2}(L_{x}^{2})}^2=J_1\lesssim \frac{E_{+}}{L\log L}+ F_{+}(t) . 
\end{equation}
Similarly, we derive
\beq\label{eq:p41}
\Bigl\|\frac{\left(\log \langle w_{-}\rangle\right)^2}{\langle w_{+}\rangle^{\frac{1}{2}} \log \langle w_{+}\rangle}\na z_{+}\Bigr\|_{L_t^{2}(L_{x}^{2})}^2\leq\Bigl\|\frac{\langle w_{-}\rangle\left(\log \langle w_{-}\rangle\right)^2}{\langle w_{+}\rangle^{\frac{1}{2}} \log \langle w_{+}\rangle}\na z_{+}\Bigr\|_{L_t^{2}(L_{x}^{2})}^2\lesssim \frac{E_{+}^0}{L\log L}+ F_{+}^{(0)}(t) . 
\eeq

{\it Step 1.5. Estimate for  $\na p_0$.} Thanks to \eqref{eq:p6} and \eqref{eq:p13}, we derive from \eqref{eq:p3} that 
\beno\begin{aligned}
&\|\langle w_+\rangle^{\f12}\log \langle w_+\rangle\cdot \left(\log \langle w_-\rangle\right)^2\na p_0\|_{L^2_t(L^2_x)}\\ 
\lesssim 
&\bigl(\sum_{l=0}^2E_-^l\bigr)^{\f12}\cdot\Bigl(\Bigl\|\frac{(\log \langle w_{-}\rangle)^2}{\langle w_{+}\rangle^{\frac{1}{2}} \log \langle w_{+}\rangle}\cdot z_{+}\Bigr\|_{L^2_t(L^2_x)}+\Bigl\|\frac{(\log \langle w_{-}\rangle)^2}{\langle w_{+}\rangle^{\frac{1}{2}} \log \langle w_{+}\rangle}\cdot\na z_{+}\Bigr\|_{L^2_t(L^2_x)}\Bigr),
\end{aligned}\eeno
which along with  \eqref{eq:p40} and \eqref{eq:p41} implies
\beno 
\|\langle w_+\rangle^{\f12}\log \langle w_+\rangle\cdot \left(\log \langle w_-\rangle\right)^2\na p_0\|_{L^2_t(L^2_x)}
\lesssim \bigl(\sum_{l=0}^2E_-^l\bigr)^{\f12}\cdot\Bigl(\frac{E_{+}+E_{+}^0}{L\log L}+ F_{+}(t)+ F_{+}^0(t)\Bigr)^{\f12}.
\eeno
Using  \eqref{eq:p40} again, we deduce from \eqref{eq:p1} that 
\beq\label{eq:p1a}
\int_{0}^{t} \int_{\Sigma_{\tau}}\left(\log \langle w_-\rangle\right)^4\left|z_+\right|\cdot \left|\nabla p_0\right| dx d\tau
\lesssim
 \bigl(\sum_{l=0}^2E_-^l\bigr)^{\f12}\cdot\Bigl(\frac{E_{+}+E_{+}^0}{L\log L}+F_{+}(t)+ F_{+}^0(t)\Bigr).
\eeq

{\bf Step 2. Estimate for  $\int_{0}^{t} \int_{\Sigma_{\tau}}\left(\log \langle w_-\rangle\right)^4\left|z_+\right|\cdot \left|\nabla p_1\right| dx d\tau $.} It is straightforward to obtain
\beq\label{eq:p15a}\begin{aligned}
\int_{0}^{t} \int_{\Sigma_{\tau}}\left(\log \langle w_-\rangle\right)^4\left|z_+\right|\cdot \left|\nabla p_1\right| dx d\tau 
&\leq T^*
\bigl\|\left(\log \langle w_-\rangle\right)^2z_+\bigr\|_{L^\infty_t(L^2_x)}\cdot\|\left(\log \langle w_-\rangle\right)^2\na p_1\|_{L^\infty_t(L^2_x)}\\ 
&\leq\log L\cdot \bigl(E_+\bigr)^{\f12}\cdot\|\left(\log \langle w_-\rangle\right)^2\na p_1\|_{L^\infty_t(L^2_x)}.
\end{aligned}\eeq
It remains to estimate $\|\left(\log \langle w_-\rangle\right)^2\na p_1\|_{L^\infty_t(L^2_x)}$. We will focus on the following term in the expression for $\na p_1$:
\beno 
\na p_{11}(t,x):=\int_{Q_L} \nabla_x G_{(0,0,1)}(x,y)\left(\partial_{i} z_{+}^{j} \partial_{j} z_{-}^{i}\right)(\tau, y) dy;
\eeno 
the estimates for the remaining terms in the expression of $\na p_1$ follow analogously. 

Since $G_{\vv k}(x,y)=\f{1}{4\pi}\cdot\f{1}{|x-y-2\vv kL|}$, we have
\beno 
|\nabla_x G_{(0,0,1)}(x,y)|\lesssim\f{1}{|x-y_{(0,0,1)}|^2}=\f{1}{|x_1-y_1|^2+|x_2-y_2|^2+|x_3-y_3-2L|^2}.
\eeno
Then we get
\beq\label{eq:p15}
|\na p_{11}(t,x)|\lesssim\int_{Q_L}\f{(|\na z_+|\cdot|\na z_-|)(t,y)}{|x-y_{(0,0,1)}|^2}\cdot\bigl( 1_{\geq\f{L}{3}}(|y_3|)+1_{<\f{L}{3}}(|y_3|)\bigr)\, dy=:I_1(t,x)+I_2(t,x).
\eeq 

\smallskip

{\it Step 2.1. Estimate for $I_1(t,x)=\int_{Q_L}\f{(|\na z_+|\cdot|\na z_-|)(t,y)}{|x-y_{(0,0,1)}|^2}\cdot1_{\geq\f{L}{3}}(|y_3|)\, dy$.} When $|y_3|\geq\f{L}{3}$, by Lemma \ref{lem:e2}, there holds $\f{L}{4}<|u_\pm(t,y)|\leq\f{9L}{8}$, which implies
\beno 
\langle w_{\pm}\rangle(t,y)\gtrsim L\gtrsim \langle w_{\pm}\rangle(t,x)\quad\text{and}\quad\log\langle w_{\pm}\rangle(t,y)\gtrsim\log L\gtrsim\log \langle w_{\pm}\rangle(t,x).
\eeno
Thereby we get
\beno\begin{aligned}
\left(\log \langle w_-\rangle\right)^2|I_1(t,x)|
&\lesssim \int_{Q_L}\f{(\left(\log \langle w_-\rangle\right)^2|\na z_+|\cdot|\na z_-|)(t,y)}{|x-y_{(0,0,1)}|^2}\cdot1_{\geq\f{L}{3}}(|y_3|)\, dy\\ 
&
\lesssim\f{1}{L^2}\int_{Q_L}\f{\bigl(f_1\cdot g_1\bigr)(t,y)}{|x-y_{(0,0,1)}|^2}\cdot1_{\geq\f{L}{3}}(|y_3|)\, dy,
\end{aligned}\eeno 
where $f_1,\,g_1$ are given by
\beno
f_1(t,x):=\bigl(\langle w_{-}\rangle(\log \langle w_{-}\rangle)^2\cdot |\na z_{+}|\bigr)(t,x),
\quad g_1(t,x):=\bigl(\langle w_{+}\rangle\cdot|\nabla z_{-}|\bigr)(t,x).
\eeno 

By virtue of Young's and H\"older's inequalities, we get
\beno\begin{aligned}
&\quad \bigl\|\int_{|x-y_{(0,0,1)}|\leq1}\f{\bigl(f_1\cdot g_1\bigr)(t,y)}{|x-y_{(0,0,1)}|^2}\cdot1_{\geq\f{L}{3}}(|y_3|)\cdot 1_{Q_L}(y)\,dy\bigr\|_{L^2_x}\\ 
&\lesssim
\bigl\|\f{1}{|x|^2}\bigr\|_{L^1(|x|\leq 1)}
\cdot\|f_1\|_{L^2_x}\cdot\|g_1\|_{L^\infty_x}
\lesssim \|f_1\|_{L^2_x}\cdot\|g_1\|_{L^\infty_x},
\end{aligned}\eeno
and 
\beno\begin{aligned}
&\quad \bigl\|\int_{|x-y_{(0,0,1)}|\geq 1}\f{\bigl(f_1\cdot g_1\bigr)(t,y)}{|x-y_{(0,0,1)}|^2}\cdot1_{\geq\f{L}{3}}(|y_3|)\cdot 1_{Q_L}(y)\,dy\bigr\|_{L^2_x}\\ 
&\lesssim
\|\f{1}{|x|^2}\|_{L^2(|x|\geq 1)}
\cdot\|f_1\|_{L^2_x}\cdot\|g_1\|_{L^2_x}
\lesssim \|f_1\|_{L^2_x}\cdot\|g_1\|_{L^2_x},
\end{aligned}\eeno
which give rise to
\beq\label{eq:p16}\begin{aligned}
\|\left(\log \langle w_-\rangle\right)^2I_1\|_{L^\infty_t(L^2_x)}&\lesssim\f{1}{L^2}\|f_1\|_{L^\infty_t(L^2_x)}\cdot\bigl(\|g_1\|_{L^\infty_{(t,x)}}+\|g_1\|_{L^\infty_t(L^2_x)}\bigr)\\ 
&\lesssim\f{1}{L^2}\bigl(\sum_{l=0}^2E_-^l\bigr)^{\f12}\cdot\bigl(E_+^0\bigr)^{\f12},
\end{aligned}\eeq
where we used the expressions of $f_1$ and $g_1$, and the weighted Sobolev inequality (Lemma \ref{lem:e3}) in the last inequality.

\smallskip

{\it Step 2.2. Estimate for $I_2(t,x)=\int_{Q_L}\f{(|\na z_+|\cdot|\na z_-|)(t,y)}{|x-y_{(0,0,1)}|^2}\cdot1_{<\f{L}{3}}(|y_3|)\, dy$.} When $|y_3|<\f{L}{3}$ and $x_3\in[-L,L]$, there holds $|x_3-y_3|\leq\f{4L}{3}$, which hints that $|x_3-y_3-2L|\geq\f{2L}{3}$. It follows that
\beno 
|x-y_{(0,0,1)}|^2\geq|x_3-y_3-2L|^2\geq\f{4}{9}L^2,
\eeno
which along with $\langle w_-\rangle\lesssim L$ gives rise to
\beno
\left(\log \langle w_-\rangle\right)^2|I_2(t,x)|
\lesssim\f{(\log L)^2}{L^2}
\int_{Q_L}(|\na z_+|\cdot|\na z_-|)(t,y)\,dy\lesssim \f{(\log L)^2}{L^2}\cdot \bigl(E_-^0\bigr)^{\f12}\cdot \bigl(E_+^0\bigr)^{\f12}.
\eeno 
Then there holds
\beq\label{eq:p16a}
\|\left(\log \langle w_-\rangle\right)^2I_2\|_{L^\infty_t(L^2_x)}
\lesssim|Q_L|^{\f12}\cdot \f{(\log L)^2}{L^2}\cdot \bigl(E_-^0\bigr)^{\f12}\cdot \bigl(E_+^0\bigr)^{\f12}
\lesssim\f{(\log L)^2}{L^{\f12}}\cdot \bigl(E_-^0\bigr)^{\f12}\cdot \bigl(E_+^0\bigr)^{\f12}.
\eeq

{\it Step 2.3. Estimate for $\int_{0}^{t} \int_{\Sigma_{\tau}}\left(\log \langle w_-\rangle\right)^4\left|z_+\right|\cdot \left|\nabla p_1\right| dx d\tau $.} Thanks to \eqref{eq:p15}, \eqref{eq:p16} and \eqref{eq:p16a}, we obtain the estimate for $\na p_{11}$, and similar bounds apply to the remaining terms in $\na p_1$. Hence, 
\beq\label{eq:p16b}
\|\left(\log \langle w_-\rangle\right)^2\na p_1\|_{L^\infty_t(L^2_x)}
\lesssim \f{(\log L)^2}{L^{\f12}}\cdot \bigl(\sum_{l=0}^2E_-^l\bigr)^{\f12}\cdot \bigl(E_+^0\bigr)^{\f12}.
\eeq
Then we deduce from \eqref{eq:p15a} that
\beq\label{eq:p17}
\int_{0}^{t} \int_{\Sigma_{\tau}}\left(\log \langle w_-\rangle\right)^4\left|z_+\right|\cdot \left|\nabla p_1\right| dx d\tau 
\lesssim\f{(\log L)^3}{L^{\f12}}\cdot \bigl(\sum_{l=0}^2E_-^l\bigr)^{\f12}\cdot \bigl(E_++E_+^0\bigr).
\eeq

{\bf Step 3. Estimate for $\int_{0}^{t} \int_{\Sigma_{\tau}}\left(\log \langle w_-\rangle\right)^4\left|z_+\right|\cdot \left|\nabla p_{\geq 2}\right| dx d\tau $.} Similarly as $\na p_1$, we first have
\beno
\int_{0}^{t} \int_{\Sigma_{\tau}}\left(\log \langle w_-\rangle\right)^4\left|z_+\right|\cdot \left|\nabla p_{\geq2}\right| dx d\tau 
\leq\log L\cdot E_+^{\f12}\cdot\|\left(\log \langle w_-\rangle\right)^2\na p_{\geq 2}\|_{L^\infty_t(L^2_x)}.
\eeno

By virtue of \eqref{eq:est for p geq 2} and $\log \langle w_-\rangle\lesssim\log L$, we get 
\beno
\left(\log \langle w_-\rangle\right)^2|\na p_{\geq 2}(t,x)|\lesssim\f{(\log L)^2}{L^4}\int_{Q_L}|z_-(t,y)|\cdot|z_+(t,y)|dy
\lesssim \frac{\left(\log L\right)^2}{L^4} E_{+}^{\frac{1}{2}} E_{-}^{\frac{1}{2}},
\eeno
which implies
\beq\label{eq:p18}
\|\left(\log \langle w_-\rangle\right)^2\na p_{\geq 2}\|_{L^\infty_t(L^2_x)}\lesssim L^{\f32}\cdot\frac{\left(\log L\right)^2}{L^4} E_{+}^{\frac{1}{2}} E_{-}^{\frac{1}{2}}=\frac{\left(\log L\right)^2}{L^{\f52}} E_{+}^{\frac{1}{2}} E_{-}^{\frac{1}{2}}.
\eeq
Then we derive that
\beq\label{eq:p19}
\int_{0}^{t} \int_{\Sigma_{\tau}}\left(\log \langle w_-\rangle\right)^4\left|z_+\right|\cdot \left|\nabla p_{\geq2}\right| dx d\tau 
\lesssim\frac{\left(\log L\right)^3}{L^{\f52}}\cdot E_{-}^{\frac{1}{2}}\cdot  E_{+}.
\eeq

{\bf Step 4. Estimate for $\int_{0}^{t} \int_{\Sigma_{\tau}}\left(\log \langle w_-\rangle\right)^4\left|z_+\right|\cdot\bigl(|\vv B_1|+|\vv B_2|\bigr) dx d\tau $.} Similarly as $\na p_1$, we first have
\beq\label{eq:p20}\begin{aligned}
&\quad \int_{0}^{t} \int_{\Sigma_{\tau}}\left(\log \langle w_-\rangle\right)^4\left|z_+\right|\cdot\bigl(|\vv B_1|+|\vv B_2|\bigr) dx d\tau\\ 
& 
\leq\log L\cdot E_+^{\f12}\cdot\bigl(\|\left(\log \langle w_-\rangle\right)^2\vv B_1\|_{L^\infty_t(L^2_x)}+\|\left(\log \langle w_-\rangle\right)^2\vv B_2\|_{L^\infty_t(L^2_x)}\bigr).
\end{aligned}\eeq

For $\|\left(\log \langle w_-\rangle\right)^2\vv B_1\|_{L^\infty_t(L^2_x)}$, the first inequality in \eqref{eq:est for B 1 2} yields
\beno 
|\vv B_1(t,x)|\lesssim\f{1}{L^2}\sum_{i=1}^3\int_{\{y_i=L\}}|z_+(t,y)|\cdot|\na z_-(t,y)|d\widehat{y_i}.
\eeno 
When $y_i=L$, Lemma \ref{lem:e2} shows that $\langle w_{\pm}\rangle(t, y)|_{y_i=L}\gtrsim L$, which gives rise to 
\beno 
\langle w_{\pm}\rangle(t, x)\lesssim L\lesssim \langle w_{\pm}\rangle(t, y)|_{y_i=L},\quad\log \langle w_{\pm}\rangle(t, x)\lesssim\log  L\lesssim\log \langle w_{\pm}\rangle(t, y)|_{y_i=L}.
\eeno 
 Hence, we obtain
\beno
\begin{aligned}
\left(\log \langle w_{-}\rangle\right)^2\left|\vv B_1(t, x)\right|
&\lesssim \frac{1}{L^3} \sum_{i=1}^3\int_{\{y_i=L\}}\left(\left(\log \langle w_{-}\rangle\right)^2 \left|z_{+}\right| \cdot\langle w_{+}\rangle\left|\nabla z_{-}\right|\right)(t, y)d\widehat{y_i}\\
& \lesssim \frac{1}{L^3}\sum_{i=1}^3 \|\big(\left(\log \langle w_{-}\rangle\right)^2 z_{+}\bigr)\big|_{y_i=L}\|_{L_{\widehat{y_i}}^{2}}\cdot\|\big(\langle w_{+}\rangle\nabla z_{-}\big)\big|_{y_i=L}\|_{L_{\widehat{y_i}}^{2}}\\
& \lesssim \frac{1}{L^3} \left(E_{+}+ E_{+}^0\right)^{\frac{1}{2}}\cdot \left(E_{-}^0+ E_{-}^1\right)^{\frac{1}{2}} ,
\end{aligned} 
\eeno
where we used Sobolev embedding theorem and Lemma \ref{lem:e1} in the last inequality. Consequently, we derive
\beq\label{eq:p21}
\|\left(\log \langle w_-\rangle\right)^2\vv B_1\|_{L^\infty_t(L^2_x)}
\lesssim\f{1}{L^{\f32}}\cdot \left(E_{-}^0+ E_{-}^1\right)^{\frac{1}{2}}\cdot \left(E_{+}+ E_{+}^0\right)^{\frac{1}{2}}.
\eeq
Similarly, using the second inequality of \eqref{eq:est for B 1 2}, we deduce that
\beq\label{eq:p22}
\|\left(\log \langle w_-\rangle\right)^2\vv B_2\|_{L^\infty_t(L^2_x)}
\lesssim\f{1}{L^{\f32}}\cdot \left(E_{-}+ E_{-}^0\right)^{\frac{1}{2}}\cdot \left(E_{+}+ E_{+}^0\right)^{\frac{1}{2}}.
\eeq

Thanks to \eqref{eq:p20}, \eqref{eq:p21} and \eqref{eq:p22}, we obtain
\beq\label{eq:p23}
\int_{0}^{t} \int_{\Sigma_{\tau}}\left(\log \langle w_-\rangle\right)^4\left|z_+\right|\cdot\bigl(|\vv B_1|+|\vv B_2|\bigr) dx d\tau
\lesssim\frac{\log L}{L^{\frac{3}{2}}}\cdot\left(E_{-}+ E_{-}^0+ E_{-}^1\right)^{\frac{1}{2}}\cdot\left(E_{+}+ E_{+}^0\right).
\eeq

{\bf Step 5. Estimate involving $\na p$.} Combining \eqref{eq:p1a}, \eqref{eq:p17}, \eqref{eq:p19}, and \eqref{eq:p23}, we conclude that
\beno
\int_{0}^{t} \int_{\Sigma_{\tau}}\left(\log \langle w_{-}\rangle\right)^4\left|z_{+}\right| \left|\nabla p\right| dx d\tau \lesssim \Big(E_{-}+ \sum_{l=0}^2 E_{-}^l \Big)^{\frac{1}{2}} \left(E_{+}+ E_{+}^{0}+ F_{+}(t)+ F_{+}^{0}(t) \right).
\eeno
This proves \eqref{eq:p0}. The proof of this proposition is completed.
\end{proof}

\subsubsection{Completion of the lowest order energy estimate}
In this subsection, we will complete the proof of Proposition \ref{prop:estimates on lowest order terms}.
\begin{proof}[\proofname\ of Proposition \ref{prop:estimates on lowest order terms}]
We only derive the lowest order energy estimate \eqref{eq:estimates on lowest order terms} for $z_+$, and the estimate for $z_-$ follows by a symmetric argument.

Thanks to Proposition \ref{prop:l1}, using \eqref{eq:l7} with $t=T^*$, $f_+=z_+$, $\rho_+=-\nabla p$ and $\lambda_+=\left(\log \langle w_-\rangle\right)^4$, we have
\begin{equation}\label{eq:c1}
\begin{aligned}
& \quad\sup_{0\leq t\leq T^*}\|\left(\log \langle w_-\rangle\right)^2z_+\|_{L^2(\Sigma_t)}^2+\mu\|\left(\log \langle w_-\rangle\right)^2\na z_+\|_{L^2_{T^*}(L^2_x)}^2\\
&\qquad + \sup\limits_{|u_+|\leq \frac{L}{4}} \int_{C^+_{u_+}} \left(\log \langle w_-\rangle\right)^4\left|z_+\right|^2 d\sigma_+\\ 
& \lesssim\|\left(\log \langle w_-\rangle\right)^2z_+\|_{L^2(\Sigma_0)}^2+ \mu\int_{0}^{T^*} \int_{\Sigma_t} \frac{\left(\log \langle w_-\rangle\right)^{2}}{\langle w_-\rangle^{2}}\left|z_+\right|^{2} dx dt \\ 
&\qquad+\mu^2\|\left(\log \langle w_-\rangle\right)^2\na^2 z_+\|_{L_{T^*}^2(L^2_x)}^2+ \int_{0}^{T^*} \int_{\Sigma_t} \left(\log \langle w_-\rangle\right)^4 \left|z_+\right|\cdot |\na p|\, dxdt,
\end{aligned} 
\end{equation}
where we have used the fact that $\frac{\left|\nabla \lambda_+\right|^2}{\lambda_+} \lesssim \frac{\left(\log \langle w_-\rangle\right)^{2}}{\langle w_-\rangle^{2}}$ for $\lambda_+=\left(\log \langle w_-\rangle\right)^4$. 

Using Proposition \ref{prop:p1} to control the last term in the RHS of \eqref{eq:c1}, then \eqref{eq:c1} is reduced to
\beq\label{eq:c2}\begin{aligned}
E_++ D_++ F_+
&\lesssim E_+(0)+ \Big(E_-+ \sum_{l=0}^2 E_-^l \Big)^{\frac{1}{2}} \left(E_++ E_+^{0}+ F_++ F_+^{0}\right)\\ 
&\qquad+ \mu D_+^{0}+\mu\int_{0}^{T^*} \int_{\Sigma_t} \frac{\left(\log \langle w_-\rangle\right)^{2}}{\langle w_-\rangle^{2}}\left|z_+\right|^{2} dx dt.
\end{aligned}\eeq

 We now estimate the term $\mu\int_{0}^{T^*} \int_{\Sigma_t} \frac{\left(\log \langle w_-\rangle\right)^{2}}{\langle w_-\rangle^{2}}\left|z_+\right|^{2} dxdt$. Let $\theta\in C_c^\infty(\R)$ be a cut-off function such that 
\beno 
0\leq\theta(r)\leq 1,\quad \theta(r)=\left\{\begin{aligned}
&1,\quad\text{if}\quad |r|\leq\f18,\\ 
&0,\quad\text{if}\quad |r|\geq\f14.
\end{aligned}\right.
\eeno
Defining $\theta_L(r):=\theta(\f{r}{L})$, we have
\beno
\theta_L(r)=1,\quad\text{when }\, |r|\leq\f{L}{8},\quad
\text{and}\quad \theta_L(r)=0,\quad\text{when }\, |r|\geq\f{L}{4}.
\eeno
Then we obtain
\beq\label{eq:c3}\begin{aligned}
\mu\int_{0}^{T^*} \int_{\Sigma_t} \frac{\left(\log \langle w_-\rangle\right)^{2}}{\langle w_-\rangle^{2}}&\left|z_+\right|^{2} dx dt
\leq\underbrace{2\mu\int_{0}^{T^*} \int_{\Sigma_t} \frac{\left(\log \langle w_-\rangle\right)^{2}}{\langle w_-\rangle^{2}}\left|\theta_L(|x^-|)z_+\right|^{2} dx dt}_{I_1}\\ 
&+\underbrace{2\mu\int_{0}^{T^*} \int_{\Sigma_t} \frac{\left(\log \langle w_-\rangle\right)^{2}}{\langle w_-\rangle^{2}}\left|\bigl(1-\theta_L(|x^-|)\bigr)z_+\right|^{2} dx dt}_{I_2}.
\end{aligned}\eeq

{\underline{\bf Estimate for $I_1$}.}
The integration  in $I_1$ is restricted to the region where $(t,x)\in[0,T^*]\times Q_L$ satisfies
\beno 
|x^-(t,x)|\leq\f{L}{4}.
\eeno
By Lemma \ref{lem:e2}, we have $u_-(t,x)=x_3^-(t,x)$ and
\beno 
\langle w_-\rangle(t,x)=\langle\overline{w_-}\rangle(t,x)
=(R^2+|x^-(t,x)|^2)^{\f12}.
\eeno
On the other hand, \eqref{eq:e9d} shows that if $|x^-(t,x)|\leq\f{L}{4}$, it follows that $x\in Q_L$.

Using the flow map $\psi_-(t,y)$ associated with $z_-$, we find that 
\beno 
x^-\bigl(t,\psi^-(t,y)\bigr)=y,\quad\forall y\in\R^3,
\eeno 
which together with $\det\,\Bigl(\f{\p\psi_-(t,y)}{\p y}\Bigr)=1$ implies
\beno 
|I_1|\lesssim\mu\int_0^{T^*}\int_{\R^3}\f{1}{|y|^2}
\bigl|\log\langle y\rangle \theta_L(|y|)z_+\bigl(t,\psi_-(t,y)\bigr) \bigr|^2dydt.
\eeno 
Here $z_+(t,\cdot)$ is treated as a $2L$-periodic vector field on $\R^3$.
By virtue of the  Hardy's inequality \footnote{On $\R^3$, the Hardy's inequality is given by 
\beno
\int_{\R^3}\f{|f(x)|^2}{|x|^2}dx\leq 4\int_{\R^3}|\na f(x)|^2dx.
\eeno
},
we obtain 
\beq\label{eq:c4}\begin{aligned}
|I_1|&\lesssim\mu\int_0^{T^*}\int_{\R^3}\Bigl|\na_y\Bigl(\log\langle y\rangle \theta_L(|y|)z_+\bigl(t,\psi_-(t,y)\bigr)\Bigr)\Bigr|^2dydt\\  
&\lesssim\mu\int_0^{T^*}\int_{\R^3}\f{1}{\langle y\rangle^2}\cdot\bigl|\theta_L(|y|)z_+\bigl(t,\psi_-(t,y)\bigr)\bigr|^2dydt\\ 
&\quad +\mu\int_0^{T^*}\int_{\R^3}\bigl(\log\langle y\rangle\bigr)^2\bigl|\theta_L'(|y|)\bigr|^2\cdot\bigl|z_+\bigl(t,\psi_-(t,y)\bigr)\bigr|^2dydt\\ 
&\quad +\mu\int_0^{T^*}\int_{\R^3}\bigl(\log\langle y\rangle\bigr)^2 \bigl|\theta_L(|y|)\bigr|^2\cdot\bigl|\na_y\bigl(z_+\bigl(t,\psi_-(t,y)\bigr)\bigr)\bigr|^2dydt\\ 
&=:I_{11}+I_{12}+I_{13}.
\end{aligned}\eeq

For $I_{11}$, using the Hardy's inequality again, we have
\beno\begin{aligned}
I_{11}&\lesssim\mu\int_0^{T^*}\int_{\R^3}\Bigl|\na_y\Bigl(\theta_L(|y|)z_+\bigl(t,\psi_-(t,y)\bigr)\Bigr)\Bigr|^2dydt\\ 
&\lesssim\mu\int_0^{T^*}\int_{\R^3}\bigl|\theta_L'(|y|)\bigr|^2\cdot\bigl|z_+\bigl(t,\psi_-(t,y)\bigr)\bigr)\Bigr|^2dydt\\
&\qquad
+\mu\int_0^{T^*}\int_{\R^3}\bigl|\theta_L(|y|)\bigr|^2\cdot\Bigl|\na_y\Bigl(z_+\bigl(t,\psi_-(t,y)\bigr)\Bigr)\Bigr|^2dydt.
\end{aligned}\eeno
Since 
\beno 
\na_y\Bigl(z_+\bigl(t,\psi_-(t,y)\bigr)\Bigr)=\Bigl(\f{\p\psi_-(t,y)}{\p y}\Bigr)^T\na_xz_+(t,x)\big|_{x=\psi_-(t,y)},
\eeno 
using the fact that $\Bigl|\f{\p\psi_-(t,y)}{\p y}-I\Bigr|\leq\f{1}{10}$ (see \eqref{eq:e4}), there holds
\beq\label{eq:c5}
\Bigl|\na_y\Bigl(z_+\bigl(t,\psi_-(t,y)\bigr)\Bigr)\Bigr|
\sim \bigl|\na z_+(t,x)|_{x=\psi_-(t,y)}\bigr|.
\eeq
Then switching back to the coordinates $x$ and noticing that $\bigl|\theta_L'(|y|)\bigr|\lesssim\f{1}{L}\cdot 1_{\f{L}{8}\leq|y|\leq\f{L}{4}}$, we use \eqref{eq:c5} to obtain 
\beno\begin{aligned}
I_{11}&\lesssim\f{\mu}{L^2}\int_0^{T^*}\int_{Q_L}|z_+(t,x)|^2dxdt
+\mu\int_0^{T^*}\int_{Q_L}\bigl|\na z_+(t,x)\bigr|^2dxdt\\ 
&\overset{T^*=\log L}{\lesssim}\f{\mu\log L}{L^2}\cdot\sup_{0\leq t\leq T^*} \|z_+\|_{L^2(\Sigma_t)}^2+\mu\int_0^{T^*}\|\na z_+\|_{L^2(\Sigma_t)}^2dt.
\end{aligned}\eeno
Thus, by the basic energy identity \eqref{eq:basic} and $L\gg1$, we obtain
\beq\label{eq:c6}
I_{11}\lesssim\int_{\Sigma_0}|z_+|^2dx\lesssim\f{1}{(\log R)^4}\cdot E_+(0).
\eeq

Similarly, we have
\beno\begin{aligned}
I_{12}&\lesssim\f{\mu}{L^2}\int_0^{T^*}\int_{\Sigma_t}(\log \langle w_-\rangle)^2|z_+(t,x)|^2dxdt\lesssim\f{\mu\log L}{L^2} E_+,\\ 
I_{13}&\lesssim\mu\int_0^{T^*}\int_{\Sigma_t}(\log\langle w_-\rangle)^2|\na z_+(t,x)|^2dxdt=\mu\|\log \langle w_-\rangle\na z_+\|_{L^2_{T^*}(L^2_x)}^2.
\end{aligned}\eeno
Therefore, we get
\beq\label{eq:c7}
|I_1|\lesssim  E_+(0)+\f{\mu\log L}{L^2}\cdot E_++\mu\|\log\langle w_-\rangle\na z_+\|_{L^2_{T^*}(L^2_x)}^2.
\eeq

We shall estimate the term $\mu\|\log\langle w_-\rangle\na z_+\|_{L^2_{T^*}(L^2_x)}^2$ later on.

{\underline{\bf Estimate for $I_2$}.}
The integration in $I_2$ is restricted to the region where $(t,x)\in[0,T^*]\times Q_L$ satisfies $|x^-(t,x)|\geq\f{L}{8}$. Since $\langle w_-\rangle\sim\langle\overline{w_-}\rangle$ (Lemma \ref{lem:e2}), we have 
\beno
\langle w_-\rangle(t,x)\gtrsim\langle\overline{w_-}\rangle(t,x)=(R^2+|x^-(t,x)|^2)^{\f12}\gtrsim L.
\eeno
Hence, using $T^*=\log L$, we get
\beq\label{eq:c8}
|I_2|\lesssim\f{\mu}{L^2(\log L)^2}\int_{0}^{T^*} \int_{\Sigma_t} \left(\log \langle w_-\rangle\right)^{4}\left|z_+\right|^{2}dx dt
\lesssim\f{\mu}{L^2\log L} E_+.
\eeq 

{\underline{\bf Estimate for $\mu\|\log\langle w_-\rangle\na z_+\|_{L^2_{T^*}(L^2_x)}^2$.}} Thanks to Proposition \ref{prop:l2}, using \eqref{eq:l36} with $f_+=z_+$, $\rho_+=-\nabla p$ and $\lambda_+=\left(\log \langle w_-\rangle\right)^2$, we have
\beno
\begin{aligned}
& \quad\sup_{0\leq t\leq T^*}\|\log \langle w_-\rangle z_+\|_{L^2(\Sigma_t)}^2+\mu\|\log \langle w_-\rangle\na z_+\|_{L^2_{T^*}(L^2_x)}^2\\
& \lesssim\|\log \langle w_-\rangle z_+\|_{L^2(\Sigma_0)}^2+ \mu\int_{0}^{T^*} \int_{\Sigma_t} \frac{\left|z_+\right|^{2}}{\langle w_-\rangle^{2}} dx dt \\ 
&\qquad+ \f{\mu\e}{L(\log L)^2}\|\log \langle w_-\rangle\na^2 z_+\|_{L_{T^*}^2(L^2_x)}^2+ \int_{0}^{T^*} \int_{\Sigma_{\tau}} \left(\log \langle w_-\rangle\right)^2 \left|z_+\right|\cdot |\na p|\,dxdt,
\end{aligned} 
\eeno
where we have used the fact that $\frac{\left|\nabla \lambda_+\right|^2}{\lambda_+} \lesssim \frac{1}{\langle w_-\rangle^{2}}$ for $\lambda_+=\left(\log \langle w_-\rangle\right)^2$.  Using Proposition \ref{prop:p1} to control the last term in the RHS of the above inequality,  we get
\beq\label{eq:c9}\begin{aligned}
\mu\|\log \langle w_-\rangle\na z_+\|_{L^2_{T^*}(L^2_x)}^2 
&\lesssim E_+(0)+ \Big(E_-+ \sum_{l=0}^2 E_-^l \Big)^{\frac{1}{2}} \left(E_++ E_+^{0}+ F_++ F_+^{0}\right)\\ 
&\qquad+\f{\e}{L(\log L)^2} D_+^{0}+\mu\int_{0}^{T^*} \int_{\Sigma_t} \frac{\left|z_+\right|^{2}}{\langle w_-\rangle^{2}} dx dt.
\end{aligned}\eeq

For the last term in \eqref{eq:c9},  we first have
\beq\label{eq:c10}\begin{aligned}
\mu\int_{0}^{T^*} \int_{\Sigma_t} \frac{\left|z_+\right|^{2}}{\langle w_-\rangle^{2}} dx dt
&\leq2\mu\int_{0}^{T^*} \int_{\Sigma_t} \frac{1}{\langle w_-\rangle^{2}}\left|\theta_L(|x^-|)z_+\right|^{2} dx dt\\ 
&+2\mu\int_{0}^{T^*} \int_{\Sigma_t} \frac{1}{\langle w_-\rangle^{2}}\left|\bigl(1-\theta_L(|x^-|)\bigr)z_+\right|^{2} dx dt.
\end{aligned}\eeq
The first term in the RHS of \eqref{eq:c10} has the same estimate as $I_{11}$. The second term is handled similarly to $I_2$, which satisfies
\beno 
2\mu\int_{0}^{T^*} \int_{\Sigma_t} \frac{1}{\langle w_-\rangle^{2}}\left|\bigl(1-\theta_L(|x^-|)\bigr)z_+\right|^{2} dx dt
\lesssim\f{\mu}{L^2(\log L)^3} E_+.
\eeno 
Hence, we obtain
\beno
\mu\int_{0}^{T^*} \int_{\Sigma_t} \frac{\left|z_+\right|^{2}}{\langle w_-\rangle^{2}} dx dt
\lesssim E_+(0)+\f{\mu\log L}{L^2}\cdot E_+,
\eeno
which along with  \eqref{eq:c9} shows that
\beq\label{eq:c11}\begin{aligned}
\mu\|\log \langle w_-\rangle\na z_+\|_{L^2_{T^*}(L^2_x)}^2 &\lesssim E_+(0)+ \Big(E_-+ \sum_{l=0}^2 E_-^l \Big)^{\frac{1}{2}} \left(E_++ E_+^{0}+ F_++ F_+^{0}\right)\\ 
&\qquad+ \mu D_+^{0}
+\f{\mu\log L}{L^2}\cdot E_+,
\end{aligned}\eeq
where we used that fact that $L\geq e^{\f{1}{\mu}}$ (or $\f{1}{\log L}\leq\mu$).

\smallskip

Combining \eqref{eq:c7}, \eqref{eq:c8} and \eqref{eq:c11}, we derive from \eqref{eq:c3} that
\beq\label{eq:c11a}
\mu\int_{0}^{T^*} \int_{\Sigma_t} \frac{\bigl(\log\langle w_-\rangle\bigr)^{2}}{\langle w_-\rangle^{2}}\left|z_+\right|^{2} dxdt
\lesssim\text{the RHS terms in \eqref{eq:c11}},
\eeq
which along with \eqref{eq:c2} gives rise to 
\beq\label{eq:c12}
E_++ D_++ F_+
\lesssim E_{\pm}(0)+ \Big(E_-+ \sum_{l=0}^2 E_-^l \Big)^{\frac{1}{2}} \left(E_++ E_+^{0}+ F_++ F_+^{0}\right)
+ \mu D_+^{0}+\f{\mu\log L}{L^2}\cdot E_+.
\eeq

Since $\Big(E_{\mp}+ \sum_{l=0}^2 E_{\mp}^l \Big)^{\frac{1}{2}}\leq 2C_1\e$ by the ansatz \eqref{eq:a3}, it follows that for sufficiently small $\e\ll1$ and sufficiently large $L\gg1$, inequality \eqref{eq:c12} implies
\beno
E_++ D_++ F_+
\lesssim E_{\pm}(0)+ \Big(E_-+ \sum_{l=0}^2 E_-^l \Big)^{\frac{1}{2}} \left(E_+^{0}+ F_+^{0}\right)
+ \mu D_+^{0}.
\eeno
This establishes \eqref{eq:estimates on lowest order terms} for $z_+$. The estimate for $z_-$ follows by a symmetric argument.
Thus, it completes the proof of the proposition.
\end{proof}

\subsection{Energy estimates for the first order terms}
In this section, our goal is to derive energy estimates for $\nabla z_{\pm}$. For this purpose, we first commute one derivative with \eqref{eq:MHD} and we obtain
\begin{equation}\label{eq:first order eqn}
\left\{\begin{array}{l}
\partial_{t} \partial z_{+}+ Z_{-}\cdot \nabla \partial z_{+}- \mu\Delta \partial z_{+}= -\partial \nabla p- \partial z_{-} \cdot \nabla z_{+}, \\
\partial_{t} \partial z_{-}+ Z_{+}\cdot \nabla \partial z_{-}- \mu\Delta \partial z_{-}= -\partial \nabla p- \partial z_{+} \cdot \nabla z_{-},
\end{array}\right. 
\end{equation}
where $\partial z_{\pm}$ denotes for some $\partial_i z_{\pm}$ with $i=1,2,3$. We will apply Proposition \ref{prop:l1} to system \eqref{eq:first order eqn} with the weight functions $\lambda_{\pm}=\langle w_{\mp}\rangle^2\left(\log \langle w_{\mp}\rangle\right)^4$.
The main result of this section is given by the following proposition.
\begin{proposition}\label{prop:estimates on first order terms}
Assume that R is sufficiently large (independent of $\e$, $\mu$ and $L$), and the assumptions \eqref{eq:a1}--\eqref{eq:a3} hold, we have
\begin{equation}\label{eq:estimates on first order terms}
E_{\pm}^{0}+ F_{\pm}^{0}+ D_{\pm}^{0}
\lesssim E_{\pm}(0)+ E_{\pm}^{0}(0)+ \Big(E_{\mp}+ \sum_{k=0}^3 E_{\mp}^k \Big)^{\frac{1}{2}} \left(E_{\pm}+ E_{\pm}^{1}+ F_{\pm}+ F_{\pm}^{1}\right)+ \mu D_{\pm}^{1}.
\end{equation}
\end{proposition}
Before going further, we first recall \eqref{eq:p41} as follows: 
\beq\label{eq:p1-3}
\Bigl\|\frac{\langle w_{\mp}\rangle\left(\log \langle w_{\mp}\rangle\right)^2}{\langle w_{\pm}\rangle^{\frac{1}{2}} \log \langle w_{\pm}\rangle}\na z_{\pm}\Bigr\|_{L_t^{2}(L_{x}^{2})}^2\lesssim \frac{E_{\pm}^0}{L\log L}+ F_{\pm}^0(t) . 
\eeq
Thanks to the div-curl iteration formula \eqref{eq:div-curl iterate}, we have for any $1\leq m\leq N_*$, 
\beno
\bigl\|\sqrt\lambda\na^m\na z_\pm\bigr\|_{L^2_x}^2\lesssim \sum_{k=1}^{m} \bigl\|\sqrt{\lambda} j_{\pm}^{(k)}\bigr\|_{L^{2}_x}^{2}+ \bigl\|\sqrt{\lambda} \nabla z_{\pm}\bigr\|_{L^{2}_x}^{2}
+\frac{\varepsilon}{L\log L}\bigl\|\sqrt{\lambda} \nabla^2 z_{\pm}^{(m)}\bigr\|_{L^{2}_x}^{2},
\eeno
with $\lambda=\f{\langle w_\mp\rangle^2\left(\log \langle w_\mp\rangle\right)^4}{\langle w_\pm\rangle\bigl(\log \langle w_\pm\rangle\bigr)^2}$.
Then a similar derivation as \eqref{eq:p40} and \eqref{eq:p41} yields
\beq\label{eq:p1-3aa}
\Bigl\|\f{\langle w_\mp\rangle\left(\log \langle w_\mp\rangle\right)^2}{\langle w_\pm\rangle^{\f12}\log \langle w_\pm\rangle}\na^m\na z_\pm\Bigr\|_{L^2_t(L^2_x)}^2\lesssim\sum_{k=0}^m\bigl(\f{ E_\pm^k}{L\log L}+ F_\pm^k(t)\bigr)+\frac{\varepsilon}{L}E_\pm^{m+1},
\eeq
\beq\label{eq:p1-3a}
\Bigl\|\f{\langle w_\mp\rangle\left(\log \langle w_\mp\rangle\right)^2}{\langle w_\pm\rangle^{\f12}\log \langle w_\pm\rangle}\na^m\na z_\pm\Bigr\|_{L^2_t(L^2_x)}^2\overset{L\geq e^{\f{1}{\mu}}}{\lesssim}\sum_{k=0}^m\bigl(\f{ E_\pm^k}{L\log L}+ F_\pm^k(t)\bigr)+\frac{\varepsilon}{L}D_\pm^m(t).
\eeq

\subsubsection{Estimates on the pressure}
In this subsection, we derive estimates for the pressure term $\p\na p$, which arises in the first-order energy estimate for the MHD system.
\begin{proposition}\label{prop:p1-1}
Let $R\geq 100$ and $L\geq e^{\f{1}{\mu}}$. Then under the assumptions \eqref{eq:a1}--\eqref{eq:a3}, for all $t\in [0, T^*]$, we have
\begin{equation}\label{eq:p1-0}\begin{aligned}
&\qquad\int_{0}^{t} \int_{\Sigma_{\tau}} \langle w_{\mp}\rangle^2\left(\log \langle w_{\mp}\rangle\right)^4\left|\nabla z_{\pm}\right| \left|\nabla^2 p\right| dx d\tau\\ 
& \lesssim  \Big(E_{\mp}+ \sum_{k=0}^3 E_{\mp}^k \Big)^{\frac{1}{2}} \left(E_{\pm}^{0}+ E_{\pm}^{1}+ F_{\pm}^{0}(t)+ F_{\pm}^{1}(t)+ \frac{\varepsilon}{L}D_{\pm}^{1}(t) \right).
\end{aligned}
\end{equation}
\end{proposition}
\begin{proof}
The proof follows a similar strategy to Proposition \ref{prop:p1}. We sketch the estimate for the term $\int_{0}^{t} \int_{\Sigma_{\tau}}\langle w_-\rangle^2\left(\log \langle w_-\rangle\right)^4\left|\na z_+\right|\cdot \left|\nabla^2 p\right| dx d\tau $, and the remaining part of \eqref{eq:p1-0} follows similarly.  

 We first use \eqref{eq:pressure 2} to decompose $\p\nabla p$ as
\beq\label{eq:p1-1}
\p\na p=\p\na\wt p_0+\p\na\wt p_1+\p\na\wt p_{\geq 2}+\wt{\vv B_1}+\wt{\vv B_2},
\eeq
where $\p\na\wt p_{\geq 2}$, $\wt{\vv B_1}$ and $\wt{\vv B_2}$ are as defined in Corollary \ref{cor:pressure}, and 
\beno\begin{aligned} 
&\p\na\wt p_0(t,x):=\frac{1}{4\pi}\sum_{i,j=1}^3 \int_{Q_L} \nabla_x \Big(\frac{1}{|x-y|}\Big)\p\left(\partial_{i} z_{+}^{j} \partial_{j} z_{-}^{i}\right)(t, y) dy,\\ 
&\p\na\wt p_1(t,x):=\sum_{\substack{k_1,k_2,k_3 =-1 \\ \max\{ |k_1|,|k_2|,|k_3|\}=1}}^1\sum_{i,j=1}^3 \int_{Q_L} \nabla_x G_{\vv k}(x,y)\p\left(\partial_{i} z_{+}^{j} \partial_{j} z_{-}^{i}\right)(t, y) dy.
\end{aligned}\eeno

We shall divide the estimates for $\int_{0}^{t} \int_{\Sigma_{\tau}}\langle w_-\rangle^2\left(\log \langle w_-\rangle\right)^4\left|\na z_+\right|\cdot \left|\nabla^2 p\right| dx d\tau $ into five parts corresponding to the decomposition of $\p\na p$. 

{\bf Step 1. Estimate for  $\int_{0}^{t} \int_{\Sigma_{\tau}}\langle w_-\rangle^2\left(\log \langle w_-\rangle\right)^4\left|\na z_+\right|\cdot \left|\p\nabla\wt p_0\right| dx d\tau $.} We shall split the integral into several terms according to the properties of $\p\na\wt p_0$. We take the term $\p_1\na\wt p_0$ for an example.

{\it Step 1.1. Decomposition of $\p_1\na\wt p_0$.} Introducing the cut-off function $\theta(r)\in C_c^\infty(\R)$ such that $0\leq \theta(r)\leq 1$ and 
\beno
\theta(r)=1\quad\text{when } |r|\leq 1\quad\text{and }\quad \theta(r)=0\quad\text{when } |r|\geq 2,
\eeno
we decompose $\p_1\na\wt p_0$ into two parts as follows:
\beq\label{eq:p1-4}\begin{aligned}
\p_1\na\wt p_0(t,x)&=\frac{1}{4\pi}\sum_{i,j=1}^3\int_{Q_L} \nabla_x\Big(\frac{1}{|x-y|}\Big)\cdot \theta(|x-y|)\cdot\p_1\left(\partial_{i} z_{+}^{j} \partial_{j} z_{-}^{i}\right)(t, y) dy\\ 
&\qquad+\frac{1}{4\pi}\sum_{i,j=1}^3\int_{Q_L} \nabla_x\Big(\frac{1}{|x-y|}\Big)\cdot\bigl(1- \theta(|x-y|)\bigr)\cdot \partial_1\left(\p_i z_{+}^{j} \partial_{j} z_{-}^{i}\right)(t, y) dy\\ 
&=:\frac{1}{4\pi}\sum_{i,j=1}^3A_{1}^{ij}(t,x)+\frac{1}{4\pi}\sum_{i,j=1}^3A_{2}^{ij}(t,x).
\end{aligned}\eeq

{\it Step 1.2. Estimate for  $\int_{0}^{t} \int_{\Sigma_{\tau}}\langle w_-\rangle^2\left(\log \langle w_-\rangle\right)^4\left|\na z_+\right|\cdot|A_{1}^{ij}|\, dx d\tau $.} 
By H\"older's inequality, we have
\beq\label{eq:p1-2}\begin{aligned}
&\qquad\int_{0}^{t} \int_{\Sigma_{\tau}}\langle w_-\rangle^2\left(\log \langle w_-\rangle\right)^4|\na z_+|\cdot |A_{1}^{ij}| dx d\tau\\ 
&\leq
\Bigl\|\f{\langle w_-\rangle\left(\log \langle w_-\rangle\right)^2}{\langle w_+\rangle^{\f12}\log \langle w_+\rangle}\na z_+\Bigr\|_{L^2_t(L^2_x)}\cdot\|\langle w_+\rangle^{\f12}\log \langle w_+\rangle\cdot \langle w_-\rangle\left(\log \langle w_-\rangle\right)^2A_{1}^{ij}\|_{L^2_t(L^2_x)}.
\end{aligned}\eeq
Similar derivation as \eqref{eq:p6} shows that
\beno\begin{aligned}
&\qquad\|\langle w_+\rangle^{\f12}\log \langle w_+\rangle\cdot \langle w_-\rangle\left(\log \langle w_-\rangle\right)^2A_{1}^{ij}\|_{L^2_t(L^2_x)}\\ 
&\lesssim\sum_{l=0}^1\|\langle w_+\rangle(\log \langle w_+\rangle)^2\na^l\na z_-\|_{L^\infty_{(t,x)}}\cdot\sum_{m=0}^1\Bigl\|\f{\langle w_-\rangle\left(\log \langle w_-\rangle\right)^2}{\langle w_+\rangle^{\f12}\log \langle w_+\rangle}\na^m\na z_+\Bigr\|_{L^2_t(L^2_x)},
\end{aligned}\eeno
which along with the weighted Sobolev inequality \eqref{eq:e12}, \eqref{eq:p1-3} and \eqref{eq:p1-3a} implies
\beno
\|\langle w_+\rangle^{\f12}\log \langle w_+\rangle\cdot \langle w_-\rangle\left(\log \langle w_-\rangle\right)^2A_{1}^{ij}\|_{L^2_t(L^2_x)}
\lesssim\Bigl(\sum_{l=0}^3E_-^l\Bigr)^{\f12}\cdot\Bigl(\frac{E_+^0+E_+^1}{L\log L}+ F_+^0+ F_+^1+\f{\e}{L}D_+^1\Bigr)^{\f12}.
\eeno
Using \eqref{eq:p1-3} again, we deduce from \eqref{eq:p1-2} that 
\beq\label{eq:p1-5}
\int_{0}^{t} \int_{\Sigma_{\tau}}\langle w_-\rangle^2\left(\log \langle w_-\rangle\right)^4|\na z_+|\cdot |A_{1}^{ij}| dx d\tau
\lesssim\Bigl(\sum_{l=0}^3E_-^l\Bigr)^{\f12}\cdot\Bigl(\frac{E_+^0+E_+^1}{L\log L}+ F_+^0+ F_+^1+\f{\e}{L}D_+^1\Bigr).
\eeq

{\it Step 1.3. Estimate for $\int_{0}^{t} \int_{\Sigma_{\tau}}\langle w_-\rangle^2\left(\log \langle w_-\rangle\right)^4\left|\na z_+\right|\cdot|A_{2}^{ij}|\, dx d\tau $.} For $A_{2}^{ij}$, integration by parts yields
\beq\label{eq:p1-5star}\begin{aligned}
A_{2}^{ij}(t,x)&=\int_{Q_L} \partial_{x_1}\Bigl\{\nabla_x\Big(\frac{1}{|x-y|}\Big)\cdot\bigl(1- \theta(|x-y|)\bigr)\Bigr\}\cdot \left(\p_i z_{+}^{j} \partial_{j} z_{-}^{i}\right)(t, y)dy\\ 
&\quad+\int_{\{y_1=L\}}\nabla_x\Big(\frac{1}{|x-y|}\Big)\cdot\bigl(1- \theta(|x-y|)\bigr)\cdot \left(\p_i z_{+}^{j} \partial_{j} z_{-}^{i}\right)(t, y)dy_2dy_3\\ 
&\quad-\int_{\{y_1=-L\}}\nabla_x\Big(\frac{1}{|x-y|}\Big)\cdot\bigl(1- \theta(|x-y|)\bigr)\cdot \left(\p_i z_{+}^{j} \partial_{j} z_{-}^{i}\right)(t, y)dy_2dy_3\\ 
&:=A_{21}^{ij}(t,x)+A_{22}^{ij}(t,x)+A_{23}^{ij}(t,x).
\end{aligned}\eeq

{\underline{(a). Estimate for $A_{21}^{ij}$.}} For $A_{21}^{ij}$, we have
\beno\begin{aligned}
|A_{21}^{ij}(t,x)|&\lesssim\int_{Q_L}\f{1- \theta(|x-y|)}{|x-y|^3}\cdot\bigl(|\na z_+|\cdot|\na z_-|\bigr)(t,y)dy+\int_{Q_L}\f{|\theta'(|x-y|)|}{|x-y|^2}\cdot\bigl(|\na z_+|\cdot|\na z_-|\bigr)(t,y)dy\\ 
&=:A_{211}(t,x)+A_{212}(t,x).
\end{aligned}\eeno

{\it i). Estimate for $A_{212}$.} Following similar derivation as \eqref{eq:p6}, we get
\beq\label{eq:p1-5a}
\|\langle w_+\rangle^{\f12}\log \langle w_+\rangle\cdot\langle w_-\rangle \left(\log \langle w_-\rangle\right)^2A_{212}\|_{L^2_t(L^2_x)}
\lesssim\Bigl(\sum_{l=0}^2 E^l_-\Bigr)^{\f12}
\cdot\Bigl\|\f{\langle w_-\rangle\left(\log \langle w_-\rangle\right)^2}{\langle w_+\rangle^{\f12}\log \langle w_+\rangle}\cdot\na z_+\Bigr\|_{L^2_t(L^2_x)}.
\eeq
Since
\beno\begin{aligned}
&\qquad\int_{0}^{t} \int_{\Sigma_{\tau}}\langle w_-\rangle^2\left(\log \langle w_-\rangle\right)^4\left|\na z_+\right|\cdot \left|A_{212}\right| dx d\tau\\ 
&\leq
\Bigl\|\f{\langle w_-\rangle\left(\log \langle w_-\rangle\right)^2}{\langle w_+\rangle^{\f12}\log \langle w_+\rangle}\na z_+\Bigr\|_{L^2_t(L^2_x)}\cdot\|\langle w_+\rangle^{\f12}\log \langle w_+\rangle\cdot \langle w_-\rangle\left(\log \langle w_-\rangle\right)^2A_{212}\|_{L^2_t(L^2_x)},
\end{aligned}\eeno
using \eqref{eq:p1-5a} and \eqref{eq:p1-3}, we obtain 
\beq\label{eq:p1-5b}
\int_{0}^{t} \int_{\Sigma_{\tau}}\langle w_-\rangle^2\left(\log \langle w_-\rangle\right)^4\left|\na z_+\right|\cdot |A_{212}| dx d\tau\lesssim\Bigl(\sum_{l=0}^2E_-^l\Bigr)^{\f12}\cdot\Bigl(\frac{E_+^0}{L\log L}+ F_+^0\Bigr).
\eeq

{\it ii). Estimate for $A_{211}$.} 
By the expression of $A_{211}$, we first have
\beno 
|A_{211}(t,x)|\leq \int_{|x-y|\geq 1}\f{1}{|x-y|^3}\cdot\bigl(|\na z_+|\cdot|\na z_-|\bigr)(t,y)\cdot 1_{Q_L}(y)dy.
\eeno 
Due to \eqref{eq:p5-1} in Lemma \ref{lem:p1}, for any $x,y\in Q_L,\, |x-y|\geq 1$, there holds
\beno
\langle w_{\pm}\rangle(\log \langle w_{\pm}\rangle)^2(\tau, x)\lesssim\langle w_{\pm}\rangle\bigl(\log\langle w_{\pm}\rangle\bigr)^2(\tau,y)+|x-y|\bigl(\log (12|x-y|)\bigr)^2.
\eeno
Then we get 
\beno\begin{aligned}
&\qquad\langle w_-\rangle\left(\log \langle w_-\rangle\right)^2|A_{211}(\tau ,x)|\\ 
&
\lesssim\int_{|x-y|\geq 1}\f{1}{|x-y|^3}\cdot\Bigl(\langle w_-\rangle\left(\log \langle w_-\rangle\right)^2\cdot|\na z_+|\cdot|\na z_-|\Bigr)(\tau,y)\cdot 1_{Q_L}(y)dy\\ 
&\qquad+\int_{|x-y|\geq 1}\f{\bigl(\log (12|x-y|)\bigr)^2}{|x-y|^2}\cdot\Bigl(|\na z_+|\cdot|\na z_-|\Bigr)(\tau,y)\cdot 1_{Q_L}(y)dy\\ 
&=:K_1(\tau,x)+K_2(\tau,x).
\end{aligned}\eeno

For $K_1(\tau,x)$, by \eqref{eq:p5} in Lemma \ref{lem:p1}, we get
\beno\begin{aligned}
\langle w_+\rangle^{\f12}\log \langle w_+\rangle|K_1(\tau ,x)|
&\lesssim\int_{|x-y|\geq 1}\f{\log(13|x-y|)}{|x-y|^{\f52}}\cdot\Bigl(\f{\langle w_-\rangle\left(\log \langle w_-\rangle\right)^2}{\langle w_+\rangle^{\f12}\log \langle w_+\rangle}|\na z_+|\Bigr)(\tau,y)\cdot\\ 
&\qquad\cdot\Bigl(\langle w_+\rangle\bigl(\log \langle w_+\rangle\bigr)^2|\na z_-|\Bigr)(\tau,y)\cdot 1_{Q_L}(y)dy.
\end{aligned}\eeno
Using Young's inequality yields 
\beno\begin{aligned}
\|\langle w_+\rangle^{\f12}\log \langle w_+\rangle K_1\|_{L^2_t(L^2_x)}
&\lesssim\underbrace{\Bigl\|\f{\log(13|x|)}{|x|^{\f52}}\Bigr\|_{L^2(|x|\geq 1)}}_{\lesssim1}
\cdot\Bigl\|\f{\langle w_-\rangle\left(\log \langle w_-\rangle\right)^2}{\langle w_+\rangle^{\f12}\log \langle w_+\rangle}\na z_+\Bigr\|_{L^2_t(L^2_x)}\cdot\\ 
&\qquad\cdot\|\langle w_+\rangle\bigl(\log \langle w_+\rangle\bigr)^2\na z_-\|_{L^\infty_t(L^2_x)},
\end{aligned}\eeno
which along with \eqref{eq:p1-3} implies
\beq\label{eq:p1-5c}
\|\langle w_+\rangle^{\f12}\log \langle w_+\rangle K_1\|_{L^2_t(L^2_x)}
\lesssim \bigl(E^0_-\bigr)^{\f12}
\cdot\Bigl(\frac{E_+^0}{L\log L}+ F_+^0\Bigr)^{\f12}. 
\eeq

For $K_2(\tau,x)$, the separation property \eqref{eq:e13} of $z_+$ and $z_-$ gives rise to
\beno\begin{aligned}
&\qquad|K_2(\tau,x)|\\ 
&=\int_{|x-y|\geq 1}\f{\bigl(\log (12|x-y|)\bigr)^2}{|x-y|^2}\cdot\Bigl(\f{\langle w_-\rangle\bigl(\log \langle w_-\rangle\bigr)^2|\na z_+|\cdot\langle w_+\rangle\bigl(\log \langle w_+\rangle\bigr)^2|\na z_-|}{\langle w_-\rangle\bigl(\log \langle w_-\rangle\bigr)^2\cdot\langle w_+\rangle\bigl(\log \langle w_+\rangle\bigr)^2}\Bigr)(\tau,y)\cdot 1_{Q_L}(y)dy\\ 
&\lesssim\f{1}{\langle\tau\rangle\bigl(\log\langle\tau\rangle\bigr)^2}
\cdot\int_{|x-y|\geq 1}\f{\bigl(\log (12|x-y|)\bigr)^2}{|x-y|^2}
\cdot\langle w_-\rangle\bigl(\log \langle w_-\rangle\bigr)^2|\na z_+(\tau,y)|\cdot\\ 
&\qquad\qquad\qquad\cdot \langle w_+\rangle\bigl(\log \langle w_+\rangle\bigr)^2|\na z_-(\tau,y)|\cdot 1_{Q_L}(y)dy.
\end{aligned}\eeno
By Young's inequality, we have
\beno\begin{aligned}
\|K_2\|_{L^2(\Sigma_\tau)}
&\lesssim\f{1}{\langle\tau\rangle\bigl(\log\langle\tau\rangle\bigr)^2}
\cdot\underbrace{\Bigl\|\f{\bigl(\log (12|x|)\bigr)^2}{|x|^2}\Bigr\|_{L^2(|x|\geq1)}}_{\lesssim1}\\ 
&\qquad
\cdot\|\langle w_-\rangle\bigl(\log \langle w_-\rangle\bigr)^2\na z_+\|_{L^2(\Sigma_\tau)}\cdot \|\langle w_+\rangle\bigl(\log \langle w_+\rangle\bigr)^2\na z_-\|_{L^2(\Sigma_\tau)},
\end{aligned}\eeno
which gives rise to
\beq\label{eq:p1-5d}
\|K_2\|_{L^2(\Sigma_\tau)}\lesssim\f{\bigl(E_-^0\bigr)^{\f12}\bigl(E_+^0\bigr)^{\f12}}{\langle\tau\rangle\bigl(\log\langle\tau\rangle\bigr)^2}.
\eeq

By virtue of \eqref{eq:p1-5c}, \eqref{eq:p1-5d} and \eqref{eq:p1-3}, we get
\beno\begin{aligned}
&\qquad\int_{0}^{t} \int_{\Sigma_{\tau}}\langle w_-\rangle^2\left(\log \langle w_-\rangle\right)^4\left|\na z_+\right|\cdot \left|A_{211}\right| dx d\tau\\ 
&\lesssim
\Bigl\|\f{\langle w_-\rangle\left(\log \langle w_-\rangle\right)^2}{\langle w_+\rangle^{\f12}\log \langle w_+\rangle}\na z_+\Bigr\|_{L^2_t(L^2_x)}\cdot\|\langle w_+\rangle^{\f12}\log \langle w_+\rangle K_1\|_{L^2_t(L^2_x)}\\ 
&\qquad+\bigl\|\langle w_-\rangle\left(\log \langle w_-\rangle\right)^2\na z_+\bigr\|_{L^\infty_t(L^2_x)}\cdot\int_0^t\|K_2\|_{L^2_x}d\tau\\ 
&\lesssim  \bigl(E^0_-\bigr)^{\f12}
\cdot\Bigl(\frac{E_+^0}{L\log L}+ F_+^0\Bigr)
+\bigl(E_-^0\bigr)^{\f12} E_+^0\cdot\underbrace{\int_0^t\f{1}{\langle\tau\rangle\bigl(\log\langle\tau\rangle\bigr)^2}d\tau}_{\lesssim 1}.
\end{aligned}\eeno
Hence, we obtain
\beq\label{eq:p1-5e}
\int_{0}^{t} \int_{\Sigma_{\tau}}\langle w_-\rangle^2\left(\log \langle w_-\rangle\right)^4\left|\na z_+\right|\cdot \left|A_{211}\right| dx d\tau\lesssim \bigl(E^0_-\bigr)^{\f12}
\cdot\Bigl(E_+^0+ F_+^0\Bigr).
\eeq

Thanks to \eqref{eq:p1-5b} and \eqref{eq:p1-5e}, we derive
\beq\label{eq:p1-5f}
\int_{0}^{t} \int_{\Sigma_{\tau}}\langle w_-\rangle^2\left(\log \langle w_-\rangle\right)^4|\na z_+|\cdot |A_{21}^{ij}| dx d\tau\lesssim \Bigl(\sum_{l=0}^2E_-^l\Bigr)^{\f12}
\cdot\Bigl(E_+^0+ F_+^0\Bigr). 
\eeq

 {\underline{(b). Estimates for $A_{22}^{ij}$ and $A_{23}^{ij}$.}} A derivation similar to that of \eqref{eq:p12} for $A_{22}^{1j}$ in the proof of Proposition \ref{prop:p1} shows that
\beno\begin{aligned}
&\|\langle w_+\rangle^{\f12}\log \langle w_+\rangle\cdot \langle w_-\rangle\left(\log \langle w_-\rangle\right)^2A_{22}^{ij}|\|_{L^2_t(L^2_x)}\\ 
\lesssim 
&\bigl(E_-^0+E_-^1\bigr)^{\f12}\cdot\Bigl(\Bigl\|\frac{\langle w_-\rangle(\log \langle w_{-}\rangle)^2}{\langle w_{+}\rangle^{\frac{1}{2}} \log \langle w_{+}\rangle}\cdot\na z_{+}\Bigr\|_{L^2_t(L^2_x)}+\Bigl\|\frac{\langle w_-\rangle(\log \langle w_{-}\rangle)^2}{\langle w_{+}\rangle^{\frac{1}{2}} \log \langle w_{+}\rangle}\cdot\na^2 z_{+}\Bigr\|_{L^2_t(L^2_x)}\Bigr).
\end{aligned}\eeno
Here, we simply replace $z_+$ and its weight $\left(\log \langle w_-\rangle\right)^2$ in \eqref{eq:p12} with $\na z_+$ and  $\langle w_-\rangle\left(\log \langle w_-\rangle\right)^2$ respectively. Using \eqref{eq:p1-3} and \eqref{eq:p1-3a}, we obtain 
\beq\label{eq:p1-6}
\|\langle w_+\rangle^{\f12}\log \langle w_+\rangle\cdot \langle w_-\rangle\left(\log \langle w_-\rangle\right)^2A_{22}^{ij}|\|_{L^2_t(L^2_x)}\lesssim 
\bigl(E_-^0+E_-^1\bigr)^{\f12}\cdot\Bigl(\frac{E_+^0+E_+^1}{L\log L}+ F_+^0+ F_+^1+\f{\e}{L}D_+^1\Bigr)^{\f12}.
\eeq

Using H\"older's inequality, \eqref{eq:p1-3} and \eqref{eq:p1-6}, we deduce that
\beq\label{eq:p1-7}
\int_{0}^{t} \int_{\Sigma_{\tau}}\langle w_-\rangle^2\left(\log \langle w_-\rangle\right)^4\left|\na z_+\right|\cdot|A_{22}^{ij}|\, dx d\tau\lesssim\bigl(E_-^0+E_-^1\bigr)^{\f12}\cdot\Bigl(\frac{E_+^0+E_+^1}{L\log L}+ F_+^0+ F_+^1+\f{\e}{L}D_+^1\Bigr).
\eeq
The same estimate holds for $A_{23}^{ij}$.

{\underline{(c). Estimate for $A_{2}^{ij}$.}}
Thanks to \eqref{eq:p1-5f} and \eqref{eq:p1-7}, we deduce from \eqref{eq:p1-5star} that
\beq\label{eq:p1-8}\begin{aligned}
&\int_{0}^{t} \int_{\Sigma_{\tau}}\langle w_-\rangle^2\left(\log \langle w_-\rangle\right)^4\left|\na z_+\right|\cdot \left|A_2^{ij}\right| dx d\tau\lesssim
\Bigl(\sum_{l=0}^2E_-^l\Bigr)^{\f12}\cdot\bigl(E_+^0+E_+^1+ F_+^0+ F_+^1+\f{\e}{L}D_+^1\bigr).
\end{aligned}\eeq

{\it Step 1.4. Estimate for  $\int_{0}^{t} \int_{\Sigma_{\tau}}\langle w_-\rangle^2\left(\log \langle w_-\rangle\right)^4\left|\na z_+\right|\cdot \left|\p\nabla\wt p_0\right| dx d\tau $.} Combining \eqref{eq:p1-5} and \eqref{eq:p1-8}, we derive from \eqref{eq:p1-4} that
\beq\label{eq:p1-8a}\begin{aligned}
&\int_{0}^{t} \int_{\Sigma_{\tau}}\langle w_-\rangle^2\left(\log \langle w_-\rangle\right)^4\left|\na z_+\right|\cdot \left|\p\nabla\wt p_0\right| dx d\tau\lesssim
\Bigl(\sum_{l=0}^3E_-^l\Bigr)^{\f12}\cdot\bigl(E_+^0+E_+^1+ F_+^0+ F_+^1+\f{\e}{L}D_+^1\bigr).
\end{aligned}\eeq

{\bf Step 2. Estimate for $\int_{0}^{t} \int_{\Sigma_{\tau}}\langle w_-\rangle^2\left(\log \langle w_-\rangle\right)^4\left|\na z_+\right|\cdot \left|\p\nabla\wt p_1\right| dx d\tau $.}
The estimate is similar to the term involving $\na p_1$ in Step 2 in the proof of Proposition \ref{prop:p1}. We will focus on the following term in the expression for $\p_k\na\wt p_1$:
\beno 
\p_k\na\wt p_{11}(\tau,x):=\int_{Q_L} \nabla_x G_{(0,0,1)}(x,y)\p_k\bigl(\partial_{i} z_{+}^{j} \partial_{j} z_{-}^{i}\bigr)(\tau, y) dy,\quad k=1,2.3;
\eeno 
the estimates for the remaining terms in the expression of $\p\na\wt p_1$ follow analogously. 

\smallskip

{\it Step 2.1. Decomposition of $\p_k\na\wt p_{11}$.}
We introduce a cut-off function $\phi(r)\in C_c^\infty(\R)$ such that $0\leq\phi(r)\leq 1$ and
\beno 
\phi(r)=1\quad\text{when}\quad |r|\leq\f13,\quad\text{and}\quad
\phi(r)=0\quad\text{when}\quad |r|\geq\f34.
\eeno
Taking $\phi_L(r)=\phi\bigl(\f{r}{L}\bigr)$, we get
\beq\label{eq:p1-9a}\begin{aligned}
\p_k\na\wt p_{11}(\tau,x)&=\int_{Q_L} \nabla_x G_{(0,0,1)}(x,y)\phi_L(y_3)\cdot\p_k\bigl(\partial_{i} z_{+}^{j} \partial_{j} z_{-}^{i}\bigr)(\tau, y)\, dy\\ 
&\quad+
\int_{Q_L} \nabla_x G_{(0,0,1)}(x,y)\bigl(1-\phi_L(y_3)\bigr)\cdot\p_k\bigl(\partial_{i} z_{+}^{j} \partial_{j} z_{-}^{i}\bigr)(\tau, y)\,dy\\ 
& =:I_{1k}(\tau,x)+I_{2k}(\tau,x).
\end{aligned}\eeq
We remark that the integrations for $I_{1k}$ and $I_{2k}$ are carried out over the regions where $y \in Q_L$ satisfies $|y_3| \leq \frac{3L}{4}$ and $|y_3| \geq \frac{L}{3}$, respectively.

\smallskip

{\it Step 2.2. Estimate for $I_{1k}$.} Integration by parts yields 
\beq\label{eq:p1-9b}
I_{1k}(\tau,x)=\underbrace{\int_{Q_L}\p_{x_k}\nabla_x G_{(0,0,1)}(x,y)\phi_L(y_3)\cdot\bigl(\partial_{i} z_{+}^{j} \partial_{j} z_{-}^{i}\bigr)(\tau, y)\, dy}_{\bar{I}_{1k}(\tau,x)}+\wt{I}_{1k}(\tau,x),
\eeq
where
\beno\begin{aligned}
\wt{I}_{1k}(\tau,x)&=\int_{\{y_k=L\}}\nabla_x G_{(0,0,1)}(x,y)\phi_L(y_3)\cdot\bigl(\partial_{i} z_{+}^{j} \partial_{j} z_{-}^{i}\bigr)(\tau, y)\, d\widehat{y_k}\\ 
&\qquad-\int_{\{y_k=-L\}}\nabla_x G_{(0,0,1)}(x,y)\phi_L(y_3)\cdot\bigl(\partial_{i} z_{+}^{j} \partial_{j} z_{-}^{i}\bigr)(\tau, y)\, d\widehat{y_k},\quad\text{for }\, k=1,2,\\ 
\text{and}\quad
\wt{I}_{13}(\tau,x)&=-\int_{Q_L}\nabla_x G_{(0,0,1)}(x,y)\phi'_L(y_3)\cdot\bigl(\partial_{i} z_{+}^{j} \partial_{j} z_{-}^{i}\bigr)(\tau, y)\, dy.
\end{aligned}\eeno

Notice that for $x_3 \in [-L, L]$ and $|y_3| \leq \frac{3L}{4}$, we have $|x_3 - y_3 - 2L| \geq \frac{L}{4}$, and thus $|x-y_{(0,0,1)}|\geq \frac{L}{4}$. Then
by the properties of $G_{\vv k}(x,y)$, we have
\beno 
|\na_x G_{(0,0,1)}(x,y)|\lesssim\f{1}{|x-y_{(0,0,1)}|^2},\quad
|\na_x^2 G_{(0,0,1)}(x,y)|\lesssim\f{1}{|x-y_{(0,0,1)}|^3},
\eeno
which implies
\beno\begin{aligned} 
\sum_{k=1}^3|\bar{I}_{1k}(\tau,x)|+|\wt{I}_{13}(\tau,x)|&\lesssim\f{1}{L^3}\int_{Q_L}\bigl(|\na z_{+}|\cdot|\na z_{-}|\bigr)(\tau, y)\,dy.
\end{aligned}
\eeno
Thus, similar derivation as \eqref{eq:p16a} yields that 
\beq\label{eq:p1-9c}
\sum_{k=1}^3\|\langle w_-\rangle\left(\log \langle w_-\rangle\right)^2\bar{I}_{1k}\|_{L^\infty_t(L^2_x)}+\|\langle w_-\rangle\left(\log \langle w_-\rangle\right)^2\wt{I}_{13}\|_{L^\infty_t(L^2_x)}
\lesssim\f{(\log L)^2}{L^{\f12}}\cdot \bigl(E_-^0\bigr)^{\f12}\cdot \bigl(E_+^0\bigr)^{\f12}.
\eeq

For $\wt{I}_{11}$ and $\wt{I}_{12}$, we only need to deal with the following term:
\beno 
\wt{I}_{11,L}(\tau,x):=\int_{\{y_1=L\}}\nabla_x G_{(0,0,1)}(x,y)\phi_L(y_3)\cdot\bigl(\partial_{i} z_{+}^{j} \partial_{j} z_{-}^{i}\bigr)(\tau, y)\, dy_2dy_3.
\eeno 
Similarly as $\bar{I}_{1k}$, we first have
\beno 
|\wt{I}_{11,L}(\tau,x)|\lesssim\f{1}{L^2}\int_{\{y_1=L\}}\bigl(|\na z_+|\cdot|\na z_-|\bigr)(\tau, y)\, dy_2dy_3,
\eeno 
which along with the fact that $\langle w_-\rangle(\tau,x)\overset{\text{Lemma \ref{lem:e2}}}{\lesssim} L\lesssim\langle w_-\rangle(\tau,L,y_2,y_3)$ implies
\beno\begin{aligned}
&\langle w_-\rangle\left(\log \langle w_-\rangle\right)^2|\wt{I}_{11,L}|(\tau,x)\lesssim 
\f{(\log L)^2}{L^2}\int_{\{y_1=L\}}\bigl(\langle w_-\rangle|\na z_+|\cdot|\na z_-|\bigr)(\tau, y)\, dy_2dy_3\\ 
&\qquad\lesssim \f{(\log L)^2}{L^2}\cdot\|\langle w_-\rangle\na z_+|_{y_1=L}\|_{L^2_{(y_2,y_3)}}\cdot\|\na z_-|_{y_1=L}\|_{L^2_{(y_2,y_3)}}\\ 
&\qquad\lesssim\f{(\log L)^2}{L^2}\cdot\bigl(\|\langle w_-\rangle\na z_+\|_{L^2_y}+\|\langle w_-\rangle\na^2 z_+\|_{L^2_y}\bigr)\cdot\bigl(\|\na z_-\|_{L^2_y}+\|\na^2 z_-\|_{L^2_y}\bigr),
\end{aligned}\eeno
where we used Sobolev embedding theorem and Lemma \ref{lem:e1} in the last inequality. Therefore, we obtain
\beq\label{eq:p1-9d}
\|\langle w_-\rangle\left(\log \langle w_-\rangle\right)^2\wt{I}_{11,L}\|_{L^\infty_t(L^2_x)}
\lesssim\f{(\log L)^2}{L^{\f12}}\bigl(E_-^0+E_-^1\bigr)^{\f12}\cdot\bigl(E_+^0+E_+^1\bigr)^{\f12}.
\eeq
The same estimates hold for the remaining terms in $\wt{I}_{11}$ and $\wt{I}_{12}$, and consequently for $\wt{I}_{11}$ and $\wt{I}_{12}$ themselves.

\smallskip 

Thanks to \eqref{eq:p1-9b}, \eqref{eq:p1-9c} and \eqref{eq:p1-9d}, we obtain
\beq\label{eq:p1-9e}
\sum_{k=1}^3\|\langle w_-\rangle\left(\log \langle w_-\rangle\right)^2 I_{1k}\|_{L^\infty_t(L^2_x)}
\lesssim\f{(\log L)^2}{L^{\f12}}\bigl(E_-^0+E_-^1\bigr)^{\f12}\cdot\bigl(E_+^0+E_+^1\bigr)^{\f12}.
\eeq

{\it Step 2.3. Estimate for $I_{2k}$.} By definition, we first get
\beno
|I_{2k}(t,x)|\lesssim \int_{|y_3|\geq\f{L}{3}}\f{(|\p_k\na z_+|\cdot|\na z_-|+|\na z_+|\cdot|\p_k\na z_-|)(t,y)}{|x-y_{(0,0,1)}|^2}\cdot 1_{Q_L}(y)\, dy.
\eeno
Similar derivation as \eqref{eq:p16}(for $I_1$ in the proof of Proposition \ref{prop:p1}) yields 
\beq\label{eq:p1-9f}
\sum_{k=1}^2\|\langle w_-\rangle\left(\log \langle w_-\rangle\right)^2I_{2k}\|_{L^\infty_t(L^2_x)}\lesssim\f{1}{L}\Bigl(\sum_{l=0}^3E_-^l\Bigr)^{\f12}\cdot\bigl(E_+^0+E_+^1\bigr)^{\f12}.
\eeq

{\it Step 2.4. Estimate for $\int_{0}^{t} \int_{\Sigma_{\tau}}\langle w_-\rangle^2\left(\log \langle w_-\rangle\right)^4\left|\na z_+\right|\cdot \left|\p\nabla\wt p_1\right| dxd\tau $.} Combining \eqref{eq:p1-9a}, \eqref{eq:p1-9e} and \eqref{eq:p1-9f}, we deduce that
\beno 
\|\langle w_-\rangle\left(\log \langle w_-\rangle\right)^2\p\na\wt{p}_{11}\|_{L^\infty_t(L^2_x)}\lesssim\f{(\log L)^2}{L^{\f12}}\cdot\Bigl(\sum_{l=0}^3E_-^l\Bigr)^{\f12}\cdot\bigl(E_+^0+E_+^1\bigr)^{\f12}.
\eeno
The same estimates hold for the remaining terms in $\p\na\wt{p}_1$. Therefore, we obtain
\beq\label{eq:p1-9}\begin{aligned}
&\qquad\int_{0}^{t} \int_{\Sigma_{\tau}}\langle w_-\rangle^2\left(\log \langle w_-\rangle\right)^4\left|\na z_+\right|\cdot \left|\p\nabla\wt p_1\right| dxd\tau\\ 
&\leq\|\langle w_-\rangle\left(\log \langle w_-\rangle\right)^2\na z_+\|_{L^\infty_t(L^2)}\cdot\|\langle w_-\rangle\left(\log \langle w_-\rangle\right)^2\p\na\wt{p}_1\|_{L^\infty_t(L^2_x)}\cdot T^*\\ 
&\lesssim\f{(\log L)^3}{L^{\f12}}\cdot \Bigl(\sum_{l=0}^3E_-^l\Bigr)^{\f12}\cdot\bigl(E_+^0+E_+^1\bigr).
\end{aligned}\eeq

{\bf Step 3. Estimate for $\int_{0}^{t} \int_{\Sigma_{\tau}}\langle w_-\rangle^2\left(\log \langle w_-\rangle\right)^4\left|\na z_+\right|\cdot \left|\p\nabla\wt p_{\geq 2}\right| dx d\tau $.} By virtue of \eqref{eq:est for wt p geq 2} and $\langle w_-\rangle\lesssim L$, we first get
\beno
\langle w_-\rangle\left(\log \langle w_-\rangle\right)^2\left|\p\nabla\wt p_{\geq 2}(t,x)\right|
\lesssim\f{(\log L)^2}{L^3}\int_{Q_L}|z_-|\cdot|\na z_+|dx
\lesssim\f{(\log L)^2}{L^3}\cdot(E_-)^{\f12}\cdot\bigl(E_+^0\bigr)^{\f12}.
\eeno
Then similar derivation as \eqref{eq:p19} for $\na p_{\geq 2}$ implies
\beq\label{eq:p1-10}
\int_{0}^{t} \int_{\Sigma_{\tau}}\langle w_-\rangle^2\left(\log \langle w_-\rangle\right)^4\left|\na z_+\right|\cdot \left|\p\nabla\wt p_{\geq 2}\right| dx d\tau\lesssim\f{(\log L)^3}{L^{\f32}}\cdot (E_-)^{\f12}\cdot E_+^0.
\eeq

{\bf Step 4. Estimate for $\int_{0}^{t} \int_{\Sigma_{\tau}}\langle w_-\rangle^2\left(\log \langle w_-\rangle\right)^4\left|\na z_+\right|\cdot \bigl(|\wt{\vv B_1}|+|\wt{\vv B_2}|\bigr)dx d\tau $.} By \eqref{eq:est for wt B 1 2}, we have
\beno\begin{aligned}
&|\wt{\vv B_1}(t,x)|\lesssim\f{1}{L^3}\sum_{i=1}^3\int_{\{y_i=L\}}|\na z_+(t,y)|\cdot|z_-(t,y)|d\widehat{y_i} \\ 
&|\wt{\vv B_2}(t,x)|\lesssim\f{1}{L^2}\sum_{i=1}^3\int_{\{y_i=L\}}|\na z_+(t,y)|\cdot|\na z_-(t,y)|d\widehat{y_i}.
\end{aligned}\eeno
Similar argument as \eqref{eq:p23} for $|\vv B_1|+|\vv B_2|$ yields 
\beq\label{eq:p1-11}
\int_{0}^{t} \int_{\Sigma_{\tau}}\langle w_-\rangle^2\left(\log \langle w_-\rangle\right)^4\left|\na z_+\right|\cdot \bigl(|\wt{\vv B_1}|+|\wt{\vv B_2}|\bigr)dx d\tau
\lesssim\frac{\log L}{L^{\frac{3}{2}}}\cdot\left(E_{-}+ E_{-}^0+ E_{-}^1\right)^{\frac{1}{2}}\cdot\left(E_{+}^0+ E_{+}^1\right).
\eeq

{\bf Step 5. Estimate for $\int_{0}^{t} \int_{\Sigma_{\tau}}\langle w_-\rangle^2\left(\log \langle w_-\rangle\right)^4\left|\na z_+\right|\cdot \left|\nabla^2 p\right| dx d\tau $.} Going back to \eqref{eq:p1-1}, we combine the estimates obtained in Steps 1-4 to conclude that
\beno\begin{aligned}
&\quad\int_{0}^{t} \int_{\Sigma_{\tau}}\langle w_-\rangle^2\left(\log \langle w_-\rangle\right)^4\left|\na z_+\right|\cdot \left|\nabla^2 p\right| dx d\tau\\ 
&\lesssim\Bigl(E_-+\sum_{l=0}^3E_-^l\Bigr)^{\f12}\cdot\bigl(E_+^0+E_+^1+ F_+^0+ F_+^1+\f{\e}{L}D_+^1\bigr).
\end{aligned}\eeno 
This establishes \eqref{eq:p1-0} for $\nabla z_+$; the result for $\nabla z_-$ follows from a symmetric argument.
It completes the proof of the proposition.
\end{proof}

\subsubsection{Completion of the first order energy estimate}
In this subsection, we will complete the proof of Proposition \ref{prop:estimates on first order terms}.
\begin{proof}[\proofname\ of Proposition \ref{prop:estimates on first order terms}]
We only derive the first order energy estimate \eqref{eq:estimates on first order terms} for $z_+$; the estimate for $ z_-$ follows by a symmetric argument.

Invoking Proposition \ref{prop:l1}, we apply \eqref{eq:l7} at $t=T^*$ with  $f_+=\p z_+$, $\rho_+=-\p\nabla p-\p z_-\cdot\na z_+$ and $\lambda_+=\langle w_-\rangle^2\left(\log \langle w_-\rangle\right)^4$ to obtain
\begin{equation}\label{eq:c1-1}
\begin{aligned}
& \quad\sup_{0\leq t\leq T^*}\|\langle w_-\rangle\left(\log \langle w_-\rangle\right)^2\na z_+\|_{L^2(\Sigma_t)}^2+\mu\|\langle w_-\rangle\left(\log \langle w_-\rangle\right)^2\na^2 z_+\|_{L^2_{T^*}(L^2_x)}^2\\
&\qquad + \sup\limits_{|u_+|\leq \frac{L}{4}} \int_{C^+_{u_+}}\langle w_-\rangle^2 \left(\log \langle w_-\rangle\right)^4\left|\na z_+\right|^2 d\sigma_+\\ 
& \lesssim\|\langle w_-\rangle\left(\log \langle w_-\rangle\right)^2\na z_+\|_{L^2(\Sigma_0)}^2+ \mu\int_{0}^{T^*} \int_{\Sigma_t} \left(\log \langle w_-\rangle\right)^4\left|\na z_+\right|^{2} dx dt \\ 
&\qquad+\mu^2\|\langle w_-\rangle\left(\log \langle w_-\rangle\right)^2\na^3 z_+\|_{L_{T^*}^2(L^2_x)}^2\\ 
&\qquad+ \int_{0}^{T^*} \int_{\Sigma_t} \langle w_-\rangle^2\left(\log \langle w_-\rangle\right)^4 \left|\na z_+\right|\cdot\bigl(|\na^2 p|+|\na z_-|\cdot|\na z_+|\bigr)\, dxdt,
\end{aligned} 
\end{equation}
where we have used the fact that  $\frac{\left|\nabla \lambda_+\right|^2}{\lambda_+} \lesssim\left(\log \langle w_-\rangle\right)^4$ for $\lambda_+=\langle w_-\rangle^2\left(\log \langle w_-\rangle\right)^4$.

For the second term on the RHS of \eqref{eq:c1-1}, we apply the lowest order energy estimate \eqref{eq:estimates on lowest order terms} from Proposition \ref{prop:estimates on lowest order terms}. The first part of the last term, invovling $\na^2p$, is bounded by \eqref{eq:p1-0} in Proposition \ref{prop:p1-1}. For the second part of the last term, we find that
\beno\begin{aligned}
&\qquad\int_{0}^{T^*} \int_{\Sigma_t} \langle w_-\rangle^2\left(\log \langle w_-\rangle\right)^4 |\na z_+|\cdot|\na z_-|\cdot|\na z_+|\, dxdt\\ 
&\lesssim\|\langle w_+\rangle\left(\log \langle w_+\rangle\right)^2\na z_-\|_{L^\infty_{T^*}(L^\infty_x)}\cdot\Bigl\|\f{\langle w_-\rangle\left(\log \langle w_-\rangle\right)^2}{\langle w_+\rangle^{\f12}\log \langle w_+\rangle}\na z_+\Bigr\|_{L^2_{T^*}(L^2_x)}^2,
\end{aligned}\eeno 
which along with the weighted Sobolev inequality and \eqref{eq:p1-3} implies 
\beno
\int_{0}^{T^*} \int_{\Sigma_t} \langle w_-\rangle^2\left(\log \langle w_-\rangle\right)^4 |\na z_+|^2\cdot|\na z_-|\, dxdt
\lesssim\Big(\sum_{k=0}^2 E_-^k\Big)^{\frac{1}{2}}\cdot\bigl(\frac{E_+^0}{L\log L}+ F_+^0\bigr). 
\eeno 
The third term on the RHS of \eqref{eq:c1-1} is bounded by $\mu D_+^1$. Combining these estimates, we finally derive from \eqref{eq:c1-1} that
\beno\begin{aligned}
E_+^0+ D_+^0+ F_+^0
&\lesssim E_+(0)+E_+^0(0)+ \Big(E_-+ \sum_{l=0}^3 E_-^l \Big)^{\frac{1}{2}} \left(E_++ E_+^1+ F_++ F_+^1\right)\\ 
&\qquad+\Big(E_-+ \sum_{l=0}^3 E_-^l \Big)^{\frac{1}{2}} \left( E_+^{0}+ F_+^{0}\right)+ \mu D_+^{0}+\mu D_+^1.
\end{aligned}\eeno

By the ansatz \eqref{eq:a3}, we have $\big(E_-+ \sum_{l=0}^3 E_-^l \big)^{\frac{1}{2}}\lesssim\e$. Taking $\e$ and $\mu$ sufficiently small, we obtain
\beno 
E_+^0+ D_+^0+ F_+^0
\lesssim E_+(0)+E_+^0(0)+ \Big(E_-+ \sum_{l=0}^3 E_-^l \Big)^{\frac{1}{2}} \left(E_++ E_+^1+ F_++ F_+^1\right)+\mu D_+^1. 
\eeno
This establishes \eqref{eq:estimates on first order terms} for $\na z_+$. The result for $\na z_-$ follows from a symmetric argument.
This completes the proof of the proposition.
\end{proof}

\subsection{Energy estimates for higher order terms}
To derive higher order energy estimates, we first commute derivatives with the vorticity equations \eqref{eq:j}. Specifically, for any multi-index $\beta$ with $1 \leq|\beta| \leq N_{*}+4$, applying $\partial^{\beta}$ to \eqref{eq:j} gives rise to
\begin{equation}
\left\{\begin{array}{l}
\partial_{t} j_{+}^{(\beta)}+Z_{-} \cdot \nabla j_{+}^{(\beta)}-\mu \Delta j_{+}^{(\beta)}=\rho_{+}^{(\beta)},  \\
\partial_{t} j_{-}^{(\beta)}+Z_{+} \cdot \nabla j_{-}^{(\beta)}-\mu \Delta j_{-}^{(\beta)}=\rho_{-}^{(\beta)},
\end{array}\right. \label{eq:h1}
\end{equation}
where source terms $\rho_{\pm}^{(\beta)}$ are given by
\begin{equation}
\begin{aligned}
\rho_{+}^{(\beta)} & =-\partial^{\beta}\left(\nabla z_{-} \wedge \nabla z_{+}\right)-\left[\partial^{\beta},\, z_{-} \cdot \nabla\right] j_{+}, \\
\rho_{-}^{(\beta)} & =-\partial^{\beta}\left(\nabla z_{+} \wedge \nabla z_{-}\right)-\left[\partial^{\beta},\, z_{+} \cdot \nabla\right] j_{-} .
\end{aligned} \label{eq:h2}
\end{equation}
The following subsections are devoted to applying \eqref{eq:l7} (from Proposition \ref{prop:l1}) and \eqref{eq:l36} (from Proposition \ref{prop:l2}) to the system \eqref{eq:h1} for the cases $1 \leq|\beta| \leq N_{*}$ and $N_{*}+1 \leq|\beta|\leq N_{*}+3$, respectively. By establishing weighted energy estimates for $\nabla z_{\pm}^{(\beta)}$ at each order with suitable weights, we will complete the proof of the total energy estimate.

\subsubsection{Energy estimates for $\nabla z_{\pm}^{(\beta)}$ with $1 \leq|\beta| \leq N_{*}$}
The main result of this subsection is stated as follows:
\begin{proposition}\label{prop:h1}
Let $L \geq e^{1/\mu}\geq R\gg1$ and let $\e$ be sufficiently small. Then under the assumptions \eqref{eq:a1}--\eqref{eq:a3}, we have
\begin{equation}\label{eq:h3}
\begin{aligned}
 \sum_{k=1}^{N_{*}}\left(E_{\pm}^{k}+ F_{\pm}^{k}+ D_{\pm}^{k}\right) 
&\lesssim \sum_{k=1}^{N_{*}} E_{\pm}^{k}(0)+\sum_{l=0}^{N_{*}}\left(E_{\mp}^{l}\right)^{\frac{1}{2}}\cdot \Big(\frac{E_{\pm}^{0}}{L}+ F_{\pm}^{0}\Big)\\ 
&\qquad+ \frac{1}{R^{2}}\bigl(E_{\pm}^{0}+ D_{\pm}^{0}\bigr)+\frac{\varepsilon\mu}{L}E_{\pm}^{N_{*}+1}+\mu D_{\pm}^{N_{*}+1}.
\end{aligned} 
\end{equation}
Moreover, there holds
\begin{equation}\label{eq:h15}
\begin{aligned}
& E_{\pm}+ \sum_{k=0}^{N_{*}} E_{\pm}^{k}+ D_{\pm}+ \sum_{k=0}^{N_{*}} D_{\pm}^{k}+ F_{\pm}+ \sum_{k=0}^{N_{*}} F_{\pm}^{k}\\
\lesssim & E_{\pm}(0)+ \sum_{k=0}^{N_{*}} E_{\pm}^{k}(0)+\frac{\varepsilon\mu}{L}E_{\pm}^{N_{*}+1}+\mu D_{\pm}^{N_{*}+1}.
\end{aligned} 
\end{equation}
\end{proposition}
\begin{proof}
We divide the proof into three steps.

\textbf{Step 1. Energy estimate for \eqref{eq:h1}.} For any multi-index $\beta$ with $1 \leq |\beta| \leq N_{*}$, applying \eqref{eq:l7} to \eqref{eq:h1} with the weight functions $\lambda_{\pm}=\langle w_{\mp}\rangle^2\left(\log \langle w_{\mp}\rangle\right)^4$ yields
\begin{equation}\label{eq:h4}
\begin{aligned}
&\sup\limits_{0\leq\tau\leq t}\|\langle w_{\mp}\rangle\left(\log \langle w_{\mp}\rangle\right)^2j_{\pm}^{(\beta)}\|_{L^2(\Sigma_\tau)}^2+\mu\|\langle w_{\mp}\rangle\left(\log \langle w_{\mp}\rangle\right)^2\na j_{\pm}^{(\beta)}\|_{L^2_t(L^2_x)}^2\\ 
&\qquad+\sup\limits_{|u_\pm|\leq \frac{L}{4}} \int_{C^{\pm}_{t,u_\pm}} \langle w_{\mp}\rangle^2\left(\log \langle w_{\mp}\rangle\right)^4 |j_{\pm}^{(\beta)}|^2 d\sigma_\pm \\
\lesssim&\|\langle w_{\mp}\rangle\left(\log \langle w_{\mp}\rangle\right)^2j_{\pm}^{(\beta)}\|_{L^2(\Sigma_0)}^2+ \underbrace{\int_{0}^{t} \int_{\Sigma_{\tau}}\langle w_{\mp}\rangle^2\left(\log \langle w_{\mp}\rangle\right)^4|j_{\pm}^{(\beta)}|\cdot|\rho_{\pm}^{(\beta)}|\, dx d\tau}_{\text{nonlinear interaction term } I_{\pm}(t)} \\
& \qquad +\mu\|\left(\log \langle w_{\mp}\rangle\right)^2j_{\pm}^{(\beta)}\|_{L^2_t(L^2_x)}^2+\mu^2\|\langle w_{\mp}\rangle\left(\log \langle w_{\mp}\rangle\right)^2\na^2 j_{\pm}^{(\beta)}\|_{L^2_t(L^2_x)}^2,
\end{aligned} 
\end{equation}
where we used the fact that $\frac{\left|\nabla \lambda_{\pm}\right|^2}{\lambda_{\pm}}\lesssim\left(\log \langle w_{\mp}\rangle\right)^4$ for $\lambda_\pm=\langle w_\mp\rangle^2\left(\log \langle w_\mp\rangle\right)^4$.

For the third term on the RHS of \eqref{eq:h4}, we have
\beno 
\mu\|\left(\log \langle w_{\mp}\rangle\right)^2j_{\pm}^{(\beta)}\|_{L^2_t(L^2_x)}^2\leq\f{1}{R^2}\cdot\mu\int_0^t\int_{\Sigma_\tau} \langle w_{\mp}\rangle^2\left(\log \langle w_{\mp}\rangle\right)^4|\na z_{\pm}^{(|\beta|)}|^2\, dxd\tau=\f{1}{R^2}\cdot D_\pm^{|\beta|-1}(t).
\eeno
 While the last term of \eqref{eq:h4} is bounded by $\mu D_\pm^{|\beta|+1}(t)$. Then \eqref{eq:h4} is reduced to
\begin{equation}\label{eq:h5}
\begin{aligned}
& \sup\limits_{0\leq\tau\leq t}\|\langle w_{\mp}\rangle\left(\log \langle w_{\mp}\rangle\right)^2j_{\pm}^{(\beta)}\|_{L^2(\Sigma_\tau)}^2+\mu\|\langle w_{\mp}\rangle\left(\log \langle w_{\mp}\rangle\right)^2\na j_{\pm}^{(\beta)}\|_{L^2_t(L^2_x)}^2+F_\pm^{(\beta)}(t) \\
& \lesssim \|\langle w_{\mp}\rangle\left(\log \langle w_{\mp}\rangle\right)^2j_{\pm}^{(\beta)}\|_{L^2(\Sigma_0)}^2+I_\pm(t) +\f{1}{R^2}\cdot D_\pm^{|\beta|-1}(t)+\mu D_\pm^{|\beta|+1}(t).
\end{aligned} 
\end{equation}

\textbf{Step 2. Estimates for the nonlinear interaction term $I_{\pm}(t)$.}  By symmetry, it suffices to control $I_{+}$. Due to the expression of $\rho_{+}^{(\beta)}$ in \eqref{eq:h2}, we have 
\beno
|\rho_{+}^{(\beta)}|\lesssim \sum_{\gamma \leq \beta} |\nabla z_{-}^{(\gamma)}|\cdot |\nabla z_{+}^{(\beta-\gamma)}|+ \sum_{0\neq \gamma \leq \beta}|z_{-}^{(\gamma)}|\cdot|\nabla j_{+}^{(\beta-\gamma)}|
\lesssim \sum_{k\leq|\beta|}|\nabla z_{-}^{(k)}|\cdot|\nabla z_{+}^{(|\beta|-k)}|,
\eeno
which implies
\begin{equation}\label{eq:h6}
I_{+}(t) \lesssim \sum_{k\leq|\beta|} \int_{0}^{t} \int_{\Sigma_{\tau}} \langle w_{-}\rangle^2\left(\log \langle w_{-}\rangle\right)^4 |j_{+}^{(\beta)}|\cdot |\nabla z_{-}^{(k)}| \cdot|\nabla z_{+}^{(|\beta|-k)}|\, dx d\tau . 
\end{equation}

We split the analysis of $I_+$ into two cases according to the order of the derivatives:

\textbf{Case 1:} $1\leq|\beta|\leq N_{*}-2$. Since $k+2\leq N_{*}$, we can use the weighted Sobolev inequality \eqref{eq:e12} to control the $L^\infty_x$ norm  for $\nabla z_{-}^{(k)}$.  Hence, we have
\beno\begin{aligned}
I_+(t)&\lesssim\sum_{k\leq|\beta|} \|\langle w_{+}\rangle \left(\log \langle w_{+}\rangle\right)^2 \nabla z_{-}^{(k)}\|_{L_{(t,x)}^{\infty}}\cdot \Bigl\|\frac{\langle w_{-}\rangle\left(\log \langle w_{-}\rangle\right)^2}{\langle w_{+}\rangle^{\f12}\log \langle w_{+}\rangle}\cdot j_{+}^{(\beta)}\Bigr\|_{L^2_t(L^2_x)}\\ 
&\qquad\cdot\Bigl\|\frac{\langle w_{-}\rangle\left(\log \langle w_{-}\rangle\right)^2}{\langle w_{+}\rangle^{\f12}\log \langle w_{+}\rangle}\cdot\nabla z_{+}^{(|\beta|-k)}\Bigr\|_{L^2_t(L^2_x)},
\end{aligned}\eeno
which along with \eqref{eq:e12}, \eqref{eq:p1-3}, \eqref{eq:p1-3aa} and \eqref{eq:p1-3a} (for $\na z_+^{(|\beta|)}$) implies
\beno
I_+(t)\lesssim\sum_{k\leq|\beta|}\Bigl(\sum_{l=k}^{k+2}E_{-}^{l}\Bigr)^{\frac{1}{2}
}\cdot\Bigl(\f{ E_+^{|\beta|}}{L\log L}+ F_+^{|\beta|}\Bigr)^{\f12}\cdot\Bigl(\sum_{l=0}^{|\beta|}\bigl(\f{ E_+^l}{L}+ F_+^l\bigr)+\frac{\e}{L}D_\pm^{|\beta|}\Bigr)^{\f12}.
\eeno
Thus, we obtain
\begin{equation}\label{eq:h7}
I_{+}\lesssim \sum_{l=0}^{N_{*}}\left(E_{-}^{l}\right)^{\frac{1}{2}}\cdot\biggl(\sum_{k=0}^{|\beta|}\Big(\frac{E_{+}^{k}}{L}+ F_{+}^k\Bigr)+ \frac{\varepsilon}{L}D_{+}^{|\beta|}\biggr). 
\end{equation}

\textbf{Case 2:} $|\beta|=N_{*}-1$ or $N_{*}$. In this case, we first  rewrite $I_{+}$ as
\beno
I_{+}(t) \lesssim\Big(\underbrace{\sum_{k\leq N_{*}-2}}_{I_1(t)}+\underbrace{\sum_{N_{*}-1 \leq k \leq|\beta|}}_{I_2(t)}\Big) \int_{0}^{t} \int_{\Sigma_{\tau}}\langle w_{-}\rangle^2\left(\log \langle w_{-}\rangle\right)^4|j_{+}^{(\beta)}|\cdot|\nabla z_{-}^{(k)}|\cdot |\nabla z_{+}^{(|\beta|-k)}| dx d\tau. 
\eeno

For $I_1$, similar derivation as \eqref{eq:h7} in Case 1 gives rise to
\begin{equation}\label{eq:h8}
I_1(t)\lesssim\sum_{l=0}^{N_{*}}\left(E_{-}^{l}\right)^{\frac{1}{2}}\cdot\biggl(\sum_{k=0}^{|\beta|}\Bigl(\frac{E_{+}^{k}}{L}+ F_{+}^k\Bigr)+ \frac{\varepsilon}{L}D_{+}^{|\beta|}\biggr). 
\end{equation}

For $I_2$, since $N_{*}-1 \leq k \leq|\beta|\leq N_*$ implies $0\leq |\beta|-k\leq 1$, we can control the $L_{x}^{\infty}$ norm of $\nabla z_{+}^{(|\beta|-k)}$ by using the weighted Sobolev inequality. This yields
\beno
\begin{aligned}
I_{2}(t) 
& \leq \sum_{k=N_{*}-1}^{|\beta|}\bigl\|\langle w_{+}\rangle\left(\log \langle w_{+}\rangle\right)^2 \nabla z_{-}^{(k)}\bigr\|_{L_t^{\infty}(L_{x}^{2})}\cdot\Bigl\|\frac{\langle w_{-}\rangle\left(\log \langle w_{-}\rangle\right)^2}{\langle w_{+}\rangle^{\frac{1}{2}} \log \langle w_{+}\rangle} j_{+}^{(\beta)}\Bigr\|_{L_t^{2}(L_{x}^{2})} \\
& \quad\quad\quad\quad\cdot \Bigl\|\frac{\langle w_{-}\rangle\left(\log \langle w_{-}\rangle\right)^2}{\langle w_{+}\rangle^{\frac{1}{2}} \log \langle w_{+}\rangle} \nabla z_{+}^{(|\beta|-k)}\Bigr\|_{L_t^{2}(L_{x}^{\infty})}\\
& \lesssim \sup_{N_{*}-1\leq k\leq|\beta|}\left(E_{-}^{k}\right)^{\frac{1}{2}}\cdot\Bigl(\frac{E_{+}^{|\beta|}}{L\log L}+F_{+}^{|\beta|}(t)\Bigr)^{\frac{1}{2}}\cdot \sum_{k=0}^{|\beta|+1-N_{*}} \Bigl\|\frac{\langle w_{-}\rangle\left(\log \langle w_{-}\rangle\right)^2}{\langle w_{+}\rangle^{\frac{1}{2}} \log \langle w_{+}\rangle} \nabla z_{+}^{(k)}\Bigr\|_{L_t^{2}(L_{x}^{\infty})},
\end{aligned} 
\eeno
where we used \eqref{eq:p1-3} for $j_+^{(\beta)}$ in the last inequality.
Concerning the last term above, since  $\nabla^{i}\Big(\frac{\langle w_{-}\rangle\left(\log \langle w_{-}\rangle\right)^2}{\langle w_{+}\rangle^{\frac{1}{2}} \log \langle w_{+}\rangle}\Big) \lesssim \frac{\langle w_{-}\rangle\left(\log \langle w_{-}\rangle\right)^2}{\langle w_{+}\rangle^{\frac{1}{2}} \log \langle \overline{w_{+}}\rangle}$ holds for a piecewise weight function with $i=1,2$, by an argument similar to the proof of Lemma \ref{lem:e3}, the standard Sobolev inequality implies
\beno
\begin{aligned}
&\Bigl\|\frac{\langle w_{-}\rangle\left(\log \langle w_{-}\rangle\right)^2}{\langle w_{+}\rangle^{\frac{1}{2}} \log \langle w_{+}\rangle} \nabla z_{+}^{(k)}\Bigr\|_{L_{\tau}^{2}(L_{x}^{\infty})}^2
\lesssim \sum_{l=k}^{k+2}\Bigl\|\frac{\langle w_{-}\rangle\left(\log \langle w_{-}\rangle\right)^2}{\langle w_{+}\rangle^{\frac{1}{2}} \log \langle w_{+}\rangle} \nabla z_{+}^{(l)}\Bigr\|_{L_{\tau}^{2}(L_{x}^{2})}^2\\
&\qquad \overset{\eqref{eq:p1-3aa}}{\lesssim}
\sum_{l=k}^{k+2}\Bigl(\sum_{m=0}^l\bigl(\f{ E_+^m}{L\log L}+ F_+^m\bigr)+\frac{\varepsilon}{L}E_+^{l+1}\Bigr)
\overset{k=0,\,1}{\lesssim}\sum_{l=0}^4\f{ E_+^l}{L}+ \sum_{m=0}^3 F_+^m.
\end{aligned} 
\eeno
By this inequality, using $N_*\geq 7$,  we obtain
\beq\label{eq:h9}
I_{2}(t)
\lesssim\sup_{k\leq N_{*}}\left(E_{-}^{k}\right)^{\frac{1}{2}} \cdot\sum_{k=0}^{|\beta|}\Bigl(\frac{E_{+}^{k}}{L}+ F_{+}^k\Bigr).
\eeq

From \eqref{eq:h8} and \eqref{eq:h9}, we obtain the estimates for $I_+$ when $|\beta| = N_*-1$ and $|\beta| = N_*$. Together with \eqref{eq:h7}, this implies that for all $1 \leq |\beta| \leq N_{*}$,
\begin{equation}\label{eq:h10}
I_+(t)\lesssim\sum_{l=0}^{N_{*}}\left(E_{-}^{l}\right)^{\frac{1}{2}}\cdot\biggl(\sum_{k=0}^{|\beta|}\Bigl(\frac{E_{+}^{k}}{L}+ F_{+}^k\Bigr)+ \frac{\varepsilon}{L}D_{+}^{|\beta|}\biggr). 
\end{equation}
Similar estimate holds for $I_-(t)$.

\textbf{Step 3. Energy estimates for $\nabla z_{\pm}^{(\beta)}$ with $1\leq|\beta|\leq N_{*}$.}
Applying the div-curl lemma \ref{lem:div-cur} to $z_\pm^{(|\beta|)}$  with $\lambda=\lambda_\pm=\langle w_\mp\rangle^2\left(\log \langle w_\mp\rangle\right)^4$ and $F(L)=\f{1}{L\log L}$, we have
\beno
\begin{aligned}
&\|\langle w_{\mp}\rangle\left(\log \langle w_{\mp}\rangle\right)^2\nabla z_{\pm}^{(|\beta|)}\|_{L^2_x}^2 
\lesssim\|\langle w_{\mp}\rangle\left(\log \langle w_{\mp}\rangle\right)^2j_{\pm}^{(|\beta|)}\|_{L^2_x}^2 + \|\left(\log \langle w_{\mp}\rangle\right)^2z_{\pm}^{(|\beta|)}\|_{L^2_x}^2\\
&\qquad+\f{\e}{L\log L}\bigl(\|\langle w_{\mp}\rangle\left(\log \langle w_{\mp}\rangle\right)^2 z_{\pm}^{(|\beta|)}\|_{L^{2}_x}^{2}+\|\langle w_{\mp}\rangle\left(\log \langle w_{\mp}\rangle\right)^2\nabla^2 z_{\pm}^{(|\beta|)}\|_{L^2_x}^{2}\bigr).
\end{aligned}
\eeno
Since $\langle w_\mp\rangle\geq R$ and $L\geq e^{\f{1}{\mu}}$ (or $\f{1}{\log L}\leq\mu$), by taking $\e\leq\f{1}{R^2}$, we obtain 
\beq\label{eq:h11}
E_\pm^{|\beta|}\lesssim\sup_{0\leq t\leq T^*}\|\langle w_{\mp}\rangle\left(\log \langle w_{\mp}\rangle\right)^2j_{\pm}^{(|\beta|)}\|_{L^2_x}^2+\f{1}{R^2} E_\pm^{|\beta|-1}+\f{\e\mu}{L}E_\pm^{|\beta|+1}.
\eeq

 Similarly, for $z_\pm^{(|\beta|+1)}$, we obtain
\beno
\begin{aligned}
&\mu\|\langle w_{\mp}\rangle\left(\log \langle w_{\mp}\rangle\right)^2\nabla z_{\pm}^{(|\beta|+1)}\|_{L^2_t(L^2_x)}^2\lesssim\mu\|\langle w_{\mp}\rangle\left(\log \langle w_{\mp}\rangle\right)^2\na j_{\pm}^{(|\beta|)}\|_{L^2_t(L^2_x)}^2 \\
&\qquad+\f{\mu}{R^2}\|\langle w_{\mp}\rangle\left(\log \langle w_{\mp}\rangle\right)^2\na z_{\pm}^{(|\beta|)}\|_{L^{2}_t(L^2_x)}^{2}+\f{\e}{L}\cdot\mu^2\|\langle w_{\mp}\rangle\left(\log \langle w_{\mp}\rangle\right)^2\nabla^2 z_{\pm}^{(|\beta|+1)}\|_{L^2_t(L^2_x)}^{2},
\end{aligned} 
\eeno
which implies
\beq\label{eq:h12}
D_\pm^{|\beta|}\lesssim\mu\|\langle w_{\mp}\rangle\left(\log \langle w_{\mp}\rangle\right)^2\na j_{\pm}^{(|\beta|)}\|_{L^2_{T^*}(L^2_x)}^2
+\f{1}{R^2} D_\pm^{|\beta|-1}+\f{\e\mu}{L} D_\pm^{|\beta|+1}.
\eeq

Noticing that the first terms on the RHS of \eqref{eq:h11} and \eqref{eq:h12} correspond to the first two terms on the LHS of \eqref{eq:h5} evaluated at  $t=T^*$, we combine them with the estimate \eqref{eq:h10} for $I_\pm$. Then we deduce from \eqref{eq:h5} that 
\begin{equation}\label{eq:h13}
\begin{aligned}
 E_{\pm}^{|\beta|}+D_{\pm}^{|\beta|}+F_{\pm}^{|\beta|}
&\lesssim E_{\pm}^{|\beta|}(0)+ \sum_{l=0}^{N_{*}}\left(E_{\mp}^{l}\right)^{\frac{1}{2}}\cdot \bigg(\frac{1}{L}\Big(E_{\pm}^{0}+\sum_{k=1}^{|\beta|}E_{\pm}^{k}\Big)+ F_{\pm}^{0}+ \sum_{k=1}^{|\beta|}F_{\pm}^{k}\bigg)\\
&\qquad + \frac{1}{R^2}\bigl(E_{\pm}^{|\beta|-1}+ D_{\pm}^{|\beta|-1}\bigr)+\frac{\varepsilon\mu}{L}E_{\pm}^{|\beta|+1}+\mu D_{\pm}^{|\beta|+1}.
\end{aligned} 
\end{equation}
Next we sum up \eqref{eq:h13} over all multi-indices $1 \leq |\beta| \leq N_*$ to obtain 
\beno\begin{aligned}
&\sum_{k=1}^{N_*}\bigl(E_{\pm}^{k}+D_{\pm}^{k}+F_{\pm}^{k}\bigr)
\lesssim\sum_{k=1}^{N_*}E_{\pm}^{k}(0)+\sum_{l=0}^{N_{*}}\left(E_{\mp}^{l}\right)^{\frac{1}{2}}\cdot\Bigl(\f{E_\pm^0}{L}+F_\pm^0\Bigr)+\frac{1}{R^2}\bigl(E_{\pm}^0+D_{\pm}^{0}\bigr)
\\
&\qquad+\frac{\varepsilon\mu}{L}E_{\pm}^{N_*+1}+\mu D_{\pm}^{N_*+1}+\sum_{l=0}^{N_{*}}\left(E_{\mp}^{l}\right)^{\frac{1}{2}}\cdot\sum_{k=1}^{N_*}\bigl(E_{\pm}^{k}+F_{\pm}^{k}\bigr)+\frac{1}{R^2}\sum_{k=1}^{N_*-1}\bigl(E_{\pm}^k+ D_{\pm}^k\bigr).
\end{aligned}\eeno
Finally, since the ansatz \eqref{eq:a3} implies $\sum_{l=0}^{N_{*}} E_{\mp}^{l} \leq 2 C_{1} \varepsilon^2$,  we may choose $\varepsilon$ sufficiently small and $R$ sufficiently large to conclude that
\begin{equation}\label{eq:h14}
\begin{aligned}
 \sum_{k=1}^{N_{*}}\left(E_{\pm}^{k}+ D_{\pm}^{k}+ F_{\pm}^{k}\right) 
&\lesssim\sum_{k=1}^{N_*}E_{\pm}^{k}(0)+\sum_{l=0}^{N_{*}}\left(E_{\mp}^{l}\right)^{\frac{1}{2}}\cdot\Bigl(\f{E_\pm^0}{L}+F_\pm^0\Bigr)\\
&\qquad+\frac{1}{R^2}\bigl(E_{\pm}^0+D_{\pm}^{0}\bigr)+ \frac{\varepsilon\mu}{L}E_{\pm}^{N_*+1}+\mu D_{\pm}^{N_*+1}.
\end{aligned} 
\end{equation}
This is \eqref{eq:h3}.

{\bf Step 4. Proof of \eqref{eq:h15}.} Combining \eqref{eq:estimates on lowest order terms}, \eqref{eq:estimates on first order terms} and \eqref{eq:h14}, by using the ansatz \eqref{eq:a3} and taking $\e$ sufficiently small and $R$ sufficiently large, we obtain \eqref{eq:h15}.
This completes the proof of Proposition \ref{prop:h1}.
\end{proof}

\subsubsection{Energy estimates for $\nabla z_{\pm}^{(\beta)}$ with $ N_{*}+1\leq |\beta|\leq N_{*}+3$}
The goal of this subsection is to provide a parabolic type estimate to control the additional higher order terms on the right-hand side of \eqref{eq:h15} caused by the presence of non-zero viscosity and the boundary terms when applying the linear energy estimate \eqref{eq:l7} and the div-curl lemma (Lemma \ref{lem:div-cur}). We shall prove the following proposition.
\begin{proposition}\label{prop:h2}
Under the same assumptions as Proposition \ref{prop:h1}, we have
\begin{equation}\label{eq:h16}
\begin{aligned}
&\mu\bigl(E_{\pm}^{N_{*}+1}+  D_{\pm}^{N_{*}+1}\bigr)+\f{\mu}{\log L}\bigl(E_{\pm}^{N_*+2}+D_{\pm}^{N_*+2}\bigr)+\f{\mu}{(\log L)^2}\bigl(E_{\pm}^{N_{*}+3}+  D_{\pm}^{N_{*}+3}\bigr)\\ 
&\lesssim\mu E_{\pm}^{N_{*}+1}(0)+\f{\mu}{\log L}E_{\pm}^{N_*+2}(0)+\f{\mu}{(\log L)^2}E_{\pm}^{N_{*}+3}(0)
+\mu\bigl(E_{\pm}^{N_*}+  D_{\pm}^{N_*}\bigl)\\ 
&\qquad+\e\cdot\f{\mu}{L}\bigl(E_{\pm}^{N_{*}+4}+ D_{\pm}^{N_{*}+4}\bigr)
+ \e\cdot\sum_{l=0}^{N_{*}}\bigl(E_{\pm}^{l}+D_{\pm}^l\bigr) .
\end{aligned} 
\end{equation}
\end{proposition}
\begin{proof}
We prove the estimates for $\na z_\pm^{(\beta)}$ with $|\beta| $ from $N_*+1$ to $N_*+3$ separately.

{\bf Step 1. Estimates for $\na z_\pm^{(\beta)}$ with $|\beta|=N_*+1$.} Similarly as the proof of Proposition \ref{prop:h1}, we apply \eqref{eq:l36} to \eqref{eq:h1} for $|\beta|= N_{*}+1$ with the weight functions $\lambda_{\pm}=\mu\langle w_{\mp}\rangle^2\left(\log \langle w_{\mp}\rangle\right)^4$ to obtain
\begin{equation}\label{eq:h17}
\begin{aligned}
&\mu\Bigl(\sup\limits_{0\leq \tau \leq t}\|\langle w_{\mp}\rangle\left(\log \langle w_{\mp}\rangle\right)^2j_{\pm}^{(\beta)}\|_{L^2(\Sigma_\tau)}^2+\mu\|\langle w_{\mp}\rangle\left(\log \langle w_{\mp}\rangle\right)^2\na j_{\pm}^{(\beta)}\|_{L^2_t(L^2_x)}^2\Bigr)\\
& \lesssim \mu E_{\pm}^{N_{*}+1}(0)+\mu D_\pm^{N_*}+\underbrace{\f{\e\mu^2}{L\log L}\int_{0}^{t} \int_{\Sigma_{\tau}}\langle w_{\mp}\rangle^2\left(\log \langle w_{\mp}\rangle\right)^4|\na^2j_{\pm}^{(\beta)}|^2 dx d\tau}_{J_\pm}\\
&\qquad  +\underbrace{\mu\int_{0}^{t} \int_{\Sigma_{\tau}}\langle w_{\mp}\rangle^2\left(\log \langle w_{\mp}\rangle\right)^4|j_{\pm}^{(\beta)}|\cdot|\rho_{\pm}^{(\beta)}| dx d\tau}_{\text{nonlinear interaction term } I_{\pm}}.
\end{aligned} 
\end{equation}

{\it Step 1.1. Estimates for $J_\pm$ and $I_\pm.$}
Since Lemma \ref{lem:e2} implies $\langle w_\mp\rangle\lesssim L$, we have
\beq\label{eq:h18}
J_\pm\lesssim\f{\e\mu^2}{\log L}\int_{0}^{t} \int_{\Sigma_{\tau}}\langle w_{\mp}\rangle\left(\log \langle w_{\mp}\rangle\right)^4|\na^2j_{\pm}^{(\beta)}|^2 dx d\tau\lesssim\f{\e\mu}{\log L} D_\pm^{N_*+2}.
\eeq
It remains to bound the nonlinear interaction term $I_{\pm}$. By symmetry, it suffices to deal with $I_{+}$. 

Since \eqref{eq:h2} yields
\begin{equation*}
|\rho_{+}^{(\beta)}| \lesssim \sum_{k\leq N_{*}+1}|\nabla z_{-}^{(k)}|\cdot|\nabla z_{+}^{(N_{*}+1-k)}|,\quad\text{for }\,  |\beta|=N_{*}+1,
\end{equation*}
 we have
\beno
I_{+} \lesssim\Big(\underbrace{\sum_{k\leq N_{*}-2}}_{I_1}+\underbrace{\sum_{k=N_{*}-1}^{N_{*}+1}}_{I_2}\Big) \mu\int_{0}^{t} \int_{\Sigma_{\tau}}\langle w_{-}\rangle^2\left(\log \langle w_{-}\rangle\right)^4|j_{+}^{(\beta)}|\cdot|\nabla z_{-}^{(k)}|\cdot|\nabla z_{+}^{(N_{*}+1-k)}| dx d\tau. 
\eeno
For $I_1$, since $k\leq N_{*}-2$, we obtain
\beno
\begin{aligned}
I_1&\leq\sum_{k\leq N_{*}-2}\|\nabla z_{-}^{(k)}\|_{L_{(t,x)}^{\infty}}\cdot\sqrt\mu\|\langle w_{-}\rangle\left(\log \langle w_{-}\rangle\right)^2j_{+}^{(\beta)}\|_{L^2_t(L^2_x)}\\ 
&\qquad\qquad\cdot\sqrt\mu\|\langle w_{-}\rangle\left(\log \langle w_{-}\rangle\right)^2\nabla z_{+}^{(N_{*}+1-k)}\|_{L^2_t(L^2_x)}\\ 
&\overset{|\beta|=N_*+1}{\lesssim}\sum_{k\leq N_{*}-2}\bigl(\sum_{l=k}^{k+2}E_-^l\bigr)^{\f12}\cdot\bigl(D_+^{N_*}\bigr)^{\f12}\cdot \bigl(D_+^{N_*-k}\bigr)^{\f12}
\lesssim\sum_{l=0}^{N_{*}}\bigl(E_-^l\bigr)^{\f12}\cdot
\sum_{k=1}^{N_{*}} D_+^k.
\end{aligned}\eeno
While for $I_2$, since $N_{*}+1-k \leq 2 \leq N_{*}-2$, we get
\beno
\begin{aligned}
I_{2}
& \leq \sum_{k=N_{*}-1}^{N_{*}+1}\sqrt{\mu}\|\langle w_{-}\rangle\left(\log \langle w_{-}\rangle\right)^2 j_{+}^{(\beta)}\|_{L_t^2(L_x^2)}\cdot\sqrt{\mu}\|\nabla z_{-}^{(k)}\|_{L_t^2(L_x^2)}\\
& \quad\quad\quad\quad\ \cdot\|\langle w_{-}\rangle\left(\log \langle w_{-}\rangle\right)^2 \nabla z_{+}^{(N_{*}+1-k)}\|_{L^\infty_{(t,x)}} \\
&\lesssim\bigl(D_+^{N_*}\bigr)^{\f12}\cdot\sum_{k=N_{*}-1}^{N_{*}+1}\bigl(D_-^{k-1}\bigr)^{\f12}\cdot\sum_{l=N_{*}+1-k}^{N_{*}+3-k}\bigl(E_+^l\bigr)^{\f12}\\ 
& \lesssim\sum_{k=1}^{N_{*}}\big( D_{-}^{k}\big)^{\frac{1}{2}}
\cdot  \Big( \sum_{k=0}^{N_{*}} E_{+}^{k}+ \sum_{k=1}^{N_{*}} D_{+}^{k}\Big) .
\end{aligned} 
\eeno
Thus, we conclude that
\begin{equation} \label{eq:h19}
\begin{aligned}
I_{+} \lesssim \Bigl(\sum_{l=0}^{N_{*}}\bigl(E_-^l\bigr)^{\f12}+ \sum_{k=1}^{N_{*}}\bigl(D_{-}^{k}\bigr)^{\frac{1}{2}}\Bigr)\cdot\Big( \sum_{k=0}^{N_{*}} E_{+}^{k}+ \sum_{k=1}^{N_{*}} D_{+}^{k}\Big) .
\end{aligned}
\end{equation}

{\it Step 1.2. Estimate for $\na z_\pm^{(\beta)}$ with $|\beta|=N_*+1$.}
Applying the div-curl lemma \ref{lem:div-cur}  to $z_\pm^{(N_*+1)}$ with $\lambda=\lambda_\pm=\langle w_{\mp}^2\rangle\left(\log \langle w_{\mp}\rangle\right)^4$ and $F(L)=\f{1}{L\log L}$, we have
\beno
\begin{aligned}
&\mu\|\langle w_{\mp}\rangle\left(\log \langle w_{\mp}\rangle\right)^2\nabla z_{\pm}^{(N_*+1)}\|_{L^2_x}^2 
\lesssim\mu\|\langle w_{\mp}\rangle\left(\log \langle w_{\mp}\rangle\right)^2j_{\pm}^{(N_*+1)}\|_{L^2_x}^2 \\ 
&\qquad\qquad+\mu\|\left(\log \langle w_{\mp}\rangle\right)^2z_{\pm}^{(N_*+1)}\|_{L^2_x}^2
+\f{\mu\e}{L\log L}\|\langle w_{\mp}\rangle\left(\log \langle w_{\mp}\rangle\right)^2 z_{\pm}^{(N_*+1)}\|_{L^{2}_x}^{2}\\ 
&\qquad\qquad+\underbrace{\f{\mu\e}{L\log L}\int_{\Sigma_t}\langle w_{\mp}\rangle^2\left(\log \langle w_{\mp}\rangle\right)^4|\nabla^2 z_{\pm}^{(N_*+1)}|^2dx}_{\lesssim \frac{\mu\varepsilon}{\log L}E_{\pm}^{N_{*}+2}\quad\text{by the fact that } \langle w_{\mp}\rangle\lesssim L},
\end{aligned}
\eeno
which  implies
\beq\label{eq:h20}
\mu E_\pm^{N_*+1}\lesssim\sup_{0\leq t\leq T^*}\mu\|\langle w_{\mp}\rangle\left(\log \langle w_{\mp}\rangle\right)^2j_{\pm}^{(N_*+1)}\|_{L^2_x}^2+\mu E_\pm^{N_*}+\f{\mu\e}{\log L}E_\pm^{N_*+2}.
\eeq

Similar derivation as \eqref{eq:h18} and \eqref{eq:h20} yields
\beq\label{eq:h21}
\mu D_\pm^{N_*+1}\lesssim\mu^2\|\langle w_{\mp}\rangle\left(\log \langle w_{\mp}\rangle\right)^2\nabla j_{\pm}^{(N_*+1)}\|_{L^2_{T^*}(L^2_x)}^2+\mu D_\pm^{N_*}+\f{\mu\e}{\log L}D_\pm^{N_*+2}.
\eeq 

Noticing that the first terms on the RHS of \eqref{eq:h20} and \eqref{eq:h21} are the summations for $|\beta|=N_*+1$ from the  LHS of \eqref{eq:h17} evaluated at  $t=T^*$, we combine them with the estimates \eqref{eq:h18} and \eqref{eq:h19}  to obtain
\begin{equation}\label{eq:h21a}
\begin{aligned}
&\mu\bigl(E_{\pm}^{N_{*}+1}+  D_{\pm}^{N_{*}+1}\bigr)
\lesssim\mu\bigl(E_{\pm}^{N_{*}+1}(0)+  E_{\pm}^{N_{*}}+  D_{\pm}^{N_{*}}\bigl)+ \frac{\mu\varepsilon}{\log L}\bigl(E_{\pm}^{N_{*}+2}+ D_{\pm}^{N_{*}+2}\bigr)\\
&\qquad\qquad\qquad + \Bigl( \sum_{l=0}^{N_{*}}\left(E_{\mp}^{l}\right)^{\frac{1}{2}}+ \Big(\sum_{l=1}^{N_{*}} D_{\mp}^l\Big)^{\frac{1}{2}}\Bigr)\cdot \Big( \sum_{l=0}^{N_{*}} E_{\pm}^{l}+ \sum_{l=1}^{N_{*}} D_{\pm}^{l}\Big).
\end{aligned} 
\end{equation}

{\bf Step 2. Estimate for $\na z_\pm^{(\beta)}$ with $|\beta|=N_*+2$.} Similar derivation as \eqref{eq:h21a} yields
\beno
\begin{aligned}
\f{\mu}{\log L}\bigl(E_{\pm}^{N_{*}+2}+  D_{\pm}^{N_{*}+2}\bigr)
&\lesssim\f{\mu}{\log L}\bigl(E_{\pm}^{N_{*}+2}(0)+  E_{\pm}^{N_{*}+1}+  D_{\pm}^{N_{*}+1}\bigl)+ \frac{\mu\varepsilon}{(\log L)^2}\bigl(E_{\pm}^{N_{*}+3}+ D_{\pm}^{N_{*}+3}\bigr)\\
&\quad + \Bigl( \sum_{l=0}^{N_{*}}\left(E_{\mp}^{l}\right)^{\frac{1}{2}}+ \Big(\f{1}{\log L}\sum_{l=1}^{N_{*}+1} D_{\mp}^l\Big)^{\frac{1}{2}}\Bigr)\cdot \Big( \sum_{l=0}^{N_{*}} E_{\pm}^{l}+\f{1}{\log L}\sum_{l=1}^{N_{*}+1} D_{\pm}^{l}\Big),
\end{aligned} 
\eeno
which along with the fact that $\f{1}{\log L}\leq\mu$ implies
\beq\label{eq:h21b}
\begin{aligned}
\f{\mu}{\log L}\bigl(E_{\pm}^{N_*+2}+D_{\pm}^{N_*+2}\bigr)
&\lesssim\f{\mu}{\log L}\bigl(E_{\pm}^{N_*+2}(0)+E_{\pm}^{N_*+1}+D_{\pm}^{N_*+1}\bigl)+ \frac{\mu\varepsilon}{(\log L)^2}\bigl(E_{\pm}^{N_*+3}+ D_{\pm}^{N_*+3}\bigr)\\ 
&\quad+\Bigl(\sum_{l=0}^{N_*}\left(E_{\mp}^l\right)^{\frac{1}{2}}+ \Bigl(\mu\sum_{l=1}^{N_{*}+1} D_{\mp}^l\Bigr)^{\frac{1}{2}}\Bigr)\cdot \Bigl(\sum_{l=0}^{N_*} E_{\pm}^l+\mu\sum_{l=1}^{N_*+1} D_{\pm}^{l}\Bigr).
\end{aligned}\eeq

{\bf Step 3. Estimate for $\na z_\pm^{(\beta)}$ with $|\beta|=N_*+3$.} Applying \eqref{eq:l36} to \eqref{eq:h1} for $|\beta|= N_{*}+3$ with the weight functions $\lambda_{\pm}=\f{\mu}{(\log L)^2}\left(\log \langle w_{\mp}\rangle\right)^4$ and $F(L)=\f{1}{L(\log L)^2}$  yields
\beq\label{eq:h21c}
\begin{aligned}
&\f{\mu}{(\log L)^2}\Bigl(\sup\limits_{0\leq \tau \leq t}\|\left(\log \langle w_{\mp}\rangle\right)^2j_{\pm}^{(\beta)}\|_{L^2(\Sigma_\tau)}^2+ \mu\|\left(\log \langle w_{\mp}\rangle\right)^2\na j_{\pm}^{(\beta)}\|_{L^2_t(L^2_x)}^2\Bigr)\\
& \lesssim\f{\mu}{(\log L)^2}\bigl(E_{\pm}^{N_{*}+3}(0)+D_\pm^{N_*+2}\bigr)+\underbrace{\f{\e\mu^2}{L(\log L)^4}\int_{0}^{t} \int_{\Sigma_{\tau}}\left(\log \langle w_{\mp}\rangle\right)^4|\na^2j_{\pm}^{(\beta)}|^2 dx d\tau}_{\wt{J}_\pm}\\
&\qquad  +\underbrace{\f{\mu}{(\log L)^2}\int_{0}^{t} \int_{\Sigma_{\tau}}\left(\log \langle w_{\mp}\rangle\right)^4|j_{\pm}^{(\beta)}|\cdot|\rho_{\pm}^{(\beta)}| dx d\tau}_{\text{nonlinear interaction term } \wt{I}_{\pm}}.
\end{aligned}
\eeq

For $\wt J_\pm$, using the fact that $\log \langle w_{\mp}\rangle\lesssim\log L$, we derive 
\beq\label{eq:h21d}
\wt J_\pm\lesssim \f{\e\mu^2}{L}\int_{0}^{t} \int_{\Sigma_{\tau}}|\na^2j_{\pm}^{(\beta)}|^2 dx d\tau\lesssim\f{\mu\e}{L}\cdot\mu\int_{0}^{T_*} \int_{\Sigma_{\tau}}|\na^2z_{\pm}^{(N_*+4)}|^2 dx d\tau=\f{\mu\e}{L} D_\pm^{N_*+4}.
\eeq

Similar argument as \eqref{eq:h19} for $I_\pm$ gives rise to 
\begin{equation}\label{eq:h21e}
\begin{aligned}
\wt I_\pm &\lesssim \bigg( \sum_{l=0}^{N_*}\left(E_\mp^{l}\right)^{\frac{1}{2}}+ \Big(\frac{1}{\left(\log L\right)^2}\sum_{k=1}^{N_{*}+2} D_\mp^k\Big)^{\frac{1}{2}}\bigg)\cdot\Big( \sum_{l=0}^{N_*} E_\pm^l+ \frac{1}{\left(\log L\right)^2}\sum_{k=1}^{N_{*}+2} D_\pm^k\Big) \\
&\overset{\f{1}{\log L}\leq\mu}{\lesssim} \bigg(\sum_{l=0}^{N_*}\left(E_\mp^{l}\right)^{\frac{1}{2}}+ \Big(\frac{\mu}{\log L}\sum_{k=1}^{N_{*}+2} D_\mp^k\Big)^{\frac{1}{2}}\bigg) \Big(\sum_{l=0}^{N_*} E_\pm^l+ \frac{\mu}{\log L}\sum_{k=1}^{N_{*}+2} D_\pm^k\Big).
\end{aligned} 
\end{equation}

To derive the energy estimate for $\na z_\pm^{(\beta)}$ with $|\beta|=N_*+3$, we apply the div-curl lemma \ref{lem:div-cur}-with $\lambda=\lambda_\pm=\left(\log \langle w_{\mp}\rangle\right)^4$ and $F(L)=\frac{1}{L(\log L)^2}$-to $z_\pm^{(N_*+3)}$, using an argument similar to \eqref{eq:h20}, \eqref{eq:h21}, and \eqref{eq:h21d} to obtain
\beq\label{eq:h21f}\begin{aligned}
\f{\mu}{(\log L)^2} E_\pm^{N_*+3}&\lesssim\sup_{0\leq t\leq T^*}\f{\mu}{(\log L)^2}\|\left(\log \langle w_{\mp}\rangle\right)^2j_{\pm}^{(N_*+3)}\|_{L^2_x}^2+\f{\mu}{(\log L)^2} E_\pm^{N_*+2}+\f{\mu\e}{L} E_\pm^{N_*+4},\\ 
\f{\mu}{(\log L)^2}D_\pm^{N_*+3}
&\lesssim\f{\mu^2}{(\log L)^2}\|\left(\log \langle w_{\mp}\rangle\right)^2\nabla j_{\pm}^{(N_*+3)}\|_{L^2_{T^*}(L^2_x)}^2+\f{\mu}{(\log L)^2} D_\pm^{N_*+2}+\f{\mu\e}{L} D_\pm^{N_*+4}.
\end{aligned}\eeq

Thanks to \eqref{eq:h21d}, \eqref{eq:h21e} and \eqref{eq:h21f}, we deduce from \eqref{eq:h21c} that
\begin{equation}\label{eq:h21g}
\begin{aligned}
&\f{\mu}{(\log L)^2}\bigl(E_{\pm}^{N_{*}+3}+  D_{\pm}^{N_{*}+3}\bigr)
\lesssim\f{\mu}{(\log L)^2}\bigl(E_{\pm}^{N_{*}+3}(0)+  E_{\pm}^{N_{*}+2}+ D_{\pm}^{N_*+2}\bigl)+\f{\mu\e}{L}\bigl(E_{\pm}^{N_{*}+4}+ D_{\pm}^{N_{*}+4}\bigr)\\ 
&\qquad\qquad\qquad+ \Bigl( \sum_{l=0}^{N_*}\left(E_{\mp}^{l}\right)^{\frac{1}{2}}+ \Big(\f{\mu}{\log L}\sum_{l=1}^{N_{*}+2} D_{\mp}^l\Big)^{\frac{1}{2}}\Bigr)\cdot \Big( \sum_{l=0}^{N_{*}} E_{\pm}^{l}+ \f{\mu}{\log L}\sum_{l=1}^{N_*+2} D_{\pm}^{l}\Big) .
\end{aligned} 
\end{equation}

{\bf Step 4. Total energy estimate for $\na z_\pm^{(\beta)}$ with $N_*+1\leq |\beta|\leq N_*+3$.} Adding up \eqref{eq:h21a}, \eqref{eq:h21b} and \eqref{eq:h21g}, using the ansatz \eqref{eq:a3} and taking $\e\ll1$  and $L\gg 1$, we obtain the desired estimate \eqref{eq:h16}. This completes the proof of the proposition.
\end{proof}
\begin{remark}
In view of \eqref{eq:h15} and \eqref{eq:h16}, to close the {\it a priori} energy estimates for MHD, it remains to prove the energy estimate for $\na z_\pm^{(\beta)}$ with $|\beta|=N_*+4$.
\end{remark}

\subsection{Top-order parabolic estimates}
In this section, we derive the top-order energy estimate as follows:
\begin{proposition}\label{prop:h4}
Under the same assumptions as Proposition \ref{prop:h1}, we have 
\begin{equation}\label{eq:h22}
\begin{aligned}
&\qquad\f{\mu}{L}\bigl(E_{\pm}^{N_{*}+4}+D_{\pm}^{N_{*}+4}\bigr)\\ 
&\lesssim \f{\mu}{L} E_{\pm}^{N_{*}+4}(0)+ \Big(\sum_{l=0}^{N_*}\left(E_\mp^l\right)^{\f12}+ \Big(\f{1}{L}\sum_{l=1}^{N_*+3} D_\mp^l\Big)^{\f12}\Big)\cdot\Big(\sum_{l=0}^{N_*} E_{\pm}^{l}+ \f{1}{L}\sum_{l=1}^{N_{*}+3} D_\pm^l\Big) .
\end{aligned} 
\end{equation}
\end{proposition}
\begin{proof}
Taking the $L^2_x$-inner product of \eqref{eq:h1} with $\f{\mu}{L}\cdot j_\pm^{(\beta)}$ and integrating the resulting equality over $[0,t]$ yields
\beq\label{eq:h22a} 
\f{\mu}{L}\bigl(\|j_\pm^{(\beta)}\|_{L^2(\Sigma_t)}^2+\mu\|\na j_\pm^{(\beta)}\|_{L^2_t(L^2_x)}^2\bigr)\lesssim\f{\mu}{L}\|j_\pm^{(\beta)}\|_{L^2(\Sigma_0)}^2
+\underbrace{\f{\mu}{L}\int_0^t\int_{\Sigma_\tau}|\r_\pm^{(\beta)}| |j_\pm^{(\beta)}|dxd\tau}_{I_\pm}.
\eeq 
Here we remark that the integration by parts for the diffusion term produces no boundary terms due to the $2L$-periodicity of $j_\pm^{(\beta)}$.

Now we estimate the last term in \eqref{eq:h22a} for $|\beta|=N_{*}+4$. Since \eqref{eq:h2} implies
\beno
|\rho_{+}^{(\beta)}|\lesssim\sum_{k\leq|\beta|}|\nabla z_{-}^{(k)}|\cdot|\nabla z_{+}^{(|\beta|-k)}|=\sum_{k\leq N_*+4}|\nabla z_{-}^{(k)}|\cdot|\nabla z_{+}^{(N_*+4-k)}|,
\eeno
we have
\beno
I_+\lesssim\bigg(\underbrace{\sum_{k\leq N_{*}-2}}_{I_1}+\underbrace{\sum_{k=N_{*}-1}^{N_{*}+4}}_{I_2}\bigg) \frac{\mu}{L}\int_{0}^{t} \int_{\Sigma_{\tau}}|j_+^{(\beta)}|\cdot|\nabla z_-^{(k)}|\cdot|\nabla z_+^{(N_{*}+4-k)}|\,  dx d\tau. 
\eeno
For $I_1$, since $k\leq N_{*}-2$, by Sobolev inequality , we have
\begin{equation*}
\begin{aligned}
I_1&\leq\f{1}{L}\sum_{k\leq N_{*}-2}\bigl\|\nabla z_{-}^{(k)}\bigr\|_{L_{(t,x)}^{\infty}}\cdot\sqrt\mu\|j_{+}^{(\beta)}\|_{L^2_t(L^2_x)}\cdot\sqrt\mu\|\na z_{+}^{(N_{*}+4-k)}\|_{L^2_t(L^2_x)}\\
&\overset{|\beta|=N_*+4}{\lesssim} \sum_{l=0}^{N_*}\left(E_-^l\right)^{\frac{1}{2}}\cdot \frac{1}{L}\sum_{k=1}^{N_*+3} D_+^k.
\end{aligned}
\end{equation*}
While for $I_2$, since $N_{*}+4-k\leq 5 \leq N_{*}-2$, by Sobolev inequality,  we get
\begin{equation*}
\begin{aligned}
 I_2
& \leq \sum_{k=N_{*}-1}^{N_{*}+4}\sqrt{\frac{\mu}{L}}\|j_{+}^{(\beta)}\|_{L_t^2(L_x^2)}\cdot \sqrt{\frac{\mu}{L}}\|\nabla z_{-}^{(k)}\|_{L_t^2(L_x^2)}\cdot\|\nabla z_{+}^{(N_{*}+4-k)}\|_{L_{(t,x)}^\infty}\\
&\overset{|\beta|=N_*+4}{\lesssim}\Big(\frac{1}{L}\sum_{k=1}^{N_{*}+3} D_-^{k}\Big)^{\frac{1}{2}} \Big( \sum_{l=0}^{N_*}E_+^l+ \frac{1}{L}\sum_{k=1}^{N_{*}+3} D_+^k\Big) .
\end{aligned}
\end{equation*}
Finally, we conclude that
\begin{equation}\label{eq:h23}
\begin{aligned}
I_+\lesssim \Big(\sum_{l=0}^{N_*}\left(E_-^l\right)^{\frac{1}{2}}+\Big(\frac{1}{L}\sum_{k=1}^{N_{*}+3} D_{-}^{k}\Big)^{\frac{1}{2}}\Big)\cdot \Big(\sum_{l=0}^{N_*} E_+^l+ \frac{1}{L}\sum_{k=1}^{N_*+3} D_+^k\Big).
\end{aligned} 
\end{equation}
Similar estimate also holds for $I_-$.

Since $\dive\,z_\pm=0$ gives rise to $\|\na z_\pm^{(\beta)}\|_{L^2_x}\sim\|j_\pm^{(\beta)}\|_{L^2_x}$, using \eqref{eq:h23} for $I_\pm$, we deduce the desired estimate \eqref{eq:h22} from \eqref{eq:h22a}.
 This completes the proof of Proposition \ref{prop:h4}.
\end{proof}

The estimates \eqref{eq:h15}, \eqref{eq:h16}, and \eqref{eq:h22}, together with the ansatz \eqref{eq:a3}, lead to a closed a priori estimate for the total energy. We formalize this result in the following proposition.
\begin{proposition}\label{prop:h4*}
Under the assumptions of Proposition \ref{prop:h1}, we have
\begin{equation}\label{eq:h23a}
\sup_{0\leq t\leq T^*}\mathcal{E}_\pm^{w}(t)+\mathcal{D}_\pm^{w}(T^*)+ F_{\pm}+ \sum_{k=0}^{N_{*}}F_{\pm}^k\lesssim\mathcal{E}_\pm^w(0),
\end{equation}
where $\mathcal{E}_{\pm}^{w}(t)$ and $\mathcal{D}_{\pm}^{w}(t)$ are defined in Subsection 1.4.1.
\end{proposition}

\subsection{Proof of Theorem \ref{thm:a1}}
In this section, we shall close the continuity argument (from Section 2.1) and hence complete the proof of Theorem \ref{thm:a1}. 

\begin{proof}[Proof of Theorem \ref{thm:a1}] We divide the proof into three steps.

\textbf{Step 1. Improvement of the ansatz \eqref{eq:a3} and \eqref{eq:a2}.}
Assume $L \geq e^{1/\mu}$ and let $R > 0$ be a constant chosen sufficiently large, independent of $\varepsilon$, $\mu$, and $L$. Then, by Proposition \ref{prop:h4*}, for $\varepsilon$ sufficiently small, there exists a constant $C_1 > 0$ such that
\begin{equation*}
\sup_{0\leq t\leq T^*}\mathcal{E}_\pm^{w}(t)+\mathcal{D}_\pm^{w}(T^*)+ F_{\pm}+ \sum_{k=0}^{N_{*}}F_{\pm}^k\leq C_{1}\mathcal{E}^{w}(0) \leq C_{1}\varepsilon^2,
\end{equation*}
where $\mathcal{E}^{w}(0)=\mathcal{E}_+^{w}(0)+\mathcal{E}_-^{w}(0)$.
This improves the ansatz \eqref{eq:a3}.

By Lemma \ref{lem:e3} (weighted Sobolev inequality) and \eqref{eq:a3}, for  sufficiently small $\varepsilon$, we have
\begin{equation*}
\begin{aligned}
\left\|z_{\pm}\right\|_{L^{\infty}_{(t,x)}} \leq C\frac{\left(E_{\pm}+E_{\pm}^0+ E_{\pm}^1\right)^{\frac{1}{2}}}{\left(\log R\right)^2} \leq \frac{C\sqrt{2C_{1}}\varepsilon}{\left(\log R\right)^2} \leq \frac{1}{4},\quad\forall\,t\in[0,T^*].
\end{aligned} 
\end{equation*}
This improves the ansatz \eqref{eq:a2}.

\medskip

\textbf{Step 2. Improvement of the ansatz \eqref{eq:a1} (or equivalently \eqref{eq:e4}).}  Since $\psi_{\pm}(t,y)$ are the flows generated by $Z_{\pm}$, we have (see \eqref{eq:e2})
\begin{equation}\label{eq:i1}
\psi_{\pm}(t,y)= y\pm tB_{0}+ \int_{0}^{t} z_{\pm}\big(\tau, \psi_{\pm}(\tau,y)\big) d\tau,\quad\forall\, y\in\R^3. 
\end{equation}

By symmetry, we only improve the ansatz \eqref{eq:e4} for $\psi_{+}$. In view of \eqref{eq:i1}, we have
\begin{equation}\label{eq:i2}
\frac{\partial \psi_{+}(t, y)}{\partial y}=I+\int_{0}^{t}\left(\nabla z_{+}\right)\big(\tau, \psi_{+}(\tau, y)\big) \frac{\partial \psi_{+}(\tau, y)}{\partial y} d\tau. 
\end{equation}

{\it Step 2.1. Improvement for the term $\frac{\partial \psi_{+}(t, y)}{\partial y}-I$. }
From \eqref{eq:i2}, we deduce that
\begin{equation}\label{eq:i3}
\begin{aligned}
\left|\frac{\partial \psi_{+}(t, y)}{\partial y}-I\right| &\leq \int_{0}^{t}\left|\left(\nabla z_{+}\right)\big(\tau, \psi_{+}(\tau, y)\big)\right|\cdot \left|\frac{\partial \psi_{+}(\tau, y)}{\partial y}-I\right| d\tau \\
& \quad\ +\int_{0}^{t}\left|\left(\nabla z_{+}\right)\big(\tau, \psi_{+}(\tau, y)\big)\right| d \tau=:G(t,y).
\end{aligned} 
\end{equation}
Since 
\begin{equation*}
\begin{aligned}
\frac{d}{dt} G(t,y) & =\left|\left(\nabla z_{+}\right)\big(t, \psi_{+}(t, y)\big)\right|\cdot\left|\frac{\partial \psi_{+}(t, y)}{\partial y}-I\right|+\left|\left(\nabla z_{+}\right)\big(t, \psi_{+}(t, y)\big)\right| \\
&\overset{\eqref{eq:i3}}{\leq}\left|\left(\nabla z_{+}\right)\big(t, \psi_{+}(t, y)\big)\right|\cdot G(t, y)+\left|\left(\nabla z_{+}\right)\big(t, \psi_{+}(t, y)\big)\right|,
\end{aligned}
\end{equation*}
 we apply Gronwall's inequality to conclude that
\begin{equation}
\begin{aligned}
G(t, y) & \leq \exp \left(\int_{0}^{t}\left|\left(\nabla z_{+}\right)\big(\tau, \psi_{+}(\tau, y)\big)\right| d\tau\right)\cdot \underbrace{\int_{0}^{t}\left|\left(\nabla z_{+}\right)\big(\tau, \psi_{+}(\tau, y)\big)\right| d\tau}_{A(t,y)} .
\end{aligned} \label{eq:i4}
\end{equation}

We now derive the bound of $A(t,y)$. First, for any $(t,y)\in [0, T^*]\times \mathbb{R}^3$, by the ansatz $\|z_+\|_{L^\infty_{(t,x)}}\leq\f12$, we deduce from \eqref{eq:i1}  that
\begin{equation}\label{eq:i5}
\begin{aligned}
&|\psi_{+}^i(t,y)-y_i|\leq\f32t\leq\f32\log L\overset{L\gg1}{<}\f{L}{20},\quad i=1,2,3.
\end{aligned} 
\end{equation}
Then we split the estimate of $A(t,y)$ into the following two cases:

\smallskip 

{\bf Case 1:} $y-2\vv kL\in [-\frac{3L}{4},\frac{3L}{4}]^3$ with $\vv k=(k_1,k_2,k_3)\in\mathbb{Z}^3$.

In this case, using \eqref{eq:i5}, we have for $i=1,2,3$
\begin{equation}\label{eq:i7}
\left|\psi_{+}^{i}(\tau,y)-2k_iL\right|\leq \frac{5L}{6},\quad\text{i.e.,}\quad \psi_{+}(\tau,y)-2\vv k L\in Q_L.
\end{equation}
By the $2L$-periodicity of $z_{+}$, Lemma \ref{lem:e3} (weighted Sobolev inequality), and the ansatz \eqref{eq:a3}, we obtain
\begin{equation}\label{eq:i8}
\begin{aligned}
\left|\nabla z_{+}\big(\tau, \psi_{+}(\tau, y)\big)\right|
& = \left|\nabla z_{+}\big(\tau, \psi_{+}(\tau, y)-2\vv k L\big)\right|\lesssim \frac{\varepsilon}{\langle w_{-}\rangle\left(\log \langle w_{-}\rangle\right)^2}\bigg|_{x=\psi_+(\tau, y)-2\vv k L}\\
& \lesssim \frac{\varepsilon}{(R^2+|u_{-}|^2)^{\frac{1}{2}}\big(\log (R^{2}+|u_{-}|^2)^{\frac{1}{2}}\big)^{2}}\bigg|_{x=\psi_{+}(\tau, y)-2\vv k L}.
\end{aligned} 
\end{equation}

We shall switch the variable $\tau$ to $u_{-}$ in the integral $A$. By virtue of \eqref{eq:i7} and Lemma \ref{lem:e2}, we have
\begin{equation*}
\begin{aligned}
&\frac{d}{d \tau} u_{-}\big(\tau, \psi_{+}(\tau, y)-2\vv k L\big)=\frac{d}{d \tau} x_{3}^{-}\big(\tau, \psi_{+}(\tau, y)-2\vv k L\big)\\
&\qquad = (\p_\tau x_{3}^{-})\big(\tau, \psi_{+}(\tau, y)-2\vv k L\big)  + \p_\tau \psi_{+}(\tau, y)\cdot(\nabla x_{3}^{-})\big(\tau, \psi_{+}(\tau, y)-2\vv k L\big)\\
&\qquad =\bigl(\p_\tau x_{3}^{-}+Z_{+}\cdot\na_xx_{3}^{-}\bigr)\big|_{x=\psi_{+}(\tau, y)-2\vv k L},
\end{aligned}
\end{equation*}
which along with $\p_tx_3^-+Z_-\cdot\na_xx_3^-=0$ implies
\beno\begin{aligned}
&\frac{d}{d \tau} u_{-}\big(\tau, \psi_{+}(\tau, y)-2\vv k L\big)=(Z_+-Z_-)\cdot\na x_3^-\big|_{x=\psi_{+}(\tau, y)-2\vv k L}\\ 
&\qquad=\bigl(2\p_3 x_{3}^{-}+(z_+-z_-)\cdot\na x_3^-\bigr)\big|_{x=\psi_{+}(\tau, y)-2\vv k L}\geq 2-6C_0\e,
\end{aligned}\eeno
where we used the ansatz \eqref{eq:a1} and \eqref{eq:a2} in the last inequality.
Taking $\e$ sufficiently small, we obtain 
\beno 
\frac{d}{d \tau} u_{-}\big(\tau, \psi_{+}(\tau, y)-2\vv k L\big)\geq 1.
\eeno 
Hence, using \eqref{eq:i8} and the above estimates, by changing the variable from $\tau$ to $u_-$, we derive
\begin{equation}\label{eq:i9}
\begin{aligned}
A(t,y) & \lesssim \e\int_{0}^{t} \frac{1}{(R^2+|u_{-}|^2)^{\frac{1}{2}}\big(\log (R^{2}+|u_{-}|^2)^{\frac{1}{2}}\big)^{2}}\bigg|_{x=\psi_{+}(\tau, y)-2\vv k L} d\tau \\
& \lesssim\e\int_{0}^{\infty} \frac{1}{\big(R^{2}+\left|u_{-}\right|^{2}\big)^{\frac{1}{2}}\big(\log \big(R^{2}+\left|u_{-}\right|^{2}\big)^{\frac{1}{2}}\big)^{2}} du_{-} \lesssim \varepsilon.
\end{aligned} 
\end{equation}

{\bf Case 2:} $y-2\vv k L\in [-L,L)^3 \setminus [-\frac{3L}{4},\frac{3L}{4}]^3$ with $\vv k\in\mathbb{Z}^3$. 

In this case, there exists $i\in\{1,2,3\}$ such that $\frac{3L}{4}< \left|y_{i}-2k_iL\right|\leq L$. It then follows from  \eqref{eq:i5} that  
\begin{equation*}
\frac{L}{2}<\left|\psi_{+}^{i}(\tau,y)-2k_iL\right|< \frac{3L}{2}, 
\end{equation*}
which implies 
\beno\begin{aligned}
&\f{L}{2}<|\psi_{+}^{i}(\tau,y)-2k_iL|<L
\quad\text{or}\quad -L\leq\psi_{+}^{i}(\tau,y)-2(k_i+1)L<-\f{L}{2}\\ 
&\text{or}\quad \f{L}{2}<\psi_{+}^{i}(\tau,y)-2(k_i-1)L\leq L,\quad\forall\,\tau\leq t.
\end{aligned}\eeno
Thus, by the $2L$-periodicity of $z_{+}$, weighted Sobolev inequality (Lemma \ref{lem:e3}), Lemma \ref{lem:e2} and  \eqref{eq:a3}, we obtain
\beno
\left|\nabla z_{+}\big(\tau, \psi_{+}(\tau, y)\big)\right| \lesssim\sup_{\f{L}{2}<|x_i|\leq L}\frac{\left(E_+^0+E_+^1+ E_+^2\right)^{\frac{1}{2}}}{\langle w_{-}\rangle\left(\log \langle w_{-}\rangle\right)^{2}} \lesssim \frac{\bigl(\mathcal{E}^w(0)\bigr)^{\f12}}{L\left(\log L\right)^2},
\eeno
which yields
\begin{equation}\label{eq:i10}
A(t,y) \lesssim\int_0^{T^*}\left|\nabla z_{+}\big(\tau, \psi_{+}(\tau, y)\big)\right| d\tau\overset{T^*=\log L}{\lesssim} \frac{\bigl(\mathcal{E}^w(0)\bigr)^{\f12}}{L\log L} \lesssim \varepsilon. 
\end{equation}

Thanks to \eqref{eq:i9} and \eqref{eq:i10}, we deduce from \eqref{eq:i3} and \eqref{eq:i4} that 
\begin{equation}\label{eq:i12}
\left|\frac{\partial \psi_{+}(t, y)}{\partial y}-I\right| \leq G(t,y)\leq e^{A(t,y)} A(t,y) \leq C_{0}^{\prime \prime}\bigl(\mathcal{E}^w(0)\bigr)^{\f12}\leq C_0'\e\leq\f{1}{20}, 
\end{equation}
where we took $C_0'\geq C_0''$. Then we improve the first part of \eqref{eq:e4}.

{\it Step 2.2. Improvement for the term $\na_y\Bigl(\frac{\partial \psi_{+}(t, y)}{\partial y}\Bigr)$. }
Applying $\partial_{k}$ (with $k=1,2,3$) to both sides of \eqref{eq:i2} leads to
\begin{equation*}
\begin{aligned}
\partial_{k}\bigg(\frac{\partial \psi_{+}(t, y)}{\partial y}\bigg)& = \int_{0}^{t}\left(\nabla z_{+}\right)\big(\tau, \psi_{+}(\tau, y)\big) \partial_{k}\Bigl(\frac{\partial \psi_{+}(\tau, y)}{\partial y}\Bigr) d\tau \\
& \quad\ +\int_{0}^{t} \partial_{k}\Big(\left(\nabla z_{+}\right)\big(\tau, \psi_{+}(\tau, y)\big)\Big)\Bigl(\frac{\partial \psi_{+}(\tau, y)}{\partial y}\Bigr) d\tau.
\end{aligned}
\end{equation*}
By virtue of Gronwall's inequality, we get
\begin{equation*}
\begin{aligned}
\Bigl|\partial_{k}\Bigl(\frac{\partial \psi_{+}(t, y)}{\partial y}\Bigr)\Bigr|& \leq \exp \Bigl(\int_{0}^{t}\left|\left(\nabla z_{+}\right)\big(\tau, \psi_{+}(\tau, y)\big)\right| d\tau\Bigr)\cdot\int_{0}^{t}\left|\left(\nabla^{2} z_{+}\right)\big(\tau, \psi_{+}(\tau, y)\big)\right|\cdot\Bigl|\frac{\partial \psi_{+}(\tau, y)}{\partial y}\Bigr|^{2} d\tau \\
& \stackrel{\eqref{eq:i12}}{\leq} 2\exp \Bigl(\underbrace{\int_{0}^{t}\left|\left(\nabla z_{+}\right)\big(\tau, \psi_{+}(\tau, y)\big)\right| d\tau}_{A}\Bigr)\cdot\underbrace{\int_{0}^{t}\left|\left(\nabla^2 z_{+}\right)\big(\tau, \psi_{+}(\tau, y)\big)\right| d\tau}_{B} .
\end{aligned}
\end{equation*}
Similar derivation as \eqref{eq:i10} for $A$ gives rise to $B \lesssim \varepsilon$, which together with \eqref{eq:i10} implies 
\begin{equation}\label{eq:i13}
\Bigl|\nabla_y \Bigl(\frac{\partial \psi_{+}(t,y)}{\partial y}\Bigr)\Bigr| \leq  C_1^{\prime\prime} \varepsilon\leq C_0'\e,
\end{equation}
where we took $C_{0}^{\prime} \geq\max\{C_0'',\,C_1''\}$. Thereby the second part of \eqref{eq:e4} is improved.

\medskip 

{\bf Step 3. Existence and uniquess.} The short-time existence for smooth data is standard. The existence and uniquess on time interval $[0,\log L]$ follow directly from the {\it a priori } energy estimate \eqref{eq:h23a} and the continuity argument. Thus, we complete the proof of Theorem \ref{thm:a1}.
\end{proof}

\section{Proof of Theorem \ref{thm:d1}}
Based on the weighted a priori estimates for $z_{\pm}$ obtained in Theorem \ref{thm:a1}, we prove the global nonlinear stability (Theorem \ref{thm:d1}) in this section. The proof relies heavily on the decay mechanism of the local solutions on the time interval $[0, \log L]$ and the global existence theory for parabolic systems with small initial data.

\subsection{The linear and nonlinear decomposition of solutions.}
We consider the solutions $(z_+, z_-)$ to the MHD system on the time interval $[0, \log L]$ obtained from Theorem \ref{thm:a1} in this and the following subsections. The interval $[0, \log L]$  is partitioned into $n_0$ subintervals as 
\beno
[0, \log L] = \bigcup_{n=1}^{n_0}[t_{n-1}, t_n], \quad\text{where}\quad t_n=nT,\,\ 0\leq n\leq n_0,\quad T =\f{\log L}{n_0}.
\eeno
The values of $n_0$ and the subinterval length $T$ will be determined in the course of the proof.

Now, we decompose $z_\pm$ on $\left[t_{n-1}, t_{n}\right]$ into the {\it linear} and {\it nonlinear} parts as follows:
\beno
z_{\pm} = z_{\pm,n,\text{(lin)}}+ z_{\pm,n,\text{(non)}},
\eeno
where the linear part $z_{\pm,n,\text{(lin)}}$  satisfies
\begin{equation}\label{eq:d1}
\left\{\begin{array}{l}
\partial_{t} z_{\pm,n,\text{(lin)}}+ Z_{\mp}\cdot\nabla z_{\pm,n,\text{(lin)}}- \mu\Delta z_{\pm,n,\text{(lin)}}= 0,\quad t\in(t_{n-1},t_n],\,\,x\in Q_L,\\ 
\left.z_{\pm,n,\text{(lin)}}\right|_{t=t_{n-1}}= z_{\pm}(t_{n-1},x),
\end{array}\right. 
\end{equation}
and the nonlinear part $z_{\pm,n,\text{(non)}}$ satisfies
\begin{equation}\label{eq:d2}
\left\{\begin{array}{l}
\partial_{t} z_{\pm,n,\text{(non)}}+ Z_{\mp}\cdot\nabla z_{\pm,n,\text{(non)}}- \mu\Delta z_{\pm,n,\text{(non)}}= -\nabla p, \quad t\in(t_{n-1},t_n],\,\,x\in Q_L,\\ 
\operatorname{div} z_{\pm,n,\text{(non)}}= -\operatorname{div} z_{\pm,n,\text{(lin)}},\\
\left.z_{\pm,n,\text{(non)}}\right|_{t=t_{n-1}}= 0.
\end{array}\right. 
\end{equation}

We define the {\it weighted energies} $E^{(\alpha)}_{\pm,\text{(lin)}}(t)$ and $E^{(\alpha)}_{\pm,\text{(non)}}(t)$ for $z_{\pm,\text{(lin)}}$ and $z_{\pm,\text{(non)}}$, respectively, using the same weight functions as for $z_\pm$. On the interval $(t_{n-1}, t_n]$, these functions are defined by: 
\beno 
z_{\pm,\text{(lin)}}:=z_{\pm,n,\text{(lin)}}\quad\text{and}\quad  z_{\pm,\text{(non)}}:=z_{\pm,n,\text{(non)}}.
\eeno
The weighted energies $E^{k}_{\pm,\text{(lin)}}(t)$ and $E^{k}_{\pm,\text{(non)}}(t)$ are defined analogously to $E^k_\pm(t)$ for $z_\pm$. We also define the corresponding {\it lowest-order energies} as $E_{\pm,\text{(lin)}}(t)$ and $E_{\pm,\text{(non)}}(t)$.

The {\it total weighted energy} at time $t$ for $z_{\pm,\text{(lin)}}$ is defined as follows:
\beno
\begin{aligned}
\mathcal{E}_{\pm,\text{(lin)}}^{w}(t) & := E_{\pm,\text{(lin)}}(t)+ \sum_{|\alpha|\leq N_{*}} E_{\pm,\text{(lin)}}^{(\alpha)}(t)+ \mu \sum_{|\alpha|=N_{*}+1} E_{\pm,\text{(lin)}}^{(\alpha)}(t)+ \frac{\mu}{\log L}\sum_{|\alpha|=N_{*}+2} E_{\pm,\text{(lin)}}^{(\alpha)}(t)\\
& \qquad  + \frac{\mu}{\left(\log L\right)^2}\sum_{|\alpha|=N_{*}+3} E_{\pm,\text{(lin)}}^{(\alpha)}(t)+\f{\mu}{L} \sum_{|\alpha|=N_{*}+4} E_{\pm,\text{(lin)}}^{(\alpha)}(t).
\end{aligned} 
\eeno
Similarly, we can define $\mathcal{E}_{\pm,\text{(non)}}^{w}(t)$ for $z_{\pm,\text{(non)}}$ at time $t$.

We shall also define the lowest- and higher-order {\it weighted diffusions} for $z_{\pm,n,\text{(lin)}}$ as follows:
\beno\begin{aligned}
&D_{\pm,n,\text{(lin)}}(t):= \mu\int_{t_{n-1}}^{t} \int_{\Sigma_{\tau}} \left(\log \langle w_{\mp}\rangle\right)^{4}\left|\nabla z_{\pm,n,\text{(lin)}}\right|^{2} dx d\tau,\quad\forall\,t\in(t_{n-1},\,t_n],\\ 
&D^{(\al)}_{\pm,n,\text{(lin)}}(t):=\mu\int_{t_{n-1}}^{t} \int_{\Sigma_{\tau}}\lambda_\pm\left|\nabla^2 z^{(\al)}_{\pm,n,\text{(lin)}}\right|^{2} dx d\tau, \ \text{ for }0\leq |\al|\leq N_*+4,
\end{aligned}\eeno
where the weights $\lambda_\pm$ are identical to those used in $D^{(\al)}_\pm(t)$ for $z_\pm$. Similarly, we can also define the weighted diffusions $D_{\pm,n,\text{(non)}}(t)$ and $D^{(\al)}_{\pm,n,\text{(non)}}(t)$. The diffusions $D^k_{\pm,n,\text{(lin)}}(t)$ and $D^k_{\pm,n,\text{(non)}}(t)$ are defined analogously to $D^k_\pm(t)$. 

The {\it total weighted diffusion} over time interval $(t_{n-1},t]\subset(t_{n-1},t_n]$ for $z_{\pm,\text{(lin)}}$ is defined as follows:
\beno\begin{aligned}
\mathcal{D}^w_{\pm,n,\text{(lin)}}(t):&=D_{\pm,n,\text{(lin)}}(t)+ \sum_{k=0}^{N_{*}} D_{\pm,n,\text{(lin)}}^k(t)+\mu D_{\pm,n,\text{(lin)}}^{N_{*}+1}(t)\\
&\quad + \frac{\mu}{\log L} D_{\pm,n,\text{(lin)}}^{N_{*}+2}(t)+ \frac{\mu}{\left(\log L\right)^2} D_{\pm,n,\text{(lin)}}^{N_{*}+3}(t)+ \f{\mu}{L}D_{\pm,n,\text{(lin)}}^{N_{*}+4}(t).
\end{aligned}\eeno
Similarly, we can define $\mathcal{D}_{\pm,n,\text{(non)}}^{w}(t)$ for $z_{\pm,\text{(non)}}$ over time interval $(t_{n-1},t]\subset(t_{n-1},t_n]$.

The lowest- and higher-order {\it weighted energy fluxes} for $z_{\pm,n,\text{(lin)}}$  are defined by:
\beno
\begin{aligned}
F_{\pm,n,\text{(lin)}}(t) & = \sup_{|u_{\pm}|\leq \frac{L}{4}}\int_{C_{[t_{n-1},\,t],\,u_{\pm}}^{\pm}} \left(\log \langle w_{\mp}\rangle\right)^{4}\left|z_{\pm,n,\text{(lin)}}\right|^{2} d\sigma_{\pm},\quad\forall\,t\in(t_{n-1},\,t_n],\\
F_{\pm,n,\text{(lin)}}^{(\alpha)}(t) & = \sup_{|u_{\pm}|\leq \frac{L}{4}}\int_{C_{[t_{n-1},\,t],\,u_{\pm}}^{\pm}} \langle w_{\mp}\rangle^{2}\left(\log \langle w_{\mp}\rangle\right)^{4}\left|\nabla z_{\pm,n,\text{(lin)}}^{(\alpha)}\right|^{2} d\sigma_{\pm}, \ \text{ for }0\leq |\alpha| \leq N_*,
\end{aligned} 
\eeno
while the weighted fluxes for $z_{\pm,n,\text{(non)}}$  are defined by:
\beno
\begin{aligned}
F_{\pm,n,\text{(non)}}(t) & = \sup_{|u_{\pm}|\leq \frac{L}{4}}\int_{C_{[t_{n-1},\,t],\,u_{\pm}}^{\pm}} \left(\log \langle w_{\mp}\rangle\right)^{4}\left|z_{\pm,n,\text{(non)}}\right|^{2} d\sigma_{\pm},\quad\forall\,t\in(t_{n-1},\,t_n],\\
F_{\pm,n,\text{(non)}}^{(0)}(t) & = \sup_{|u_{\pm}|\leq \frac{L}{4}}\int_{C_{[t_{n-1},\,t],\,u_{\pm}}^{\pm}} \langle w_{\mp}\rangle^{2}\left(\log \langle w_{\mp}\rangle\right)^{4}\left|\nabla z_{\pm,n,\text{(non)}}\right|^{2} d\sigma_{\pm},\\
F_{\pm,n,\text{(non)}}^{(\alpha)}(t) & = \sup_{|u_{\pm}|\leq \frac{L}{4}}\int_{C_{[t_{n-1},\,t],\,u_{\pm}}^{\pm}} \langle w_{\mp}\rangle^{2}\left(\log \langle w_{\mp}\rangle\right)^{4}\left|j_{\pm,n,\text{(non)}}^{(\alpha)}\right|^{2} d\sigma_{\pm}, \ \text{ for }1\leq |\alpha| \leq N_*,
\end{aligned} 
\eeno
where $j_{\pm,n,\text{(non)}}=\operatorname{curl} z_{\pm,n,\text{(non)}}$. 
The fluxes $F^k_{\pm,n,\text{(lin)}}(t)$, $F^k_{\pm,n,\text{(non)}}(t)$ are defined analogously to $F^k_\pm(t)$. 

\smallskip

In addition to the notations above, we define the {\it total physical energy} for $z_{\pm,\text{(lin)}}$ at time $t$ as follows:
\beno\begin{aligned}
\mathcal{E}_{\pm,\text{(lin)}}(t)&:=\sum_{|\al|\leq N_*+1}\|z^{(\al)}_{\pm,\text{(lin)}}\|_{L^2(\Sigma_t)}^2 
+\mu\sum_{|\al|=N_*+1}\|\na z^{(\al)}_{\pm,\text{(lin)}}\|_{L^2(\Sigma_t)}^2,
\end{aligned}\eeno
where $z_{\pm,\text{(lin)}}(t):=z_{\pm,n,\text{(lin)}}(t)$ for $t\in(t_{n-1},t_n]$. 
The corresponding {\it total diffusion} over $(t_{n-1},t]\subset(t_{n-1},t_n]$ is defined by
\beno\begin{aligned}
\mathcal{D}_{\pm,n,\text{(lin)}}(t)&:=\mu\sum_{|\al|\leq N_*+1}\int_{t_{n-1}}^t\|\na z^{(\al)}_{\pm,n,\text{(lin)}}\|_{L^2(\Sigma_\tau)}^2d\tau+\mu^2\sum_{|\al|=N_*+1}\int_{t_{n-1}}^t\|\na^2 z^{(\al)}_{\pm,n,\text{(lin)}}\|_{L^2(\Sigma_\tau)}^2d\tau.
\end{aligned}\eeno
Similarly, we can define $\mathcal{E}_{\pm,\text{(non)}}(t)$ for $z_{\pm,\text{(non)}}$. While the total unweighted diffusion over $(t_{n-1},t]\subset(t_{n-1},t_n]$ for $z_{\pm,\text{(non)}}$ is defined by
\beno\begin{aligned}
\mathcal{D}_{\pm,n,\text{(non)}}(t):&=\mu\sum_{0\leq|\al|\leq 1}\int_{t_{n-1}}^t\|\na z^{(\al)}_{\pm,n,\text{(non)}}\|_{L^2(\Sigma_\tau)}^2d\tau+ \mu\sum_{1\leq|\al|\leq N_*}\int_{t_{n-1}}^t\|\na j^{(\al)}_{\pm,n,\text{(non)}}\|_{L^2(\Sigma_\tau)}^2d\tau\\
&\quad + \mu^2\sum_{|\al|=N_*+1}\int_{t_{n-1}}^t\|\na j^{(\al)}_{\pm,n,\text{(non)}}\|_{L^2(\Sigma_\tau)}^2d\tau.
\end{aligned}\eeno
The {\it total energy } for MHD system at time $t$ is defined by
\beno 
\mathcal{E}(t):=\mathcal{E}_+(t)+\mathcal{E}_-(t),
\eeno
where $\mathcal{E}_\pm(t)$ denote the total physical energy for $z_\pm$, defined in the same manner as $\mathcal{E}_{\pm,\text{(lin)}}(t)$.

\begin{lemma}\label{lem:d1}
For all $1\leq n\leq n_0$ and any $t\in(t_{n-1},t_n]$, the following estimates hold:
\begin{equation}\label{eq:d3}
\begin{aligned}
&\mathcal{E}_{\pm,\text{(lin)}}^{w}(t)+\mathcal{D}_{\pm,n,\text{(lin)}}^{w}(t)+ F_{\pm,n,\text{(lin)}}(t)+ \sum_{k=0}^{N_{*}} F_{\pm,n,\text{(lin)}}^{k}(t)\lesssim \mathcal{E}^{w}(0),\\ 
&\mathcal{E}_{\pm,\text{(non)}}^{w}(t)+\mathcal{D}_{\pm,n,\text{(non)}}^{w}(t)+ F_{\pm,n,\text{(non)}}(t)+ \sum_{k=0}^{N_{*}} F_{\pm,n,\text{(non)}}^k(t)\lesssim \mathcal{E}^{w}(0).
\end{aligned} 
\end{equation}
Moreover, for any $t\in(t_{n-1},t_n]$, we have
\beq\label{eq:d4}\begin{aligned}
&\mathcal{E}_{\pm,\text{(lin)}}(t)+\mathcal{D}_{\pm,n,\text{(lin)}}(t)+\sum_{|\al|\leq N_*+1}\sup_{|u_{\pm}|\leq \frac{L}{4}}\int_{C_{[t_{n-1},\,t],\,u_{\pm}}^{\pm}}\left|z^{(\al)}_{\pm,n,\text{(lin)}}\right|^{2} d\sigma_{\pm}\lesssim\mathcal{E}_{\pm,\text{(lin)}}(t_{n-1}),\\ 
&\mathcal{E}_{\pm,\text{(non)}}(t)+\mathcal{D}_{\pm,n,\text{(non)}}(t)+\sum_{|\al|=0}^1\sup_{|u_{\pm}|\leq \frac{L}{4}}\int_{C_{[t_{n-1},\,t],\,u_{\pm}}^{\pm}}\left|z^{(\al)}_{\pm,n,\text{(non)}}\right|^{2}d\sigma_{\pm}\\ 
&\qquad+\sum_{1\leq|\al|\leq N_*}\sup_{|u_{\pm}|\leq \frac{L}{4}}\int_{C_{[t_{n-1},\,t],\,u_{\pm}}^{\pm}}\left|j^{(\al)}_{\pm,n,\text{(non)}}\right|^{2} d\sigma_{\pm}\lesssim\mathcal{E}_{\pm,\text{(lin)}}(t)
+\bigl(\mathcal{E}^w(0)\bigr)^{\f12}\cdot\mathcal{E}(t_{n-1}),
\end{aligned}\eeq
and 
\beq\label{eq:d4a}
\mathcal{E}(t)
\lesssim\sum_{+,-}\mathcal{E}_{\pm,\text{(lin)}}(t)
+\big(\mathcal{E}^w(0)\big)^{\f12}\cdot\mathcal{E}(t_{n-1}).
\eeq
\end{lemma}
\begin{proof}
For the linear system \eqref{eq:d1}, treating $Z_{\mp}=\mp B_0+z_\mp$ as given divergence-free periodic vector fields and taking $t_{n-1}$ as the initial time, a derivation similar to the {\it a priori} estimates for $z_{\pm}$ in \eqref{eq:thm2} leads to 
\beno 
\mathcal{E}_{\pm,\text{(lin)}}^{w}(t)+\mathcal{D}_{\pm,n,\text{(lin)}}^{w}(t)+ F_{\pm,n,\text{(lin)}}(t)+ \sum_{k=0}^{N_{*}} F_{\pm,n,\text{(lin)}}^{k}(t)\lesssim \mathcal{E}^{w}(t_{n-1})\lesssim \mathcal{E}^{w}(0),
\eeno
where we used \eqref{eq:thm2} in the last inequality.  This establishes the first inquality in \eqref{eq:d3}. 
Similarly, using the $2L$-periodicity of  $z_{\pm,n,\text{(lin)}}$, we obtain the first inequality in \eqref{eq:d4} for the unweighted energy estimate.

\smallskip

By taking curl of the nonlinear system \eqref{eq:d2}, we obtain the following system for $j_{\pm,n,\text{(non)}}$:
\begin{equation}\label{eq:d5}
\left\{\begin{array}{l}
\partial_{t} j_{\pm,n,\text{(non)}}+ Z_{\mp}\cdot\nabla j_{\pm,n,\text{(non)}}- \mu\Delta j_{\pm,n,\text{(non)}}= -\nabla z_{\mp} \wedge \nabla z_{\pm,n,\text{(non)}}, \\ 
\left.j_{\pm,n,\text{(non)}}\right|_{t=t_{n-1}}= 0.
\end{array}\right.
\end{equation}
In view of the a priori estimates \eqref{eq:thm2} for $z_{\pm}$ and the obtained estimate in \eqref{eq:d3} for the linear part, by using the div-curl lemma \ref{lem:div-cur} and proceeding in the same manner as derivation of \eqref{eq:thm2}, we derive the second inequality in \eqref{eq:d3} for $z_{\pm,n,\text{(non)}}$.

For the second inequality in \eqref{eq:d4}, since $\dive\,z_{\pm,n,\text{(non)}}=-\dive z_{\pm,n,\text{(lin)}}$, we have the following div-curl formula:
\beno 
\|\na z_{\pm,n,\text{(non)}}\|_{L^2_x}^2\sim \|j_{\pm,n,\text{(non)}}\|_{L^2_x}^2+\|\dive z_{\pm,n,\text{(lin)}}\|_{L^2_x}^2.
\eeno
By this equivalence and an argument analogous to the derivation of \eqref{eq:thm2}, together with the first inequality in \eqref{eq:d4}, the second inequality in \eqref{eq:d4} follows. 

Thanks to \eqref{eq:d4} and  
\beno
\mathcal{E}(t)=\sum_{+,-}\mathcal{E}_\pm(t)
\lesssim\sum_{+,-}\bigl(\mathcal{E}_{\pm,\text{(lin)}}(t)
+\mathcal{E}_{\pm,\text{(non)}}(t)\bigr),
\eeno
we obtain \eqref{eq:d4a}.
This concludes the proof of the lemma.
\end{proof}

\subsection{Estimate on the total unweighted linear energy $\mathcal{E}_{\pm,\text{(lin)}}(t)$}\label{sec:3.2}
By symmetry, it suffices to bound $\mathcal{E}_{+,\text{(lin)}}$. We introduce a cut-off function $\theta(r)\in C_{c}^{\infty}\left(\mathbb{R}\right)$ satisfying $0\leq\theta(r)\leq 1,$
\begin{equation*}
\theta(r)=
\begin{cases}
& 1, \quad \text{for }|r|\leq 1,\\
& 0, \quad \text{for }|r|\geq 2,
\end{cases}\quad\text{and define}\quad  \theta_{L}(r):=\theta\Big(\frac{r}{L/8}\Big)=
\begin{cases}
& 1, \quad \text{for }|r|\leq \frac{L}{8},\\
& 0, \quad \text{for }|r|\geq \frac{L}{4}.
\end{cases}
\end{equation*}
Using this, we split $z_{+,n,\text{(lin)}}(t,x)$ as follows:
\beno
z_{+,n,\text{(lin)}}=z_{+,n,\text{(lin)}}\theta_{L}(|x^-|)+ z_{+,n,\text{(lin)}}\left(1-\theta_{L}(|x^-|)\right)=:z_n+r_n, 
\eeno
where
\beq\label{eq:d6}
\text{supp}\, z_n(t,\cdot)\subset\Bigl\{x\in Q_L\,|\,|x^-(t,x)|\leq\f{L}{4}\Bigr\},\quad\text{supp}\, r_n(t,\cdot)\subset\Bigl\{x\in Q_L\,|\,|x^-(t,x)|>\f{L}{8}\Bigr\}.
\eeq
Then, for $t_{n-1}< t\leq t_{n}$, we have
\begin{equation}\label{eq:d7}
\begin{aligned}
\mathcal{E}_{+,\text{(lin)}}(t) & \leq 2\Bigl(\left\|z_{n} \right\|_{L^{2}\left(\Sigma_{t}\right)}^2+ \sum_{|\alpha|=1}^{N_*+1}\left\|\partial^{\alpha} z_{n} \right\|_{L^{2}\left(\Sigma_{t}\right)}^2+ \mu\sum_{|\alpha|=N_*+2}\left\|\partial^{\alpha} z_{n} \right\|_{L^{2}\left(\Sigma_{t}\right)}^2 \Bigr)\\
& \qquad + 2\Bigl(\left\|r_{n} \right\|_{L^{2}\left(\Sigma_{t}\right)}^2+ \sum_{|\alpha|=1}^{N_*+1}\left\|\partial^{\alpha} r_{n} \right\|_{L^{2}\left(\Sigma_{t}\right)}^2+ \mu\sum_{|\alpha|=N_*+2}\left\|\partial^{\alpha} r_{n} \right\|_{L^{2}\left(\Sigma_{t}\right)}^2 \Bigr) .
\end{aligned} 
\end{equation}

{\bf Estimate for $r_n$.} Since \eqref{eq:d6} and Lemma \ref{lem:e2} imply
\beno 
\langle w_-\rangle(t,x)\gtrsim \langle\overline{w_-}\rangle(t,x)=\left(R^2+|x^-|^2\right)^\frac{1}{2}\gtrsim L, \quad
\forall\,x\in\text{supp}\, r_n(t,\cdot),
\eeno
using the first inequality in \eqref{eq:d3} for $z_{+,n,\text{(lin)}}$ and the bound $\big|\partial^{\alpha}\big(\theta_{L}(|x^-|)\big)\big|\lesssim \frac{1}{L}$, we obtain
\begin{equation}\label{eq:d8}
\left\|r_{n} \right\|_{L^{2}\left(\Sigma_{t}\right)}^2+ \sum_{|\alpha|=1}^{N_*+1}\left\|\partial^{\alpha} r_{n} \right\|_{L^{2}\left(\Sigma_{t}\right)}^2+ \mu\sum_{|\alpha|=N_*+2}\left\|\partial^{\alpha} r_{n} \right\|_{L^{2}\left(\Sigma_{t}\right)}^2\lesssim \frac{\mathcal{E}^{w}(0)}{\left(\log L\right)^{4}} . 
\end{equation}

{\bf Estimate for $z_n$.} By \eqref{eq:d1} and the definition of $z_n$, there holds
\begin{equation}\label{eq:d9}
\left\{\begin{array}{l}
\partial_{t} z_{n}+ Z_{-}\cdot\nabla z_{n}- \mu\Delta z_{n}= \rho_{n}, \\ 
\left.z_{n}\right|_{t=t_{n-1}}= \left.z_{+}\theta_{L}(|x^-|)\right|_{t=t_{n-1}},
\end{array}\right. 
\end{equation}
where 
\beno 
 \r_n=-2\mu\nabla\big(\theta_{L}(|x^-|)\big) \cdot\nabla z_{+, n,\text{(lin)}}-\mu\Delta\big(\theta_{L}(|x^-|)\big) z_{+, n,\text{(lin)}}.
\eeno
Applying $\partial^{\alpha}$ with $|\alpha|\geq 1$ to the first equation in \eqref{eq:d9}, we obtain
\begin{equation} \label{eq:d10}
\partial_{t} z_{n}^{(\alpha)}+ Z_{-}\cdot\nabla z_{n}^{(\alpha)}- \mu\Delta z_{n}^{(\alpha)}=-\left[\partial^{\alpha},\, z_{-}\cdot \nabla\right] z_{n} + \partial^{\alpha}\rho_{n} . 
\end{equation}

The general energy estimates for \eqref{eq:d9} and \eqref{eq:d10} yield 
\begin{equation} \label{eq:d11}
\begin{aligned}
& \frac{1}{2}\frac{d}{d t}\Bigl(\sum_{|\alpha|=0}^{N_*+1}\|z_{n}^{(\alpha)}\|_{L^{2}\left(\Sigma_{t}\right)}^2+ \mu\sum_{|\alpha|=N_*+2}\|z_{n}^{(\alpha)}\|_{L^{2}\left(\Sigma_{t}\right)}^2 \Bigr)\\
&\qquad + \mu\Bigl(\sum_{|\alpha|=0}^{N_*+1}\|\nabla z_{n}^{(\alpha)}\|_{L^{2}\left(\Sigma_{t}\right)}^2+ \mu\sum_{|\alpha|=N_*+2}\|\nabla z_{n}^{(\alpha)}\|_{L^{2}\left(\Sigma_{t}\right)}^2 \Bigr)\leq f_{n}(t) ,
\end{aligned}
\end{equation}
where 
\begin{equation}\label{eq:d12}
f_{n}(t)= \left|\int_{\Sigma_{t}} \rho_{n}\cdot z_{n} \,dx\right|+ \Bigl(\sum_{|\alpha|=1}^{N_*+1}+ \mu\sum_{|\alpha|=N_*+2}\Bigr)\left|\int_{\Sigma_{t}} \big(-\left[\partial^{\alpha},\, z_{-}\cdot \nabla\right] z_{n}+ \partial^{\alpha}\rho_{n}\big)\cdot \partial^{\alpha} z_{n} \,dx\right|. 
\end{equation}

{\underline{\it i). Estimate for $f_n$.}
Since $\big|\partial^{\beta}\big(\theta_{L}(|x^-|)\big)\big|\lesssim \frac{1}{L}$ for $|\beta|\geq 1$, we have
\beq\label{eq:d13}\begin{aligned}
&|\p^{\alpha}\rho_{n}|\lesssim \frac{\mu}{L}\Big(|z_{+,n,\text{(lin)}}|+ \sum_{k\leq|\alpha|}|\nabla z_{+,n,\text{(lin)}}^{(k)}|\Big),\quad\text{for}\, |\al|\leq N_*+2,\\ 
&|\partial^{\alpha}z_{n}|\lesssim \frac{1}{L}|z_{+,n,\text{(lin)}}|+ \sum_{k\leq|\alpha|-1}|\nabla z_{+,n,\text{(lin)}}^{(k)}|,\quad\text{for}\, 1\leq|\al|\leq N_*+2,
\end{aligned}\eeq
which implies that
\beno\begin{aligned}
&\quad\Bigl(\sum_{|\alpha|=0}^{N_*+1}+ \mu\sum_{|\alpha|=N_*+2}\Bigr)\int_{t_{n-1}}^{t_n}\left|\int_{\Sigma_t}\p^\al\rho_n\cdot \p^\al z_n \,dx\right|dt\\ 
&
\lesssim\f{\mu}{L}\sup_{t_{n-1}\leq t\leq t_n}\Bigl(\sum_{k=0}^{N_*+1}\|z^{(k)}_{+,n,\text{(lin)}}\|_{L^2(\Sigma_t)}^2
+\mu\|z^{(N_*+2)}_{+,n,\text{(lin)}}\|_{L^2(\Sigma_t)}^2
\Bigr)\cdot(t_n-t_{n-1})\\ 
&\qquad
+\f{\mu}{L}\int_{t_{n-1}}^{t_n}\int_{\Sigma_t}|\na z^{(N_*+1)}_{+,n,\text{(lin)}}|^2\,dxdt+\f{\mu^2}{L}\int_{t_{n-1}}^{t_n}\int_{\Sigma_t}|\na z^{(N_*+2)}_{+,n,\text{(lin)}}|^2\,dxdt\\ 
&\lesssim\f{\mu}{L}\sup_{t_{n-1}\leq t\leq t_n}\mathcal{E}_{+,\text{(lin)}}(t)\cdot (t_n-t_{n-1})+\f{1}{L}\cdot\mathcal{D}_{+,n,\text{(lin)}}(t_n).
\end{aligned}\eeno
Then, using the first inequality in \eqref{eq:d4} and $|t_n-t_{n-1}|=T\leq\log L$, we get
\beq\label{eq:d14}
\Bigl(\sum_{|\alpha|=0}^{N_*+1}+ \mu\sum_{|\alpha|=N_*+2}\Bigr)\int_{t_{n-1}}^{t_n}\left|\int_{\Sigma_t}\p^\al\rho_n\cdot \p^\al z_n \,dx\right|dt\lesssim\f{\mu\log L+1}{L}\cdot\mathcal{E}_{\pm,\text{(lin)}}(t_{n-1}).
\eeq

For the remainding terms in $f_n$,
since $\text{supp}\, z_n(t,\cdot)\subset\Bigl\{x\in Q_L\,|\,|x^-(t,x)|\leq\f{L}{4}\Bigr\}$, 
we have
\beno\begin{aligned}
&\int_{t_{n-1}}^{t_n}\left|\int_{\Sigma_t}\left[\partial^{\alpha},\, z_{-}\cdot \nabla\right] z_n\cdot \p^\al z_n \,dx\right|dt\lesssim\iint_{W_{[t_{n-1},t_n],+}^{[-\f{L}{4},\f{L}{4}]}}\left|\left[\partial^{\alpha},\, z_{-}\cdot \nabla\right] z_n\right|\cdot|\p^\al z_n| dxdt,
\end{aligned}\eeno
and
\beno
\begin{aligned}
\left|\left[\partial^{\alpha},\, z_{-}\cdot \nabla\right] z_{n}\right| & \lesssim \sum_{1\leq k\leq|\alpha|}|\nabla z_{-}^{(k-1)}|\cdot |\nabla z_{n}^{(|\alpha|-k)}|\\
 & \lesssim \sum_{1\leq k\leq|\alpha|}|\nabla z_{-}^{(k-1)}|\cdot \Big(\frac{1}{L}|z_{+,n,\text{(lin)}}|+ \sum_{l\leq|\alpha|-k}|\nabla z_{+,n,\text{(lin)}}^{(l)}|\Big),
\end{aligned} 
\eeno
where $W_{[t_{n-1},t_n],+}^{[-\f{L}{4},\f{L}{4}]}:=\bigcup_{-\f{L}{4}\leq u_+\leq\f{L}{4}} C^+_{[t_{n-1},t_n],u_+}$.
Then using \eqref{eq:d13}, a derivation similar to those in \eqref{eq:h10} and \eqref{eq:p39} leads to
\beno\begin{aligned}
&\sum_{1\leq|\al|\leq N_*+1}\int_{t_{n-1}}^{t_n}\left|\int_{\Sigma_t}\left[\partial^{\alpha},\, z_{-}\cdot \nabla\right] z_n\cdot \p^\al z_n \,dx\right|dt\\ 
&\lesssim\sum_{l=0}^{N_*}\bigl(E_-^l\bigr)^{\f12}
\cdot\sum_{k=0}^{N_*+1}\sup_{|u_+|\leq\f{L}{4}}\int_{C_{[t_{n-1},t_n],u_+}^+}|z^{(k)}_{+,n,\text{(lin)}}|^2d\sigma_+.
\end{aligned}\eeno
Proceeding in the same manner as derivation of \eqref{eq:h19}, we also find that
\beno\begin{aligned}
&\quad\mu\sum_{|\al|=N_*+2}\int_{t_{n-1}}^{t_n}\left|\int_{\Sigma_t}\left[\partial^{\alpha},\, z_{-}\cdot \nabla\right] z_n\cdot \p^\al z_n \,dx\right|dt\\ 
&\lesssim \Bigl(\sum_{l=0}^{N_{*}}\bigl(E_-^l\bigr)^{\f12}+ \Bigl(\sum_{k=1}^{N_{*}} D_{-}^{k}\Bigr)^{\frac{1}{2}}\Bigr)\cdot\Bigl(\sup_{t_{n-1}\leq t\leq t_n}\sum_{k=0}^{N_*+1}\|z^{(k)}_{+,n,\text{(lin)}}\|_{L^2(\Sigma_t)}^2\\
&\quad\quad\ +\f{\mu(t_n-t_{n-1})}{L^2}
\sup_{t_{n-1}\leq t\leq t_n}\|z_{+,n,\text{(lin)}}\|_{L^2(\Sigma_t)}^2+\mu\sum_{l=0}^{N_*+1}\int_{t_{n-1}}^{t_n}\|\na z^{(l)}_{+,n,\text{(lin)}}\|_{L^2(\Sigma_t)}^2dt\Bigr)\\ 
&\lesssim \Bigl(\sum_{l=0}^{N_{*}}\bigl(E_-^l\bigr)^{\f12}+ \Bigl(\sum_{k=1}^{N_{*}} D_{-}^{k}\Bigr)^{\frac{1}{2}}\Bigr)
\cdot\Bigl(\bigl(1+\f{\mu\log L}{L^2}\bigr)\sup_{t_{n-1}\leq t\leq t_n}\mathcal{E}_{+,\text{(lin)}}(t)+\mathcal{D}_{+,n,\text{(lin)}}(t_n)\Bigr). 
\end{aligned}\eeno
Using \eqref{eq:thm2} and the first inequality in \eqref{eq:d4}, we obtain 
\beq\label{eq:d15}
\Bigl(\sum_{1\leq|\al|\leq N_*+1}+\mu\sum_{|\al|=N_*+2}\Bigr)\int_{t_{n-1}}^{t_n}\left|\int_{\Sigma_t}\left[\partial^{\alpha},\, z_{-}\cdot \nabla\right] z_n\cdot \p^\al z_n \,dx\right|dt\lesssim\bigl(\mathcal{E}^w(0)\bigr)^{\f12}\cdot\mathcal{E}_{+,\text{(lin)}}(t_{n-1}).
\eeq

Thus, by \eqref{eq:d14} and \eqref{eq:d15}, and since $L\gg1$, we derive
\beq\label{eq:d16}
\int_{t_{n-1}}^{t_{n}} f_{n}(t)\,  dt \lesssim\bigl(\mathcal{E}^w(0)\bigr)^{\f12}\cdot \mathcal{E}(t_{n-1}). 
\eeq

{\underline{\it ii). A refined energy of $z_n$ in the Lagrangian form.}} To simplify notation in this and subsequent subsections, we define the Lagrangian form $\widetilde{v}(t,y)$ of  $v(t,x)$ on $\mathbb{R}^{3}$ by
\beno
\widetilde{v}(t,y)=v\left(t, \psi_{-}(t,y)\right), 
\eeno
where $\psi_{-}(t,y)$ is the flow generated by $Z_{-}$, as given in \eqref{eq:e2}, and $y\in\mathbb{R}^{3}$ denotes the label at time $t=0$. Since $\dive\, Z_-=0$ implies $\operatorname{det}\left(\frac{\partial \psi_{-}(t,y)}{\partial y}\right)=1$, we have 
\beq\label{eq:d17} 
\|v\|_{L^{2}\left(\mathbb{R}^{3}\right)}=\|\widetilde{v}\|_{L^{2}\left(\mathbb{R}^{3}\right)}.
\eeq
Whenever $\nabla \widetilde{v}(t,y)$ is written, the derivative $\nabla$ is always taken with respect to the variable $y$. 

Since the support of $z_n(t,\cdot)$ is contained in $\left\{x\in Q_L\,|\, |x^-(t,x)|\leq \frac{L}{4}\right\}$, we may extend $z_n$ by zero outside $\left[-L, L\right]^{3}$, thereby regarding it as a compactly supported vector field on $\mathbb{R}^{3}$. 

For $t_{n-1}\leq t \leq t_{n}$, we define
\begin{equation}\label{eq:d18}
X_{n}(t)=\left\|\widetilde{z_n}\right\|_{L^{2}\left(\mathbb{R}^{3}\right)}^2 + \sum_{|\alpha|=1}^{N_*+1}\Big\|\widetilde{z_n^{(\alpha)}}\Big\|_{L^{2}\left(\mathbb{R}^{3}\right)}^2 + \mu\sum_{|\alpha|=N_*+2}\Big\|\widetilde{z_n^{(\alpha)}}\Big\|_{L^{2}\left(\mathbb{R}^{3}\right)}^2.
\end{equation}
By \eqref{eq:d17}, we have
\beq\label{eq:d19}
X_n(t)=\left\|z_{n} \right\|_{L^{2}\left(\Sigma_{t}\right)}^2+ \sum_{|\alpha|=1}^{N_*+1}\left\|\partial^{\alpha} z_{n} \right\|_{L^{2}\left(\Sigma_{t}\right)}^2+ \mu\sum_{|\alpha|=N_*+2}\left\|\partial^{\alpha} z_{n} \right\|_{L^{2}\left(\Sigma_{t}\right)}^2.
\eeq
By virtue of \eqref{eq:e4}, \eqref{eq:d11} and \eqref{eq:d19}, there exists a universal constant $c>0$ such that 
\begin{equation}\label{eq:d20}
\frac{d}{dt} X_{n}(t)+ c\mu\bigg(\left\|\nabla\widetilde{z_n}\right\|_{L^{2}}^2 + \sum_{|\alpha|=1}^{N_*+1}\Big\|\nabla\widetilde{z_n^{(\alpha)}}\Big\|_{L^{2}}^2 + \mu\sum_{|\alpha|=N_*+2}\Big\|\nabla\widetilde{z_n^{(\alpha)}}\Big\|_{L^{2}}^2\bigg) \leq 2f_{n}(t). 
\end{equation}
Here and in what follows, the notation $\|\cdot\|_{L^2}$ applied to $\wt f(t,y)$ always refers to the norm in $L^2(\R^3)$.

Let $\widehat{f}$ denote the Fourier transform of $f$ on $\mathbb{R}^{3}$. By Plancherel's theorem, we have
\beno
\|\nabla\widetilde{f}\|_{L^{2}}^2\sim\int_{\mathbb{R}^3} |\xi|^2\big|\widehat{\widetilde{f}}(\xi) \big|^2 d\xi\geq h(t)^2 \bigl\|\widetilde{f}\bigr\|_{L^2}^2 - h(t)^2\int_{|\xi|\leq h(t)}\big|\widehat{\widetilde{f}}(\xi) \big|^2 d\xi, 
\eeno
where $h(t)>0$ is a smooth function to be determined in the following Lemma \ref{lem:d3}. Using this inequality, we deduce from \eqref{eq:d20} that
\begin{equation}\label{eq:d21}
\begin{aligned}
& \frac{d}{dt} X_{n}(t)+ c\mu h(t)^{2} X_{n}(t) \\
\leq & \,2f_{n}(t)+ c\mu h(t)^{2}\int_{|\xi|\leq h(t)}\bigg(\bigl|\widehat{\widetilde{z_n}}(t,\xi)\bigr|^2+\Bigl(\sum_{|\alpha|=1}^{N_*+1}+ \mu\sum_{|\alpha|=N_*+2}\Bigr)\Big|\widehat{\widetilde{z_n^{(\alpha)}}}(t,\xi) \Big|^2\bigg) d\xi.
\end{aligned} 
\end{equation}

By combining \eqref{eq:d7}, \eqref{eq:d8}, \eqref{eq:d16}, and \eqref{eq:d21}, we obtain the main result of this subsection, stated in the following lemma.

\begin{lemma}\label{lem:d2} 
For $t_{n-1}< t \leq t_{n}$, we have
\beq\label{eq:d22}
\mathcal{E}_{+,\text{(lin)}}(t)\lesssim X_{n}(t)+ \frac{\mathcal{E}^{w}(0)}{\left(\log L\right)^{4}}, 
\eeq
where $X_n(t)$ defined in \eqref{eq:d18} satisfies \eqref{eq:d19} and \eqref{eq:d21}, and $f_n$ in \eqref{eq:d21}  satisfies \eqref{eq:d16}.
\end{lemma}

\subsection{Decay estimates for the principal total linear energy} Starting from \eqref{eq:d21}, we now derive the decay estimate of $X_n(t)$ presented in the following lemma.

\begin{lemma}\label{lem:d3} For $1\leq n\leq n_0$, we have
\begin{equation}\label{eq:d23}
\begin{aligned}
X_{n}(t_{n}) &\lesssim \frac{\log\big(\log(\mu T+e)+e\big)}{\log\big(\mu T+e\big)}\cdot \mathcal{E}^{w}(0)\\
& \qquad+\bigl(\mathcal{E}^w(0)\bigr)^{\f12}\cdot \mathcal{E}(t_{n-1}) +\int_{\mathbb{R}^{3}}\bigl|\widehat{\widetilde{z_n}}(t_{n-1}, \xi)\bigr|^2\cdot \theta\Bigl(\frac{|\xi|}{\delta(T)}\Bigr) d\xi, 
\end{aligned}
\end{equation}
where 
\begin{equation}\label{eq:d23a}
\delta(T)^2:=\frac{4}{c}\big(\log(\mu T+e)+e\big)^{-1}, 
\end{equation}
and $\theta$  denotes the cut-off function defined at the beginning of  Section \ref{sec:3.2}.
\end{lemma}
\begin{proof}
In view of \eqref{eq:d21}, we divide the proof into several steps.

{\bf Step 1. Estimates on $\int_{|\xi|\leq h(t)} \Big|\widehat{\widetilde{z_n^{(\alpha)}}}(t,\xi) \Big|^2 d\xi$ for $1\leq|\al|\leq N_*+2$.} Defining 
\beno
a_{1}(t,y)=\mathcal{F}^{-1}\Bigl(\theta\Bigl(\frac{|\xi|}{h(t)}\Big)\Bigr)=h(t)^{3} \mathcal{F}^{-1}(\theta)\big(h(t) y\big)\in\mathcal{S}(\R^3),\quad\text{for } \, t\in(t_{n-1},t_n],
\eeno
we have
\begin{equation}\label{eq:d24}
\left\|a_{1}\right\|_{L^{1}}\lesssim 1, \quad \left\|\nabla^2 a_{1}\right\|_{L^{1}}\lesssim h(t)^{2}, \quad \left\|a_{1}\right\|_{L^{2}}\lesssim h(t)^{\frac{3}{2}}. 
\end{equation}

By Plancherel's theorem and H\"{o}lder's inequality and the fact that $\theta(r)=1$ for all $|r|\leq 1$, we have
\begin{equation}\label{eq:d25}
\begin{aligned}
\int_{|\xi|\leq h(t)} \Big|\widehat{\widetilde{z_n^{(\alpha)}}}(t,\xi) \Bigr|^2 d\xi  & \leq \int_{\mathbb{R}^3} \Bigl|\widehat{\widetilde{z_n^{(\alpha)}}}(t,\xi) \Bigr|^2\theta\Bigl(\frac{|\xi|}{h(t)}\Bigr) d\xi\lesssim \left|\int_{\mathbb{R}^3} \widetilde{z_n^{(\alpha)}}(t,y)\cdot a_1 * \widetilde{z_n^{(\alpha)}}(t,y)\, dy\right|\\
& \lesssim \Bigl\|\left(R^2+|y|^2\right)^{\frac{1}{2}}\widetilde{z_n^{(\alpha)}}\Bigr\|_{L^2}\cdot\Bigl\|\frac{1}{|y|}\bigl(a_1 *\widetilde{z_n^{(\alpha)}}\bigr)\Bigr\|_{L^2} .
\end{aligned} 
\end{equation}
Due to Hardy's and Young's inequalities, we get
\beno
\begin{aligned}
\Bigl\|\frac{1}{|y|}\bigl(a_1 *\widetilde{z_n^{(\alpha)}}\bigr)\Bigr\|_{L^2} & \stackrel{\text{Hardy}}{\lesssim} \bigl\|\nabla \big(a_1 *\widetilde{z_n^{(\alpha)}}\big)\bigr\|_{L^2} \lesssim \bigl\| (\Delta a_1) * \big(|\nabla|^{-1}\widetilde{z_n^{(\alpha)}}\big)\bigr\|_{L^2}\\
& \stackrel{\text{Young}}{\lesssim} \left\|\nabla^2 a_{1}\right\|_{L^{1}}\cdot\bigl\| \left|\nabla\right|^{-1}\widetilde{z_n^{(\alpha)}}\bigr\|_{L^2}\stackrel{\eqref{eq:d24}}{\lesssim} h(t)^2 \Big\|\frac{1}{|\xi|}\widehat{\widetilde{z_n^{(\alpha)}}}\Big\|_{L^2}\\
& \stackrel{\text{Hardy}}{\lesssim} h(t)^2\cdot \Big\|\nabla_{\xi}\,\widehat{\widetilde{z_n^{(\alpha)}}}\Big\|_{L^2}\lesssim h(t)^2\cdot\bigl\||y|\,\widetilde{z_n^{(\alpha)}}\bigr\|_{L^2} .
\end{aligned} 
\eeno
Using this inequality and the fact that  $x^{-}\left(t,\psi_{-}(t,y)\right)=y$, we deduce from \eqref{eq:d25} that
\beno
\begin{aligned}
&\int_{|\xi|\leq h(t)}\Big|\widehat{\widetilde{z_n^{(\alpha)}}}(t,\xi) \Big|^2 d\xi\lesssim h(t)^2\cdot \bigl\|\langle y\rangle\widetilde{z_n^{(\alpha)}}\bigr\|_{L^2}^2\lesssim h(t)^2\cdot \bigl\|\langle x^-\rangle z_n^{(\alpha)} \bigr\|_{L^2(\Sigma_t)}^2 \\
&\quad\overset{\eqref{eq:d13}}{\lesssim} h(t)^2\cdot \Big(\|z_{+,n,\text{(lin)}}\|_{L^{2}(\Sigma_t)}^2 + \sum_{k\leq |\alpha|-1} \bigl\|\langle w_-\rangle\nabla z_{+,n,\text{(lin)}}^{(k)}\bigr\|_{L^{2}(\Sigma_t)}^2 \Big),
\end{aligned} 
\eeno 
which along with the first inequality in \eqref{eq:d3} implies
\begin{equation} \label{eq:d26}
\Big(\sum_{|\alpha|=1}^{N_*+1}+ \mu\sum_{|\alpha|=N_*+2}\Big)\int_{|\xi|\leq h(t)}\Big|\widehat{\widetilde{z_n^{(\alpha)}}}(t,\xi) \Big|^2 d\xi\lesssim h(t)^2\cdot \sup_{0\leq t\leq\log L}\mathcal{E}_{+,\text{(lin)}}^w(t)\lesssim h(t)^2\cdot \mathcal{E}^w(0).
\end{equation}

{\bf Step 2. Estimate on $\int_{|\xi|\leq h(t)}\big|\widehat{\widetilde{z_n}}(t,\xi)\big|^2 d\xi$.}
Since $\frac{d}{dt}\psi_{-}(t,y)=Z_-(t,\psi_{-}(t,y))$ for any $y\in\mathbb{R}^3$, it follows that
\begin{equation}\label{eq:d27}
\partial_{t}\tilde{f}(t,y)=\left(\partial_{t}+ Z_{-}\cdot\nabla\right) f|_{x=\psi_{-}(t,y)}, \quad \nabla_y\tilde{f}(t,y)=\left(\frac{\partial \psi_{-}(t,y)}{\partial y}\right)^T\nabla f|_{x=\psi_{-}(t,y)} . 
\end{equation}
For simplicity, one may regard $\na_x=\left(\frac{\partial \psi_{-}(t,y)}{\partial y}\right)^{-T}\na_y$. 

Let $A_{-}=\left(\frac{\partial \psi_-(t,y)}{\partial y}\right)^{-1} \left(\frac{\partial \psi_-(t,y)}{\partial y}\right)^{-T}$.
Since $\dive Z_-=0$ implies $\operatorname{det}\left(\frac{\partial \psi_{-}(t,y)}{\partial y}\right)=1$,  $\left(\frac{\partial \psi_-(t,y)}{\partial y}\right)^{-1}$ is the adjoint matrix of $\frac{\partial \psi_-(t,y)}{\partial y}$ and hence each of its column vectors is divergence free, i.e., $\operatorname{div}\left(\frac{\partial \psi_{-}(t,y)}{\partial y}\right)^{-1}=0$. Consequently,
\beno
\Delta f|_{x=\psi_-(t,y)}= \nabla_y\cdot \left( A_-\nabla_y \tilde{f}(t,y)\right), 
\eeno
from which and \eqref{eq:d9}, we deduce that
\begin{equation}\label{eq:d28}
\partial_{t} \widetilde{z_n}- \mu\Delta \widetilde{z_n}=  \underbrace{\mu\operatorname{div}\big((A_{-}- I)\nabla\widetilde{z_n}\big)}_{J_n^1}+ \underbrace{\widetilde{\rho_n}}_{J_n^2}=: J_n. 
\end{equation}

In frequency space, we have
\beno
\frac{d}{dt}\bigl|\widehat{\widetilde{z_n}}\bigr|^2 + 2\mu |\xi|^2 \bigl|\widehat{\widetilde{z_n}}\bigr|^2= 2\operatorname{Re}\bigl(\widehat{J_n}\cdot \overline{\widehat{\widetilde{z_n}}}\, \bigr), 
\eeno
which implies
\beq\label{eq:d28a}
\bigl|\widehat{\widetilde{z_n}}(t,\xi)\bigr|^2  = e^{-2\mu(t-t_{n-1})|\xi|^2} \bigl|\widehat{\widetilde{z_n}}(t_{n-1},\xi)\bigr|^2 + 2\operatorname{Re} \int_{t_{n-1}}^{t} e^{-2\mu(t-s)|\xi|^2}\big(\widehat{J_n}\cdot \overline{\widehat{\widetilde{z_n}}}\,\big)(s,\xi)ds.
\eeq
Using the fact that $\theta(r)=1$ for all $|r|\leq 1$, we have
\begin{equation}\label{eq:d29}
\begin{aligned}
\int_{|\xi|\leq h(t)} \bigl|\widehat{\widetilde{z_n}}(t,\xi)\bigr|^2 d\xi
&\leq \int_{\mathbb{R}^3} e^{-2\mu (t-t_{n-1})|\xi|^2} \bigl|\widehat{\widetilde{z_n}}(t_{n-1},\xi)\bigr|^2 \theta\Bigl(\frac{|\xi|}{h(t)}\Bigr) d\xi\\ 
&\quad+ 2\operatorname{Re} \int_{t_{n-1}}^{t} \int_{\mathbb{R}^3} e^{-2\mu(t-s)|\xi|^2} \theta\Bigl(\frac{|\xi|}{h(t)}\Bigr)\big(\widehat{J_n}\cdot \overline{\widehat{\widetilde{z_n}}}\,\big) (s,\xi) \,d\xi ds.
\end{aligned} 
\end{equation}

Following a similar approach as for $a_1$, we define 
\beno
a_{2}(s,y)=\mathcal{F}^{-1}\Bigl(e^{-2\mu(t-s)|\xi|^{2}}\Bigr)=\left(2\sqrt{\mu(t-s)}\right)^{-3} e^{-\frac{|y|^2}{8\mu(t-s)}} \in\mathcal{S}(\R^3),\quad\text{for }\, s<t.
\eeno
Using Plancherel's theorem, we have
\begin{equation}\label{eq:d30}
\Bigl|\int_{t_{n-1}}^t \int_{\mathbb{R}^3} e^{-2\mu(t-s)|\xi|^2} \theta\Bigl(\frac{|\xi|}{h(t)}\Bigr)\big(\widehat{J_n}\cdot \overline{\widehat{\widetilde{z_n}}}\,\big)(s,\xi)\, d\xi ds \Bigr|\lesssim\sum_{i=1}^2\underbrace{\Bigl|\int_{t_{n-1}}^t \int_{\mathbb{R}^3} J_n^i\cdot a_1 * a_2 * \widetilde{z_n}\, dy ds \Bigr|}_{F_{ni}(t)}.
\end{equation}

{\underline{\bf Estimate for $F_{n1}$.}}
By the expression of $J_n^1$ given in \eqref{eq:d28}, integration by parts gives rise to 
\beno\begin{aligned}
F_{n1}(t)&
\leq \mu\Bigl|\int_{t_{n-1}}^t \int_{\mathbb{R}^3} (A_- -I)\nabla \widetilde{z_n}\cdot \nabla\left(a_1 * a_2 * \widetilde{z_n}\right)\,dy ds\Bigr|\\ 
&\leq\mu\|A_- -I\|_{L_{t,y}^{\infty}}\cdot \int_{t_{n-1}}^t \Big(\int_{|y|\leq h(t)^{-\frac{1}{2}}}|\nabla\widetilde{z_n}| \cdot|\nabla(a_1 * a_2 * \widetilde{z_n})|\,dy\\ 
&\qquad\qquad\qquad\qquad\qquad\qquad
+ \int_{|y|\geq h(t)^{-\frac{1}{2}}} |\nabla\widetilde{z_n}| \cdot|\nabla(a_1 * a_2 * \widetilde{z_n})|\,dy\Bigr) ds\\ 
&\stackrel{\text{H\"{o}lder}}{\leq} \mu\|A_- -I\|_{L_{t,y}^{\infty}}\cdot\int_{t_{n-1}}^t \Big( h(t)^{-\frac{3}{4}}\left\|\nabla \widetilde{z_n}\right\|_{L^2}\cdot \left\|a_{1} * a_{2} *\left(\nabla \widetilde{z_{n}}\right)\right\|_{L^{\infty}} \\
& \quad\quad + \big(\log \left(R^2+h(t)^{-1}\right)^{\frac{1}{2}}\big)^{-1}\cdot\|\log (R^2+|y|^2)^{\frac{1}{2}}\nabla\widetilde{z_n}\|_{L^2}\cdot \|a_{1} * a_{2} *(\nabla \widetilde{z_{n}})\|_{L^{2}}\Big) ds.
\end{aligned}\eeno
Applying Young's inequality,  we get
\beno\begin{aligned}
&\left\|a_{1} * a_{2} *\left(\nabla \widetilde{z_{n}}\right)\right\|_{L^{\infty}} \lesssim\left\|a_{1}\right\|_{L^{2}}\left\|a_{2}\right\|_{L^1}\left\|\nabla \widetilde{z_{n}}\right\|_{L^2} \lesssim h(t)^{\frac{3}{2}}\left\|\nabla \widetilde{z_{n}}\right\|_{L^{2}}, \\ 
&\left\|a_{1} * a_{2} *\left(\nabla \widetilde{z_{n}}\right)\right\|_{L^{2}} \lesssim\|a_1\|_{L^{1}}\|a_2\|_{L^1}\left\|\nabla \widetilde{z_{n}}\right\|_{L^{2}} \lesssim \left\|\nabla \widetilde{z_{n}}\right\|_{L^{2}}.
\end{aligned}
\eeno
The estimate \eqref{eq:i12} leads to
\beq\label{eq:d30a} 
\|A_- -I\|_{L_{t,y}^{\infty}}\lesssim\bigl(\mathcal{E}^w(0)\bigr)^{\f12}.
\eeq
Hence, we derive that
\begin{equation}\label{eq:d31}
\begin{aligned}
F_{n1}(t)
& \lesssim \mu \Big(h(t)^{\frac{3}{4}}+ \big(\log \left(R^2+h(t)^{-1}\right)^{\frac{1}{2}}\big)^{-1}\Big)\cdot\|A_{-}-I \|_{L_{t,y}^{\infty}}
\cdot\int_{t_{n-1}}^t \big\|\log\langle y\rangle  \widetilde{\nabla z_n}\big\|_{L^{2}}^2 ds \\
& \lesssim\mu \Big(h(t)^{\frac{3}{4}}+ \big(\log \left(R^2+h(t)^{-1}\right)^{\frac{1}{2}}\big)^{-1}\Big)\cdot\bigl(\mathcal{E}^w(0)\bigr)^{\f12}\cdot \int_{t_{n-1}}^t \|\log \langle w_{-}\rangle \nabla z_{n}\|_{L^{2}}^2 ds .
\end{aligned} 
\end{equation}

From the second inequality in \eqref{eq:d13}, it follows that
\beno
\left\|\log \langle w_{-}\rangle \nabla z_{n}\right\|_{L^{2}}^2 \lesssim \frac{1}{L^2}\left\|\log \langle w_{-}\rangle z_{+,n,\text{(lin)}} \right\|_{L^{2}\left(\Sigma_{s}\right)}^2+ \left\|\log \langle w_{-}\rangle \nabla z_{+,n,\text{(lin)}} \right\|_{L^{2}\left(\Sigma_{s}\right)}^2 .
\eeno
Combining this inequality with \eqref{eq:d3} and the fact that $T^* \leq \log L$, we derive from \eqref{eq:d31} that
\begin{equation}\label{eq:d32}
F_{n1}(t) \lesssim\Big(h(t)^{\frac{3}{4}}+ \big(\log \left(R^2+h(t)^{-1}\right)^{\frac{1}{2}}\big)^{-1}\Big)\cdot\bigl(\mathcal{E}^w(0)\bigr)^{\f32}. 
\end{equation}

{\underline{\bf Estimate for $F_{n2}$.}} By virtue of  \eqref{eq:d13} and Young's inequality, we have
\beno
\begin{aligned}
F_{n2}(t) & \leq \int_{t_{n-1}}^{t}\left\|\rho_{n}\right\|_{L^{2}}\cdot\left\|a_{1} * a_{2} * \widetilde{z_{n}}\right\|_{L^{2}} ds \\
& \lesssim\frac{\mu}{L}\int_{t_{n-1}}^{t}\bigl(\|z_{+,n,\text{(lin)}}\|_{L^{2}(\Sigma_{s})}+\|\nabla z_{+,n,\text{(lin)}}\|_{L^{2}(\Sigma_{s})}\bigr)\cdot\|\widetilde{z_{n}}\|_{L^{2}} ds \\
& \lesssim \frac{\mu}{L}\int_{t_{n-1}}^{t}\left(\left\|z_{+,n,\text{(lin)}}\right\|_{L^{2}\left(\Sigma_{s}\right)}^{2}+\left\|\nabla z_{+,n,\text{(lin)}}\right\|_{L^{2}\left(\Sigma_{s}\right)}^{2}\right) ds,
\end{aligned} 
\eeno
which together with \eqref{eq:d3} implies
\beq\label{eq:d33}
F_{n2}(t)\lesssim \frac{\mu}{L}\cdot(t-t_{n-1})\cdot\mathcal{E}^{w}(0).
\eeq

Finally, combining \eqref{eq:d29}, \eqref{eq:d30}, \eqref{eq:d32} and \eqref{eq:d33}, for all $t_{n-1}< t \leq t_{n}$, we conclude that
\begin{equation}\label{eq:d34}
\begin{aligned}
\int_{|\xi|\leq h(t)}\bigl|\widehat{\widetilde{z_{n}}}(t, \xi)\bigr|^{2} d\xi&\leq C\Big(h(t)^{\frac{3}{4}}+ \big(\log \left(R^2+h(t)^{-1}\right)^{\frac{1}{2}}\big)^{-1}\Big)\cdot\bigl(\mathcal{E}^{w}(0)\bigr)^{\f32}\\
&\quad+ \frac{C\mu}{L}\cdot\left(t-t_{n-1}\right)\cdot\mathcal{E}^{w}(0) + \int_{\mathbb{R}^{3}}\bigl|\widehat{\widetilde{z_{n}}}(t_{n-1} ,\xi)\bigr|^{2}\cdot \theta\Bigl(\frac{|\xi|}{h(t)}\Bigr) d\xi .
\end{aligned} 
\end{equation}

{\bf Step 3. The decay estimate of $X_n(t)$.}
Thanks to \eqref{eq:d26} and \eqref{eq:d34}, for all $t_{n-1}< t \leq t_{n}$, we deduce from \eqref{eq:d21} that 
\begin{equation}\label{eq:d35}
\begin{aligned}
&\frac{d}{dt} X_{n}(t)+ c\mu h(t)^{2} X_{n}(t)\leq 2f_{n}(t)+ C\mu h(t)^{4} \mathcal{E}^{w}(0)\\ 
&\qquad+ C\mu h(t)^{2}\Big(h(t)^{\frac{3}{4}}+ \big(\log \left(R^2+h(t)^{-1}\right)^{\frac{1}{2}}\big)^{-1}\Big)\cdot\bigl(\mathcal{E}^{w}(0)\bigr)^{\f32} \\
& \qquad + \frac{C\mu^{2}\left(t-t_{n-1}\right) h(t)^{2}}{L}\mathcal{E}^{w}(0)+ C\mu h(t)^{2} \underbrace{\int_{\mathbb{R}^{3}}\left|\widehat{\widetilde{z_{n}}}(t_{n-1}, \xi)\right|^{2} \theta\left(\frac{|\xi|}{h(t)}\right) d\xi}_{H_{n}(t)},
\end{aligned} 
\end{equation}
which implies
\begin{equation}\label{eq:d36}
\frac{d}{dt}\Big(e^{c\mu\int_{t_{n-1}}^{t} h(\tau)^{2} d\tau} X_{n}(t)\Big) \leq e^{c\mu\int_{t_{n-1}}^{t} h(\tau)^{2} d\tau}\times\text{RHS of \eqref{eq:d35}}. 
\end{equation}

Taking
\beno
ch(t)^2=\big(\mu(t-t_{n-1})+e\big)^{-1} \Big(\log\big(\mu(t-t_{n-1})+e\big)\Big)^{-1}, \quad t_{n-1}< t \leq t_{n} , 
\eeno
we have
\beno
e^{c\mu\int_{t_{n-1}}^{t} h(\tau)^{2} d\tau}= \log\big(\mu(t-t_{n-1})+e\big), 
\eeno
from which and \eqref{eq:d36}, we deduce that
\beno
\begin{aligned}
&\frac{d}{dt}\Big(\log\big(\mu(t-t_{n-1})+e\big) X_{n}(t)\Big)\lesssim  \log\big(\mu(t-t_{n-1})+e\big) f_n(t)+ \mu\big(\mu(t-t_{n-1})+e\big)^{-2}\cdot\mathcal{E}^{w}(0) \\
& \qquad\qquad + \mu\big(\mu(t-t_{n-1})+e\big)^{-1} \Big(\log\big(\mu(t-t_{n-1})+e\big)\Big)^{-1}\cdot\bigl(\mathcal{E}^{w}(0)\bigr)^{\f32}
+ \frac{\mu}{L}\cdot\mathcal{E}^w(0) \\
&\qquad \qquad + \mu\big(\mu(t-t_{n-1})+e\big)^{-1}\cdot H_{n}(t) .
\end{aligned} 
\eeno
Integrating the above inequality on $\left[t_{n-1}, t_{n}\right]$ and using \eqref{eq:d16} for $f_n$, we obtain
\begin{equation} \label{eq:d37}
\begin{aligned}
&\log\big(\mu(t_{n}-t_{n-1})+e\big) X_{n}(t_{n})
\lesssim  X_{n}(t_{n-1})+ \log\big(\mu(t_{n}-t_{n-1})+e\big)\cdot\bigl(\mathcal{E}^{w}(0)\bigr)^{\f12}\cdot \mathcal{E}(t_{n-1}) \\
&\qquad \qquad+\Bigl(1+\frac{\mu\log L}{L}\Bigr)\mathcal{E}^{w}(0) +\log\Big(\log\big(\mu(t_{n}-t_{n-1})+e\big)\Big)\cdot\bigl(\mathcal{E}^{w}(0)\bigr)^{\f32}\\
&\qquad \qquad +\mu \int_{t_{n-1}}^{t_{n}} \frac{H_{n}(s)}{\mu\left(s-t_{n-1}\right)+e} ds .
\end{aligned}
\end{equation}

By definition, it is easy to check that 
\beno 
|\p^\al z_n(t_{n-1},x)|\lesssim\f{1}{L}|z_+(t_{n-1},x)|
+\sum_{k\leq|\al|-1}|\na z_+^{(k)}(t_{n-1,x})|,\quad 1\leq|\al|\leq N_*+2,
\eeno
which together with \eqref{eq:d19} and Lemma \ref{lem:d1} yields
\beno 
X_n(t_{n-1})\lesssim\mathcal{E}_+(t_{n-1})\lesssim\mathcal{E}_{+,\text{(lin)}}(t_{n-1})+\mathcal{E}_{+,\text{(non)}}(t_{n-1})\overset{\eqref{eq:d3}}{\lesssim}\mathcal{E}^w(0).
\eeno
Consequently, using the facts that $L\gg1$, $\mathcal{E}^{w}(0)\leq\e^2<1$ and $t_n-t_{n-1}=T$, we deduce from \eqref{eq:d37} that
\begin{equation} \label{eq:d38}
\begin{aligned}
&X_{n}(t_n)
\lesssim\f{\log\big(\log(\mu T+e)+e\big)}{\log(\mu T+e)}\cdot\mathcal{E}^{w}(0)+\bigl(\mathcal{E}^{w}(0)\bigr)^{\f12}\cdot \mathcal{E}(t_{n-1}) \\
&\qquad \qquad +\f{1}{\log(\mu T+e)}\cdot\mu \int_{t_{n-1}}^{t_{n}} \frac{H_{n}(s)}{\mu\left(s-t_{n-1}\right)+e} ds .
\end{aligned}
\end{equation}

Now we estimate the last term on the RHS of \eqref{eq:d38}.

(i). For $t_{n-1}\leq s\leq \frac{1}{\mu}\log(\mu T+e)+t_{n-1}$, since
 \beno
 H_n(s)\lesssim \|z_{+,n,\text{(lin)}}\|^2_{L^2(\Sigma_{t_{n-1}})}=\|z_{+}\|^2_{L^2(\Sigma_{t_{n-1}})} \lesssim \mathcal{E}^{w}(0),
 \eeno
we have
\begin{equation}\label{eq:d39}
\begin{aligned}
&\mu\int_{t_{n-1}}^{\frac{1}{\mu}\log (\mu T+e)+t_{n-1}} \frac{H_{n}(s)}{\mu\left(s-t_{n-1}\right)+e} ds \lesssim \log \big(\log(\mu T+e)+e\big)\cdot \mathcal{E}^{w}(0). 
\end{aligned}
\end{equation}

(ii). For $\frac{1}{\mu}\log(\mu T+e)+t_{n-1}\leq s\leq t_n$, we use $h(s)^{2}\leq \frac{1}{c}\big(\log(\mu T+e)+e\big)^{-1}$ and the definition of $H_{n}(s)$ to get
\begin{equation}\label{eq:d40}
\begin{aligned}
& \quad\ \mu \int_{\frac{1}{\mu}\log(\mu T+e)+t_{n-1}}^{t_{n}} \frac{H_{n}(s)}{\mu\left(s-t_{n-1}\right)+e} ds \\
& \lesssim \log(\mu T+e)\cdot\int_{|\xi|^{2}\leq \frac{4}{c}\bigl(\log (\mu T+e)+e\bigr)^{-1}}\bigl|\widehat{\widetilde{z_n}}(t_{n-1}, \xi)\bigr|^2 d\xi .
\end{aligned}
\end{equation}

Thanks to \eqref{eq:d39} and \eqref{eq:d40}, we derive from \eqref{eq:d38} that
\begin{equation}\label{eq:d41}
\begin{aligned}
X_{n}(t_{n}) &\lesssim \frac{\log\big(\log(\mu T+e)+e\big)}{\log(\mu T+e)} \mathcal{E}^{w}(0)+\bigl(\mathcal{E}^w(0)\bigr)^{\f12}\cdot \mathcal{E}(t_{n-1})\\
& \quad\ +\int_{\mathbb{R}^{3}}\bigl|\widehat{\widetilde{z_n}}(t_{n-1}, \xi)\bigr|^{2} \theta\Bigl(\frac{|\xi|}{\delta(T)}\Bigr) d\xi, \quad 1\leq n\leq n_0,
\end{aligned} 
\end{equation}
where $\delta(T)>0$ is given by
\beno
\delta(T)^{2}:=\frac{4}{c}\big(\log(\mu T+e)+e\big)^{-1}.
\eeno
This is the desired estimate \eqref{eq:d23}. The lemma is proved.
\end{proof}

\subsection{Estimate of the last term in (\ref{eq:d23}) (or (\ref{eq:d41}))}
 To establish the iterative decay estimate for  $X_n$, it remains to bound the last term in \eqref{eq:d23}, as detailed in the following lemma.

\begin{lemma}\label{lem:d4}
For $1\leq n\leq n_0$, there holds
\begin{equation}\label{eq:d42}
\int_{\mathbb{R}^{3}}\left|\widehat{\widetilde{z_n}}(t_{n-1}, \xi)\right|^{2} \theta\left(\frac{|\xi|}{\delta(T)}\right) d\xi\lesssim\f{2^n}{\log\big(\log(\mu T+e)+e\big)}\cdot\mathcal{E}^w(0).
\end{equation}
 The implicit constant in $\lesssim$ is independent of n and $\mu$.
\end{lemma}
\begin{proof}
For $1\leq n\leq n_0$ and $t\in[t_{n-1},t_n]$,  since 
\beno 
z_n=z_{+,n,\text{(lin)}}\theta_{L}(|x^-|),\quad z_+=z_{+,n,\text{(lin)}}+z_{+,n,\text{(non)}},
\eeno
it follows from \eqref{eq:d9} and \eqref{eq:d1} that
\begin{equation*}
\begin{aligned}
z_n|_{t=t_{n-1}}&=z_{+}\theta_{L}(|x^-|)|_{t=t_{n-1}}\overset{\text {(if } n\geq 2)}{=} z_{n-1}|_{t=t_{n-1}}+w_{n-1}|_{t=t_{n-1}},
\end{aligned}
\end{equation*}
where 
\beno 
w_n:=z_{+, n,\text{(non)}}\theta_{L}(|x^-|),\quad\text{for}\quad 1\leq n\leq n_0.
\eeno
Therefore, for $2\leq n\leq n_0$, we have
\begin{equation}\label{eq:d43}
\begin{aligned}
\int_{\mathbb{R}^{3}}\bigl|\widehat{\widetilde{z_n}}(t_{n-1}, \xi)\bigr|^2 \theta\Bigl(\frac{|\xi|}{\delta(T)}\Bigr) d\xi
&\leq  \,2\int_{\mathbb{R}^{3}}\bigl|\widehat{\widetilde{z_{n-1}}}(t_{n-1},\xi)\bigr|^2 \theta\Bigl(\frac{|\xi|}{\delta(T)}\Bigr) d\xi\\ 
&\qquad+ 2\int_{\mathbb{R}^{3}}\bigl|\widehat{\widetilde{w_{n-1}}}(t_{n-1}, \xi)\bigr|^2 \theta\Bigl(\frac{|\xi|}{\delta(T)}\Bigr) d\xi.
\end{aligned} 
\end{equation}

{\bf Step 1.  Estimate on $\int_{\mathbb{R}^{3}}\bigl|\widehat{\widetilde{z_n}}(t_{n}, \xi)\bigr|^2\theta\Bigl(\frac{|\xi|}{\delta(T)}\Bigr) d\xi $ for $1\leq n\leq n_0$.}
By \eqref{eq:d28a}, we have
\begin{equation}\label{eq:d44}
\begin{aligned}
 &\int_{\mathbb{R}^{3}}\bigl|\widehat{\widetilde{z_n}}(t_{n}, \xi)\bigr|^2\theta\Bigl(\frac{|\xi|}{\delta(T)}\Bigr) d\xi 
 = \int_{\mathbb{R}^{3}} e^{-2\mu\left(t_{n}-t_{n-1}\right)|\xi|^2}\bigl|\widehat{\widetilde{z_{n}}}(t_{n-1} ,\xi)\bigr|^2 \theta\Bigl(\frac{|\xi|}{\delta(T)}\Bigr) d\xi \\
&\qquad\qquad\qquad + 2\operatorname{Re} \int_{t_{n-1}}^{t_{n}} \int_{\mathbb{R}^{3}} e^{-2\mu\left(t_{n}-s\right)|\xi|^{2}} \theta\Bigl(\frac{|\xi|}{\delta(T)}\Bigr)\Big(\widehat{J_{n}}\cdot \overline{\widehat{\widetilde{z_n}}}\Big)(s, \xi) \,d\xi ds .
\end{aligned} 
\end{equation}
The RHS of \eqref{eq:d44}  is analogous to that of \eqref{eq:d29}, with  $h(t)$ and $t$ replaced by $\delta(T)$ and $t_n$, respectively. Therefore, using a derivation similar to that leading to \eqref{eq:d34} and  $t_n-t_{n-1}=T$, we obtain
\begin{equation}\label{eq:d45}
\begin{aligned}
 \int_{\mathbb{R}^{3}}\bigl|\widehat{\widetilde{z_n}}(t_{n}, \xi)\bigr|^2\theta\Bigl(\frac{|\xi|}{\delta(T)}\Bigr) d\xi 
 &\leq C\big(\log\delta(T)^{-1}\big)^{-1}\cdot\bigl(\mathcal{E}^{w}(0)\bigr)^{\f32} + \frac{C\mu T}{L}\cdot\mathcal{E}^{w}(0) \\
& \qquad + \int_{\mathbb{R}^{3}} \bigl|\widehat{\widetilde{z_{n}}}(t_{n-1}, \xi)\bigr|^2\cdot\theta\Bigl(\frac{|\xi|}{\delta(T)}\Bigr) d\xi,
\end{aligned} 
\end{equation}
where the inequality $\delta^{\f34}\lesssim \big(\log\delta^{-1}\big)^{-1}$ for sufficiently small $\delta$ has been applied to determine the coefficient of the first term on the RHS of \eqref{eq:d45}.

{\bf Step 2. Estimate on $\int_{\mathbb{R}^{3}}\bigl|\widehat{\widetilde{w_n}}(t_n, \xi)\bigr|^2 \theta\Bigl(\frac{|\xi|}{\delta(T)}\Bigr) d\xi$ for $1\leq n\leq n_0$.}
From the definition of $w_{n}$ and \eqref{eq:d2}, it follows that
\begin{equation}\label{eq:d46}
\left\{
\begin{aligned}
&\p_t w_n+ Z_-\cdot\nabla w_n- \mu\Delta w_n= -\nabla p\,\theta_{L}(|x^-|)+g_n,\quad t\in\left(t_{n-1}, t_{n}\right] \\ 
&w_n|_{t=t_{n-1}}=0,
\end{aligned}
\right. 
\end{equation}
where 
\beno 
g_n:=-2\mu\nabla\big(\theta_{L}(|x^-|)\big) \cdot\nabla z_{+, n,\text{(non)}}-\mu\Delta\big(\theta_{L}(|x^-|)\big) z_{+, n,\text{(non)}}.
\eeno
We rewrite \eqref{eq:d46} in the Lagrangian coordinates as follows:
\beq\label{eq:d47}
\partial_{t} \widetilde{w_n}-\mu\Delta \widetilde{w_{n}}=\underbrace{\mu\operatorname{div}\big((A_{-}- I)\nabla\widetilde{w_n}\big)}_{K_n^1}+\underbrace{\widetilde{g_n}}_{K_n^2}-\underbrace{\widetilde{\nabla p\, \theta_{L}(|x^-|)}}_{K_n^3}=: K_n, \quad t\in(t_{n-1}, t_{n}].
\eeq
 Similar to \eqref{eq:d44}, we obtain
\begin{equation}\label{eq:d48}
\int_{\mathbb{R}^{3}}\bigl|\widehat{\widetilde{w_{n}}}(t_{n}, \xi)\bigr|^2 \theta\Bigl(\frac{|\xi|}{\delta(T)}\Bigr) d\xi = 2\operatorname{Re} \int_{t_{n-1}}^{t_{n}} \int_{\mathbb{R}^{3}} e^{-2\mu\left(t_{n}-s\right)|\xi|^{2}} \theta\Bigl(\frac{|\xi|}{\delta(T)}\Bigr)\Big(\widehat{K_{n}} \cdot \overline{\widehat{\widetilde{w_{n}}}}\Big)(s, \xi)\, d\xi ds . 
\end{equation}

Now, we define $b_{1}, b_{2}\in\mathcal{S}(\mathbb{R}^{3})$ by
\begin{equation}
\begin{aligned}
b_{1}(y)&=\mathcal{F}^{-1}\Bigl(\theta\Bigl(\frac{|\xi|}{\delta(T)}\Bigr)\Bigr)=\delta(T)^{3} \mathcal{F}^{-1}(\theta)(\delta(T) y),\\ 
b_{2}(s,y)&=\mathcal{F}^{-1}\Bigl(e^{-2\mu\left(t_{n}-s\right)|\xi|^{2}}\Bigr)=\left(2 \sqrt{\mu\left(t_{n}-s\right)}\right)^{-3} e^{-\frac{|y|^{2}}{8 \mu\left(t_{n}-s\right)}},
\end{aligned} \label{eq:d48star}
\end{equation}
which satisfy
\beno
\|b_1\|_{L^1}+\|b_2(s,\cdot)\|_{L^1}\lesssim 1, \quad \|b_1\|_{L^2}\lesssim \delta(T)^{\frac{3}{2}}.
\eeno
Applying Plancherel's theorem, we then derive
\begin{equation}\label{eq:d49}
\begin{aligned}
& \left|\int_{t_{n-1}}^{t_{n}} \int_{\mathbb{R}^3} e^{-2\mu({t_{n}}-s)|\xi|^2} \theta\Bigl(\frac{|\xi|}{\delta(T)}\Bigr)\Big(\widehat{K_n}\cdot \overline{\widehat{\widetilde{w_n}}}\Big)(s,\xi)\, d\xi ds \right|
\leq\sum_{j=1}^{3}\underbrace{\left|\int_{t_{n-1}}^{t_{n}} \int_{\mathbb{R}^3} K_n^j\cdot b_1 * b_2 * \widetilde{w_n}\, dy ds \right|}_{P_{nj}} .
\end{aligned} 
\end{equation}

{\it Step 2.1. Estimates for $P_{n1}$ and $P_{n2}$.} Following derivations analogous to those in \eqref{eq:d32} and \eqref{eq:d33}, by virtue of \eqref{eq:d3}, we obtain
\beq\label{eq:d50}
P_{n1}\lesssim\big(\log\delta(T)^{-1}\big)^{-1}\cdot\bigl(\mathcal{E}^w(0)\bigr)^{\f32},\quad
P_{n2}\lesssim\f{\mu}{L}(t_n-t_{n-1})\cdot\mathcal{E}^w(0)
=\f{\mu T}{L}\cdot\mathcal{E}^w(0).
\eeq

{\it Step 2.2.  Estimate for $P_{n3}$.} Using an approach similar to the derivation of \eqref{eq:d31} for $F_{n1}$, applying H\"{o}lder's and Young's inequalities, we have
\beno
\begin{aligned}
P_{n3} & \leq \int_{t_{n-1}}^{t_{n}} \Big(\int_{|y|\leq \delta(T)^{-\frac{1}{2}}}+ \int_{|y|\geq \delta(T)^{-\frac{1}{2}}}\Big) \bigl|\widetilde{\nabla p\,\theta_{L}(|x^-|)}\bigr|\cdot|b_1 * b_2 *\widetilde{w_n}|\,dy ds\\
& \lesssim \Big(\delta(T)^{\frac{3}{4}}+ \big(\log\delta(T)^{-1}\big)^{-1}\Big) \int_{t_{n-1}}^{t_{n}} \bigl\|\log \left(R^2+|y|^2\right)^{\frac{1}{2}}\widetilde{\nabla p\,\theta_{L}(|x^-|)}\bigr\|_{L^2}\cdot\|\widetilde{w_n}\|_{L^{2}} ds \\
& \lesssim \big(\log\delta(T)^{-1}\big)^{-1} \int_{t_{n-1}}^{t_{n}} \|\log \langle w_{-}\rangle \nabla p\|_{L^2(\Sigma_{s})}\cdot\|z_{+,n,\text{(non)}}\|_{L^{2}(\Sigma_{s})} ds,
\end{aligned}
\eeno
which together with \eqref{eq:d3} implies
\beq\label{eq:d51}
P_{n3}\lesssim\big(\log\delta(T)^{-1}\big)^{-1}\cdot \big(\mathcal{E}^w(0)\big)^{\frac{1}{2}}\cdot\int_{t_{n-1}}^{t_{n}} \left\|\log \langle w_{-}\rangle \nabla p \right\|_{L^{2}\left(\Sigma_{s}\right)}ds.
\eeq
Hence, it suffices to control $\int_{0}^{T^*} \left\|\log \langle w_{-}\rangle \nabla p \right\|_{L^{2}\left(\Sigma_{\tau}\right)} d\tau$.

By the expression of $\na p$ in \eqref{eq:pressure}, we decompose $\nabla p$ as
\beno
\na p=\na p_0+\na p_1+\na p_{\geq 2}+\vv B_1+\vv B_2,
\eeno
where $\na p_1,\, \na p_{\geq 2}$, $\vv B_1$ and $\vv B_2$ are as defined in the proof of Proposition \ref{prop:p1}, and 
\beno
\na p_0(t,x):=\frac{1}{4\pi}\sum_{i,j=1}^3 \int_{Q_L} \nabla_x \Big(\frac{1}{|x-y|}\Big)\left(\partial_{i} z_{+}^{j} \partial_{j} z_{-}^{i}\right)(\tau, y) dy.
\eeno

{\underline{\it i). Estimate for $\int_{0}^{T^*} \left\|\log \langle w_{-}\rangle \nabla p_0 \right\|_{L^{2}\left(\Sigma_{\tau}\right)} d\tau$.}} 
Let $\theta(r)\in C_c^\infty(\R)$ be the cut-off function satisfying $\theta(r)=1$ for $|r|\leq 1$ and $\theta(r)=0$ for $|r|\geq 2$. We split $\na p_0$ into two parts as follows:
\begin{equation*}
\begin{aligned}
\na p_0(\tau, x)= & \underbrace{\frac{1}{4\pi}\sum_{i,j=1}^3 \int_{Q_L} \nabla_x\Big(\frac{1}{|x-y|}\Big)\cdot \theta(|x-y|)\cdot \left(\partial_{i} z_{+}^{j} \partial_{j} z_{-}^{i}\right)(\tau, y) dy}_{A_1(\tau, x)} \\
& +\underbrace{\frac{1}{4\pi} \sum_{i,j=1}^3\int_{Q_L} \nabla_x\Big(\frac{1}{|x-y|}\Big)\cdot \big(1-\theta(|x-y|)\big)\cdot \left(\partial_{i} z_{+}^{j} \partial_{j} z_{-}^{i}\right)(\tau, y) dy}_{A_2(\tau, x)} .
\end{aligned}
\end{equation*}

For $A_1$, since \eqref{eq:p4} (Lemma \ref{lem:p1}) implies $\log \langle w_{-}\rangle(\tau,x)\lesssim\log \langle w_{-}\rangle(\tau,y) $ for $|x-y|\leq 2$, and the seperation property \eqref{eq:e13} of $z_{\pm}$ gives rise to
\beq\label{eq:d52}
\f{1}{\langle w_-\rangle^{\f12}\log \langle w_-\rangle\cdot\langle w_+\rangle^{\f12}\log \langle w_+\rangle}\lesssim \f{1}{\langle\tau\rangle^{\f12}\log \langle\tau\rangle},
\eeq
we have
\beno\begin{aligned}
\log \langle w_{-}\rangle|A_1(\tau, x)|
&\lesssim \int_{|x-y|\leq 2} \frac{\log \langle w_{-}\rangle(\tau, y)|\nabla z_{+}(\tau, y)|\cdot |\nabla z_{-}(\tau, y)|}{|x-y|^{2}}\cdot 1_{Q_L}(y)dy\\ 
&\lesssim\frac{1}{\langle\tau\rangle^{\f12}\log \langle\tau\rangle}\cdot\int_{|x-y|\leq 2} \frac{1}{|x-y|^{2}} \cdot \Bigl(\frac{\langle w_{-}\rangle^{\frac{1}{2}}(\log \langle w_{-}\rangle)^2\left|\nabla z_{+}\right|}{\langle w_{+}\rangle^{\frac{1}{2}} \log \langle w_{+}\rangle}\\ 
&\qquad\qquad\qquad\qquad\cdot \langle w_{+}\rangle\left(\log \langle w_{+}\rangle\right)^2\left|\nabla z_{-}\right|\Bigr)(\tau, y)\cdot 1_{Q_L}(y) dy.
\end{aligned}\eeno 
Applying Young's inequality, we obtain
\begin{equation}\label{eq:d53}
\begin{aligned}
\|\log \langle w_{-}\rangle A_1\|_{L^{2}\left(\Sigma_{\tau}\right)}&\lesssim \f{1}{\langle\tau\rangle^{\f12}\log \langle\tau\rangle}\cdot\Bigl\|\frac{1}{|x|^{2}}\Bigr\|_{L^{1}(|x|\leq2)} \cdot \Bigl\|\frac{\langle w_-\rangle^{\frac{1}{2}}(\log \langle w_-\rangle)^2 \nabla z_+}{\langle w_+\rangle^{\f12} \log \langle w_+\rangle}\Bigr\|_{L^2(\Sigma_{\tau})}\\ 
&\qquad\cdot \|\langle w_+\rangle(\log \langle w_+\rangle)^2 \nabla z_-\|_{L^{\infty}(\Sigma_{\tau})}\\
& \lesssim \f{1}{\langle\tau\rangle^{\f12}\log \langle\tau\rangle}\cdot  \left\|\frac{\langle w_{-}\rangle \nabla z_{+}}{\langle w_{+}\rangle^{\frac{1}{2}} \log \langle w_{+}\rangle}\right\|_{L^{2}\left(\Sigma_{\tau}\right)}\cdot \Big(\sum_{k=0}^2 E_{-}^k \Big)^{\frac{1}{2}},
\end{aligned} 
\end{equation}
where we used the fact that $(\log \langle w_{-}\rangle)^2\lesssim \langle w_{-}\rangle^{\frac{1}{2}}$ and the weighted Sobolev inequality \eqref{eq:e12} in the last inequality.

Similarly, for $A_2$, since \eqref{eq:p5} (Lemma \ref{lem:p1}) implies
\beno 
\log\langle w_{-}\rangle(\tau,x)\lesssim\log(13|x-y|)\cdot \log\langle w_{-}\rangle(\tau,y),\quad\text{for}\quad|x-y|\geq 1,
\eeno
 using \eqref{eq:d52} and Young's inequality, we obtain
\begin{equation}\label{eq:d54}
\begin{aligned}
\|\log \langle w_-\rangle A_2\|_{L^{2}(\Sigma_{\tau})} 
&\lesssim \f{1}{\langle\tau\rangle^{\f12}\log \langle\tau\rangle}\cdot \Bigl\|\frac{\log (13|x|)}{|x|^{2}} \Bigr\|_{L^{2}(|x|\geq 1)} \cdot \Bigl\|\frac{\langle w_-\rangle^{\frac{1}{2}}(\log \langle w_{-}\rangle)^2 \nabla z_+}{\langle w_+\rangle^{\frac{1}{2}} \log \langle w_+\rangle}\Bigr\|_{L^2(\Sigma_{\tau})}\\ 
&\qquad\cdot\|\langle w_+\rangle(\log \langle w_+\rangle)^2 \nabla z_-\|_{L^{2}(\Sigma_{\tau})}\\
& \lesssim\f{1}{\langle\tau\rangle^{\f12}\log \langle\tau\rangle}\cdot \Bigl\|\frac{\langle w_-\rangle \nabla z_+}{\langle w_+\rangle^{\frac{1}{2}} \log \langle w_+\rangle}\Bigr\|_{L^{2}(\Sigma_{\tau})}\cdot\left(E_{-}^0\right)^{\frac{1}{2}} .
\end{aligned} 
\end{equation}

Therefore, thanks to \eqref{eq:d53} and \eqref{eq:d54}, we get
\begin{equation}\label{eq:d55}
\begin{aligned}
\int_{0}^{T^*}\|\log \langle w_{-}\rangle
\na p_0\|_{L^{2}(\Sigma_{\tau})} d\tau
&\lesssim \Big(\sum_{k=0}^2 E_{-}^k \Big)^{\frac{1}{2}}\cdot\Bigl\|\f{1}{\langle\tau\rangle^{\f12}\log \langle\tau\rangle}\Bigr\|_{L^2(0,T^*)}\cdot\Bigl\|\frac{\langle w_{-}\rangle \nabla z_{+}}{\langle w_{+}\rangle^{\frac{1}{2}} \log \langle w_{+}\rangle}\Bigr\|_{L_{T^*}^{2}(L_x^2)} \\
&\overset{\eqref{eq:p1-3}}{\lesssim}\Big(\sum_{k=0}^2 E_{-}^k \Big)^{\frac{1}{2}}\cdot\left(\frac{1}{L\log L}E_{+}^{0}+ F_{+}^{0}\right)^{\frac{1}{2}} \lesssim \mathcal{E}^w(0).
\end{aligned} 
\end{equation}

{\underline{\it ii). Estimate for  $\int_{0}^{T^*} \|\log \langle w_{-}\rangle(\na p-\na p_0) \|_{L^{2}(\Sigma_{\tau})} d\tau$.}} Since $\na p-\na p_0=\na p_1+\na p_{\geq 2}+\vv B_1+\vv B_2$, using the estimates \eqref{eq:p16b}, \eqref{eq:p18}, \eqref{eq:p21}, and \eqref{eq:p22} for $\na p_1$, $\na p_{\geq 2}$, $\vv B_1$,  and $\vv B_2$, respectively,   we derive
\beno\begin{aligned}
&\quad\|\log \langle w_{-}\rangle(\na p-\na p_0) \|_{L^\infty_t(L^2_x)}\lesssim\|(\log \langle w_{-}\rangle)^2(\na p_1+\na p_{\geq 2}+\vv B_1+\vv B_2) \|_{L^\infty_t(L^2_x)}\\ 
&\lesssim\f{(\log L)^2}{L^{\f12}}\cdot \Bigl(\sum_{l=0}^2E_-^l\Bigr)^{\f12}\cdot \bigl(E_+^0\bigr)^{\f12}
+\frac{\left(\log L\right)^2}{L^{\f52}} E_{+}^{\frac{1}{2}} E_{-}^{\frac{1}{2}}+\f{1}{L^{\f32}}\cdot \left(E_{-}+E_{-}^0+ E_{-}^1\right)^{\frac{1}{2}}\cdot \left(E_{+}+ E_{+}^0\right)^{\frac{1}{2}}\\ 
&\lesssim\Bigl(\f{(\log L)^2}{L^{\f12}}+\frac{\left(\log L\right)^2}{L^{\f52}}+\f{1}{L^{\f32}}\Bigr)\cdot\mathcal{E}^w(0)
\overset{L\gg1}{\lesssim}\f{(\log L)^2}{L^{\f12}}\cdot\mathcal{E}^w(0),\quad\forall\,0\leq t\leq T^*=\log L,
\end{aligned}\eeno
which implies
\begin{equation}\label{eq:d56}
\int_{0}^{T^*} \left\|\log \langle w_{-}\rangle(\na p-\na p_0)\right\|_{L^{2}\left(\Sigma_{\tau}\right)} d\tau\lesssim \frac{(\log L)^3}{L^{\frac{1}{2}}}\cdot\mathcal{E}^w(0)\lesssim \mathcal{E}^w(0). 
\end{equation}

Combining \eqref{eq:d55} and \eqref{eq:d56}, we conclude that
\begin{equation}\label{eq:d57}
\int_{0}^{T^*} \left\|\log \langle w_{-}\rangle \nabla p\right\|_{L^{2}\left(\Sigma_{\tau}\right)} d\tau\lesssim \mathcal{E}^w(0),
\end{equation}
from which and  \eqref{eq:d51}, we deduce that
\begin{equation}\label{eq:d58}
P_{n3} \lesssim \big(\log\delta(T)^{-1}\big)^{-1}\big(\mathcal{E}^{w}(0)\big)^{\frac{3}{2}} . 
\end{equation}

{\it Step 2.3. Estimate for $\int_{\mathbb{R}^{3}}\bigl|\widehat{\widetilde{w_n}}(t_n, \xi)\bigr|^2 \theta\Bigl(\frac{|\xi|}{\delta(T)}\Bigr) d\xi$.}
Thanks to \eqref{eq:d50}, \eqref{eq:d58} and the fact that $t_n-t_{n-1}=T\leq\log L$, we  deduce from \eqref{eq:d48} and \eqref{eq:d49} that
\begin{equation}\label{eq:d59}
\begin{aligned}
&\int_{\mathbb{R}^3} \bigl|\widehat{\widetilde{w_{n}}}(t_n, \xi)\bigr|^2 \theta\Bigl(\frac{|\xi|}{\delta(T)}\Bigr)d\xi \lesssim\big(\log\delta(T)^{-1}\big)^{-1} \big(\mathcal{E}^w(0)\big)^\frac{3}{2} + \frac{\mu\log L}{L}\mathcal{E}^w(0). 
\end{aligned} 
\end{equation}

{\bf Step 3. Estimate on $\int_{\mathbb{R}^{3}}\left|\widehat{\widetilde{z_n}}(t_{n-1}, \xi)\right|^{2} \theta\left(\frac{|\xi|}{\delta(T)}\right) d\xi$.}
By virtue of \eqref{eq:d45}, \eqref{eq:d59}, and the fact that $T\leq\log L$,  we derive from \eqref{eq:d43} that
\begin{equation}\label{eq:d60}
\begin{aligned}
&\int_{\mathbb{R}^{3}}\left|\widehat{\widetilde{z_n}}(t_{n-1}, \xi)\right|^{2} \theta\left(\frac{|\xi|}{\delta(T)}\right) d\xi\leq 2\int_{\mathbb{R}^{3}}\left|\widehat{\widetilde{z_{n-1}}}(t_{n-2}, \xi)\right|^{2} \theta\left(\frac{|\xi|}{\delta(T)}\right) d\xi\\ 
&\qquad\qquad\qquad+ C\big(\log\delta(T)^{-1}\big)^{-1} \big(\mathcal{E}^w(0)\big)^\frac{3}{2}+ \frac{C\mu\log L}{L}\mathcal{E}^w(0),\quad\text{for}\,\, 2\leq n\leq n_0.
\end{aligned} 
\end{equation}
Iterating this estimate yields
\begin{equation}\label{eq:d61}
\begin{aligned}
&\int_{\mathbb{R}^{3}}\left|\widehat{\widetilde{z_n}}(t_{n-1}, \xi)\right|^{2} \theta\left(\frac{|\xi|}{\delta(T)}\right) d\xi \leq 2^{n-1}\int_{\mathbb{R}^{3}}\left|\widehat{\widetilde{z_{1}}}(0, \xi)\right|^{2} \theta\left(\frac{|\xi|}{\delta(T)}\right) d\xi\\ 
&\qquad\qquad\qquad+ C 2^{n-1}\big(\log\delta(T)^{-1}\big)^{-1} \big(\mathcal{E}^w(0)\big)^\frac{3}{2}+ \frac{C\mu\log L}{L}2^{n-1}\mathcal{E}^w(0).
\end{aligned} 
\end{equation}

To bound the first term on the RHS of \eqref{eq:d61}, we take $b_{1}\in\mathcal{S}(\mathbb{R}^{3})$ as defined in \eqref{eq:d48star}, through a similar derivation as to \eqref{eq:d51}, by using H\"{o}lder's and Young's inequalities, we obtain
\beno
\begin{aligned}
&\int_{\mathbb{R}^3}\left|\widehat{\widetilde{z_1}}(0, \xi)\right|^2\theta\left(\frac{|\xi|}{\delta(T)}\right) d\xi \lesssim \left|\int_{\mathbb{R}^3} \widetilde{z_1}(0, y)\cdot b_1 * \widetilde{z_1}(0,y)\, dy \right| \\
&\qquad \lesssim \Big(\int_{|y|\leq \delta(T)^{-\frac{1}{2}}}+ \int_{|y|\geq \delta(T)^{-\frac{1}{2}}}\Big) |\widetilde{z_1}(0, y)|\cdot|b_1 * \widetilde{z_1}(0, y)|\,dy\\
&\qquad \lesssim \Big(\delta(T)^{\frac{3}{4}}+ \big(\log\delta(T)^{-1}\big)^{-1}\Big)\bigl\|\log \left(R^2+|y|^2\right)^{\frac{1}{2}}\widetilde{z_1}(0,y)\bigr\|_{L^2}^2 \\
&\qquad \lesssim \big(\log\delta(T)^{-1}\big)^{-1}\|\log\langle x\rangle z_1(0, x)\|_{L^2(\mathbb{R}^3)}^2,
\end{aligned} 
\eeno
which  along with the fact that $\|\log\langle x\rangle z_1(0, x)\|_{L^2(\mathbb{R}^3)}\leq\|\log \langle w_{-}\rangle z_{+}\|_{L^{2}(\Sigma_{0})}$ implies
\beq\label{eq:d61star}
\int_{\mathbb{R}^3}\left|\widehat{\widetilde{z_1}}(0, \xi)\right|^2\theta\left(\frac{|\xi|}{\delta(T)}\right) d\xi \lesssim \left|\int_{\mathbb{R}^3} \widetilde{z_1}(0, y)\cdot b_1 * \widetilde{z_1}(0,y)\, dy \right|\lesssim \big(\log\delta(T)^{-1}\big)^{-1}\mathcal{E}^w(0). 
\eeq

Using \eqref{eq:d61star} and the facts that $T\leq\log L$, $L\gg1$ and $\mathcal{E}^w(0)\leq1$, we deduce from \eqref{eq:d61} that for $1\leq n\leq n_0$,
\begin{equation}\label{eq:d62}
\begin{aligned}
&\int_{\mathbb{R}^{3}}\left|\widehat{\widetilde{z_n}}(t_{n-1}, \xi)\right|^{2} \theta\Bigl(\frac{|\xi|}{\delta(T)}\Bigr) d\xi\lesssim 2^{n}\Bigl(\big(\log\delta(T)^{-1}\big)^{-1}+\frac{\mu\log L}{L}\Bigr)\mathcal{E}^w(0) \\
&\qquad \lesssim 2^{n}\Bigl(\f{1}{\log\big(\log(\mu T+e)+e\big)}+\frac{\mu\log L}{L}\Bigr)\mathcal{E}^w(0) \\
&\qquad \lesssim \f{2^n}{\log\big(\log(\mu T+e)+e\big)}\cdot\mathcal{E}^w(0).
\end{aligned} 
\end{equation}
This establishes \eqref{eq:d42} and the proof of the lemma is completed.
\end{proof}

\subsection{Decay mechanism and global extension.}
In this subsection, we show that by time $T^* = \log L$, the total unweighted energy of the local solution decays to a value $\e_\mu$ that is sufficiently small compared to the viscosity $\mu$. The solution then enters the classical small-data parabolic regime,  enabling the local solution on $[0, \log L]$ to be extended globally. This completes the proof of Theorem \ref{thm:d1}.

\begin{proof}[Proof of Theorem \ref{thm:d1}] We divide the proof into two steps.

{\bf Step 1. Decay mechanism of the energy.}  Thanks to Lemmas \ref{lem:d2}, \ref{lem:d3}, and \ref{lem:d4}, we obtain for $1\leq n\leq n_0$,
\beno
\begin{aligned}
\mathcal{E}_{+,\text{(lin)}}(t_n)  &\lesssim \frac{\log\big(\log(\mu T+e)+e\big)}{\log\big(\mu T+e\big)}\cdot \mathcal{E}^{w}(0)+\bigl(\mathcal{E}^w(0)\bigr)^{\f12}\cdot \mathcal{E}(t_{n-1}) \\ 
&\qquad+\f{2^{n}}{\log\big(\log(\mu T+e)+e\big)}\cdot\mathcal{E}^w(0) + \frac{\mathcal{E}^{w}(0)}{\left(\log L\right)^{4}}, 
\end{aligned}
\eeno
which along with the facts that $T<\log L$ and $L\gg 1$ implies
\beno\begin{aligned}
\mathcal{E}_{+,\text{(lin)}}(t_n)
&\lesssim\Bigl(\frac{\log\big(\log(\mu T+e)+e\big)}{\log\big(\mu T+e\big)}+\f{2^{n}}{\log\big(\log(\mu T+e)+e\big)}\Bigr)\cdot\mathcal{E}^w(0) +\bigl(\mathcal{E}^w(0)\bigr)^{\f12}\cdot \mathcal{E}(t_{n-1}).
\end{aligned}\eeno
The same estimate also holds for $\mathcal{E}_{-,\text{(lin)}}(t_n)$. Combining these estimates and \eqref{eq:d4a} in Lemma \ref{lem:d1}, there exists a constant $C_0>1$ such that for all $1\leq n\leq n_0$,
\beq\label{eq:d64}\begin{aligned}
\mathcal{E}(t_n)
&\leq C_0\Bigl(\frac{\log\big(\log(\mu T+e)+e\big)}{\log\big(\mu T+e\big)}+\f{2^{n}}{\log\big(\log(\mu T+e)+e\big)}\Bigr)\cdot\mathcal{E}^w(0)+C_0\bigl(\mathcal{E}^w(0)\bigr)^{\f12}\cdot \mathcal{E}(t_{n-1}).
\end{aligned}\eeq 

{\bf Step 2. Decay to the small-data parabolic regime.} First, there exists a constant $\e_1\in(0,\e_0]$ such that
\beno 
C_0\bigl(\mathcal{E}^w(0)\bigr)^{\f12}\leq C_0\e_1<\f12,
\eeno
provided that $\mathcal{E}^w(0)\leq\e_1^2$. Let $\e_\mu=O(\mu)$ be the threshold for entering the small-data parabolic regime. That is, if the initial energy $\mathcal{E}(0)$ is less than $\e_\mu^2$, the MHD system admits a unqiue global solution.

In view of the iteration scheme \eqref{eq:d64}, we take $n_{0}=\Bigl\lfloor \frac{2\log \varepsilon_{\mu}}{\log\bigl(2C_0(\mathcal{E}^w(0))^{\f12}\bigr)}\Bigr\rfloor-1$, such that
\beq\label{eq:d65}
\left(2C_0\bigl(\mathcal{E}^w(0)\bigr)^{\f12}\right)^{n_{0}+2} < \varepsilon_{\mu}^2 \leq \left(2C_0\bigl(\mathcal{E}^w(0)\bigr)^{\f12}\right)^{n_{0}+1},
\eeq
 where $\left\lfloor m \right\rfloor$ denotes the maximum integer that does not exceed $m$. Then there exists a sufficiently large constant $L_\mu\geq e^{\f{1}{\mu}}$ such that for any $L\geq L_\mu$ and with $T=\f{\log L}{n_0}$ chosen large enough, we have 
 \beno 
\frac{\log\big(\log(\mu T+e)+e\big)}{\log(\mu T+e)}<\f{\varepsilon_{\mu}^2}{4C_0} 
\quad\text{and}\quad\f{2^{n_0}}{\log\big(\log(\mu T+e)+e\big)}<\f{\varepsilon_{\mu}^2}{4C_0}.
\eeno
Hence, we deduce from \eqref{eq:d64} that
\beno
\mathcal{E}(t_n)\leq \f{\e_\mu^2}{2}\cdot\mathcal{E}^w(0)+C_0\bigl(\mathcal{E}^w(0)\bigr)^{\f12}\cdot \mathcal{E}(t_{n-1}).
\eeno
Consequently, by iteration and the assumption that $C_0\bigl(\mathcal{E}^w(0)\bigr)^{\f12}<\f12$, we obtain
\beno\begin{aligned}
\mathcal{E}(t_n)&\leq \f{\e_\mu^2}{2}\cdot\sum_{k=0}^{n-1}\bigl(C_0(\mathcal{E}^w(0))^{\f12}\bigr)^k\cdot\mathcal{E}^w(0)+\bigl(C_0(\mathcal{E}^w(0))^{\f12}\bigr)^n\mathcal{E}(t_0)\\ 
&
\leq\e_\mu^2\cdot\mathcal{E}^w(0)+\bigl(C_0(\mathcal{E}^w(0))^{\f12}\bigr)^n\mathcal{E}^w(0),
\end{aligned}\eeno
which together with \eqref{eq:d65} implies
\beq\label{eq:d66}
\mathcal{E}(t_n)\leq\bigl(2C_0(\mathcal{E}^w(0))^{\f12}\bigr)^{n+2},\quad\text{for}\quad 1\leq n\leq n_0.
\eeq
Then the following estimate holds:
\beq\label{eq:d67}
\mathcal{E}(T^*)=\mathcal{E}(t_{n_0})\leq\bigl(2C_0(\mathcal{E}^w(0))^{\f12}\bigr)^{n_0+2}\leq\e_\mu^2.
\eeq
In particular, the $H^2$-norm of the system at time $T^*$ is bounded by $\e_\mu$. Hence the MHD system enter the classical small-data parabolic regime, and the local solution $(z_+,z_-)$ obtained in Therorem \ref{thm:a1} can be extended  globally. This completes the proof of Theorem \ref{thm:d1}.
\end{proof}

\section{Appendix}

\subsection{Divergence theorem for piecewise smooth vector fields}
In this subsection, we state a modification of  divergence theorem for piecewise smooth vector fields as follows:
\begin{lemma}[Divergence Theorem for Piecewise Smooth Vector Fields]\label{lem:divergence}
Let $\Omega\subset\R^n$ be a domain partitioned into $\Omega = \bigcup_{k=1}^N \Omega_k$ with pairwise disjoint interiors and each $\Omega_k$ be a Lipschitz domain. Denote by $\Gamma_{ik} = \partial\Omega_i \cap \partial\Omega_k$ the interfaces between subdomains. For a piecewise smooth vector field $\mathbf{F}$ satisfying $\mathbf{F}|_{\Omega_k} \in C^1(\overline{\Omega_k})$ for each $1 \leq k \leq N$, the following divergence theorem holds
\beq\label{eq:div thm}
\int_\Omega \na\cdot\vv Fdx=\int_{\p\Omega}\vv F\cdot\vv ndS+\sum_{i<k}\int_{\Gamma_{ik}}[\vv F]_{ik}\cdot\vv n_{ik} dS,
\eeq
where $\vv n$ is the outward unit normal  on $\p\Omega$, $\vv n_{ik}$ is the unit normal vector pointing from $\Omega_i$ to $\Omega_k$ and $[\vv F]_{ik}=\vv F|_{\Omega_i}-\vv F|_{\Omega_k}$ denotes the jump of $\vv F$ across $\Gamma_{ik}$.
\end{lemma}
\begin{proof}
Since $\mathbf{F}|_{\Omega_k} \in C^1(\overline{\Omega_k})$ for each $k$, we may apply the classical divergence theorem to obtain
\beno 
\int_{\Omega_k}\dive\vv Fdx=\int_{\p\Omega_k}\vv F\cdot\vv n_k dS,\quad1\leq k\leq N,
\eeno
where $\vv n_k$ denotes the outward unit normal to $\partial\Omega_k$.
Summing over all subdomains and accounting for the interface jumps yields \eqref{eq:div thm}. The lemma is proved.
\end{proof}

\subsection{Proof of Lemma \ref{prop:b1}} In this subsection, we give detailed proof of Lemma \ref{prop:b1}.
\begin{proof}[Proof of Lemma \ref{prop:b1}]
We only derive the estimates for $\lambda_+$, as the estimates for $\lambda_{-}$ can be obtained in the same manner. We first verify \eqref{eq:b1} for 
$\lambda_{+}=\left(\log \langle w_{-}\rangle\right)^4$.

{\bf Step 1. Estimate of  $\bigl\|\frac{\lambda_+|_{x_1=L}-\lambda_+|_{x_1=-L}}{\lambda_+|_{x_1=L}}\bigr\|_{L^\infty_{(t, x_{2}, x_{3})}}$.}
By virtue of \eqref{eq:g16}, the periodicity relation $x_{3}^-(t, x_{1}, x_{2}, x_{3})=x_{3}^-(t, x_{1}+2L, x_{2}, x_{3})$ holds for all $(t, x)\in \left[0, T^{*}\right] \times \mathbb{R}^{3}$. Consequently, it suffices to prove
\begin{equation}\label{eq:b2}
\Bigl\|\frac{\left(\log \langle \overline{w_{-}}\rangle\right)^4|_{x_1=L}- \left(\log \langle \overline{w_{-}}\rangle\right)^4|_{x_1=-L}}{\left(\log \langle \overline{w_{-}}\rangle\right)^4|_{x_1=L}}\Bigr\|_{L^{\infty}([0,T^*]\times[-2L,L]^2)}
\lesssim \frac{\varepsilon}{L\left(\log L\right)^2} .
\end{equation}

For any fixed $(t,x_2,x_3) \in [0,T^*] \times [-2L,L]^2$, we define 
\beno
c:=\big(R^{2}+\left|x_{2}^{-}\right|^{2}+\left|x_{3}^{-}\right|^{2}\big)^{\frac{1}{2}}(t, L, x_{2}, x_{3}),\quad g(r):=\Big(\log \left(c^{2}+r^{2}\right)^{\frac{1}{2}}\Big)^{4}.
\eeno 
According to \eqref{eq:g16}, we have
\beno
x_{2}^{-}(t, L, x_{2}, x_{3})=x_{2}^{-}(t,-L, x_{2}, x_{3}), \quad x_{3}^{-}(t, L, x_{2}, x_{3})=x_{3}^{-}(t,-L, x_{2}, x_{3}), 
\eeno
which along with $\langle \overline{w_{-}}\rangle=\big(R^{2}+\left|x_{1}^{-}\right|^{2}+\left|x_{2}^{-}\right|^{2}+\left|x_{3}^{-}\right|^{2}\big)^{\frac{1}{2}}$ implies that 
\beno
\langle\overline{w_{-}}\rangle(t,\pm L,x_2,x_3)=(c^2+\left|x_{1}^{-}(t,\pm L,x_2,x_3)\right|^{2})^{\f12 }.
\eeno 
Then we get
\begin{equation}
\begin{aligned}
& \quad\quad\quad\bigl|\left(\log \langle \overline{w_{-}}\rangle\right)^4(t,L,x_2,x_3)- \left(\log \langle \overline{w_{-}}\rangle\right)^4(t,-L,x_2,x_3)\bigr| \\
& \quad\quad = \left|g\big(x_{1}^{-}(t,L,x_2,x_3)\big)- g\big(x_{1}^{-}(t,-L,x_2,x_3)\big)\right| \\
& \ \ \overset{\scriptscriptstyle g(L)=g(-L)}{\leq} \left|g\big(x_{1}^{-}(t,L,x_2,x_3)\big)- g(L)\right|+ \left|g(-L)- g\big(x_{1}^{-}(t,-L,x_2,x_3)\big)\right| \\
& \overset{\scriptscriptstyle \text{mean value theorem}}{=} \left|g^{\prime}(\xi_1)\right| \cdot \left|x_{1}^{-}(t,L,x_2,x_3)-L\right|+ \left|g^{\prime}(\xi_2)\right| \cdot \left|x_{1}^{-}(t,-L,x_2,x_3)-(-L)\right|,
\end{aligned} \label{eq:b4}
\end{equation}
where $\xi_1$ is a point between $x_{1}^{-}(t,L,x_2,x_3)$ and $L$, $\xi_2$ is a point between $x_{1}^{-}(t,-L,x_2,x_3)$ and $-L$. 

Thanks to \eqref{eq:e9a} in the proof of Lemma \ref{lem:e2}, we have
\beq\label{eq:b6} 
\left|x_{1}^{-}(t,L,x_2,x_3)-L\right|\leq\f32\int_{0}^{t} \left|z_{-}(\tau,L,x_2,x_3+t-\tau)\right|\,d\tau\leq\f{L}{20},
\eeq
which shows that $x_{1}^{-}(t,L,x_2,x_3)\in[\f{19}{20}L,\f{21}{20}L]$, i.e., $x_{1}^{-}(t,L,x_2,x_3)\sim L$. Then using weighted Sobolev inequality \eqref{eq:e11}, the $2L$-periodicity of $z_-$,  Lemma \ref{lem:e2} (3) and ansatz \eqref{eq:a3} and $T^*=\log L$, we deduce from \eqref{eq:b6} that
\begin{equation}\label{eq:b7}
\left|x_{1}^{-}(t,L,x_2,x_3)-L\right| 
\lesssim \frac{\left(E_{-}+ E_{-}^0+ E_{-}^1\right)^{\frac{1}{2}}}{\left(\log L\right)^{2}}\cdot T^* \lesssim \frac{\varepsilon}{\log L},
\end{equation}
Similarly, there also holds $\left|x_{1}^{-}(t,-L,x_2,x_3)-(-L)\right|\lesssim \frac{\varepsilon}{\log L}$.
Based on these estimates, we derive that $|\xi_{1}|\gtrsim L$ and $|\xi_{2}|\gtrsim L$, thereby
\beno
\left|g^{\prime}(\xi_1)\right| \lesssim \frac{\left(\log L\right)^{3}}{L},\quad \left|g^{\prime}(\xi_2)\right| \lesssim \frac{\left(\log L\right)^{3}}{L},
\eeno
from which and \eqref{eq:b7}, we deduce from \eqref{eq:b4} that
\begin{equation}\label{eq:b9}
\bigl|\left(\log \langle \overline{w_{-}}\rangle\right)^4(t,L,x_2,x_3)- \left(\log \langle \overline{w_{-}}\rangle\right)^4(t,-L,x_2,x_3)\bigr| \lesssim \frac{\left(\log L\right)^{2}}{L} \varepsilon . 
\end{equation}

On the other hand, since $x_{1}^{-}(t,L,x_2,x_3)\sim L$, we have $\left(\log \langle \overline{w_{-}}\rangle\right)^4(t,L,x_2,x_3)\gtrsim \left(\log L\right)^{4}$, which along with \eqref{eq:b9} gives rise to \eqref{eq:b2}. Thus, we get
\begin{equation}\label{eq:b10}
\Bigl\|\frac{\lambda_{\pm}|_{x_1=L}-\lambda_{\pm}|_{x_1=-L}}{\lambda_{\pm}|_{x_1=L}}\Bigr\|_{L^\infty_{(t, x_{2}, x_{3})}}\lesssim \frac{\varepsilon}{L\left(\log L\right)^2}. 
\end{equation}
The same estimates also holds for term $\bigl\|\frac{\lambda_+|_{x_2=L}-\lambda_+|_{x_2=-L}}{\lambda_+|_{x_2=L}}\bigr\|_{L^\infty_{(t, x_{1}, x_{3})}}$.

\smallskip 

{\bf Step 2. Estimate of $\delta_{+}(\lambda_{+}^{1}-\lambda_{+}^{2})=\bigl\|\frac{\lambda_{+}^{1}-\lambda_{+}^{2}}{\lambda_{+}^{2}}\bigr\|_{L^{\infty}(\overline{C}_{L}^{-})}$.}  For any fixed $(t, x) \in \overline{C}_{L}^{-}$, we have 
\beno 
x_{3}^{-}(t,x)=L\quad \text{and}\quad  u_-(t,x)=x_{3}^{-}(t,x_1,x_2,x_3-2L).
\eeno
 Lemma \ref{lem:e2} (2) then yields that $L\geq |x_{3}|>\frac{5L}{6}\geq\f{L}{3}$, which along with Lemma \ref{lem:e2} (1) implies that
\beno 
\left|u_{-}(t, x)\right|>\frac{L}{4}\quad \text{and}\quad \lambda_{+}^{2}(t,x)\gtrsim \left(\log L\right)^{4}.
\eeno

By virtue of \eqref{eq:g16}, we have
\begin{equation*}
x_{1}^{-}(t,x)=x_{1}^{-}(t,x_1,x_2,x_3-2L), \quad x_{2}^{-}(t,x)=x_{2}^{-}(t,x_1,x_2,x_3-2L) .  
\end{equation*}
Defining $d=\big(R^{2}+\left|x_{1}^{-}\right|^{2}+\left|x_{2}^{-}\right|^{2}\big)^{\frac{1}{2}}(t,x)$ and $h(r)=\Big(\log \left(d^{2}+r^{2}\right)^{\frac{1}{2}}\Big)^{4}$, using the definitions of $\lambda_{+}^{i}(t,x)$ (for $i=1,2$), we get
\begin{equation}
\begin{aligned}
\left|\lambda_{+}^{1}(t,x)- \lambda_{+}^{2}(t,x)\right|
&  = \left|h\big(x_{3}^{-}(t,x)\big)- h\big(x_{3}^{-}(t,x_1,x_2,x_3-2L)\big)\right| \\
&  = \left|h(L)- h\big(x_{3}^{-}(t,x_1,x_2,x_3-2L)\big)\right| \\
& = \left|h(-L)- h\big(x_{3}^{-}(t,x_1,x_2,x_3-2L)\big)\right| \\
& = \left|h^{\prime}(\eta)\right|\cdot \left|x_{3}^{-}(t,x_1,x_2,x_3-2L)-(-L)\right|,
\end{aligned} \label{eq:b11}
\end{equation}
where $\eta$ is a point between $x_{3}^{-}(t,x_1,x_2,x_3-2L)$ and $-L$. 

According to \eqref{eq:e8}, by using the ansatz \eqref{eq:a1}, \eqref{eq:a3} and Lemma \ref{lem:e2}, Lemma \ref{lem:e3}, we have
\beq\label{eq:b11a}
\begin{aligned}
\left|x_{3}^{-}(t,x)-(x_{3}+t)\right| 
& \leq\f32\int_{0}^{t} \left|z_{-}(\tau,x_1,x_2,x_3+t-\tau)\right|\,d\tau\\
& \lesssim \frac{\left(E_{-}+ E_{-}^0+ E_{-}^1\right)^{\frac{1}{2}}}{\left(\log L\right)^{2}}\cdot T^* \lesssim \frac{\varepsilon}{\log L} ,
\end{aligned}
\eeq
where we used the facts that $L\geq|x_3|>\f{5}{6}L$ and then $|x_3+t-\tau|\geq |x_3|-T^*> \frac{L}{3}$. Consequently, using the fact $x_{3}^{-}(t,x)=L$, we obtain
\beq\label{eq:b12}
\left|-L-(x_{3}+t-2L)\right|= \left|x_{3}^{-}(t,x)-2L-(x_{3}+t-2L)\right|
=\left|x_{3}^{-}(t,x)-(x_{3}+t)\right|\lesssim \frac{\varepsilon}{\log L},
\eeq
whicn implies that $x_{3}+t-2L\in[-\f98L,-\f78L]$.

Since $x_{3}+t-2L\in[-\f98L,-\f78L]$, we apply \eqref{eq:e8} with  $x_{3}$ replaced by $x_{3}-2L$. Following the same argument as in \eqref{eq:b11a}, we obtain
\begin{equation}\label{eq:b13}
\ \left|x_{3}^{-}(t,x_1,x_2,x_3-2L)-(x_{3}+t-2L)\right|
\lesssim \frac{\varepsilon}{\log L} .
\end{equation}

Combining \eqref{eq:b12} and \eqref{eq:b13}, we obtain
\begin{equation}\label{eq:b14}
\left|x_{3}^{-}(t,x_1,x_2,x_3-2L)-(-L)\right|\lesssim \frac{\varepsilon}{\log L},
\end{equation}
which yields that $x_{3}^{-}(t,x_1,x_2,x_3-2L)\in[-\f98L,-\f78L]$. 
Since $\eta$ is between $x_{3}^{-}(t,x_1,x_2,x_3-2L)$ and $-L$, then  $|\eta| \sim L$ and thus $\left|h^{\prime}(\eta)\right| \sim \frac{\left(\log L\right)^{3}}{L}$. Thus, we deduce from \eqref{eq:b11} that
\beno
\left|\lambda_{+}^{1}(t,x)- \lambda_{+}^{2}(t,x)\right|\lesssim \frac{\left(\log L\right)^{2}}{L} \varepsilon,
\eeno
which along with $\lambda_{+}^{2}(t,x)\gtrsim \left(\log L\right)^{4}$ implies that
\beno
\Bigl|\frac{\lambda_{+}^{1}(t,x)-\lambda_{+}^{2}(t,x)}{\lambda_{+}^{2}(t,x)}\Bigr|\lesssim \frac{\varepsilon}{L\left(\log L\right)^{2}}\quad \text{for all } (t,x)\in \overline{C}_{L}^{-}. 
\eeno
Thus, we get
\beq\label{eq:b16}
\delta_{+}(\lambda_{+}^{1}-\lambda_{+}^{2})\lesssim \frac{\varepsilon}{L\left(\log L\right)^{2}}.
\eeq

\smallskip

Combining \eqref{eq:b10} and \eqref{eq:b16}, we obtain the first part of \eqref{eq:b1} for $\lambda_+=\left(\log \langle w_{-}\rangle\right)^4$.

\medskip 

The second part of \eqref{eq:b1} for $\lambda_{+}=\langle w_{-}\rangle^2\left(\log \langle w_{-}\rangle\right)^4$ and $\langle w_{-}\rangle\left(\log \langle w_{-}\rangle\right)^4$ follows similarly to the case for $\lambda_+=\left(\log \langle w_{-}\rangle\right)^4$. This completes the proof of the lemma.
\end{proof}

\subsection{Proof of Lemma \ref{prop:div-curl-apply}} In this subsection, we sketch the proof of Lemma \ref{prop:div-curl-apply}.

\begin{proof}[Proof of Lemma \ref{prop:div-curl-apply}]
We only verify \eqref{eq:div-curl *} for $\lambda(t,x)=\f{\langle w_-\rangle^2(\log \langle w_-\rangle)^4}{\langle w_+\rangle(\log \langle w_+\rangle)^2}$, since \eqref{eq:div-curl *} for $\lambda(t,x)=\f{\langle w_+\rangle^2(\log \langle w_+\rangle)^4}{\langle w_-\rangle(\log \langle w_-\rangle)^2}$ can be derived in a similar way. 

{\bf Step 1. Estimate of $\bigl\|\frac{\lambda|_{x_1=L}-\lambda|_{x_1=-L}}{\lambda|_{x_1=L}}\bigr\|_{L^\infty_{(t, x_{2}, x_{3})}}.$} 
Since $x_{3}^{\pm}(t, x_{1}, x_{2}, x_{3})=x_{3}^{\pm}(t, x_{1}+2L, x_{2}, x_{3})$ holds for all $(t, x_{1}, x_{2}, x_{3})\in \left[0, T^{*}\right] \times \mathbb{R}^{3}$, it suffices to prove that
\begin{equation}\label{eq:A1}
\Bigl\|\frac{\lambda_{i}|_{x_1=L}-\lambda_{i}|_{x_1=-L}}{\lambda_{i}|_{x_1=L}}\Bigr\|_{L^\infty([0,T^*]\times[-2L,2L]^2)}\lesssim \frac{\varepsilon}{L\log L}, \quad \text{for }\ i=1,2, 3.
\end{equation}

For any fixed $(t,x_2,x_3) \in [0,T^*] \times [-2L,2L]^2$, according to \eqref{eq:g16}, we have
\begin{equation}
x_{2}^{\pm}(t, L, x_{2}, x_{3})=x_{2}^{\pm}(t,-L, x_{2}, x_{3}), \quad x_{3}^{\pm}(t, L, x_{2}, x_{3})=x_{3}^{\pm}(t,-L, x_{2}, x_{3}) .  \label{eq:div-curl L,-L}
\end{equation}
Denoting by $b=\big(R^{2}+\left|x_{2}^{-}\right|^{2}+\left|x_{3}^{-}\right|^{2}\big)^{\frac{1}{2}}(t, L, x_{2}, x_{3})$ and $c=\big(R^{2}+\left|x_{2}^{+}\right|^{2}+\left|x_{3}^{+}\right|^{2}\big)^{\frac{1}{2}}(t, L, x_{2}, x_{3})$,  we define two functions 
\beno 
g(r)=\left(b^{2}+r^{2}\right)\Big(\log \left(b^{2}+r^{2}\right)^{\frac{1}{2}}\Big)^{4}\quad\text{and}\quad h(r)=\left(c^{2}+r^{2}\right)^{\frac{1}{2}}\Big(\log \left(c^{2}+r^{2}\right)^{\frac{1}{2}}\Big)^{2},
\eeno 
which along with the fact $\langle \overline{w_{\pm}}\rangle=\big(R^{2}+\left|x^{\pm}\right|^{2}\big)^{\frac{1}{2}}$ and \eqref{eq:div-curl L,-L} imply that
\begin{equation}\label{eq:A2}
\begin{aligned}
&\left|\lambda_{1}(t, L, x_{2}, x_{3})-\lambda_{1}(t,-L, x_{2}, x_{3})\right|  = \Bigl|\frac{g\big(x_{1}^{-}(t,L,x_{2},x_{3})\big)}{h\big(x_{1}^{+}(t,L,x_{2},x_{3})\big)}- \frac{g\big(x_{1}^{-}(t,-L,x_{2},x_{3})\big)}{h\big(x_{1}^{+}(t,-L,x_{2},x_{3})\big)}\Bigr| \\
&\quad \leq \Bigl|\frac{g\big(x_{1}^{-}(t,L,x_{2},x_{3})\big)}{h\big(x_{1}^{+}(t,L,x_{2},x_{3})\big)}- \frac{g(L)}{h(L)}\Bigr| + \Bigl|\frac{g(-L)}{h(-L)}- \frac{g\big(x_{1}^{-}(t,-L,x_{2},x_{3})\big)}{h\big(x_{1}^{+}(t,-L,x_{2},x_{3})\big)}\Bigr| 
\end{aligned}
\end{equation}

Thanks to \eqref{eq:e9a} in the proof of Lemma \ref{lem:e2}, we obtain by the similar argumet as \eqref{eq:b7} that
\beno
\left|x_{1}^{\pm}(t, L, x_{2}, x_{3})-L\right| \lesssim \frac{\varepsilon}{\log L}, \quad \left|x_{1}^{\pm}(t, -L, x_{2}, x_{3})-(-L)\right| \lesssim \frac{\varepsilon}{\log L}, 
\eeno
which gives rise to
\beno
\left|x_{1}^{+}(t,\pm L,x_{2},x_{3})\right|\sim L,\quad \left|x_{1}^{-}(t,\pm L,x_{2},x_{3})\right|\sim L,\quad
\left|h\big(x_{1}^{+}(t,\pm L,x_{2},x_{3})\big)\right|\gtrsim L\left(\log L\right)^{2},
\eeno 
\beno 
\text{and}\quad 
\lambda_{1}(t, L, x_{2}, x_{3})=\frac{g\big(x_{1}^{-}(t,L,x_{2},x_{3})\big)}{h\big(x_{1}^{+}(t,L,x_{2},x_{3})\big)}\gtrsim \frac{g(L)}{h(L)}\gtrsim L\left(\log L\right)^{2}.
\eeno

Following similar derivation as \eqref{eq:b9}, we deduce from \eqref{eq:A2} that
\begin{equation*}
\left|\lambda_{1}(t, L, x_{2}, x_{3})-\lambda_{1}(t,-L, x_{2}, x_{3})\right| \lesssim \Big(\frac{1}{L\left(\log L\right)^{2}}\cdot L\left(\log L\right)^{4}+ \left(\log L\right)^{2}\Big) \cdot \frac{\varepsilon}{\log L} \lesssim \varepsilon \log L . 
\end{equation*}
Thus, we get for all $(t, x_{2}, x_{3})\in [0, T^{*}]\times [-2L,2L]^2$,
\begin{equation}\label{eq:A3}
\Bigl|\frac{\lambda_{1}(t,L,x_2,x_3)-\lambda_{1}(t,-L,x_2,x_3)}{\lambda_{1}(t,L,x_2,x_3)}\Bigr|\lesssim \frac{\varepsilon}{L\log L}.
\end{equation}
This is exactly \eqref{eq:A1} for $i=1$. Similar arguments as \eqref{eq:A3} show that \eqref{eq:A1} also holds for $i=2,3$. 

Therefore, we obtain
\beq\label{eq:A4}
\Bigl\|\frac{\lambda|_{x_1=L}-\lambda|_{x_1=-L}}{\lambda|_{x_1=L}}\Bigr\|_{L^\infty_{(t, x_{2}, x_{3})}}\lesssim \frac{\varepsilon}{L\log L}.
\eeq 
The same estimate also holds for $\bigl\|\frac{\lambda|_{x_2=L}-\lambda|_{x_2=-L}}{\lambda|_{x_2=L}}\bigr\|_{L^\infty_{(t, x_{2}, x_{3})}}$. 

\smallskip 

{\bf Step 2. Estimate of $\bigl\|\frac{\lambda_{1}-\lambda_{2}}{\lambda_{2}}\bigr\|_{L^{\infty}(\overline{C}_{L}^{-})}$.} For any fixed $(t, x) \in \overline{C}_{L}^{-}$, we have
\beno\begin{aligned}
\Bigl|\frac{\lambda_{1}(t,x)-\lambda_{2}(t,x)}{\lambda_{2}(t,x)}\Bigr|&= \Bigl|\frac{\big(\langle \overline{w_{-}}\rangle^2\left(\log \langle \overline{w_{-}}\rangle\right)^4\big)(t,x)-\big(\langle \overline{w_{-}}\rangle^2\left(\log \langle \overline{w_{-}}\rangle\right)^4\big)(t,x_1,x_2,x_3-2L)}{\big(\langle \overline{w_{-}}\rangle^2\left(\log \langle \overline{w_{-}}\rangle\right)^4\big)(t,x_1,x_2,x_3-2L)}\Bigr|\\
&=\Bigl|\frac{\lambda_+^{1}(t,x)-\lambda_+^{2}(t,x)}{\lambda_+^{2}(t,x)}\Bigr|,
\end{aligned}
\eeno
where $\lambda_+^{1}(t,x)$ and $\lambda_+^{2}(t,x)$ are notations used in Lemma \ref{prop:b1}. 

Thus, by virtue of Lemma \ref{prop:b1} for $\lambda(t,x)=\langle w_-\rangle^2
(\log\langle w_-\rangle)^4$, we get
\beq\label{eq:A5}
\Bigl\|\frac{\lambda_{1}-\lambda_{2}}{\lambda_{2}}\Bigr\|_{L^{\infty}(\overline{C}_{L}^{-})}=\Bigl\|\frac{\lambda_+^{1}(t,x)-\lambda_+^{2}(t,x)}{\lambda_+^{2}(t,x)}\Bigr\|_{L^{\infty}(\overline{C}_{L}^{-})}\lesssim\frac{\varepsilon}{L\log L}.
\eeq 

\smallskip 

{\bf Step 3. Estimate of $\bigl\|\frac{\lambda_{1}-\lambda_{3}}{\lambda_{3}}\bigr\|_{L^{\infty}(\overline{C}_{-L}^{+})}$.}
For any fixed $(t, x) \in \overline{C}_{-L}^{+}$, we have
\beno
\Bigl|\frac{\lambda_{1}(t,x)-\lambda_{3}(t,x)}{\lambda_{3}(t,x)}\Bigr|= \Bigl|\frac{\big(\langle \overline{w_{+}}\rangle \left(\log \langle \overline{w_{+}}\rangle\right)^2\big)(t,x)-\big(\langle \overline{w_{+}}\rangle \left(\log \langle \overline{w_{+}}\rangle\right)^2\big)(t,x_1,x_2,x_3+2L)}{\big(\langle \overline{w_{+}}\rangle \left(\log \langle \overline{w_{+}}\rangle\right)^2\big)(t,x)}\Bigr|.
\eeno
Then similar derivation as \eqref{eq:b16} yields that
\beq\label{eq:A6}
\Bigl\|\frac{\lambda_{1}-\lambda_{3}}{\lambda_{3}}\Bigr\|_{L^{\infty}(\overline{C}_{-L}^{+})}\lesssim \frac{\varepsilon}{L\log L}.
\eeq

Combining \eqref{eq:A4},  \eqref{eq:A5} and  \eqref{eq:A6}, we finally obtain
 \eqref{eq:div-curl *} for $\lambda(t,x)=\f{\langle w_-\rangle^2(\log \langle w_-\rangle)^4}{\langle w_+\rangle(\log \langle w_+\rangle)^2}$. 
This completes the proof of the lemma.
\end{proof}

\subsection{A lemma on weight functions.} This subsection presents a lemma concerning the weight functions, which allows the transfer of dependence from the $x$-variables to the $y$-variables.
\begin{lemma}\label{lem:p1}
Let $R\geq 100$. We have
\begin{itemize}
\item[(1).] for any $x,y\in Q_L,\, |x-y|\leq 2$, there hold
\begin{equation}\label{eq:p4}
\langle w_{\pm}\rangle(\tau, x)\leq 12\langle w_{\pm}\rangle(\tau, y),\quad \log \langle w_{\pm}\rangle(\tau, x)\leq 2\log \langle w_{\pm}\rangle(\tau, y). 
\end{equation}  
\item[(2).] for any $x,y\in Q_L,\,|x-y|\geq 1$, there hold
\beq\label{eq:p5}\begin{aligned}
&\langle w_{\pm}\rangle(\tau, x) \leq 12\langle w_{\pm}\rangle(\tau, y)+6|x-y|\leq 13|x-y|\cdot\langle w_{\pm}\rangle(\tau, y),\\ 
&
\log \langle w_{\pm}\rangle(\tau, x) \leq 2\log (13|x-y|)\cdot \log \langle w_{\pm}\rangle(\tau, y). 
\end{aligned}\eeq
\item[(3).] for any $x,y\in Q_L,\,|x-y|\geq 1$, there hold
\beq\label{eq:p5-1}
\langle w_{\pm}\rangle(\log \langle w_{\pm}\rangle)^2(\tau, x)\leq48\langle w_{\pm}\rangle\bigl(\log\langle w_{\pm}\rangle\bigr)^2(\tau,y)+12|x-y|\bigl(\log (12|x-y|)\bigr)^2.
\eeq
\end{itemize}
\end{lemma}
\begin{proof}
{\bf (1). Proof of \eqref{eq:p4}.} For any $x,y\in Q_L$ with $|x-y|\leq 2$, by virtue of Lemma \ref{lem:e2}, we have
\begin{itemize}
\item if $|y_3|\geq \frac{L}{3}$, then $|u_{\pm}(\tau,y)|\geq \frac{L}{4}$, which implies
\begin{equation*}
\langle w_{\pm}\rangle(\tau, x) \leq 3L\leq 12|u_{\pm}(\tau,y)|\leq 12\langle w_{\pm}\rangle(\tau, y);
\end{equation*}
\item if $|y_3|< \frac{L}{3}$, then $|x_3|\leq|y_3|+|x_3-y_3|\overset{L\gg1}{\leq} \frac{L}{2}$, from which and Lemma \ref{lem:e2} shows that 
\beno 
u_{\pm}(\tau,x)=x_{3}^{\pm}(\tau,x),\quad u_{\pm}(\tau,y)=x_{3}^{\pm}(\tau,y).
\eeno 
Thanks to \eqref{eq:e5} and the differential mean value theorem, we have
\beno
\begin{aligned}
\langle w_{\pm}\rangle(\tau, x) = \langle \overline{w_{\pm}}\rangle(\tau, x)&\leq \langle \overline{w_{\pm}}\rangle(\tau, y)+ |x-y|\cdot \left\|\nabla \langle \overline{w_{\pm}}\rangle\right\|_{L^{\infty}}\\ 
&\leq \langle \overline{w_{\pm}}\rangle(\tau, y)+ 4
 \leq 2\langle \overline{w_{\pm}}\rangle(\tau, y)= 2\langle w_{\pm}\rangle(\tau, y).
\end{aligned}
\eeno
\end{itemize}

Thus, based on the above analysis, we obtain the first inequality in \eqref{eq:p4}. Using this inequality, we obtain
\begin{equation*}
\log \langle w_{\pm}\rangle(\tau, x)\leq \log 12 + \log \langle w_{\pm}\rangle(\tau, y) \overset{R\geq100}{\leq} 2\log \langle w_{\pm}\rangle(\tau, y).
\end{equation*} 
This is the second inequality in \eqref{eq:p4}.

\medskip 

{\bf (2). Proof of \eqref{eq:p5}.}  For any $x,y\in Q_L$ with $|x-y|\geq 1$, by Lemma \ref{lem:e2},  we have
\begin{itemize}
\item if $|y_3|\geq \frac{L}{3}$, then $|u_{\pm}(\tau,y)|\geq \frac{L}{4}$, which gives rise to
\begin{equation*}
\langle w_{\pm}\rangle(\tau, x) \leq 3L\leq 12|u_{\pm}(\tau,y)|\leq 12\langle w_{\pm}\rangle(\tau, y);
\end{equation*}
\item if $|x_3|\geq\frac{5L}{6}$ and $|y_3|< \frac{L}{3}$, then $|x-y|\geq |x_3-y_3|> \frac{L}{2}$, and
\begin{equation*}
\langle w_{\pm}\rangle(\tau, x) \leq 3L\leq 6|x-y|;
\end{equation*}
\item if $|x_3|<\frac{5L}{6}$ and $|y_3|< \frac{L}{3}$, then  $u_{\pm}(\tau,x)=x_{3}^{\pm}(\tau,x)$, $u_{\pm}(\tau,y)=x_{3}^{\pm}(\tau,y)$, thereby $\langle w_{\pm}\rangle(\tau, x)=\langle \overline{w_{\pm}}\rangle(\tau, x)$, $\langle w_{\pm}\rangle(\tau, y)=\langle \overline{w_{\pm}}\rangle(\tau, y)$. In view of \eqref{eq:e5} and the mean value theorem, we have
\begin{equation*}
\langle w_{\pm}\rangle(\tau, x) \leq \langle w_{\pm}\rangle(\tau, y)+2|x-y|.
\end{equation*}
\end{itemize}

Therefore, based on the above analysis and $|x-y|\geq 1$, we conclude that
\begin{equation*}
\langle w_{\pm}\rangle(\tau, x) \leq 12\langle w_{\pm}\rangle(\tau, y)+6|x-y|\overset{R\geq 100}{\leq} 13|x-y|\langle w_{\pm}\rangle(\tau, y).
\end{equation*}
This is the first inequality in \eqref{eq:p5}. Using this inequality, we get
\begin{equation*}
\log \langle w_{\pm}\rangle(\tau, x)\leq \log (13|x-y|)+ \log \langle w_{\pm}\rangle(\tau, y) \leq 2\log (13|x-y|) \log \langle w_{\pm}\rangle(\tau, y),
\end{equation*} 
which is the second inequality in \eqref{eq:p5}. 

\medskip

{\bf (3). Proof of \eqref{eq:p5-1}.} Since \eqref{eq:p5} shows
$
\langle w_{\pm}\rangle(\tau, x) \leq 12\langle w_{\pm}\rangle(\tau, y)+6|x-y|$, it follows that
\beno 
\text{either }\  \langle w_{\pm}\rangle(\tau, x) \leq 24\langle w_{\pm}\rangle(\tau, y),\quad  \text{or }\  \langle w_{\pm}\rangle(\tau, x) \leq12|x-y|.
\eeno 
Therefore, we obtain  $\langle w_{\pm}\rangle(\log \langle w_{\pm}\rangle)^2(\tau, x)$ is bounded by 
\beno 
\text{either }\  24\langle w_{\pm}\rangle\bigl(\log(24\langle w_{\pm}\rangle)\bigr)^2(\tau, y),\quad  \text{or }\ 12|x-y|\bigl(\log(12|x-y|)\bigr)^2,
\eeno 
which along with $R\geq 100$ gives rise to \eqref{eq:p5-1}. It completes the proof of the lemma.
\end{proof}

\vspace{0.3cm}

 \noindent{\bf Data of availability statement.} Our manuscirpt has no associated data.

\vspace{0.3cm}
\noindent {\bf Acknowledgments.}
 The work of the first author was partially supported by NSF of China under grant 12171019. 

 \vspace{0.3cm}

 \noindent{\bf Conflict of interest statement.} All authors declare that they have no competing financial interests or personal relationships that could have appeared to influence the work reported in this paper.

\end{document}